\documentclass[11pt,a4paper]{amsart}

\usepackage{tikz}
\usetikzlibrary{arrows.meta,decorations.pathmorphing,decorations.markings,backgrounds,positioning,fit,shapes}

\usepackage[all,cmtip]{xy}

\usepackage{hyperref}
\usepackage[capitalize,nameinlink]{cleveref}

\usepackage[margin=1in]{geometry}   
\usepackage{graphicx}               
\usepackage{amsmath,euscript,comment}                
\usepackage{amsfonts}               
\usepackage{mathtools,amssymb,amscd,epsf,amsthm, epsfig,bm}  

\usepackage{enumerate}               
\usepackage[toc,page]{appendix}
\usepackage{amsmath,calligra,mathrsfs,xcolor}
\usepackage{float}
\usepackage[symbol]{footmisc}
\usepackage{xcolor}

\tikzset{>={Stealth[scale=1.2]}}
\tikzset{->-/.style={decoration={
			markings,
			mark=at position #1 with {\arrow{>}}},postaction={decorate}}}
\tikzset{-w-/.style={decoration={
			markings,
			mark=at position #1 with {\arrow{Stealth[fill=white,scale=1.4]}}},postaction={decorate}}}
\tikzset{->-/.default=0.65}
\tikzset{-w-/.default=0.65}
\tikzstyle{bullet}=[circle,fill=black,inner sep=0.5mm]
\tikzstyle{circ}=[circle,draw=black,fill=white,inner sep=0.5mm]
\tikzstyle{vertex}=[circle,draw=black,thick,inner sep=0.5mm]
\tikzset{darrow/.style={double distance = 4pt,>={Implies},->},
	darrowthin/.style={double equal sign distance,>={Implies},->},
	tarrow/.style={-,preaction={draw,darrow}},
	qarrow/.style={preaction={draw,darrow,shorten >=0pt},shorten >=1pt,-,double,double
		distance=0.2pt}}
\newcommand{\tikzfig}[1]{\begin{tikzpicture}[auto,baseline={([yshift=-.5ex]current bounding box.center)}]#1\end{tikzpicture}}
\tikzset{partial ellipse/.style args={#1:#2:#3}{insert path={+ (#1:#3) arc (#1:#2:#3)}}}

\newtheorem{theorem}{Theorem}[section]
\newtheorem{lemma}[theorem]{Lemma}
\newtheorem{proposition}[theorem]{Proposition}
\newtheorem{corollary}[theorem]{Corollary}

\newtheorem{lemma-definition}[theorem]{Lemma-Definition}

\theoremstyle{definition}
\newtheorem{definition}[theorem]{Definition}
\newtheorem{example}[theorem]{Example}

\newtheorem{remark}[theorem]{Remark}

\numberwithin{theorem}{section}

\newcommand{\R}{\mathbb{R}}

\newcommand{\F}{\mathbb{F}}
\newcommand{\Z}{\mathbb{Z}}
\newcommand{\LL}{\mathbb{L}}
\newcommand{\N}{\mathbb{N}}

\newcommand{\calV}{\mathcal{V}}

\newcommand{\kk}{\Bbbk}

\newcommand{\A}{\mathop{\overline{A}}\nolimits}

\newcommand{\Hom}{\mathop{\mathrm{Hom}}\nolimits}

\newcommand{\Cone}{\mathop{\mathrm{Cone}}\nolimits}

\newcommand{\id}{\mathop{\mathrm{id}}\nolimits}

\newcommand{\Tor}{\mathop{\mathrm{Tor}}\nolimits}

\newcommand{\HH}{\mathop{\mathrm{HH}}\nolimits}

\newcommand{\del}{\partial}

\newcommand{\Ahatinfty}{\widehat{A}_\infty}
\newcommand{\ev}{\mathop{\mathrm{ev}}\nolimits}
\newcommand{\co}{\mathop{\mathrm{co}}\nolimits}
\newcommand{\pt}{\mathop{\mathrm{pt}}\nolimits}

\mathchardef\mh="2D

\input xy
\xyoption{all}

\usepackage{lipsum}

\usepackage{hyperref,url}
\usepackage[backend=biber, maxbibnames=99, bibencoding=utf8, doi=false, isbn=false, url=false, style=alphabetic]{biblatex}
\DeclareRobustCommand{\SkipTocEntry}[5]{}

\begin{document}

\title{Algebraic string topology from the neighborhood of infinity}
\author{Manuel Rivera, Alex Takeda \and Zhengfang Wang}

\newcommand{\Addresses}{{ 
  \bigskip
  \footnotesize

  M.~Rivera, \textsc{Purdue University, Department of Mathematics, 150 N. University St. West Lafayette, IN 47907}\par\nopagebreak
  \textit{E-mail address}: \texttt{manuelr@purdue.edu}
  
   \medskip
   A.~Takeda, \textsc{Uppsala University, Matematiska Institutionen, L\"{a}gerhyddsv\"{a}gen 1, 752 37 Uppsala, Sweden}\par\nopagebreak
   \textit{E-mail address}: \texttt{alex.takeda@math.uu.se}

   \medskip 
   Z.~Wang, \textsc{Department of Mathematics, Nanjing University, Nanjing 210093, Jiangsu, PR China
}\par\nopagebreak
  \textit{E-mail address}: \texttt{zhengfangw@gmail.com}

 }}

\begin{abstract}
    We construct and study an algebraic analogue of the loop coproduct in string topology, also known as the Goresky-Hingston coproduct. Our algebraic setup, which under this analogy takes the place of the complex of chains on the free loop space of a possibly non-simply connected oriented manifold, is the Hochschild chain complex of a smooth $A_{\infty}$-category equipped with a pre-Calabi-Yau structure and a trivialization of a version of the Chern character of its diagonal bimodule.  The algebraic analogue of the loop coproduct is part of a more general mapping cone construction, which we describe in terms of the categorical formal punctured neighborhood of infinity associated to the underlying smooth $A_\infty$-category. We use a graphical formalism for $A_\infty$-categories and bimodules to describe explicit models for the operations and homotopies involved. We also compute explicitly the algebraic coproduct in the context of string topology of spheres.
\end{abstract}

\maketitle

\vspace{-0.5cm}
\tableofcontents

\section{Introduction}
String topology is the study of algebraic structures defined in terms of intersecting, cutting, and reconnecting families (or more precisely, \textit{chains}) of strings or loops on a manifold \cite{CS99}. One of the principal themes in string topology since its inception has been the characterization of the algebraic topology of manifolds through chain-level manifestations of Poincaré duality. In this article, we construct and study algebraic operations on the Hochschild chains of a smooth $A_\infty$-category equipped with some additional structure, inspired by the constructions in string topology, with emphasis on an algebraic analogue of the \textit{Goresky-Hingston (loop) coproduct}. We propose a construction that combines algebraic analogues for the Goresky-Hingston loop coproduct and Chas-Sullivan loop product. Our construction is inspired by the \textit{categorical formal punctured neighborhood of infinity}, a non-commutative analogue of the category of perfect complexes on the formal punctured neighborhood of the divisor at infinity for a compactification of a smooth variety \cite{efimov2017categorical}.

The Goresky-Hingston loop coproduct was originally defined geometrically as an operation on the homology of the free loop space of a manifold relative to constant loops \cite{GorHin}, \cite{HinWah0}. It is a type of secondary invariant and its construction requires choices that are more subtle than those involved in other string topology operations such as the Chas-Sullivan loop product. In fact, the Goresky-Hingston loop coproduct, in contrast to the Chas-Sullivan loop product, is sensitive to structure beyond the homotopy type of the underlying manifold \cite{Nae21} \cite{naef2023string}, \cite{NaeWil19}. The algebraic formalism proposed in the present article makes explicit the choices required to construct such a coproduct and provides a framework to analyze its properties and compatibilities with other operations.  

\subsection{Motivation} Let $M$ be a closed oriented manifold of dimension $n$. One of the first operations considered in string topology was the \textit{loop product}
\[\wedge \colon H_*(LM) \otimes H_*(LM) \to H_*(LM)\]
(Chas-Sullivan product in the literature) \cite{CS99}, where $LM = \mathrm{Map}(S^1, M)$ is the free loop space of $M$. This is a degree $-n$ product combining the intersection product on the chains on $M$ with the concatenation product of loops with the same base point. The circle action on $LM$ given by rotation of loops gives rise to an operator $B \colon H_*(LM)\to H_{*+1}(LM)$. One of the first theorems in string topology is that $(H_*(LM)[-n], \wedge, B)$ satisfies the axioms of a $BV$-\textit{algebra} structure. This algebraic structure (as well as chain level lifts taking the form of a framed $E_2$-algebra) has been constructed rigorously through different perspectives  \cite{CohJon},\cite{Iri17}, \cite{Tr02}, \cite{TrZe06}, \cite{TraZei07b}, \cite{WahWes08}, \cite{kontsevich2021precalabiyau}, \cite{BR23}, \cite{HinWah0}. The loop product has shown to be an invariant of the oriented homotopy type of the underlying manifold \cite{CKS}.

Another major string topology operation is the \textit{loop coproduct} 
\[\vee \colon H_*(LM,M) \to H_*(LM,M) \otimes H_*(LM,M)\] (Goresky-Hingston coproduct in the literature) \cite{GorHin}, \cite{HinWah0}. This is a degree $1-n$ coproduct defined by considering a $1$-parameter family of self intersections in a single chain of loops and then splitting at the points of intersection to obtain a formal sum of pairs of chains of loops. In order to obtain a well-defined operation on homology, one must work relative to constant loops or relative to a choice of base point. Constructing the loop coproduct rigorously, understanding all the necessary choices, and analyzing its properties and compatibilities with other operations has proven to be a delicate endeavor.  

The first and third named authors of this paper constructed in \cite{RivWan19} a framework combining algebraic models for the loop product and loop coproduct into a single algebra structure on the mapping cone of a map connecting the Hochschild chain and cochain complexes of a connected differential graded (dg) Frobenius algebra. The construction was coined the \textit{Tate-Hochschild complex} since it resembles Tate cohomology of finite groups. One can derive an explicit formula for the algebraic analogue of the loop coproduct in this context, which agrees with a formula given previously by \cite{Abb16}. The Tate-Hochschild construction has a counterpart in symplectic topology known as \textit{Rabinowitz-Floer homology}, a theory combining symplectic homology and cohomology of symplectic manifolds \cite{CieHinOan2}, \cite{CieHinOan1}, \cite{ganatra2022rabinowitz}

The setting of Frobenius algebras (or proper Calabi Yau algebras) is suitable for studying string topology in the \textit{simply connected} setting: the Tate-Hochschild cohomology perspective was used to deduce that, working over rational coefficients, two homotopy-equivalent simply connected closed manifolds have isomorphic Goresky-Hingston loop coalgebras \cite{RiWa22}.  
In the \textit{non-simply connected} setting, F.\ Naef observed that the loop coproduct can distinguish homotopy-equivalent but non-homeomorphic lens spaces \cite{Nae21}. This was strengthened in \cite{naef2023string} by proving two lens spaces $L_{p,q}$ and $L_{p,q'}$ are homeomorphic if and only if their Goresky-Hingston coalgebras are isomorphic. These results realized one of the original goals of string topology of constructing operations that can detect more geometric information than just the oriented homotopy type of the underlying manifold. 

The motivation for the present article is to devise an algebraic framework for string topology that 1.\ allows to keep track of the choices necessary to construct chain level operations for non-simply connected manifolds, 2.\ makes transparent the dependency of the string operations on the underlying manifold, and 3.\ is suitable for describing explicit formulas, compatibilities between operations, and carrying out computations. With this motivation in mind, we describe a construction that we believe to be a Koszul dual version of the Tate-Hochschild complex; like the latter, this construction combines algebraic models for the loop product and loop coproduct, but is now applicable to the setting of non-simply connected manifolds. 

We model the complex of chains on the free loop space of a closed manifold as the Hochschild chain complex of a smooth dg (or $A_{\infty}$) category $A$ playing the role of the dg category of paths on the underlying manifold. To recover the relevant string operations, the category $A$ is equipped with the additional data of a \textit{pre-Calabi Yau} structure manifesting chain level Poincaré duality together with certain ``trivialization" of (a version of) the Chern character of $A$. This approach is grounded on combining the perspectives and results of \cite{Goo85}, \cite{rivera2024algebraic}, \cite{kontsevich2021precalabiyau}, \cite{RivWan19}, and \cite{efimov2017categorical}. In a subsequent article, we will describe how our algebraic formalism may be used to construct string topology operations of non-simply connected manifolds directly from a triangulation.

\subsection{Summary of results}
We briefly summarize our main constructions and results while describing how the paper is organized. We outline the construction of the \textit{algebraic loop coproduct}, which is described in detail in \cref{sec:chernCharAndCoprod} building upon the graphical formalism recalled in \cref{sec:graphicalCalculus}.  Suppose $A$ is a connective (i.e.\ non-negatively graded) smooth dg category over a ring $\kk$. We also suppose that all its morphism complexes are cofibrant over $\kk$; this condition, which holds automatically when $\kk$ is a field, is called $\kk$-cofibrancy\footnote{Not to be confused with cofibrancy in any particular model structure for the category of dg categories.} in \cite{lowen2005hochschild} (see also \cite{keller2003derived}), and implies that the Hochschild complexes we write compute the appropriate derived tensors and homs. We will assume all our dg categories are $\kk$-cofibrant throughout the whole article.

Denote $A^e=A \otimes A^{op}$ and let $A^!$ be any $A$-bimodule modeling the inverse dualizing bimodule complex $\R\! \Hom_{A^e}(A, A^e)$. For any $A$-bimodule $N$, denote by $C_*(A, N)$ and $C^*(A, N)$ the Hochschild chain and cochain complexes of $A$ with values on $N$, respectively, and by $D^*(A \otimes A^{!}, A^e)$ the chain complex calculating $\R\!\Hom_{A^e \otimes A^e} (A \otimes A^!, A^e)$ obtained through appropriate bar resolutions. Smoothness implies the existence of natural quasi-isomorphisms
\[C_*(A, A^{!}) \xrightarrow{\simeq} C^*(A,A) \xleftarrow{\simeq} D^*(A \otimes A^{!}, A^e).\]
Recall that $C_*(A, A^{!})$ models the derived tensor product $A \otimes^{\mathbb{L}}_{A^e} A^{!}$ and $C^*(A,A)$ models the derived mapping complex $\mathbb{R}\!\Hom_{A^e}(A, A)$. Choose cycles $\co \in C_*(A, A^{!})$ and $\eta \in D^*(A \otimes A^{!}, A^e)$ representing the cohomology class of the identity (bimodule) morphism 
\[ [\mathrm{id}_A] \in H^0(\R\! \Hom_{A^e}(A, A)) \cong H^0(C^*(A,A))\]
under the above quasi-isomorphisms. 

We have an evaluation map
\[e \colon D^*(A \otimes A^{!}, A^e)  \otimes C_*(A, A^{!}) \to C_*(A,A) \otimes C_*(A,A),\]
with which we make the following definition.
\begin{definition}(\cref{def:ChernCharacter})
	The \emph{chain-level Chern character (of the diagonal bimodule of)} $A$ is the element
	\[ E = e(\eta, \co) \in C_*(A,A) \otimes C_*(A,A).\]
\end{definition}
This definition is a special case of a more general definition that makes sense for any perfect $A$-bimodule. Though $E$ depends on the representatives we chose for $\co$ and $\eta$, its homology class $[E] \in HH_*(A,A)\otimes HH_*(A,A)$ is well defined and may be thought of as a version of \textit{Chern character}, or of the \textit{Hattori-Stallings trace} \cite{Hat65}, for smooth dg categories. 

Using $\co$ and $\eta$ one can also define a map of graded complexes, see \cref{sec:coproduct}
\[G \colon C^*(A,A) \to C_*(A,A) \otimes C_*(A,A)[-1].\]
This is not in general a map of complexes; instead it is a chain homotopy between capping with either factor of $E$. That is, $G$ satisfies
\[ [d,G]\varphi = (\varphi \frown E') \otimes E'' - (-1)^{\deg(\varphi)\deg(E')} E'\otimes (\varphi\frown E'') \]
where $\varphi \in C^*(A,A)$, $\frown \colon C^*(A,A) \otimes C_*(A,A) \to C_*(A,A)$ denotes the classical cap product between Hochschild cochains and chains, and we have written $E= \sum E' \otimes E''=E' \otimes E''$. One can correct $G$ to a map of complexes, by choosing some extra data, as follows.
\begin{definition}(\cref{def:trivializationEnonzero})
    Suppose $W$ is some subcomplex of $C_*(A, A)$. A \emph{trivialization of $E$ onto $W$} is a pair $(E_0,H)$, where
	\[ E_0 \in W \otimes C_*(A,A) \cap C_*(A,A) \otimes W \]
	and $H \in C_*(A,A) \otimes C_*(A,A)$, such that $dH = E- E_0$ in $C_*(A,A) \otimes C_*(A,A)$.
\end{definition}

Given a trivialization of $E$ onto $W$ as above, let us denote $\overline{C}_*(A, A) = C_*(A,A)/W$ and define
\[ \widetilde{\lambda}_H  \colon C^*(A,A) \to \overline{C}_*(A,A) \otimes \overline{C}_*(A,A)[-1] \]
by the formula
\[ \widetilde\lambda_H(\varphi) = G(\varphi) - (-1)^{\deg(\varphi)}((\varphi \frown H')\otimes H'' - (-1)^{\deg(\varphi)\deg(H')} (H' \otimes (\varphi \frown H'')). \]
This is a map of complexes, and induces a map on homology
\[ \widetilde{\lambda}_H  \colon HH^*(A,A) \to H_*(\overline{C}_*(A,A) \otimes \overline{C}_*(A,A))[-1]. \]
We note that this map is defined independent on having any sort of duality structure on $A$. We now include such a duality structure, in the form of a \textit{pre-Calabi-Yau} (pre-CY) structure of dimension $n$. The formalism of pre-CY structures was developed in \cite{kontsevich2021precalabiyau} and will be recalled in \cref{subsec:preCY}. The data of a pre-CY structure of dimension $n$, in particular, gives rise to a map of $A$-bimodules $\alpha \colon A \to A^![n]$ and a map of complexes
\[ g_\alpha \colon C_*(A,A) \to C^*(A,A)[n] \]
\begin{definition}(\cref{def:chainAlgebraicLoopCoproduct})
	The \emph{chain-level loop coproduct} associated to $(A,\co,\eta,H,\alpha)$ is the map of complexes
	\[ \lambda_H \colon C_*(A,A)  \to \overline{C}_*(A,A) \otimes \overline{C}_*(A,A)[n-1] \] 
	given by the composition $\lambda_H= \widetilde{\lambda}_H \circ g_{\alpha}$. We denote equally by $\lambda_H$ the operation induced on homology
	\[ \lambda_H: HH_*(A,A) \to H_*(\overline{C}_*(A,A) \otimes \overline{C}_*(A,A))[n-1] \]
	which we call the \emph{homology loop coproduct}.
\end{definition}

In the special case $[E]=0$, one can choose $W=0$ to get a coproduct on the whole of $HH_*(A,A)$. The pre-CY structure on $A$ also defines a product $\pi$ on this complex, part of a framed $E_2$-structure, which is an algebraic model of the Chas-Sullivan loop coproduct. We show in \cref{thm:sullivan} that, for any choice of $H$, the homology coproduct $\lambda_H$ and homology loop product $\pi$ satisfy the \textit{infinitesimal bialgebra} relation, that is, $\lambda_H$ satisfies a version of the Leibniz rule with respect to $\pi$.

\vspace{1cm}

Without further assumptions, $\lambda_H$ may not have other nice properties one encounters in coproducts coming from string topology or Floer theory, namely, coassociativity and cocommutativity. A great portion of this paper is dedicated to studying what conditions one needs in order to guarantee that the homology loop coproduct has these properties.

We start by imposing a non-degeneracy condition on the pre-CY structure: such a structure is non-degenerate if $\alpha$ is a quasi-isomorphism of $A$-bimodules.  The literature on pre-CY structures has been developed over a field of characteristic zero; in that case, for example, non-degenerate pre-CY structures can be produced on any $A$ endowed with a smooth Calabi-Yau structure, using the procedure of \cite{KTV23} (see also \cite{yeung2018pre,pridham2017shifted}). Nonetheless, the definition of pre-CY structures, and the nondegeneracy condition, still make sense over any ring $\kk$.

We would also like to pick the subcomplex $W$ appropriately; we have some freedom of choice, depending on what quotient complex we want to define the homology coproduct on, but some of those choices will lead to better-behaved coproducts. 
\begin{definition}(\cref{def:balancedText})\label{def:balancedIntro}
	The triple $(W,H,\alpha)$ is \emph{balanced} if the image of $W$ under the chain-level map $\lambda_H$ is contained in $W \subset C_*(A,A) \cap C_*(A,A) \otimes W$, implying that $\lambda_H$ descends to a map of complexes
    \[ \overline{\lambda}_H \colon \overline{C}_*(A,A) \to \overline{C}_*(A,A) \otimes \overline{C}_*(A,A)[n-1]. \]
\end{definition}
If $[E]=0$, among the possible choices we have $W=0$, which always gives a balanced triple. Moreover, if $A$ is supported in non-negative homological degrees, and $H$ is a trivialization onto a subcomplex $W_0$ concentrated in degree zero, as long as the pre-CY dimension satisfies $n \ge 2$ we always have a balanced triple (\cref{prop:coproductFactors}).\\

The space of choices for the trivialization $H$, modulo exact terms, has the structure of a torsor over $H_1(C_*(A,A) \otimes C_*(A,A))$. The resulting coproduct $\lambda_H$ is not natural with respect to maps of smooth dg categories and depends on the choice of $H$.  There is a certain symmetry condition on the space of such choices for $H$, depending on $\alpha,\co,\eta$ and our choice of $W$; we say that $H$ is $\alpha$-\emph{symmetric modulo} $W$ if it satisfies the condition in \cref{def:appropriatelySymmetric}. 

This symmetry condition simplifies in some examples (including all the ones we study in this paper) to $H$ being $(-1)^n$-symmetric under the $\Z/2$ action swapping the two $C_*(A, A)$ factors, up to exact terms and terms in $W\otimes C_*(A,A)\cap C_*(A,A)\otimes W)$. We then study the resulting coproduct $\lambda_H$ under this symmetry assumption. Summarizing our results, we find:
\begin{theorem}[see \cref{thm:commCoproduct}] \label{thm:thm1intro}
    Let $A$ be a smooth connective dg (or $A_{\infty}$) category over a field $\kk$ equipped with a nondegenerate pre-Calabi Yau structure of dimension $n \geq 3$. Suppose $(E_0,H)$ is a trivialization of $E$ onto a subcomplex $W$ such that $(W,H,\alpha)$ is balanced. If $H$ is $\alpha$-symmetric modulo $W$, then $(H_*(\overline{C}_*(A, A))[n-1], \lambda_H)$ is a graded cocommutative and coassociative coalgebra.
\end{theorem}

Let us mention the relation to string topology, which we are developing in subsequent articles (\cite{RivTak} and follow-ups). Let $M$ be an oriented closed manifold with a triangulation (that is, a `combinatorial manifold'); this defines a certain dg category $A$, whose morphism spaces model the space of chains of $\kk$-modules on path spaces. This dg category is quasi-equivalent to $C^\mathrm{sing}_*(\mathcal{P}M;\kk)$, the dg category of singular chains on the topological category of (Moore) paths in $M$, and its Hochschild chains $C_*(A, A)$ is quasi-isomorphic to $\mathcal{C}_*(LM;\kk)$, the chains on the free loop space $LM$ of $M$ \cite{Goo85}.

Using the simplicial structure on $M$ given by the triangulation, in \cite{RivTak} we give explicit representatives for the $\co,\eta$ vertices, and therefore for $E$, whose class in $HH_*(A,A)^{\otimes 2} \cong H_*(LM,\kk)^{\otimes 2}$ agrees with $\chi(M) [\pt_0 \times \pt_0]$, that is, the Euler characteristic of $M$ times the square of the class of a point in the `identity connected component' of $LM$. We have some freedom in choosing the subcomplex $W$; among possibly other options, we have:
\begin{enumerate}
	\item take $W$ to be a subcomplex modeling chains on the image of the inclusion of constant loops $M \hookrightarrow LM$, or
	\item take $W$ to be the submodule of $C_0(A, A)$ generated by $\chi(M)\cdot \id_{b}$, for some choice of base point $b \in M$.
\end{enumerate}
For option (1), we will show in \emph{op.cit.} that for an appropriate choice of $H$ (in fact, $H =0$), the triple $(H,W,\alpha)$ is balanced and $H$ is $\alpha$-symmetric modulo $W$, and that the resulting coproduct $\lambda_H$ agrees on reduced homology with the geometrically defined Goresky-Hingston coproduct. Furthermore, we conjecture this coproduct defines a $BV$-\textit{coalgebra} with respect to the operator given by rotation of loops and may be lifted to a \textit{framed $E_2$-coalgebra} structure at the level of reduced Hochschild chains. 

On the other hand, a trivialization onto the choice of $W$ for option (2) involves more choices; the resulting coproduct should be thought of as the lift of the Goresky-Hingston coproduct to $H_*(LM)/\chi(M)\cdot [\pt]$ using some vector field with a single zero at a point $\pt \in M$; the data of this vector field is then encoded in the trivialization $H$.

\subsection{Outline of the proof}
\cref{thm:thm1intro} is obtained by analyzing the relationship (see \cref{thm:thm3intro} (3) below) between the coproduct $\lambda_H$ and a product $\pi_H$ on a dual complex. This product is obtained by lifting a product defined on a mapping cone construction. As a result, the coassociativity (resp.\ cocommutativity) of $\lambda_H$ is deduced from the associativity (resp.\ commutativity) of $\pi_H$. The mapping cone construction is inspired by the \textit{categorical formal punctured neighborhood of infinity} as developed by Efimov in \cite{efimov2017categorical} and recalled in \cref{sec:catFormal}. It is similar to the Tate-Hochschild complex considered in \cite{RivWan19}, but suitable for smooth categories and, consequently, for non-simply connected string topology.  

Given any smooth dg (or $A_{\infty}$) category $A$, which we may think of the non-commutative analog of a smooth but possibly non-proper variety $X$, the categorical formal punctured neighborhood of infinity is another dg category $\Ahatinfty$ serving as the analog of the category of perfect complexes on the formal punctured neighborhood of the divisor at infinity for some compactification of $X$. On `the other side of mirror symmetry', when $A$ is the wrapped Fukaya category of a Weinstein domain, a relation between $\Ahatinfty$ and the \emph{Rabinowitz Fukaya category}, an open-string analogue of the Rabinowitz Floer homology, has been recently established in \cite{ganatra2022rabinowitz}.   
As an object of the derived category of $A$-bimodules, $\Ahatinfty$ sits in a distinguished triangle
\[ A^! \otimes^\LL_A A^\vee \to A \to \Ahatinfty, \]
where $A^\vee$ denotes the \textit{linear dual} of $A$. Applying Hochschild chains we obtain a distinguished triangle
\begin{equation}\label{eq:distinguishedTriangle}
    C^*(A,A^\vee) \cong C_*(A,A)^\vee \xrightarrow{\ ^\sharp E} C_*(A,A) \to C_*(A,\Ahatinfty)
\end{equation}
in the derived category of complexes of $\kk$-modules. As suggested by the notation, we calculate that the first map $\ ^\sharp E$ is given by pairing against the element $E \in C_*(A,A)^{\otimes 2}$ we previously defined. 

When $A$ is equipped with a non-degenerate pre-CY structure of dimension $n$, we may consider another model for $\widehat{A}_{\infty}$. This model is an $A$-bimodule $M_{\alpha}$ defined as the mapping cone of certain $A$-bimodule morphism \[f_\alpha: A^\vee[-n] \to A\] induced by the  the pre-CY structure of $A$, as discussed in \cref{sec:coneoff}. Non-degeneracy of the pre-CY structure means that the structure map $\alpha \colon A \to A^!$ is a quasi-isomorphism, giving a quasi-isomorphism of $A$-bimodules $\Ahatinfty \simeq M_{\alpha}$ (\cref{prop:homotopyAlpha}). The Hochschild chain complex $C_*(A, M_{\alpha})$ is the analog of the Tate-Hochschild complex but now in the setting of smooth pre-CY categories, being given by the mapping cone of the map of complexes
\[F_\alpha: C_*(A,A^\vee)[-n] \to C_*(A, A)\] induced by $f_\alpha$. In \cref{def:productOnChains}, we construct an algebra structure on  $C_*(A, M_{\alpha})$ with product
\[ \pi_{M_{\alpha}}: C_*(A,M_{\alpha}) \otimes C_*(A,M_{\alpha}) \to C_*(A, M_{\alpha})[n] \]
extending the algebraic analogue of the loop product 
\[\pi \colon C_*(A, A) \otimes C_*(A, A) \to C_*(A, A)[n].\]
The product $\pi$ is associated to the ``pair of pants" in the graphical formalism of \cite{kontsevich2021precalabiyau}. The following statement, which is the main result of \cref{sec:extendingproducts}, summarizes the homological algebra meaning, as well as the main properties, of the product $\pi_{M_\alpha}$. 

\begin{theorem}\label{thm:thm2intro}
Let $A$ be a smooth dg (or $A_{\infty}$) category equipped with a pre-CY structure of dimension $n$. The structure map $\alpha \colon A \to A^{!}$ of the pre-CY structure determines a quasi-isomorphism \[g^{M_{\alpha}} \colon C_*(A, M_{\alpha}) \xrightarrow{\simeq} C^*(A, M_{\alpha})\] of degree $-n$ inducing a map of algebras on (co)homology, where the products on $HH_*(A, M_{\alpha})$ and $HH^*(A, M_{\alpha})$ are induced by $\pi_{M_{\alpha}}$ and an extension $\smile_{M_{\alpha}}$ of the classical Hochschild cup product, respectively. Consequently, $\pi_{M_{\alpha}}$ induces an associative and (graded) commutative product on $HH_*(A, M_{\alpha})$.
\end{theorem}

Subsequently, we explain the sense in which the map $\pi_{M_{\alpha}}$ relates to algebraic loop coproduct $\lambda_H$. We have a map of complexes $g^{A^\vee}_\alpha \colon C_*(A,A^\vee)[-n] \to C^*(A,A^\vee)$, which is a quasi-isomorphism when $\alpha$ is non-degenerate. We note that the target of this map is the linear dual of $C_*(A,A)$ and consequently a natural setting to consider the linear dual of the coproduct $\lambda_H$. The product $\pi_{M_{\alpha}}$ does not require the additional data of a trivialization $(E_0, H)$.  However, in the presence of such additional structure we can compare this product with the previously defined coproduct. For that, let us define a subcomplex
\[ K \coloneqq \{x \in C_*(A,A^\vee)[-n]\ |\ dx=0, \forall w \in W\ \text{such that}\ dw=0, \langle g^{A^\vee}_\alpha(x), w \rangle = 0 \} \]
If we have any element $E_0\in W\otimes C_*(A,A) \cap C_*(A,A) \otimes W$ closed of degree zero, denoting
\[ q \coloneqq \ ^\sharp E_0\circ g_\alpha^{A^{\vee}} \colon C_*(A, A^{\vee})[-n] \to C_*(A,A) \]
defines a subspace $\ker(q)$ satisfying the property that $K \subseteq \ker(q)$.

\begin{theorem}[see Propositions \ref{prop:productCone} and \ref{prop:compatibility}] \label{thm:thm3intro}
    Let $A$ be a connective smooth dg (or $A_{\infty}$) category equipped with a pre-Calabi Yau structure of dimension $n$. Suppose  $(E_0,H)$ is a trivialization of $E$ to a subcomplex $W$. Then there are two maps of complexes $\iota_H^1,\  \iota_H^2  \colon \emph{ker}(q)[1] \to C_*(A, M_{\alpha})$, both sections of the projection $p\colon C_*(A,M_{\alpha}) \to C_*(A, A^{\vee})[-n+1]$, such that the map of complexes
    \[ \pi_H \coloneqq p \circ \pi_M \circ (\iota^1_H|_K \otimes \iota^2_H|_K) \colon K[1] \otimes K[1] \to C_*(A,A^\vee)[1] \]
    satisfies the equation
    \[\langle \pi_H(x_1 \otimes x_2), g_{\alpha}(y)\rangle = (-1)^{\deg(x_1)\deg(x_2)}\langle x_2 \otimes x_1, (\id \otimes g_{\alpha})(g_{\alpha} \otimes \id) \lambda_H(y)\rangle, \]
    for any closed elements $x_1,x_2 \in K$ and $y \in C_*(A,A)$.
\end{theorem}
Each of the maps $\iota_H^1$ and $\iota_H^2$ is constructed by ``perturbing" the inclusion map of graded complexes $\ker(q)[1] \hookrightarrow C_*(A,M_{\alpha})$ to a map of complexes by using homotopies, induced by $(E_0,H)$. These homotopies deform the map $F_\alpha$ into the map given by pairing against $E_0$.

\begin{corollary}
    With the same assumptions and notation as in the theorem above, if $(W,H,\alpha)$ is balanced, then the image of $\pi_H$ is contained in $K$, defining thus a `homology dual product' $\pi_H \colon K/B_*(\ker(q)) \otimes K/B_*(\ker(q))  \to K/B_*(\ker(q)) [n-1]$. If moreover $\alpha$ is non-degenerate, and we are working over a field, then the homology loop coproduct $\lambda_H \colon \overline{HH}_*(A,A) \to \overline{HH}_*(A,A) \otimes \overline{HH}_*(A,A) [n-1]$ is determined by the homology dual product $\pi_H$.
\end{corollary}
\cref{thm:thm1intro} then follows by analyzing the properties of the product $\pi_H$. We emphasize our main point: \textit{\cref{thm:thm3intro} provides an interpretation of the dual of the algebraic loop coproduct $\lambda_H$ as a part of homological algebra operation $\pi_{M_{\alpha}}$, extending the interpretation of the algebraic loop product as the classical Hochschild cup product.} 

Finally, \cref{sec:examples} is devoted to computing explicit examples in the context of string topology of spheres. Our computations confirm that the structure studied in this article is non-trivial in general and coincides partially with other geometrically defined constructions in the literature.

\subsection{Relation to existing work}
String topology has been studied in the literature through different perspectives. In particular, the loop coproduct was originally constructed in \cite{GorHin} using Morse theoretic methods, in \cite{Abb16} using algebraic methods applicable to simply connected manifolds, in \cite{HinWah0} using Thom-Pontryagin theory and a Riemannian metric, in \cite{NaeWil19} using configuration spaces, and in \cite{CieHinOan1, CieHinOan2} using symplectic Floer theory. See also \cite{naef2023string} for a survey on three of these perspectives.
Some of the constructions and computations in this paper have direct analogs in the symplectic setting \cite{CieHinOan2}. The recent works \cite{naef2024simple,kenigsberg2024obstructions} have been studying the loop coproduct and its relationship to Whitehead torsion and simple homotopy following other approaches; in particular, in \cite{naef2023torsion} Naef and Safronov give a definition of an \emph{Euler structure}, which in the Euler characteristic zero case should be related to our trivializations, and describe how these give rise to volume forms on derived mapping stacks. Our algebraic perspective has been influenced and guided by their talks and by discussions we have had with them. However, we come to it from a different angle: we have arrived to explicit formulas for the algebraic loop coproduct by first considering a mapping cone construction inspired by Efimov's categorical formal punctured neighborhood of infinity (taking a sort of Koszul pre-dual of the main construction of \cite{RivWan19}) and using the formalism of pre-Calabi Yau structures (developed in \cite{kontsevich2021precalabiyau}) to incorporate chain-level Poincaré duality.

\addtocontents{toc}{\SkipTocEntry}
\section*{Acknowledgments}
We would like to thank Damien Calaque, Kai Cieliebak, Florian Naef, Sheel Ganatra, Ezra Getzler, Ralph Kaufmann, Maxim Kontsevich, Alexandru Oancea, Pavel Safronov, Bruno Vallette, Yiannis Vlassopoulos, and Nathalie Wahl for helpful conversations. AT would like to thank IHES for providing great working conditions. This project is partly supported by the Simons Collaboration on Homological Mirror Symmetry.

\addtocontents{toc}{\SkipTocEntry}
\subsection*{Conventions} We will use \textit{homological} grading convention throughout the article. Given a commutative ring $\kk$, we denote by $(\mathrm{Ch}(\kk), \otimes)$ the symmetric monoidal category of (unbounded) chain complexes of $\kk$-modules, that is, with a differential of degree $-1$. If $(V,d) \in \mathrm{Ch}(\kk)$ then $(V^{\vee}, d)$ is the chain complex with $V^{\vee}_n=\Hom_{\kk}(V_{-n}, \kk)$ and $d(f)(x)= -(-1)^{\deg(f)}f(dx)$. To unclutter notation, we will often use the generic notation $d$ for the differential on any chain complex, such as the Hochschild chain complex or the cochain complex (both with \emph{homological} grading).

A dg category is defined as a category enriched over  $(\mathrm{Ch}(\kk), \otimes)$. A connective dg category is a category enriched over the symmetric monoidal category of non-negatively graded chain complexes. We refer to \cite{Tab10, Kel06} for further generalities on dg categories. The symbol $X \cong Y$ means $X$ and $Y$ are ``isomorphic" and $X \simeq Y$ means $X$ and $Y$ are ``quasi-isomorphic". 

Given any $M=(M_*,d) \in (\mathrm{Ch}(\kk), \otimes)$, we denote by $M[n]$ the complex obtained by shifting $M$ down by $n$, namely, the complex given by $M[n]_p= M_{p-n}$ with differential $(-1)^nd$. In particular, maps $M \to N$ of degree $-n$ are equivalent to maps $M \to N[n]$ of degree $0$, products $M \otimes M \to M$ of degree $-n$ are equivalent to products $M[-n] \otimes M[-n] \to M[-n]$ of degree $0$, and coproducts $M \to M\otimes M$ of degree $-n$ are equivalent to coproducts $M[n] \to M[n] \otimes M[n]$ of degree $0$.

\section{Pre-CY categories and their graphical calculus}\label{sec:graphicalCalculus}
Throughout this paper we will be working in the setting of strictly unital $\Z$-graded $A_\infty$-categories; for ease of exposition in this section we will restrict to the case of a categories with a single object, that is, the case of $A_\infty$-algebras. Moreover, we consider differential graded algebras (dg algebras) as special types of $A_\infty$-algebras with trivial structure maps $\mu^{\ge 3}$. In the literature there are many different notations and sign conventions for such objects; we will be mostly following \cite{kontsevich2008notes, Kel06}. 

\subsection{The graphical calculus}
We will need to write formulas for morphisms and homotopies of such objects, Hochschild cycles and cocycles etc. whose formulas are often too complicated to be enlightening.

For that, instead of formulas we will use a certain graphical calculus that was developed to deal with $A_\infty$-algebras/categories, bimodules over them and Hochschild co/chains. An early version of this graphical calculus, including only Hochschild cochains, appeared in Kontsevich and Soibelman's proof \cite{kontsevich2000deformations} of the Deligne conjecture; this was extended to include Hochschild chains in \cite{kontsevich2008notes}. In \cite{kontsevich2021precalabiyau} this graphical calculus is explained more systematically in terms of `ribbon quivers', that is, acyclic oriented ribbon graphs with markings, to include what is called in \textit{op.cit.} `higher Hochschild cochains', used to define pre-Calabi-Yau structures.

\subsubsection{Hochschild chains and cochains}\label{subsubsection:chainsandcochains}
We will use the normalized complexes for Hochschild chains and cochains. Let $(A,\mu)$ be a strictly unital $A_\infty$-algebra over $\kk$, and $M$ an $A$-bimodule. Denoting $\A = A/(\kk\cdot 1_A)$, the Hochschild chain complex of $A$ with values on $M$ will be denoted by
\[ C_*(A,M) = \bigoplus_{k \ge 0} M \otimes \A[1]^{\otimes k}. \]
We will use the notation $m[a_1|\dots|a_k]$ to denote an element of this complex. The Hochschild cochain complex will be denoted by
\[ C^*(A,M) = \prod_{k \ge 0} \Hom_\kk(\A[1]^{\otimes k},M), \]
and we will assume homological convention, meaning that $C^p(A,M)$ denotes the $\kk$-module of maps of degree $-p$ in $\prod_{k \ge 0} \Hom_\kk(\A[1]^{\otimes k},M)$. 

We visualize a Hochschild cochain $\varphi$ as living in a disc with one output on the bottom, and inputs running along the top; we denote this by the diagram
\[\tikzfig{
    \node [vertex] (a) at (0,0) {$\varphi$};
    \draw [->-,orange] (a) to (0,-1);
}\]
We visualize this vertex as receiving any number $k$ of $\A[1]$-arrows, each carrying a factor of 
\[ a_1 \otimes \dots \otimes a_k \in (\A[1])^{\otimes k} \]
from left to right, and outputting $\varphi(a_k,\dots,a_k) \in M$ along the bottom $M$-arrow, which we distinguish by color.

As for Hochschild chains, we will visualize an element $m[a_1|\dots|a_k] \in M \otimes \A[1]^{\otimes k}$ as traveling down a cylinder with one $M$ line and $k$ $\A[1]$ lines running down along its length.
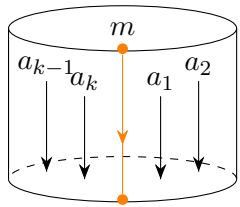
\[\tikzfig{ 
    \draw (0,2) [partial ellipse=0:360:1.5 and 0.3];
    \draw (0,0) [partial ellipse =0:-180:1.5 and 0.3];
    \draw [dashed] (0,0) [partial ellipse =0:180:1.5 and 0.3];
    \node [bullet, orange] (topM) at (0,1.73) {};
    \node [bullet, orange] (botM) at (0,-0.27) {};
    \node at (0,2) {$m$};
    \draw [->-,orange] (topM) to (botM);
    \draw (-1.5,2) to (-1.5,0);
    \draw (1.5,2) to (1.5,0);
    \draw [->-=1] (0.5,1.1) to (0.5,0);
    \draw [->-=1] (1,1.3) to (1,0.1);
    \draw [->-=1] (-0.5,1.1) to (-0.5,0);
    \draw [->-=1] (-1,1.3) to (-1,0.1);
    \node at (0.5,1.3) {$a_1$};
    \node at (1,1.5) {$a_2$};
    \node at (-0.5,1.3) {$a_{k}$};
    \node at (-1,1.5) {$a_{k-1}$};
}\]

\subsubsection{Higher Hochschild cochains}
We also have vertices with multiple outgoing edges; in a vertex with $k$ outgoing arrows, we insert an element of the graded complex
\begin{align}\label{align:C(n)}
    C^*_{(k)}(A) = \prod_{\{(n_1,\dots,n_k)\mid n_i \ge 0\}} \Hom_\kk(\A[1]^{\otimes n_1} \otimes \dots \otimes \A[1]^{\otimes n_k}, A^{\otimes k}) 
\end{align} 
which was called the complex of $k$th higher Hochschild cochains in \cite{kontsevich2021precalabiyau}. We visualize an element $\varphi$ of this complex, when $k=3$ for example, as a vertex
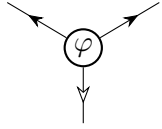
\[\tikzfig{
    \node [vertex] (a) at (0,0) {$\varphi$};
    \draw [->-] (a) to (1,0.6);
    \draw [-w-] (a) to (0,-1);
    \draw [->-] (a) to (-1,0.6);
}\]
where the first outgoing factor of $A$ is marked by the white arrowhead.

\subsubsection{Oriented ribbon quivers}\label{sec:orientedRibbon}
A ribbon quiver is a ribbon graph whose edges are directed; we always require that there are no directed cycles. We add some markings on this quiver: some of the valence one sources of this quiver, which we label by an $\times$, are places where we can input a Hochschild chain. In each other vertex with $k$ outgoing legs we can insert a $k$th higher Hochschild cochain, and at each sink vertex, which we denote by $\circ$, we can read off a Hochschild chain. Note that each sink may have more than one incoming arrow; in that case, we mark one of its incident edges with a white arrowhead, meaning that we read the output Hochschild chain starting from it.

We make the convention that on vertices labeled with a black dot we will use the $A_\infty$-structure map $\mu$. All the other vertices are either inputs or outputs; in order to fix a sign ambiguity we will consider all the inputs/outputs of the same type totally ordered
\begin{example}\label{ex:ribbonGraph}
    Consider the ribbon quiver below:
    \[\tikzfig{
	\node (v1) at (-1.5,0) {$\times$};
        \node at (-1.5,-0.3) {I};
	\node [bullet] (v4) at (0.75,0) {};
	\node (v2) at (0,0) {$\times$};
        \node at (0,-0.3) {II};
	\node [bullet] (v5) at (0,-0.75) {};
	\node [circ] (v6) at (0,-2) {};
	\node [bullet] (v7) at (-0.75,0) {};
        \node [vertex] (v8) at (2,0) {I};
	\draw [->-,shorten <=-3.5pt] (v1) to (v7);
	\draw [->-,shorten <=-3.5pt] (v2) to (v4);
	\draw [->-] (0,0.75) arc (90:0:0.75);
	\draw [->-] (0.75,0) arc (0:-90:0.75);
	\draw [-w-] (0,0.75) arc (90:180:0.75);
	\draw [->-] (-0.75,0) arc (180:270:0.75);
        \draw [->-] (v8) to (v4);
        \node [vertex,fill=white,inner sep=4pt] (v3) at (0,0.75) {};
	\draw [->-=0.4] (v5) to (v6);
        \node [vertex,fill=white] (v7) at (-1,-2) {II};
        \draw [->-] (v7) to (v6);
        \draw [-w-=1] (v5) to (v6);
    }\]
    to which we give a ribbon structure by the embedding onto the page. The ribbon quiver above lives on a surface of genus zero, with two incoming $\times$ boundary components (on which the $\times$ vertices live) and one outgoing $\circ$ boundary component (on which the $\circ$ vertex lives).

    Into the circle vertex on the top, we will input an element of $C^*_{(2)}(A)$, that is, a higher Hochschild cochain with two outputs. The white arrowhead marks which arrow receives the first factor of $A$. Into the circle vertices on the right and bottom left, we will input an element of $C^*(A,A)$, that is, a (usual) Hochschild cochain, and into the $\times$ vertices, we will input Hochschild chains. Since there are two of them we consider an ordering of the $\times$ vertices to be part of the data of the ribbon quiver; we indicate this on the diagram with Roman numerals. Out of the little circle on the bottom we read the outgoing Hochschild chain; the arriving white arrowhead indicates which factor of $A$ is to be read first.
\end{example}

In order to interpret a ribbon quiver like the one in the above example as an operation between Hochschild co/chains, with a coherent choice of signs, one must also pick an \textit{orientation}. There are many different ways this notion of orientation can be phrased; here we choose to follow a slight variation on the conventions of \cite{kontsevich2021precalabiyau}, which brings it closer to the definition of orientation on the usual ribbon graph complex, and will be more appropriate to our ends.

For each ribbon quiver $\vec \Gamma$, we consider the set
\[ \mathrm{Or}_\Gamma = \mathrm{Span}_{\Z/2\Z}(\{\text{total orderings of all edges of }\Gamma\})/X \]
where the $\Z/2\Z$-submodule $X$ is spanned by $o_1 - (-1)^{s(o_1,o_2)}o_2$ for any two total orderings $o_1, o_2$ of edges, with $s(o_1,o_2)$ being the sign of the permutation. Each oriented ribbon quiver gives a map between appropriate tensor products of complexes of (higher) Hochschild co/chains.
\begin{example}
    Returning to \cref{ex:ribbonGraph}, upon labeling all the edges we can specify an orientation, representing it by a total ordering of the edges. For example, the following data gives an oriented ribbon quiver
    \[(\vec \Gamma, o) = \tikzfig{
	\node (v1) at (-1.5,0) {$\times$};
        \node at (-1.5,-0.3) {I};
	\node [bullet] (v4) at (0.75,0) {};
	\node (v2) at (0,0) {$\times$};
        \node at (0,-0.3) {II};
	\node [bullet] (v5) at (0,-0.75) {};
	\node [circ] (v6) at (0,-2) {};
	\node [bullet] (v7) at (-0.75,0) {};
        \node [vertex] (v8) at (2,0) {I};
	\draw [->-,shorten <=-3.5pt] (v1) to (v7);
	\draw [->-,shorten <=-3.5pt] (v2) to (v4);
	\draw [->-] (0,0.75) arc (90:0:0.75);
	\draw [->-] (0.75,0) arc (0:-90:0.75);
	\draw [->-] (0,0.75) arc (90:180:0.75);
	\draw [->-] (-0.75,0) arc (180:270:0.75);
        \draw [->-] (v8) to (v4);
        \node [vertex,fill=white,inner sep=4pt] (v3) at (0,0.75) {};
	\draw [->-=0.4] (v5) to (v6);
        \node [vertex,fill=white] (v7) at (-1,-2) {II};
        \draw [->-] (v7) to (v6);
        \draw [-w-=1] (v5) to (v6);
        \node at (-0.8,0.8) {$1$};
        \node at (0.8,0.8) {$2$};
        \node at (-0.8,-0.8) {$6$};
        \node at (0.8,-0.8) {$7$};
        \node at (-1,0.25) {$3$};
        \node at (0.3,0.25) {$4$};
        \node at (1.3,0.25) {$5$};
        \node at (0.3,-1.2) {$8 $};
        \node at (-0.5,-1.7) {$9$};
    } \quad, \quad (9\ 8\ 7\ 6\ 5\ 4\ 3\ 2\ 1)\]
    for any $A_\infty$-algebra, the oriented ribbon quiver $(\vec\Gamma,o)$ gives a map
    \[ C_*(A,A) \otimes C_*(A,A) \otimes C^*(A,A) \otimes C^*(A,A) \otimes C^*_{(2)}(A) \to C_*(A,A). \]
    To get the appropriate sign, one follows the prescriptions of \cite[Subsection 2.2]{kontsevich2021precalabiyau}, using the given orientation $o$ to determine the sign of each term. We also refer to \cite[Section 6.1.4]{kontsevich2021precalabiyau} for a detailed example of how to get a map from an oriented ribbon quiver. 
\end{example}

\subsubsection{Differential on oriented ribbon quivers}
Not every map given by an oriented ribbon quiver is a map of complexes, but the failure of such a map to intertwine the differentials on the source and target can be encoded itself by a differential acting on the complex spanned by oriented ribbon quivers. This is described precisely in \cite[Section 2]{kontsevich2021precalabiyau}. For sake of example, let us just illustrate a familiar example in our language, using the operations appearing in the \textit{Tamarkin-Tsygan calculus} of Hochschild (co)chains \cite{TaTs05}.
\begin{example}\label{ex:capProducts}
    The following oriented ribbon quiver gives a `cap product' between a Hochschild cochain and a Hochschild chain:
    \[ (\vec\Gamma_1,o_1) = \tikzfig{
        \node (in) at (0,2) {$\times$};
        \node [bullet] (mid) at (0,1) {};
        \node [vertex,inner sep=4pt] (phi) at (-1,1) {};
        \node [circ] (out) at (0,0) {};
        \draw [->-,shorten <=-6 pt] (in) to node [auto] {$2$} (mid);
        \draw [->-] (phi) to node [auto,swap] {$1$} (mid);
        \draw [->-] (mid) to node [auto] {$3$} (out);
     }, (3\ 2\ 1)
    \]    
    We will denote $\varphi \frown x = (\vec\Gamma_1,o_1)(\varphi,x)$.

    There is another version of the cap product, given by the diagram
    \[ (\vec\Gamma_2,o_2) = \tikzfig{
        \node (in) at (0,2) {$\times$};
        \node [bullet] (mid) at (0,1) {};
        \node [vertex,inner sep=4pt] (phi) at (1,1) {};
        \node [circ] (out) at (0,0) {};
        \draw [->-,shorten <=-6 pt] (in) to node [auto,swap] {$1$} (mid);
        \draw [->-] (phi) to node [auto] {$2$} (mid);
        \draw [->-] (mid) to node [auto,swap] {$3$} (out);
     } , (3\ 2\ 1)
    \]
    whose corresponding operation we will denote $x \frown \varphi = (\vec\Gamma_1,o_1)(\varphi,x)$. Both of these oriented ribbon quivers are $d$-closed, and the corresponding operations descend to cohomology, giving maps $HH^*(A,A) \otimes HH_*(A,A) \to HH_*(A,A)$, interpreted as `contraction' of a vector field with a form.

    These operations, seen as maps $C^*(A,A) \otimes C_*(A,A) \to C_*(A,A)$, are homotopic. We can write down a homotopy by giving a $d$-primitive of the difference, for example:
    \[ (\vec\Gamma_3,o_3) = \tikzfig{
        \node (in) at (0,2) {$\times$};
        \node [vertex,inner sep=4pt] (mid) at (0,1) {};
        \node [circ] (out) at (0,0) {};
        \draw [->-,shorten <=-6 pt] (in) to node [auto] {$1$} (mid);
        \draw [->-] (mid) to node [auto] {$2$} (out);
     }\ ,\ (2\ 1)
    \]
    whose differential is $d(\vec\Gamma_3,o_3) = (\vec\Gamma_1,o_1) -(\vec\Gamma_2,o_2)$; this means that for any $\varphi,x$ we have
    \[ d((\vec\Gamma_3,o_3)(\varphi,x)) + (\vec\Gamma_3,o_3)(\delta\varphi,x)+(-1)^{|\varphi|}(\vec\Gamma_3,o_3)(\varphi,d x) = (\vec\Gamma_1,o_1)(\varphi,x) - (\vec\Gamma_2,o_2)(\varphi,x)\]

    We can write down another homotopy:
    \[ (\vec\Gamma_4,o_4) = \tikzfig{
        \node (in) at (0,2) {$\times$};
        \node [vertex,inner sep=4pt] (bot) at (0,-0.5) {};
        \node [circ] (out) at (0,0.5) {};
        \draw [-w-,shorten <=-6 pt] (in) to node [auto] {$1$} (out);
        \draw [->-] (bot) to node [auto] {$2$} (out);
     }\ ,\ (2\ 1)
    \]
    whose differential is $d(\vec\Gamma_4,o_4) = (\vec\Gamma_2,o_2)-(\vec\Gamma_1,o_1)$. Thus the combination $(\vec\Gamma_3,o_3)+(\vec\Gamma_4,o_4)$ is $d$-closed: we denote the corresponding map of complexes $L_\varphi = (\vec\Gamma_3,o_3)(\varphi,x)+(\vec\Gamma_4,o_4)(\varphi,x)$. This map descends to homology and gives a map $HH^*(A,A) \otimes HH_*(A,A) \to HH_*(A,A)[-1]$, interpreted as the `Lie derivative', see e.g.\ \cite[Section 3.3]{TaTs05}. The Hochschild chain differential itself is $\partial=L_\mu$.
\end{example}

We will also consider a slightly different class of diagrams, each of which outputs is instead a Hochschild \emph{cochain}. To distinguish these types of outputs, we write these diagrams inside of a large circle with one outgoing $M$-arrow at the bottom, for some $A$-bimodule $M$. The diagrams where $M=A$ and whose vertices all have exactly one output represent exactly the operations appearing in \cite{kontsevich2000deformations}, and giving the $E_2$-algebra structure on the Hochschild cochains $C^*(A,A)$.
\begin{example}
    Let $M$ be some $A$-bimodule, and let us color orange the arrows carrying factors of $M$. The oriented ribbon quiver
    \[\tikzfig{
	\draw (0,0.7) circle (1.2);
	\node [vertex] (a) at (-0.6,1) {I};
	\node [vertex] (b) at (0.6,1) {II};
	\node [bullet] (c) at (0,0.3) {};
	\draw [->-] (a) to (c);
	\draw [->-,orange] (b) to (c);
	\draw [->-,orange] (c) to (0,-0.5);
        \node at (-0.5,0.4) {$1$};
        \node at (0.5,0.4) {$2$};
        \node at (0.3,-0.1) {$3$};
    }\ ,\ (3\ 2\ 1) \]
    is $d$-closed, and thus gives a map of complexes $\smile: C^*(A,A) \otimes C^*(A,M) \to C^*(A,M)$, which can be verified to be the familiar \textit{cup product} of Hochschild cochains when $M = A$.
\end{example}

\begin{example}\label{ex:galpha}
    Let $M$ be some $A$-bimodule, and let us color orange the arrows carrying factors of $M$. The oriented ribbon quiver below
    \[\tikzfig{
        \draw (0,0) circle (1.5);
        \node [vertex,inner sep=4pt] (v3) at (0,0.8) {};
        \node (v2) at (0,0) {$\times$};
        \node [bullet] (v4) at (0.8,0) {};
        \node [bullet] (v5) at (0,-0.8) {};
        \node (v6) at (0,-1.6) {};
        \draw [->-,orange,shorten <=-6pt] (v2) to (v4);
        \draw [->-,shorten <=6pt] (0,0.8) arc (90:0:0.8);
        \draw [->-,shorten <=6pt] (0,0.8) arc (90:270:0.8);
        \draw [->-,orange] (0.8,0) arc (0:-90:0.8);
        \draw [->-,orange] (v5) to (v6);
        \node at (-1.2,0) {$1$};
        \node at (0.8,0.8) {$2$};
        \node at (0.4,-0.2) {$3$};
        \node at (0.8,-0.8) {$4$};
        \node at (0.4,-1.2) {$5$};
    }\ , (5\ 4\ 3\ 2\ 1) \]
    gives a map of complexes $C^*_{(2)}(A) \otimes C_*(A,M) \to C^*(A,M)$. For any fixed $\alpha \in C^n_{(2)}(A)$ which is closed under the differential $\delta = [\mu,-]$, we will denote the corresponding map of complexes by \[g^M_\alpha: C_*(A,M) \to C^*(A,M)[n]\] and simply by $g_\alpha$ when $M$ is the diagonal bimodule $A$.
\end{example}

We can compose oriented ribbon quivers that have compatible outputs/inputs, by identifying the outgoing arrow of $\vec\Gamma_1$ and the incoming arrow of $\vec\Gamma_2$, as a new arrow $f$ of the combined ribbon quiver $\vec\Gamma$. As for the orientation, we permute the sequences of edges (keeping track of the sign) so that the two arrows to be composed are adjacent \footnote{For more details about the composition operation in general see \cite[Section 6]{kontsevich2021precalabiyau}.}.

\begin{remark}
    The attentive reader that consults \cite{kontsevich2021precalabiyau} will note that the definitions of degree and differential that we use in this paper are slightly different from the ones presented in \textit{op.cit.} The reason for this difference is that, in this article, we are only fixing the $A_\infty$ structure maps in the $\bullet$ vertices and regard the other vertices (including all the ones with $\ge 2$ outgoing arrows) as places where we can insert possibly other (higher) Hochschild cochains, which have to be ordered. The comparison between our conventions here and the ones in \textit{op.cit.} is straightforward from the definition of the pre-CY structure equations.
\end{remark}

\subsection{Pre-CY structures}\label{subsec:preCY}
The notion of a pre-Calabi-Yau structure on an $A_\infty$-algebra or category was given by Kontsevich and Vlassopoulos and later developed in more detail by those authors together with the second-named author of this paper in \cite{kontsevich2021precalabiyau}. Let us recall the relevant definitions from that paper.

Let $(A, \mu)$ be an $A_\infty$-category; the conditions on the $A_\infty$ structure maps $\mu = \mu^1 + \mu^2 + \dots$ say exactly that the element $\mu \in C^2(A)$ satisfies the equation $\mu \circ \mu = 0$, where $\circ$ is the Gerstenhaber product. The commutator bracket  $[-,-]$  associated to the Gerstenhaber product endows $C^*(A,A)[1]$ with the structure of a dg Lie algebra. The cochain complex $(C^*(A,A), \delta=[\mu,-])$ calculates the Hochschild cohomology of $(A,\mu)$.

One can consider vertices with more outputs, and evaluate them in a similar way; the Gerstenhaber bracket $[-,-]$ then extends to vertices with more outputs. In order to discuss these extended brackets, we need to include certain shifts and signs depending on a chosen integer $n$. For each $k \ge 1$, we denote by $C^*_{(k)}(A)$ the complex spanned by vertices with $k$ outputs, see \cref{align:C(n)}. This complex has an action by the cyclic group $\Z/k\Z$; we twist this action by a sign, declaring that, besides the Koszul sign, the generator of $\Z/k\Z$ acts with an extra sign $(n-1)(k-1)$. We denote by $C^*_{(k,n)}(A) = C^*_{(k)}(A)^{\Z/k\Z}[(n-2)(k-1)]$ its (appropriately shifted) subcomplex of invariants. We can assemble all of these complexes as $C^*_{[n]}(A) := \prod_{k\ge 1} C^*_{(k,n)}(A)$; naturally we have $C^*(A,A) \subset C^*_{[n]}(A)$.

\begin{proposition}[{\cite[Proposition 10]{kontsevich2021precalabiyau}}]
    For any $k,l \ge 1$, there is a `necklace product' $\circ: C^*_{(k,n)}(A)[1] \otimes C^*_{(l,n)}(A)[1] \to C^*_{(k+l-1,n)}(A)[1]$ which agrees with the Gerstenhaber product when $k=l=1$, and whose associated `necklace bracket' gives $C^*_{[n]}(A)[1]$ the structure of a dg Lie algebra.
\end{proposition}

\begin{definition}
    A pre-CY structure of dimension $n$ on $(A,\mu)$ is an element $m = \mu + m_{(2)} + \dots \in C^2_{[n]}(A)$ such that $m_{(k)} \in C^2_{(k,n)}$  and $m \circ m=0$.
\end{definition}
That is, $m_{(k)} \in C^2_{(k,n)}(A)$ denotes the part of the pre-CY structure with $k$ outgoing legs. The part of the equation $m\circ m=0$ with one output is just the usual $A_\infty$-structure equation $\mu\circ \mu = 0$.  

\subsubsection{Truncated pre-CY structures}
For the purposes of this paper, it will be unnecessary to consider all the maps $m_{(k)}$; for that reason, we will use truncated versions of these structures. The complex $C^*_{[n]}(A)$ on which the necklace bracket acts has a decreasing filtration given by $F^\ell C^*_{[n]}(A) = \prod_{k\ge \ell} C^*_{(k,n)}(A)$. The necklace product and bracket descend to the quotients $C^*_{[n]}(A)/F^\ell C^*_{[n]}(A)$, so we can make the definition:
\begin{definition}
    A $\ell$-truncated pre-CY structure of dimension $n$ on $A$ is a solution $m$ of the equation $m \circ m=0$ on $C^*_{[n]}(A)/F^{\ell+1} C^*_{[n]}(A)$.
\end{definition}
For example, an $1$-truncated pre-CY structure of any dimension is just an $A_\infty$-structure $\mu$; a $2$-truncated pre-CY structure of dimension $n$ has, in addition, a $[\mu,-]$-closed element $m_{(2)} \in C^n_{(2)}(A)$.

\subsubsection{Product on Hochschild homology of pre-CY} \label{sssec:algebraicloopproduct}
In \cite{kontsevich2021precalabiyau}, it is described how, given a pre-CY structure $m$ of dimension $n$ on $A$, there is an action of a PROP of chains on moduli spaces of Riemann surfaces with framed punctures on the Hochschild homology of $A$. More specifically, for any $g \ge 0$, $m \ge 1,k \ge 1$, there is a map of complexes
\[ C_*(A,A)^{\otimes m} \otimes C^\mathrm{cell}_*(\mathscr{M}_{g,\vec m,\vec k}) \to C_*(A,A)^{\otimes k}[(2g+m-1)n] \]
where $C^\mathrm{cell}_*(\mathscr{M}_{g,\vec m,\vec k})$ is a certain cellular chain complex coming from a stratification of the corresponding moduli space.

Restricting attention to genus zero, and looking at operations with one output, that is, with $n=1$, we get the action of the operad of chains on the framed little $2$-disk operad; that is, $C_*(A,A)[n]$ is a framed $E_2$-algebra and thus $HH_*(A,A)[n]$ is a BV algebra. We will choose a particular representative for the chain-level product
\begin{align}\label{align:productonchains}
    \pi: C_*(A,A) \otimes C_*(A,A) \to C_*(A,A)[n] 
\end{align} 
given by the following oriented ribbon quiver:
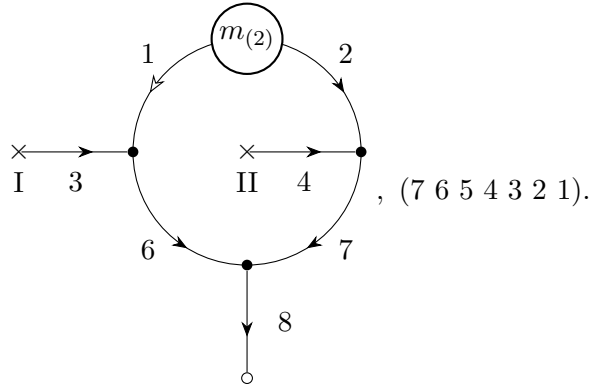
\[\tikzfig{
        \draw [->-] (0,1.5) arc (90:0:1.5);
        \draw [->-] (1.5,0) arc (0:-90:1.5);
        \draw [-w-] (0,1.5) arc (90:180:1.5);
        \draw [->-] (-1.5,0) arc (180:270:1.5);
        \node [inner sep=0pt] (v1) at (-3,0) {$\times$};
        \node [bullet] (v4) at (1.5,0) {};
        \node [inner sep=0pt] (v2) at (0,0) {$\times$};
        \node [bullet] (v5) at (0,-1.5) {};
        \node [circ] (v6) at (0,-3) {};
        \node [bullet] (v7) at (-1.5,0) {};
        \draw [->-,shorten <=-3.5pt] (v1) to (v7);
        \draw [->-,shorten <=-3.5pt] (v2) to (v4);
        \draw [->-] (v5) to (v6);
        \node at (0,-0.4) {II};
        \node at (-3,-0.4) {I};
        \node at (-1.3,1.3) {$1$};
        \node at (1.3,1.3) {$2$};;
        \node at (-1.3,-1.3) {$6$};
        \node at (1.3,-1.3) {$7$};
        \node at (-2.25,-0.4) {$3$};
        \node at (0.75,-0.4) {$4$};
        \node at (0.5,-2.25) {$8$};
		\node [vertex,fill=white] (v3) at (0,1.5) {$m_{(2)}$};
    }\ ,\ (7\ 6\ 5\ 4\ 3\ 2\ 1).\]

As the diagram suggests, into the vertex on the top we input the $m_{(2)}$ component of the pre-CY structure on $A$. We call the induced product on Hochschild homology the \textit{algebraic loop product}; this is an associative operation, as a consequence of the statements of \emph{op.cit.} We note that we only need a $2$-truncated pre-CY structure to define this product.

\subsubsection{Relations to other notions of Calabi-Yau}
Let us now recall some of the relations between the notion of pre-CY structure and other notions of Calabi-Yau categories. For this section we assume that $\kk$ is a field of characteristic zero.

\begin{theorem}[{\cite[Proposition 14]{kontsevich2021precalabiyau}}]\label{thm:comparisonCyclic}
If the category $A$ has finite-dimensional hom spaces, then the data of a pre-CY structure of dimension $n$ on $A$ is equivalent to the data to a cyclic $A_\infty$-structure with pairing of dimension $1-n$ on $A \oplus A^\vee[1-n]$, such that $A$ is an $A_\infty$-subcategory.
\end{theorem}
By $(-)^\vee$ we denote the linear dual category (same objects as $A$ and morphisms given by $A^\vee(x,y) = (A(y,x))^\vee$). The case of a finite-dimensional algebra appeared first in the work of Tradler-Zeinalian \cite{TraZei07a}, under the name of `$\calV_\infty$-algebra'.

However, this characterization will be of limited utility for us, since we will be interested in categories that are not of finite-dimension over $\kk$. Instead, our categories will be smooth; recall that $A$ is smooth if its diagonal bimodule $A$ is perfect \cite{kontsevich2008notes}. Then, its bimodule dual $A^!$ is also perfect, and there is a quasi-isomorphism $C_*(A,A) \simeq \Hom_{A^e}(A^!,A)$ between the Hochschild chains and $A_\infty$-bimodule morphisms between $A^!$ and $A$.

We recall the definition of a smooth Calabi-Yau structure on $A$. Let $n$ be an integer, the `dimension'; a $n$-CY structure on $A$ is a negative cyclic chain of degree $n$
\[ \tilde\omega = \omega + \omega_1 u + \omega_2 u^2 + \dots \in C_*(A,A)[[u]] \]
closed under the differential $\partial+uB$, whose image $\omega \in C_n(A,A)$ corresponds to a quasi-isomorphism between $A^!$ and $A[n]$ under the natural isomorphism $C_*(A,A) \simeq \R\!\Hom_{A^e}(A^!, A)$. On the other hand, there is also a quasi-isomorphism \cite[Proposition 19]{kontsevich2021precalabiyau}
\begin{equation}\label{quasi-iso-pCY} C^*_{(2)}(A) \simeq \R\!\Hom_{A^e}(A, A^!),
\end{equation}
leading to the following definition. 
\begin{definition}[{\cite[Sec.4.3]{kontsevich2021precalabiyau}}]\label{non-degenerate}
We say $\alpha \in C^n_{(2)}(A)$ is \textit{nondegenerate} if it corresponds to a quasi-isomorphism $A \xrightarrow{\simeq} A^![n]$ under the identification \cref{quasi-iso-pCY}. We say a \textit{pre-CY structure is nondegenerate} if its element $m_{(2)}$ is non-degenerate.   
\end{definition}

\begin{remark}\label{remark:nondegeneratequasiisomorphism}
    Note that for any nondegenerate cycle $\alpha \in C^n_{(2)}(A)$ the map $g_\alpha^M$ in \cref{ex:galpha} is a quasi-isomorphism. 
\end{remark}

If $\tilde\omega$ is an $n$-CY structure, there exists a closed element $\alpha \in C^n_{(2)}(A)$ whose image under the identification \cref{quasi-iso-pCY} is a quasi-inverse to the image of $\omega$. The main result of \cite{KTV23} (also appearing in \cite{pridham2017shifted} in a slightly different formulation) is that, when $\omega$ comes from a closed negative cyclic chain $\tilde\omega$, one can choose such an $\alpha$ that extends to a full pre-CY structure of dimension $n$. Conversely, given a non-degenerate pre-CY structure, there is a compatible smooth CY structure. Paraphrasing \cite{KTV23}, we have:
\begin{theorem}
    If $(A,\mu)$ is a smooth $A_\infty$-category and $\tilde\omega$ a smooth $n$-dimensional CY structure on it, then there is a pre-CY structure $m = \mu + \alpha + m_{(3)} + \dots$ such that $g_\alpha(\omega)$ is homologous to the unit cochain $1 \in C^0(A,A)$. Here $g_\alpha$ is given in \cref{ex:galpha}. Conversely, if we have such an $m$ that is non-degenerate, then there is a smooth CY structure $\tilde\omega$ whose component $\omega$ satisfies the same condition.
\end{theorem}

\section{Chain-level Chern character and coproduct}\label{sec:chernCharAndCoprod}
As argued by Shklyarov in \cite{shklyarov2013hirzebruch}, a perfect $A$-bimodule $M$ should give a certain distinguished class $[E_M] \in HH_*(A,A)\otimes HH_*(A,A)$ of degree zero, its \textit{Chern character}. We now propose, using the graphical calculus explained in the previous section, what we believe to be an explicit representative $E_M$ for this class (\cref{def:ChernCharacter}). Moreover, we  show that given a smooth $A$ satisfying appropriate conditions, a pre-CY structure $m$ together with a  trivialization of $E_A$ onto a subcomplex $W \subseteq C_*(A,A)$ give rise to a chain-level coproduct on $C_*(A,A)/W$ \cref{def:chainAlgebraicLoopCoproduct}; this will be the \textit{algebraic Goresky-Hingston loop coproduct}. 

\subsection{Chern character}\label{sec:chernCharacter}
Let $A$ be an $A_\infty$-category. Given any $A$-bimodule $M$, there is an $A$-bimodule $M^!$ called its \textit{bimodule dual}, see e.g.\ \cite[Section 8]{kontsevich2008notes}; if $M$ is a perfect bimodule then $M^!$ is also perfect, there are quasi-isomorphisms $M^{!!} \simeq M$ and 
\[ N \otimes^\LL_{A^e} M^! \simeq \R\!\Hom_{A^e}(M,N), \quad N \otimes^\LL_{A^e}  M \simeq \R\!\Hom_{A^e}(M^!,N), \]
for any $A$-bimodule $N$, see \cite[Remark 8.2.4]{kontsevich2008notes}.

\subsubsection{Coevaluation vertex}
Let $M$ be a perfect bimodule.
\begin{definition}\label{def:coevaluation}
    A \textit{coevaluation vertex} $\co_M \in M \otimes^\LL_{A^e} M^!$ is any element representing the class corresponding to the identity morphism $\mathrm{id}_M$ under the quasi-isomorphism $M \otimes^\LL_{A^e} M^! \simeq \R\!\Hom_{A^e}(M,M)$.
\end{definition}
We call this element $\co_M$ a `vertex' since we visualize it as a drawing
\begin{align}\label{align:coevfigure}
    \tikzfig{
    \node [vertex] (center) at (0,0) {$\co_M$};
    \node (left) at (-1.3,0) {$M$};
    \node (right) at (1.3,0) {$M^!$};
    \draw [->-,orange] (center) to (-1,0);
    \draw [->-,cyan] (center) to (1,0);
    \draw [->-] (0.35,-0.35) to (0.7,-0.7);
    \draw [->-] (0,-0.5) to (0,-1);
    \draw [->-] (-0.35,-0.35) to (-0.7,-0.7);
    \draw [->-] (-0.15,-0.43) to (-0.4,-0.86);
    \draw [->-] (0.15,-0.43) to (0.4,-0.86);
    \draw [->-] (-0.35,0.35) to (-0.7,0.7);
    \draw [->-] (-0.15,0.43) to (-0.4,0.86);
    \draw [->-] (0.15,0.43) to (0.4,0.86);
    \draw [->-] (0.35,0.35) to (0.7,0.7);
    \draw [->-] (0,0.5) to (0,1);
}
\end{align}
where the black arrows schematically denote outgoing factors of the bar complex $B\overline{A}[1]$. Explicitly, $\co_M$ is an element of the complex $\bigoplus_{i, j \geq 0}M \otimes \A[1]^{\otimes i} \otimes M^! \otimes \A[1]^{\otimes j}$  modeling the two-sided derived tensor product $M \otimes^\LL_{A^e} M^!$.

It follows from the quasi-isomorphisms above that composition with a coevaluation vertex gives a map realizing the quasi-isomorphism
\[ \R\!\Hom_{A^e}(M^!,N) \xrightarrow{\simeq} N \otimes^\LL_{A^e} M. \]

\subsubsection{The canonical pairing vertex}
Consider the complex of vertices with one $M$ input into one side, one $M^!$ input into the other, and two $A$ outputs on either side of the inputs, with $A[1]$ inputs into the four corners, as in the following diagram:
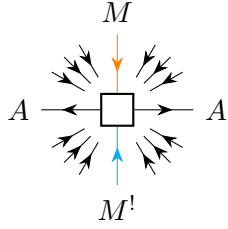
\[\tikzfig{
    \node [rectangle,draw,thick,inner sep=6pt] (center) at (0,0) {};
    \node (left) at (-1.3,0) {$A$};
    \node (right) at (1.3,0) {$A$};
    \node (top) at (0,1.3) {$M$};
    \node (bottom) at (0,-1.3) {$M^!$};
    \draw [->-] (center) to (-1,0);
    \draw [->-] (center) to (1,0);
    \draw [->-,orange] (0,1) to (center);
    \draw [->-,cyan] (0,-1) to (center);
    \draw [->-] (0.86, 0.5) to (0.43,0.25);
    \draw [->-] (0.5, 0.86) to (0.25,0.43);
    \draw [->-] (0.7,0.7) to (0.35,0.35);
    \draw [->-] (-0.86, 0.5) to (-0.43,0.25);
    \draw [->-] (-0.5, 0.86) to (-0.25,0.43);
    \draw [->-] (-0.7,0.7) to (-0.35,0.35);
    \draw [->-] (0.86, -0.5) to (0.43,-0.25);
    \draw [->-] (0.5, -0.86) to (0.25,-0.43);
    \draw [->-] (0.7,-0.7) to (0.35,-0.35);
    \draw [->-] (-0.86, -0.5) to (-0.43,-0.25);
    \draw [->-] (-0.5,-0.86) to (-0.25,-0.43);
    \draw [->-] (-0.7,-0.7) to (-0.35,-0.35);
}\]
We use a square to distinguish which incoming arrows carry bimodules, and which incoming arrows carry factors of the bar complexes. The complex of vertices as above is the complex $\R\!\Hom_{A^e \otimes A^e}(M\otimes M^!,A^{op}\otimes A)$ calculated by using the bar resolution of the bimodules $M,M^!$, namely the complex $\prod_{i, j, k, l\geq 0} \Hom(\A[1]^{\otimes i} \otimes M \otimes \A[1]^{\otimes j} \otimes \A[1]^{\otimes k} \otimes M^! \otimes \A[1]^{\otimes l}, A \otimes A)$.

The differential on this complex can be deduced from the differential on this derived complex of homs. When part of a larger diagram, this differential acts by splitting a multiplication away from the square, staying within each angle. For example:
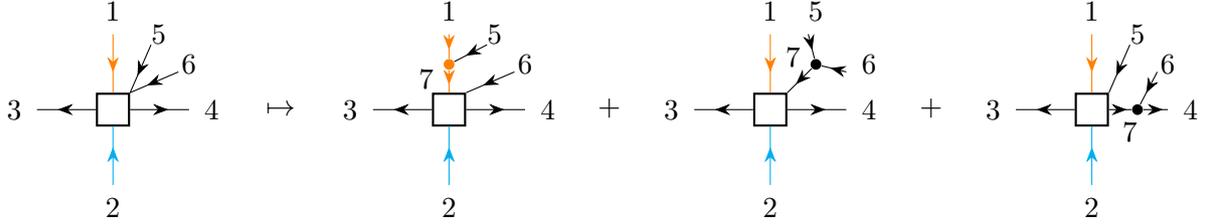
\[\tikzfig{
	\node [rectangle,draw,thick,inner sep=6pt] (center) at (0,0) {};
	\node (left) at (-1.3,0) {$3$};
	\node (right) at (1.3,0) {$4$};
	\node (top) at (0,1.3) {$1$};
	\node (bottom) at (0,-1.3) {$2$};
	\draw [->-] (center) to (-1,0);
	\draw [->-] (center) to (1,0);
	\draw [->-,orange] (0,1) to (center);
	\draw [->-,cyan] (0,-1) to (center);
	\draw [->-] (0.86, 0.5) to (center.north east);
	\draw [->-] (0.5, 0.86) to (center.north east);
	\node at (1,0.6) {$6$};
	\node at (0.6,1) {$5$};
} \quad \mapsto \quad \tikzfig{
\node [rectangle,draw,thick,inner sep=6pt] (center) at (0,0) {};
\node (left) at (-1.3,0) {$3$};
\node (right) at (1.3,0) {$4$};
\node (top) at (0,1.3) {$1$};
\node (bottom) at (0,-1.3) {$2$};
\node [bullet,orange] (mu) at (0,0.6) {}; 
\draw [->-] (center) to (-1,0);
\draw [->-] (center) to (1,0);
\draw [->-,orange] (0,1) to (mu);
\draw [->-,orange] (mu) to (center);
\draw [->-,cyan] (0,-1) to (center);
\draw [->-] (0.86, 0.5) to (center.north east);
\draw [->-] (0.5, 0.86) to (mu);
\node at (1,0.6) {$6$};
\node at (0.6,1) {$5$};
\node at (-0.3,0.4) {$7$};
} \quad + \quad \tikzfig{
\node [rectangle,draw,thick,inner sep=6pt] (center) at (0,0) {};
\node (left) at (-1.3,0) {$3$};
\node (right) at (1.3,0) {$4$};
\node (top) at (0,1.3) {$1$};
\node (bottom) at (0,-1.3) {$2$};
\node [bullet] (mu) at (0.6,0.6) {};
\draw [->-] (center) to (-1,0);
\draw [->-] (center) to (1,0);
\draw [->-,orange] (0,1) to (center);
\draw [->-,cyan] (0,-1) to (center);
\draw [->-] (1, 0.5) to (mu);
\draw [->-] (0.5, 1) to (mu);
\draw [->-] (mu) to (center.north east);
\node at (1.3,0.6) {$6$};
\node at (0.6,1.3) {$5$};
\node at (0.3,0.7) {$7$};
} \quad + \quad \tikzfig{
\node [rectangle,draw,thick,inner sep=6pt] (center) at (0,0) {};
\node (left) at (-1.3,0) {$3$};
\node (right) at (1.3,0) {$4$};
\node (top) at (0,1.3) {$1$};
\node (bottom) at (0,-1.3) {$2$};
\node [bullet] (mu) at (0.6,0) {}; 
\draw [->-] (center) to (-1,0);
\draw [->-=.8] (center) to (mu);
\draw [->-=.8] (mu) to (1,0);
\draw [->-,orange] (0,1) to (center);
\draw [->-,cyan] (0,-1) to (center);
\draw [->-] (0.86, 0.5) to (mu);
\draw [->-] (0.5, 0.86) to (center.north east);
\node at (1,0.6) {$6$};
\node at (0.6,1) {$5$};
\node at (0.5,-0.3) {$7$};
}\]
with orientation $(\dots\ 6\ 5\ 4\ 3\ 2\ 1\ \dots)$ on the left-hand side and $(7\ \dots 6\ 5\ 4\ 3\ 2\ 1\ \dots)$ on the right-hand side. Consider now the quasi-isomorphisms
\begin{equation*}
    \R\!\Hom_{A^e \otimes A^e}(M\otimes M^!,A^{op}\otimes A) \xrightarrow{\sim} \R\!\Hom_{A^e}(M, A\otimes^\LL_A M\otimes^\LL_A A) \xrightarrow{\sim} \R\!\Hom_{A^e}(M,M)
\end{equation*}
given by composition on the bottom with the coevaluation vertex, followed by the isomorphism $A\otimes^\LL_A M\otimes^\LL_A A \xrightarrow{\sim} M$: 
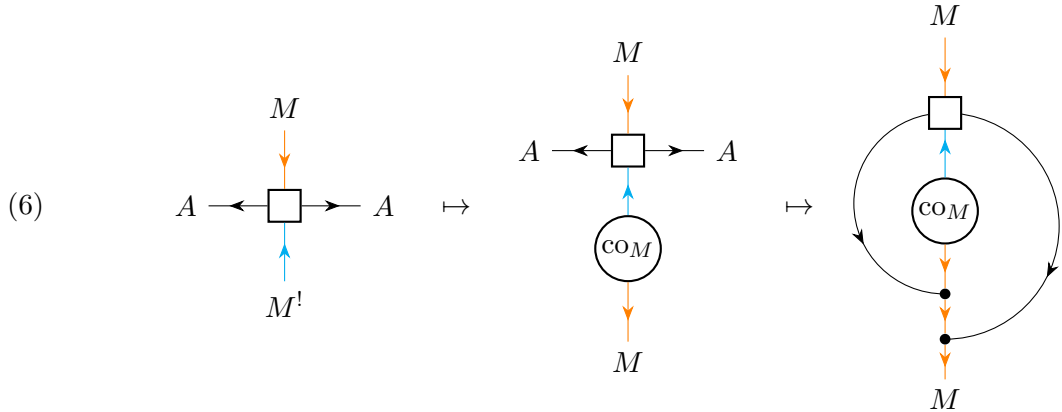
\begin{align}\label{align:pairingvertex}
    \tikzfig{
    \node [rectangle,draw,thick,inner sep=6pt] (center) at (0,0) {};
    \node (left) at (-1.3,0) {$A$};
    \node (right) at (1.3,0) {$A$};
    \node (top) at (0,1.3) {$M$};
    \node (bottom) at (0,-1.3) {$M^!$};
    \draw [->-] (center) to (-1,0);
    \draw [->-] (center) to (1,0);
    \draw [->-,orange] (0,1) to (center);
    \draw [->-,cyan] (0,-1) to (center);
} \quad \mapsto \quad \tikzfig{
    \node [rectangle,draw,thick,inner sep=6pt] (center) at (0,0) {};
    \node (left) at (-1.3,0) {$A$};
    \node (right) at (1.3,0) {$A$};
    \node (top) at (0,1.3) {$M$};
    \node [vertex] (coev) at (0,-1.3) {$\co_M$};
    \node (bottom) at (0,-2.8) {$M$};
    \draw [->-] (center) to (-1,0);
    \draw [->-] (center) to (1,0);
    \draw [->-,orange] (0,1) to (center);
    \draw [->-,cyan] (coev) to (center);
    \draw [->-,orange] (coev) to (bottom);
} \quad \mapsto \quad \tikzfig{
    \node [bullet] (center) at (0,0) {};
    \node (top) at (0,1.3) {$M$};
    \node [vertex] (coev) at (0,-1.3) {$\co_M$};
    \node [bullet] (mu1) at (0,-2.4) {};
    \node [bullet] (mu2) at (0,-3) {};
    \node (bottom) at (0,-3.8) {$M$};
    \draw [->-,orange] (0,1) to (center);
    \draw [->-,cyan] (coev) to (center);
    \draw [->-,orange] (coev) to (mu1);
    \draw [->-] (center) arc (90:270:1.2);
    \draw [->-] (center) arc (90:-90:1.5);
    \draw [->-,orange] (mu1) to (mu2);
    \draw [->-,orange] (mu2) to (bottom);
    \node [rectangle,draw,thick,inner sep=6pt,fill=white] at (0,0) {};
} 
\end{align}
\begin{definition}\label{def:canonicalPairing}
    A \emph{pairing vertex} $\eta_M$ of a perfect bimodule $M$ is a closed element of degree zero in the complex $\R\!\Hom_{A^e\otimes A^e}(M\otimes M^!,A^{op}\otimes A)$ which maps, in homology, to the identity morphism in $\R\!\Hom_{A^e}(M,M)$.
\end{definition}

Strictly speaking, only the cohomology classes of $\co_M$ and $\eta_M$ are well-defined from the data of $A$ and $M$, while the chain-level representatives depend on choices. We will just fix a choice for these vertices, and refer to them as \textit{the} coevaluation and pairing vertices.

\subsubsection{Definition of chain-level Chern character}
Let $A$ be any $A_\infty$-category, not necessarily smooth, and $M$ a perfect bimodule. In this subsection we give a chain-level expression for a certain element associated to $M$, using (a choice of) the vertices $\co_M,\eta_M$.
\begin{definition}\label{def:ChernCharacter}
    The \textit{chain-level Chern character of the perfect bimodule $M$ with respect to $\co_M$ and $\eta_M$} is the element of degree zero
    \[ E_M \in C_*(A,A) \otimes C_*(A,A)\]
    given by the evaluating the following oriented diagram on the `elbow' (genus 0 surface with two outputs):
    \[\tikzfig{
        \draw (-2,0) [partial ellipse =0:-180:1.5 and 0.3];
        \draw [dashed] (-2,0) [partial ellipse =0:180:1.5 and 0.3];
        \draw (2,0) [partial ellipse =0:-180:1.5 and 0.3];
        \draw [dashed] (2,0) [partial ellipse =0:180:1.5 and 0.3];
        \draw (0.5,0) arc (0:180:0.5);
        \draw (3.5,0) arc (0:180:3.5);
        \node [vertex] (coev) at (-0.5,2.5) {$\co_M$};
        \node (eta) at (-0.5,1) {};
        \draw [->-,orange] (coev) to (eta);
        \draw (0,2) [->-,partial ellipse=90:-90:0.3 and 1.5,cyan,dashed];
        \draw [->-,bend left=20,cyan] (coev) to (0,3.5);
        \draw [bend left=20,cyan] (0,0.5) to (eta);
        \draw [->-] (eta) arc (115:2.5:1.5); 
        \draw [->-] (eta) arc (100:184.5:1.2); 
        \node [rectangle,draw,thick,fill=white] at (-0.5,1) {$\eta_M$};
        \node at (-1,2) {$1$};
        \node at (0.8,2.5) {$2$};
        \node at (-1.5,0.7) {$3$};
        \node at (2,1) {$4$}; 
    } \quad , \quad (4\ 3\ 2\ 1) \]
\end{definition}
Since $\co_M$ and $\eta_M$ are closed, so is $E_{M} \in C_*(A,A) \otimes C_*(A,A)$. We can define the Chern character class of $M$ as the class $[E_{M}] \in H_0(C_*(A,A) \otimes C_*(A,A))$. The element $E_M$ is only well-defined at the chain level once we make choices for $\eta_M$ and $\co_M$, but its homology class is canonically well-defined since those two vertices have canonically well-defined classes.  
Recall that the $A_\infty$-category $A$ is \emph{smooth} if its diagonal bimodule $A$ is perfect; in that case we denote simply $\co = \co_A, \eta = \eta_A$ and $E = E_A$.

\begin{remark}
    For a general smooth $A$, there is no reason for the Chern character $E$ to have any sort of symmetry under the $\Z/2\Z$ action on $C_*(A,A) \otimes C_*(A,A)$, \textit{even at the homology level}. On the other hand, we will see later in \cref{thm:symmetry} that the existence of a Calabi-Yau structure on $A$ imposes (skew)symmetry at the homology level.
\end{remark}

\begin{remark}
    We may see from \cite[Proposition 1.2.4]{PolishchukVaintrob} that the definition of $E_M$ above coincides with the Chern character of $M$ of Shklyarov \cite{shklyarov2013hirzebruch}.
\end{remark}

\subsection{Coproduct}\label{sec:coproduct}
Let $A$ be a smooth $A_\infty$-category, and let us fix representatives for the vertices $\co$ and $\eta$. We define a map of graded complexes
\[ G: C^*(A,A) \to C_*(A,A) \otimes C_*(A,A)[-1] \]
where $G(\varphi)$ is given by evaluating the following combination of diagrams:
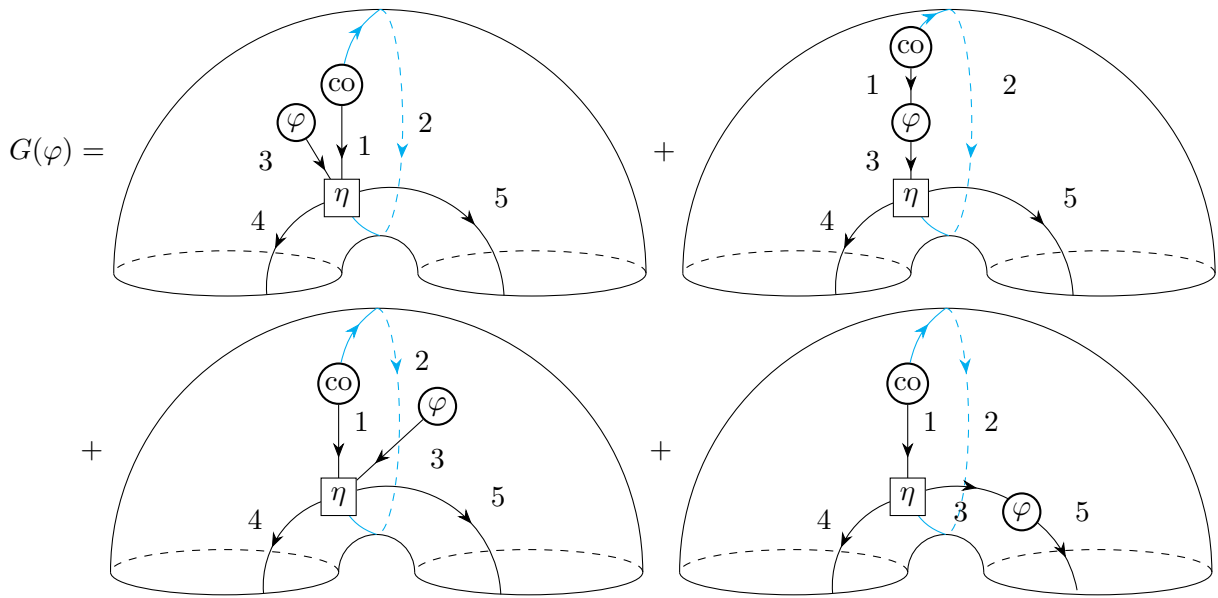
\begin{align*}
    &+ \tikzfig{
        \draw (-2,0) [partial ellipse =0:-180:1.5 and 0.3];
        \draw [dashed] (-2,0) [partial ellipse =0:180:1.5 and 0.3];
        \draw (2,0) [partial ellipse =0:-180:1.5 and 0.3];
        \draw [dashed] (2,0) [partial ellipse =0:180:1.5 and 0.3];
        \draw (0.5,0) arc (0:180:0.5);
        \draw (3.5,0) arc (0:180:3.5);
        \node [vertex] (coev) at (-0.5,2.5) {$\co$};
        \node [vertex] (phi) at (-1.1,2) {$\varphi$};
        \node (eta) at (-0.5,1) {};
        \draw [->-] (coev) to (eta);
        \draw [->-] (phi) to (eta);
        \draw (0,2) [->-,partial ellipse=90:-90:0.3 and 1.5,cyan,dashed];
        \draw [->-,bend left=20,cyan] (coev) to (0,3.5);
        \draw [bend left=20,cyan] (0,0.5) to (eta);
        \draw [->-] (eta) arc (115:2.5:1.5); 
        \draw [->-] (eta) arc (100:184.5:1.2); 
        \node [rectangle,draw,fill=white] at (-0.5,1) {$\eta$};
        \node at (-0.2,1.7) {$1$};
        \node at (0.6,2) {$2$};
        \node at (-1.5,1.5) {$3$};
        \node at (-1.6,0.7) {$4$};
        \node at (1.6,1) {$5$};
    } + \tikzfig{
        \draw (-2,0) [partial ellipse =0:-180:1.5 and 0.3];
        \draw [dashed] (-2,0) [partial ellipse =0:180:1.5 and 0.3];
        \draw (2,0) [partial ellipse =0:-180:1.5 and 0.3];
        \draw [dashed] (2,0) [partial ellipse =0:180:1.5 and 0.3];
        \draw (0.5,0) arc (0:180:0.5);
        \draw (3.5,0) arc (0:180:3.5);
        \node [vertex] (coev) at (-0.5,3) {$\co$};
        \node [vertex] (phi) at (-0.5,2) {$\varphi$};
        \node (eta) at (-0.5,1) {};
        \draw [->-] (coev) to (phi);
        \draw [->-] (phi) to (eta);
        \draw (0,2) [->-,partial ellipse=90:-90:0.3 and 1.5,cyan,dashed];
        \draw [->-,bend left=20,cyan] (coev) to (0,3.5);
        \draw [bend left=20,cyan] (0,0.5) to (eta);
        \draw [->-] (eta) arc (115:2.5:1.5); 
        \draw [->-] (eta) arc (100:184.5:1.2); 
        \node [rectangle,draw,fill=white] at (-0.5,1) {$\eta$};
        \node at (-1,2.5) {$1$};
        \node at (0.8,2.5) {$2$};
        \node at (-1,1.5) {$3$};
        \node at (-1.6,0.7) {$4$};
        \node at (1.6,1) {$5$};
    } \\
    &+ \tikzfig{
        \draw (-2,0) [partial ellipse =0:-180:1.5 and 0.3];
        \draw [dashed] (-2,0) [partial ellipse =0:180:1.5 and 0.3];
        \draw (2,0) [partial ellipse =0:-180:1.5 and 0.3];
        \draw [dashed] (2,0) [partial ellipse =0:180:1.5 and 0.3];
        \draw (0.5,0) arc (0:180:0.5);
        \draw (3.5,0) arc (0:180:3.5);
        \node [vertex] (coev) at (-0.5,2.5) {$\co$};
        \node [vertex] (phi) at (0.8,2.2) {$\varphi$};
        \node (eta) at (-0.5,1) {};
        \draw [->-] (coev) to (eta);
        \draw [->-] (phi) to (eta);
        \draw (0,2) [->-=0.3,partial ellipse=90:-90:0.3 and 1.5,cyan,dashed];
        \draw [->-,bend left=20,cyan] (coev) to (0,3.5);
        \draw [bend left=20,cyan] (0,0.5) to (eta);
        \draw [->-] (eta) arc (115:2.5:1.5); 
        \draw [->-] (eta) arc (100:184.5:1.2); 
        \node [rectangle,draw,fill=white] at (-0.5,1) {$\eta$};
        \node at (-0.2,2) {$1$};
        \node at (0.6,2.8) {$2$};
        \node at (0.8,1.5) {$3$};
        \node at (-1.6,0.7) {$4$};
        \node at (1.6,1) {$5$};
    } + \tikzfig{
        \draw (-2,0) [partial ellipse =0:-180:1.5 and 0.3];
        \draw [dashed] (-2,0) [partial ellipse =0:180:1.5 and 0.3];
        \draw (2,0) [partial ellipse =0:-180:1.5 and 0.3];
        \draw [dashed] (2,0) [partial ellipse =0:180:1.5 and 0.3];
        \draw (0.5,0) arc (0:180:0.5);
        \draw (3.5,0) arc (0:180:3.5);
        \node [vertex] (coev) at (-0.5,2.5) {$\co$};
        \node (phi) at (1,0.8) {};
        \node (eta) at (-0.5,1) {};
        \draw [->-] (coev) to (eta);
        \draw (0,2) [->-=0.3,partial ellipse=90:-90:0.3 and 1.5,cyan,dashed];
        \draw [->-,bend left=20,cyan] (coev) to (0,3.5);
        \draw [bend left=20,cyan] (0,0.5) to (eta);
        \draw [->-] (eta) arc (115:60:1.5);
        \draw [->-] (phi) arc (60:10:1.5); 
        \draw [->-] (eta) arc (100:184.5:1.2); 
        \node [vertex,fill=white] at (1,0.8) {$\varphi$};
        \node [rectangle,draw,fill=white] at (-0.5,1) {$\eta$};
        \node at (-0.2,2) {$1$};
        \node at (0.6,2) {$2$};
        \node at (0.2,0.8) {$3$};
        \node at (-1.6,0.7) {$4$};
        \node at (1.8,0.8) {$5$};
    }
\end{align*}
all four with orientation $(5\ 4\ 3\ 2\ 1)$. Note that $G$ is \textit{not} a map of complexes. Instead, computing the differential gives
\begin{align} \label{align:homotopyG}
d(G(\varphi)) + G(d\varphi) = (\varphi \frown \otimes \id)E  - (\id \otimes \varphi \frown)E.
\end{align}
where the cap product is the operation we discussed in \cref{ex:capProducts}.

\begin{definition}\label{def:trivializationEnonzero}
    Suppose $W \subseteq C_*(A,A)$ is a subcomplex of $\kk$-modules, such that the induced map on homology $H_*(W) \to HH_*(A,A)$ is injective. A \textit{trivialization of $E$ onto $W$} consists of a pair $(E_0, H)$ such that 
    \[ E_0 \in W \otimes C_*(A,A) \cap C_*(A,A) \otimes W \subseteq C_*(A,A) \otimes C_*(A,A) \]
    and $H \in C_*(A,A) \otimes C_*(A,A)$ satisfies $dH = E - E_0$.
\end{definition}

Note that since $E$ is closed, if $(E_0,H)$ is a trivialization then $E_0$ is automatically closed. It follows from the definitions and \cref{align:homotopyG} that the map of graded $\kk$-modules $\widetilde\lambda_H$ defined by
\[ \widetilde\lambda_H(\varphi) \coloneqq G(\varphi) - (-1)^{\deg(\varphi)}((\varphi \frown \otimes \id)H - (\id \otimes \varphi \frown)H) \]
gives a map of complexes
\[ \widetilde\lambda_H: C^*(A,A)\to \frac{C_*(A,A) \otimes C_*(A,A)}{(W \otimes C_*(A,A)) + (C_*(A,A) \otimes W)}[n-1], \]
that is, a map intertwining the differentials and thus giving a map on homology. Recall now that a closed element $\alpha \in C^n_{(2)}(A)$ defines a map of complexes $g_\alpha: C_*(A,A) \to C^*(A,A)[n]$.

\begin{definition}\label{def:chainAlgebraicLoopCoproduct}
    Let $A$ be a smooth $A_{\infty}$-category, a trivialization $H$ of $E$ onto $W$, and a closed element $\alpha \in C^n_{(2)}(A)$. The \emph{chain-level algebraic loop coproduct} associated to this data is the map of complexes
    \[ \lambda_H \colon C_*(A,A) \to \frac{C_*(A,A) \otimes C_*(A,A)}{(W \otimes C_*(A,A)) + (C_*(A,A) \otimes W)}[n-1] \]
    given by $\lambda_H \coloneqq \widetilde\lambda_H \circ g_\alpha$. The \emph{homology algebraic loop coproduct} is the map induced by $\lambda_H$ on homology.
\end{definition}

\begin{proposition}
    The homology algebraic loop coproduct $\lambda_H$ only depends on $H$ up to exact terms and up to terms living in $(W \otimes C_*(A,A)) \cap (C_*(A,A) \otimes W)$.
\end{proposition}

\cref{def:chainAlgebraicLoopCoproduct} makes sense for any ring $\kk$ over which the data $(A,\co,\eta,\alpha,H)$ is defined. When $\kk$ is a field, the K\"unneth map is an isomorphism and with the appropriate shift we get a map
\[ \lambda_H \colon HH_*(A,A)[n-1] \to \overline{HH}_*(A,A)[n-1] \otimes \overline{HH}_*(A,A)[n-1], \]
where $\overline{HH}_*(A,A) = H_*(C_*(A,A)/W) \cong HH_*(A,A)/H_*(W)$ since we required $H_*(W) \hookrightarrow \HH_*(A,A)$. Let us now define a certain compatibility condition which will be important for our study of this operation.
\begin{definition}\label{def:balancedText}
    The triple $(W,H,\alpha)$ is \emph{balanced} when the image of the subcomplex $W$ under the map of graded $\kk$-modules
    \[ \widetilde\lambda_H \circ g_\alpha \colon C_*(A,A) \to C_*(A,A) \otimes C_*(A,A)[n-1] \]
    is contained in $C_*(A,A) \otimes W \cap W \otimes C_*(A,A)$.
\end{definition}
In particular, if $(W,H,\alpha)$ is balanced, the map of complexes $\lambda_H$ descends to a map of complexes
\[ \overline\lambda_H \colon \frac{C_*(A,A)}{W} \to \frac{C_*(A,A) \otimes C_*(A,A)}{(W \otimes C_*(A,A)) + (C_*(A,A) \otimes W)}[n-1], \]
whose induced map on homology, over a field $\kk$, gives an honest coproduct on reduced Hochschild homology $\overline{HH}_*(A,A)[n-1]$. We will be interested in this coproduct when $\alpha = m_{(2)}$ comes from a (truncated) pre-CY structure on $A$; it will not be coassociative or cocommutative in general, but will behave nicely for some choices of $H$, compare \cref{thm:commCoproduct}.

\begin{remark}
    Note that in \cref{def:balancedText}, we really do mean the image is contained in the smaller subcomplex $W \otimes C_*(A,A) \cap C_*(A,A) \otimes W $, and not just in $W \otimes C_*(A,A) + C_*(A,A) \otimes W$; this is a stronger condition than saying that $\lambda_H$ descends to the quotient, but it turns out to be necessary for our later results.
\end{remark}

We do not require $W$ to be concentrated in degree zero, since in some applications we may want to take quotients in other degrees as well. Nonetheless, since $[E]$ is a class of degree zero, one can always find trivializations onto some $W_0$ concentrated in degree zero; under some extra assumptions this is always a balanced choice.
\begin{proposition}\label{prop:coproductFactors}
   If $A$ is connective, $n \ge 2$, and $(E_0,H)$ is a trivialization onto a subcomplex $W$ concentrated in degree zero, then the triple $(W,H,\alpha)$ is balanced.
\end{proposition}
\begin{proof}
    Since the subcomplex $W$ is concentrated in degree zero and $n \ge 2$, the image $\lambda_H(W)$ is in negative degrees so it vanishes since by the connectivity assumption.
\end{proof}

When the homology class of $E$ is zero, we can pick $W = 0$, in which case we will just call some $H$ of degree $1$ in $C_*(A,A)\otimes C_*(A,A)$ such that $d H = E$ a \textit{trivialization of $E$}. For any such choice, and any closed $\alpha \in C^*_{(2)}(A)$, the triple $(W,H,\alpha)$ is automatically balanced, and we get a map of complexes $\lambda_H: C_*(A,A) \to C_*(A,A) \otimes C_*(A,A)[n-1]$ giving a product on Hochschild homology $HH_*(A,A)[n-1]$.

\subsection{Covariance of the space of trivializations with respect to choices}\label{sec:covariance}
The map of complexes $\widetilde\lambda_H$ in \cref{def:chainAlgebraicLoopCoproduct} is well-defined on homology for a given equivalence class of the data $(\co,\eta, H)$.  In fact, let us be more precise and use subscripts to write
\[ \widetilde\lambda_{\co,\eta,H}(\varphi) = G_{\co,\eta}(\varphi) - (-1)^{\deg(\varphi)}((\varphi \frown \otimes \id - \id \otimes \varphi \frown)H). \] 
If $H-H'$ is exact, then $\widetilde\lambda_{\co,\eta,H}$ and $\widetilde\lambda_{\co,\eta,H'}$ are chain homotopic maps of complexes and that space of choices of $H$ modulo exact terms is a torsor over $H_1(C_*(A,A)\otimes C_*(A,A))$. Suppose now that we have two representatives for the coevaluation element, differing by an exact term, that is $\co'= \co + dz$ for some $z \in A \otimes^\LL_{A^e} A^!$. We then have
\[ G_{\co',\eta}(\varphi) \simeq G_{\co,\eta}(\varphi)-(-1)^{\deg(\varphi)}((\varphi \frown \otimes \id- \id \otimes \varphi \frown)H_z)  \]
where $H_z$ is an element of degree $1$ in $C_*(A,A) \otimes C_*(A,A)$, given by evaluating the same diagram as for the Chern character $E$ (\cref{def:ChernCharacter}) but with $z$ in the place of $\co$. This element satisfies  $dH_z = E_{\co',\eta} - E_{\co,\eta}$.

In other words, changing $\co$ by an exact term shifts the space of partial trivializations. We can compensate by changing $H$ accordingly to $H-H_z$; there is a homotopy $\lambda_{\co+dz,\eta,H-H_z} \simeq \lambda_{\co,\eta,H}$. An analogous description applies to changing the chain-level representative for $\eta$. We summarize these observations in the following
\begin{proposition}
    At the homology level, $\lambda_{\co,\eta,H}$ only depends on the class of the data $(\co,\eta,H)$ under the equivalence relation generated by
    \[ (\co,\eta,H) \sim (\co + dz,\eta+dw, H-H_{z,w}) \]
    where $H_{z,w}$ is an element of degree $1$ in $C_*(A,A)\otimes C_*(A,A)$ satisfying $dH_{z,w} = E_{\co+dz,\eta+dw} - E_{\co,\eta}$.
\end{proposition}

\subsection{Infinitesimal bialgebra relation}
We will now study the relation between the algebraic loop coproduct we described above and the product which is part of the framed $E_2$-algebra structure induced by the pre-CY structure.

Let us start by stating and proving a lemma about the operation $\widetilde\lambda_H$. Recall the cup product operation $\smile$ on Hochschild cochains, which is commutative on homology. We can pick an explicit homotopy realizing this commutativity: let $Q:C^*(A,A)^{\otimes 2} \to C^*(A,A)[-1]$ be the operation given by evaluating the oriented diagram
\[
    \tikzfig{
	\draw (0,0.7) circle (1.2);
	\node [vertex] (a) at (0,1.3) {I};
	\node [vertex] (b) at (0,0.3) {II};
	\draw [->-] (a) to (b);
	\draw [->-] (b) to (0,-0.5);
        \node at (-0.3,0.8) {$1$};
        \node at (-0.3,-0.1) {$2$};
    }\ ,\ (2\ 1)
\]
giving the relation $[d,Q](\varphi,\psi) = \varphi \smile \psi - (-1)^{\deg(\varphi)\deg(\psi)} \psi \smile \varphi$.

Let us pick, as before, an element $H$ such that $dH = E - E_0$. 
\begin{lemma}\label{lem:compatibilityWithCup}
    For any $\varphi, \psi \in C^*(A,A)$, we have the relation
    \begin{align*}
        \widetilde\lambda_H(\varphi \smile \psi) &\simeq (-1)^{\deg(\varphi)} (\varphi \frown \otimes \id) \widetilde\lambda_H(\psi) + (-1)^{\deg(\varphi)\deg(\psi) + \deg(\psi)} (\id \otimes \psi \frown)\widetilde\lambda_H(\varphi)\\
        &-(\id \otimes Q(\varphi,\psi)\frown)E_0
    \end{align*}
\end{lemma}
\begin{proof}
    Let us start with the terms involving the map $G$. We consider the two maps $C^*(A,A)^{\otimes 2} \to C_*(A,A)[-1]$ which take $(\varphi,\psi) \mapsto G(\varphi\smile\psi)$ and
    \[ (\varphi,\psi) \mapsto (-1)^{\deg(\varphi)}(\varphi\frown\otimes\id)G(\psi) + (-1)^{\deg(\varphi)\deg(\psi)+\deg(\psi)}(\id\otimes \psi\frown)G(\psi) -(\id\otimes Q(\varphi,\psi)\frown)E \]
    and find a homotopy between them, given in \cref{app:ibl}.

    As for the terms involving $H$, we have easier homotopies (which we omit) giving relations
    \begin{align*}
        ((\varphi\smile\psi)\frown\otimes \id)H &\simeq (\varphi\frown\otimes\id)(\psi\frown\otimes\id)H \\
        (\id\otimes(\varphi\smile\psi - (-1)^{\deg(\varphi)\deg(\psi)} \psi\smile\varphi)\frown)H &\simeq (-1)^{\deg(\varphi)+\deg(\psi)} (\id\otimes Q(\varphi\smile\psi))(E-E_0)
    \end{align*}
    Combining the three relations above, and then adding and subtracting the term
    \[ (-1)^{\deg(\varphi)+\deg(\psi)}(\varphi\frown\otimes\id)(\id\otimes\psi\frown)H \]
    on the right-hand side gives us the desired relation.
\end{proof}

Given a pre-CY structure on $A$, in \cref{sssec:algebraicloopproduct} we described a map of complexes
\[ \pi: C_*(A,A)[-n] \otimes C_*(A,A)[-n] \to C_*(A,A)[-n] \]
inducing an associative and $(-1)^n$-commutative product on $HH_*(A,A)$. We establish a compatibility relation between $\pi$ and $\lambda_H$ when $E_0 = 0$.
\begin{theorem}\label{thm:sullivan}
   Suppose $A$ is an $A_{\infty}$-category over a field $\kk$ equipped with a ($2$-truncated) pre-CY structure $m$ of dimension $n$. Furthermore, suppose $[E]=0$ and $H$ is a trivialization of $E$. The algebraic loop product
   \[ \pi: HH_*(A,A)[-n] \otimes HH_*(A,A)[-n] \to HH_*(A,A)[-n] \]
   defined by $m$ and the algebraic loop coproduct 
   \[ \lambda_H \colon HH_*(A,A)[n-1] \to HH_*(A,A) [n-1] \otimes HH_*(A,A) [n-1] \]
   satisfy the following relation, known as the \emph{infinitesimal bialgebra equation} (also called \emph{Sullivan's relation} in \cite{CieHinOan2}):
   \[ \lambda_H \circ \pi = (\pi \otimes \id)\circ(\id\otimes \lambda_H) + (\id \otimes \pi) \circ (\lambda_H \otimes \id).\]
\end{theorem}
\begin{proof}
    The product $\pi$ and the cap product $\frown$ are related by means of the map $g_\alpha$ of \cref{ex:galpha}; we can easily write down a homotopy giving the relation $\pi(x,y) \simeq g_\alpha(x) \frown y$. The desired result then follows by using this relation combined with the relation in  \cref{lem:compatibilityWithCup}, applied to $\varphi=g_\alpha(x),\psi=g_\alpha(y)$.
\end{proof}

\section{Categorical formal punctured neighborhood of infinity}\label{sec:catFormal}
In this section we recall the notion of the \emph{categorical formal punctured neighborhood of infinity} $\Ahatinfty$ of an $A_\infty$-category $A$. This notion was first defined by Efimov for dg categories in \cite{efimov2017categorical}, as a noncommutative analog of the category of perfect complexes on the formal neighborhood of the divisor at infinity in a compactification of a noncompact smooth variety $X$. We also establish a relationship between $\Ahatinfty$ and the Chern character $E_A$ of $A$ defined in \cref{def:ChernCharacter}.

\subsection{Definition}
The object $\Ahatinfty$ has recently appeared `on the other side of mirror symmetry' in \cite{ganatra2022rabinowitz} where it is proven to give an algebraic model for the `Rabinowitz Fukaya category' of a noncompact symplectic manifold. We will refer the reader to Section 2 of that paper for the precise definition of $\Ahatinfty$.  
For the purposes of this article we follow Proposition 2.11 in there and define $\Ahatinfty$ as follows.
\begin{definition}\label{Definition:Psi}
   Let $A$ be a smooth $A_{\infty}$-category. The bimodule $\Ahatinfty$ is the cone of the canonical map
    \[
    \Psi: C^*(A,A^\vee \otimes A) \to C^*(A,\Hom_\kk(A,A)), 
    \]
\end{definition}
Here $A^\vee$ is the bimodule given by the linear dual of $A$. We \emph{do not} assume that $A$ is finite-dimensional or proper; note that, without this assumption, $A^\vee$ will be `bigger' than $A$, in the sense that its own linear dual $A^{\vee\vee}$ will not be equivalent to $A$. In particular, the natural map $A^\vee \otimes A \to \Hom_\kk(A,A)$ is not a (quasi-)isomorphism and thus neither is $\Psi$ in general. 

Even though $A$ is not proper, there is a canonical evaluation map $\ev: A \otimes A^\vee \to \kk$. The map $\ev$ gives a map of complexes
\[ A \otimes^\LL_{A^e} A^\vee \to \kk, \]
which we represent in the graphical calculus as a vertex taking one arrow carrying the $A$ bimodule and one arrow carrying the $A^\vee$ bimodule.
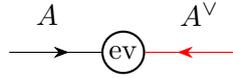
\[\tikzfig{
    \node [vertex] (center) at (0,0) {$\ev$};
    \node (left) at (-1,0.5) {$A$};
    \node (right) at (1,0.5) {$A^\vee$};
    \draw [->-] (-1.5,0) to (center);
    \draw [->-,red] (1.5,0) to (center);
}\]
The $A^\vee$ arrows will always be colored red from now on, to distinguish them more easily.

\subsubsection{Graphical calculus definition}
Let us describe the map of bimodules $\Psi$ in our graphical calculus. The $A$-bimodule structure on the complex $C^*(A, A^\vee \otimes A)$ comes from the structure of an tetramodule on $A^\vee \otimes A$; the two actions of $A$ on $C^*(A, A^\vee \otimes A)$ come from acting inside, on either side of $\otimes$. Again, to avoid discussing tetramodules, we just visualize an element in $C^*(A, A^\vee \otimes A)$ as a vertex of the following form.
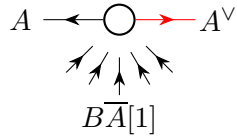
\[\tikzfig{
    \node [vertex,inner sep=4pt] (center) at (0,0) {};
    \node (left) at (-1.3,0) {$A$};
    \node (right) at (1.3,0) {$A^\vee$};
    \node at (0,-1.3) {$B\overline{A}[1]$};
    \draw [->-] (center) to (-1,0);
    \draw [->-,red] (center) to (1,0);
    \draw [->-] (0.7,-0.7) to (0.35,-0.35);
    \draw [->-] (0,-1) to (0,-0.5);
    \draw [->-] (-0.7,-0.7) to (-0.35,-0.35);
    \draw [->-] (-0.4,-0.86) to (-0.15,-0.43);
    \draw [->-] (0.4,-0.86) to (0.15,-0.43);
}\]
The map $\Psi$ in \cref{Definition:Psi} is then given by the diagram
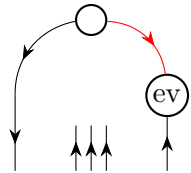
\[\tikzfig{
    \node [vertex,inner sep=4pt] (beta) at (-1,2) {};
    \node [vertex] (ev) at (0,1) {$\ev$};
    \draw [->-,red,bend left=40] (beta) to (ev);
    \draw [->-,bend right=40] (beta) to (-2,1);
    \draw [->-] (-2,1) to (-2,0);
    \draw [->-] (0,0) to (ev);
    \draw [->-] (-1.2,0) to (-1.2,0.6);
    \draw [->-] (-1,0) to (-1,0.6);
    \draw [->-] (-0.8,0) to (-0.8,0.6);
}\]
Here, the three little arrows at the bottom just schematically indicate the direction of the $\overline{A}[1]$-arrows at the bottom (the inputs of a Hochschild cochain in $C^*(A,\Hom_\kk(A,A))$).

\subsubsection{Cup product}
The differential and composition that endow $\Ahatinfty$ with the structure of a dg category are maps of $A$-bimodules.
\[ \mu^1_\infty: \Ahatinfty \to \Ahatinfty[1]\quad \text{and} \quad \mu^2_\infty: \Ahatinfty \otimes^\LL_A \Ahatinfty \to \Ahatinfty \]
We note that these structure maps give a cup product on Hochschild cochains valued in this bimodule:
\begin{definition}
    The cup product $\smile_\infty: C^*(A,\Ahatinfty) \otimes C^*(A,\Ahatinfty) \to C^*(A,\Ahatinfty)$ is defined by the oriented ribbon quiver
    \[\tikzfig{
	\draw (0,0.7) circle (1.2);
	\node [vertex,inner sep=4pt] (a) at (-0.6,1) {};
	\node [vertex,inner sep=4pt] (b) at (0.6,1) {};
	\node [bullet,orange] (c) at (0,0.3) {};
        \node at (-0.6,1.4) {$1$};
        \node at (+0.6,1.4) {$2$};
	\draw [->-,orange] (a) to (c);
	\draw [->-,orange] (b) to (c);
	\draw [->-,orange] (c) to (0,-0.5);
        \node at (-0.5,0.4) {$e_1$};
        \node at (0.5,0.4) {$e_2$};
        \node at (0.3,-0.1) {$e_3$};
    }\ ,\ (e_3\ e_2\ e_1) \]
    where the orange lines carry $\Ahatinfty$ and the orange dot in the middle represents $\mu^2_\infty$.
\end{definition}

\begin{remark}
    The only difference between the conventions we use here and the definition in \cite{ganatra2022rabinowitz} is that what we call $\Ahatinfty$ in their notation would technically be called $\widehat{(A^{op})}_\infty$, meaning the $A^{op}$-bimodule corresponding to the opposite $A_\infty$-category, but seen as an $A$-bimodule by reversing the two actions. This difference is mainly due to the fact that \textit{op.cit.} (and also \cite{ganatra2013symplectic}, for example) uses a convention for $A_\infty$ maps that is opposite to the one we use following \cite{kontsevich2021precalabiyau}.
\end{remark}

\subsection{Relation to Chern character of diagonal bimodule}
The definition for the bimodule $\Ahatinfty$ above makes sense for any $A_\infty$-category, smooth or not. But from now on, we will assume $A$ is smooth. We show that using the vertices $\co,\eta$ we defined before, it is possible to give an explicit relation between the categorical formal punctured neighborhood of infinity and the Chern character we defined in \cref{sec:chernCharacter}. 

\subsubsection{Another description of $\widehat{A}_\infty$}\label{sec:anotherDescr}
Let $A$ be a smooth $A_\infty$-category. We will define another bimodule $\Ahatinfty'$ associated to $A$ that will be quasi-isomorphic to $\Ahatinfty$. We pick vertices $\co$ and $\eta$ as in \cref{sec:chernCharacter}, and also let us pick a bimodule morphism $\beta: A \to A$ of degree $1$ which witnesses the relation between $\co$ and $\eta$, that is,
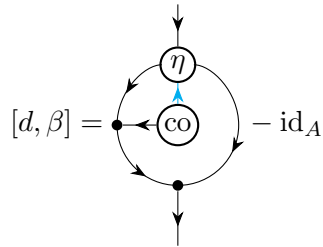
\[ [d,\beta] = \tikzfig{
    \node [vertex] (top) at (0,0.8) {$\eta$};
    \node [vertex] (mid) at (0,0) {$\co$};
    \node [bullet] (left) at (-0.8,0) {};
    \node [bullet] (down) at (0,-0.8) {};
    \draw [->-=0.8,cyan] (mid) to (top);
    \draw [->-] (mid) to (left);
    \draw [->-,shorten <=6pt] (0,0.8) arc (90:180:0.8);
    \draw [->-,shorten <=6pt] (0,0.8) arc (90:-90:0.8);
    \draw [->-] (-0.8,0) arc (180:270:0.8);
    \draw [->-] (down) to (0,-1.6);
    \draw [->-] (0,1.6) to (top);
} - \id_A \]

We define the bimodule $\Ahatinfty'$ to be the cone of the morphism
\[ \Psi': C^*(A, A^\vee \otimes A) \otimes^\LL_A A \to A \]
given by the following diagram, with some appropriate orientation:
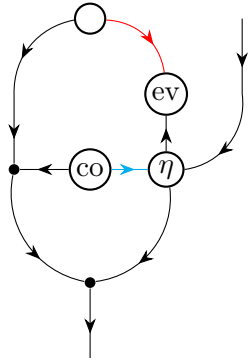
\[\tikzfig{
    \node [vertex,inner sep=4pt] (beta) at (-1,2) {};
    \node [vertex] (ev) at (0,1) {$\ev$};
    \node [vertex] (eta) at (0,0) {$\eta$};
    \node [vertex] (coev) at (-1,0) {$\co$};
    \node [bullet] (left) at (-2,0) {};
    \node [bullet] (down) at (-1,-1.5) {};
    \draw [->-,red,bend left=40] (beta) to (ev);
    \draw [->-,bend right=40] (beta) to (-2,1);
    \draw [->-] (-2,1) to (left);
    \draw [->-,bend right=45] (left) to (down);
    \draw [->-,bend left=45] (eta) to (down);
    \draw [->-] (down) to (-1,-2.5);
    \draw [->-] (1,2) to (1,1);
    \draw [->-,bend left=45] (1,1) to (eta);
    \draw [->-,cyan] (coev) to (eta);
    \draw [->-] (coev) to (left);
    \draw [->-] (eta) to (ev);
}\]

\begin{proposition}\label{prop:isoOfBimodules}
   Let $A$ be a smooth $A_\infty$-category. Then the bimodule $\Ahatinfty'$ is quasi-isomorphic to $\Ahatinfty$.
\end{proposition}
\begin{proof}
The natural quasi-isomorphism of bimodules $\varphi: A \to C^*(A,\Hom_\kk(A, A))$ can be given explicitly by the diagram 
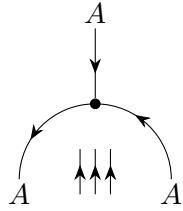
\[\tikzfig{
        \node at (0,1.2) {$A$};
        \node [bullet] (bottom) at (0,0) {};
        \node at (-1,-1.2) {$A$};
        \node at (1,-1.2) {$A$};
        \draw [->-] (0,1) to (bottom);
        \draw [->-] (1,-1) arc (0:90:1);
        \draw [->-] (0,0) arc (90:180:1);
        \draw [->-] (-0.2,-1.2) to (-0.2,-0.6);
        \draw [->-] (0,-1.2) to (0,-0.6);
        \draw [->-] (0.2,-1.2) to (0.2,-0.6);
}\]
Again, the three small arrows on the bottom just indicate schematically where the inputs are in the Hochschild cochain structure. In order to see that $\Ahatinfty'$ is quasi-isomorphic to $\Ahatinfty$, it suffices to exhibit a homotopy between the maps $\varphi \circ \Psi'$ and $\Psi \circ R$, where $R$ is the canonical quasi-isomorphism given by $\otimes^\LL_A A$. In diagrams, we need to find a homotopy
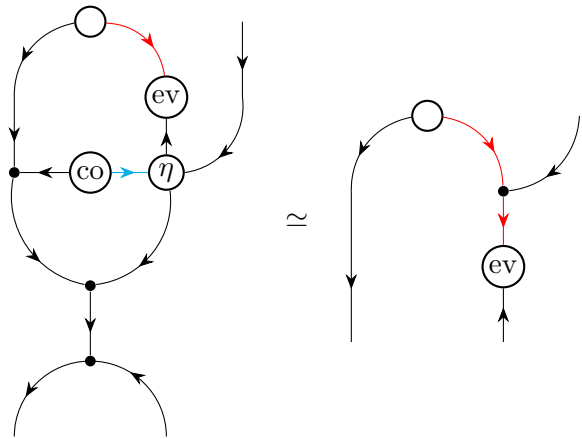
\[ \tikzfig{
    \node [vertex,inner sep=4pt] (beta) at (-1,2) {};
    \node [vertex] (ev) at (0,1) {$\ev$};
    \node [vertex] (eta) at (0,0) {$\eta$};
    \node [vertex] (coev) at (-1,0) {$\co$};
    \node [bullet] (left) at (-2,0) {};
    \node [bullet] (down) at (-1,-1.5) {};
    \node [bullet] (down2) at (-1,-2.5) {};
    \draw [->-,red,bend left=40] (beta) to (ev);
    \draw [->-,bend right=40] (beta) to (-2,1);
    \draw [->-] (-2,1) to (left);
    \draw [->-,bend right=45] (left) to (down);
    \draw [->-,bend left=45] (eta) to (down);
    \draw [->-] (down) to (down2);
    \draw [->-] (1,2) to (1,1);
    \draw [->-,bend left=45] (1,1) to (eta);
    \draw [->-,cyan] (coev) to (eta);
    \draw [->-] (coev) to (left);
    \draw [->-] (eta) to (ev);
    \draw [->-] (0,-3.5) arc (0:90:1);
    \draw [->-] (-1,-2.5) arc (90:180:1);
} \quad \simeq \quad \tikzfig{
    \node [vertex,inner sep=4pt] (beta) at (-1,3) {};
    \node [vertex] (ev) at (0,1) {$\ev$};
    \node [bullet] (mu) at (0,2) {};
    \draw [->-,red,bend left=40] (beta) to (mu);
    \draw [->-,red] (mu) to (ev);
    \draw [->-,bend left=40] (1,3) to (mu);
    \draw [->-,bend right=40] (beta) to (-2,2);
    \draw [->-] (-2,2) to (-2,0);
    \draw [->-] (0,0) to (ev);
}\]
for some appropriate orientation; we do this by using the element $\beta$ in \cref{app:AhatinftyBimodules}.
\end{proof}

\subsubsection{Triangle of Hochschild chain complexes}
Consider the distinguished triangle of complexes
\[ C_*(A, C^*(A,A^\vee \otimes A)) \xrightarrow{F} C_*(A,A) \to C_*(A,\Ahatinfty) \]
obtained from applying Hochschild chains to the distinguished triangle of bimodules defining $\Ahatinfty$. Since $A$ is smooth, we have quasi-isomorphisms of complexes
\[ C^*(A,A^\vee \otimes A) \simeq A^! \otimes^\LL_{A^e}(A^\vee \otimes A) \simeq A^! \otimes^\LL_A A^\vee \]
and also $C_*(A, A^! \otimes^\LL_A A^\vee) = A^! \otimes^\LL_{A^e} A^\vee \simeq C^*(A,A^\vee)$. Picking any inverse for these quasi-isomorphisms, we call the resulting map 
\[
F:C^*(A,A^\vee) \to C_*(A,A);\] then we must have a quasi-isomorphism $C^*(A, \Ahatinfty) \simeq \Cone(F)$. Note also that for any $A$, smooth or not, we have an isomorphism of complexes $C^*(A, A^\vee) \cong (C_*(A,A))^\vee$ induced by the canonical pairing
\[ \langle-,-\rangle: C^*(A, A^\vee) \otimes C_*(A,A) \to \kk .\]
 
\subsubsection{Comparison}
The following result is stated by Efimov in \cite[Subsection 10.2]{efimov2017categorical}, and also mentioned as a `folklore' lemma in \cite[Lemma 6.5]{ganatra2022rabinowitz}.
\begin{theorem}\label{thm:EulerCharacter}
    Let $A$ be smooth. The map $ ^\sharp E: C^*(A, A^\vee) \to C_*(A,A)$ defined by pairing with the first component of $E$, that is,
    \[  ^\sharp E(\varphi) = \langle \varphi, E' \rangle E'' \]
    (where we use Sweedler's notation $E = E' \otimes E''$) is homotopic to $F$. In other words, there is a quasi-isomorphism $C_*(A,\Ahatinfty) \simeq \Cone( ^\sharp E)$.
\end{theorem}
\begin{proof}
    Let us now apply $C_*(A,-)$ to the distinguished triangle given by $\Psi'$, to get a distinguished triangle of complexes
    \[ C_*(A,C^*(A,A^\vee \otimes A))\otimes^\LL_A A = C^*(A, A^\vee \otimes^\LL_A A \otimes^\LL_A A) \xrightarrow{F'} C_*(A,A) \to C_*(A,\Ahatinfty') \]
    whose map $F'$ is given by the diagram
    \[\tikzfig{
        \node (beta) at (-1,2) {};
        \node [vertex] (ev) at (0,1) {$\ev$};
        \node [vertex] (eta) at (0,0) {$\eta$};
        \node [vertex] (coev) at (-1,0) {$\co$};
        \node [bullet] (left) at (-2,0) {};
        \node [bullet] (down) at (-1,-1.5) {};
        \node [circ] (out) at (-1,-2.5) {};
        \draw [->-,red,bend left=40] (beta) to (ev);
        \draw [->-,bend right=40] (beta) to (-2,1);
        \draw [->-] (-2,1) to (left);
        \draw [->-,bend right=45] (left) to (down);
        \draw [->-,bend left=45] (eta) to (down);
        \draw [->-] (down) to (out);
        \draw [->-] (1,2) to (1,1);
        \draw [->-,bend left=45] (1,1) to (eta);
        \draw [->-,cyan] (coev) to (eta);
        \draw [->-] (coev) to (left);
        \draw [->-] (eta) to (ev);
        \draw [->-] (beta) arc (180:0:1);
        \node [vertex,fill=white,inner sep=4pt] at (-1,2) {};
    }\]

    Now we look at the map given by mapping $C^*(A, A^\vee\otimes^\LL_A A \otimes^\LL_A A) \xrightarrow{\cong} C^*(A,A^\vee)$ and then pairing with the first term in $E$. 
    \[\tikzfig{
        \node [bullet] (mu1) at (0,0) {};
        \node [bullet] (mu2) at (0.8,0.8) {};
        \node [vertex] (gamma) at (0,1.6) {};
        \node [vertex] (co) at (0,3) {$\co$};
        \node [vertex] (ev) at (0,-1) {$\ev$};
        \node [vertex] (eta) at (0,-2) {$\eta$};
        \node [circ] (out) at (0,-3) {};
        \draw [->-,cyan] (co) arc (90:270:2.5);
        \draw [->-] (co) arc (90:-90:2.5);
        \draw [->-,red] (gamma) arc (90:270:0.8);
        \draw [->-] (mu2) arc (0:-90:0.8);
        \draw [->-] (gamma) arc (90:0:0.8);
        \draw [->-] (gamma) arc (180:270:0.8);
        \draw [->-] (eta) to (ev);
        \draw [->-] (eta) to (out);
        \draw [->-,red] (mu1) to (ev);
        \node [vertex,fill=white,inner sep=4pt] (gamma) at (0,1.6) {};
        \node [vertex,fill=white] (eta) at (0,-2) {$\eta$};
        \node [vertex,fill=white] (co) at (0,3) {$\co$};
    }\]
    To prove the statement we must give a homotopy between the two diagrams above; we do that in \cref{app:relationToE}.
\end{proof}

\section{Pre-CY structures and product on their cone}\label{sec:preCYandProduct}
We now address one of the main points of this paper, which is that the formalism of pre-CY structures developed in \cite{kontsevich2021precalabiyau} as we recalled in \cref{sec:graphicalCalculus} can be used to explicitly describe a product on $C_*(A,\Ahatinfty)$ when that structure is non-degenerate. The graphical calculus can then be used to show the desired properties of the induced product on $\HH(A,\Ahatinfty)$. 

More generally, we will see that any (3-truncated) pre-CY structure $m = \mu + \alpha + \tau$, even if $\alpha$ is degenerate, already gives a product with those properties, but on $C_*(A,M_\alpha)$ for another bimodule $M_\alpha$. As suggested by the notation, this bimodule depends on $\alpha$; however, we will show that when $\alpha$ is non-degenerate it is quasi-isomorphic to $\Ahatinfty$.

\subsection{The cone bimodule of a pre-CY category}\label{sec:coneoff}
Let $(A,m=\mu+\alpha + \tau +\sigma)$ be a 4-truncated pre-CY category of dimension $n$. That is, we have $\alpha \in C^n_{(2)}(A), \tau \in C^{2n-2}_{(3)}(A)$ and $\sigma \in C^{3n-4}_{(4)}(A)$ such that
\[ \mu \circ \mu = 0, \quad [\mu,\alpha] = 0, \quad \alpha \circ \alpha + [\mu,\tau] = 0. \]
\subsubsection{Definition}\label{def:morphismFalpha}
We use the data of $\alpha$ to give an $A_\infty$-morphism of bimodules
\begin{align}\label{align:falpha} f_\alpha: A^\vee[-n] \to A, \qquad f_\alpha = \tikzfig{
    \node [vertex] (left) at (-1,0) {$\ev$};
    \node [vertex] (right) at (0,0) {$\alpha$};
    \draw [->-,red] (-2,0) to (left);
    \draw [->-] (right) to (left);
    \draw [->-] (right) to (1,0);
}\end{align}
Note that the diagram above encapsulates the data of many maps of complexes, since $\alpha$ can take any number of $A[1]$-arrows on each side.

\begin{remark}
    We would like to point out a possible source of confusion, which is that the nondegeneracy of $\alpha$, in the sense that the map $A^! \to A$ it induces is a quasi-isomorphism, does not imply that $f_\alpha:A^\vee \to A$ will be a quasi-isomorphism. In fact, as we will see in \cref{sec:examples}, $\alpha$ can be non-degenerate at the same time that $f_\alpha$ is nullhomotopic.
\end{remark}

\begin{definition}\label{definition:calM}
    The bimodule $M_\alpha$ is the cone of the map $f_\alpha$ above, that is, it is given by the bimodule
    \[ M_\alpha = (B\overline{A}[1] \otimes A^\vee[-n] \otimes B\overline{A}[1])[1] \oplus A \]
    with the differential given by the usual differentials on the two summands plus $f_\alpha$.
\end{definition}

\subsubsection{Extension of pre-CY structure}\label{subsubsec:extension}
We now describe how the data of the pre-CY structure on $A$ canonically extends to certain operations on $M_{\alpha}$. Let $(A,m=\mu + \alpha + \tau + \sigma)$ be a 4-truncated pre-CY structure, and let us denote $M = M_{\alpha}$ for the cone bimodule.
\begin{proposition}\label{prop:extension}
   Let $(A,m=\mu + \alpha + \tau + \sigma)$ be a 4-truncated pre-CY of dimension $n$. Then there are canonically defined vertices
    \[\tikzfig{
        \node [vertex] (mu) at (0,0) {$\mu^2_M$};
        \node (topleft) at (-1,1.2) {$M$};
        \node (topright) at (1,1.2) {$M$};
        \node (bottom) at (0,-1.5) {$M$};
        \draw [->-] (topleft) to (mu);
        \draw [->-] (topright) to (mu);
        \draw [->-] (mu) to (bottom);
    } \qquad \tikzfig{
        \node [vertex] (mu) at (0,0) {$\mu^3_M$};
        \node (topleft) at (-1,1.2) {$M$};
        \node (topcenter) at (0,1.2) {$M$};
        \node (topright) at (1,1.2) {$M$};
        \node (bottom) at (0,-1.5) {$M$};
        \draw [->-] (topleft) to (mu);
        \draw [->-] (topcenter) to (mu);
        \draw [->-] (topright) to (mu);
        \draw [->-] (mu) to (bottom);
    }\]
    of degrees zero and one, where $\mu^2_M$ is closed, extends the product $\mu^2$ on $A$, and $\mu^3_M$ satisfies the equation
    \[ [d,\mu^3_M] = \mu^2_M(\mu^2_M(-,-),-) -\mu^2_M(-,\mu^2_M(-,-))\]
    that is, $\mu^2_M$ is associative up to the boundary of $\mu^3_M$, and also a vertex
    \[\tikzfig{
        \node [vertex] (center) at (0,0) {$\psi$};
        \node (top) at (0,1.5) {$M$};
        \node (left) at (-1,0.5) {$M$};
        \node (right) at (1,0.5) {$M$};
        \draw [->-] (top) to (center);
        \draw [->-] (center) to (-1.5,0);
        \draw [->-] (center) to (1.5,0);
    }\]
    satisfying the equation
    \[ d\psi = \tikzfig{
        \node [vertex] (center) at (-0.4,0) {$\alpha$};
        \node [vertex] (mu) at (1,0) {$\mu^2_M$};
        \node (top) at (0,1.5) {};
        \draw [->-] (0.5,1) to (mu);
        \draw [->-] (center) to (-1.2,0);
        \draw [->-] (center) to (mu);
        \draw [->-] (mu) to (2.2,0);
    } - \tikzfig{
        \node [vertex] (mu) at (0,0) {$\mu^2_M$};
        \node [vertex] (center) at (1.4,0) {$\alpha$};
        \node (top) at (0,1.5) {};
        \draw [->-] (0.5,1) to (mu);
        \draw [->-] (center) to (2.2,0);
        \draw [->-] (center) to (mu);
        \draw [->-] (mu) to (-1.2,0);
    }\]
\end{proposition}
If one does not want to phrase this in the graphical language of vertices: $\mu^2_M$ is a map of $A_\infty$-bimodules
\[ \mu^2_M: M \otimes^\LL_A M \to M \]
while $\mu^3_M$ is a `nonclosed $A_\infty$-morphism' (a pre-morphism of bimodules in the language of \cite{ganatra2013symplectic})
\[ \mu^3_M: M \otimes^\LL_A M \otimes^\LL_A M \rightsquigarrow M [-1] \]
Likewise, $\psi$ is a `nonclosed' $A_\infty$-morphism'
\[ \psi: M \rightsquigarrow C^*(A,M\otimes M) \]

\begin{proof}
    To unclutter notation, let us denote $Z = A^\vee[1-n]$, $f = f_\alpha$, and $\tau = m_{(3)}$ (the part of the pre-CY structure on $A$ with 3 outputs). Since $M = (B\overline{A}[1]\otimes Z \otimes B\overline{A}[1]) \oplus A$, we have to describe 8 types of vertices, with inputs either in $A$ or in $(B\overline{A}[1]\otimes Z \otimes B\overline{A}[1])$, which we abbreviate by subscripts $A,Z$, respectively. We set:
    \begin{itemize}
        \item $\mu^2_{AA} = \mu^2$ (usual multiplication on $A$)
        \item $\mu^2_{AA,Z} = 0$
        \item $\mu^2_{AZ,Z}$ and $\mu^2_{ZA,Z}$ are the $A$-actions on the bimodule $Z$.
        \item $\mu^2_{AZ,A} = \tikzfig{
            \node [vertex] (ev) at (0.5,0.3) {$\ev$};
            \node [vertex] (alpha) at (0,-0.4) {$\alpha$};
            \draw [->-] (-1,1) to (alpha);
            \draw [->-,red] (1,1) to (ev);
            \draw [->-=0.8] (alpha) to (ev);
            \draw [->-] (alpha) to (0,-1.2);
        }, \qquad \mu^2_{ZA} = \tikzfig{
            \node [vertex] (ev) at (-0.5,0.3) {$\ev$};
            \node [vertex] (alpha) at (0,-0.4) {$\alpha$};
            \draw [->-,red] (-1,1) to (ev);
            \draw [->-] (1,1) to (alpha);
            \draw [->-=0.7] (alpha) to (ev);
            \draw [->-] (alpha) to (0,-1.2);
        }$
        \item $\mu^2_{ZZ,Z} = \tikzfig{
            \node [vertex] (evr) at (0.5,0.3) {$\ev$};
            \node [vertex] (evl) at (-0.5,0.3) {$\ev$};
            \node [vertex] (alpha) at (0,-0.5) {$\alpha$};
            \draw [->-,red] (-1,1) to (evl);
            \draw [->-,red] (1,1) to (evr);
            \draw [->-=0.7] (alpha) to (evr);
            \draw [->-=0.7] (alpha) to (evl);
            \draw [->-] (0,-1.2) to (alpha);
        }$, by which we mean that the output is an element of $Z$ which evaluates on an element of $A$ by plugging it into the bottom arrow.
        \item $\mu^2_{ZZ,A} = \tikzfig{
            \node [vertex] (evr) at (0.5,0.3) {$\ev$};
            \node [vertex] (evl) at (-0.5,0.3) {$\ev$};
            \node [vertex] (alpha) at (0,-0.5) {$\tau$};
            \draw [->-,red] (-1,1) to (evl);
            \draw [->-,red] (1,1) to (evr);
            \draw [->-=0.7] (alpha) to (evr);
            \draw [->-=0.7] (alpha) to (evl);
            \draw [->-] (alpha) to (0,-1.2);
        }$ 
    \end{itemize}
    One can check that summing all the terms above gives an operation
    \[ \mu^2_M: B\overline{A}[1] \otimes M \otimes B\overline{A}[1] \otimes M \otimes B\overline{A}[1] \to M \]
    which is closed, that is, a map of complexes giving a morphism in the category of bimodules of degree zero.

    In other words, the terms in $\mu^2_M$ are systematically obtained from diagrams with the vertices $\alpha = m_{(2)},\tau=m_{(3)}$, by dualizing the appropriate arrows. The formulas for $\mu^3_M$ are analogous, and involve $\alpha,\tau$ and $\sigma$. It remains to construct the element $\psi$, with one $M$ input and two $M$ outputs. Again, there are 8 terms to be specified.
    \begin{itemize}
        \item $\psi_{AA} = \tikzfig{
            \node [vertex] (center) at (0,0) {$\alpha$};
            \node (top) at (0,1.5) {};
            \draw [->-] (top) to (center);
            \draw [->-] (center) to (-1.5,0);
            \draw [->-] (center) to (1.5,0);
        }$
        \item $\psi_{AZ},\psi_{A,ZA},\psi_{A,ZZ},\psi_{Z,AZ}$ and $\psi_{Z,ZA}$ are all zero.
        \item $\psi_{Z,AA} = \tikzfig{
            \node [bullet] (center) at (0,0) {};
            \node [vertex] (top) at (0,1) {$\ev$};
            \draw [->-,red] (0,2) to (top);
            \draw [->-] (center) to (top);
            \draw [->-] (-1.5,0) to (center);
            \draw [->-] (1.5,0) to (center);
            \node at (0,-0.5) {};
        }$, interpreting, as we did for $\mu^2_{ZZ,Z}$ above, the bottom $A$ inputs as $Z \otimes Z$ outputs.
        \item $\psi_{Z,AA} = \tikzfig{
            \node [vertex] (center) at (0,0) {$\tau$};
            \node [vertex] (top) at (0,1) {$\ev$};
            \draw [->-,red] (0,2) to (top);
            \draw [->-] (center) to (top);
            \draw [->-] (center) to (-1.5,0);
            \draw [->-] (center) to (1.5,0);
        }$
    \end{itemize}
    It follows from the equations $[\mu,\alpha] = 0, [\mu,\tau] = \alpha \circ \alpha$, and from the previous definition of $\mu^2_M$, that these formulas give the desired equation for $d\psi$.
\end{proof}

\subsection{Extending products on chains and cochains}\label{sec:extendingproducts}
We continue to assume that $(A,m=\mu+\alpha+\tau+\sigma)$ is a 4-truncated pre-CY category of dimension $n$. We will now use the vertices defined in \cref{prop:extension} in order to construct a product structure on $C_*(A,M_{\alpha})$, and prove some of its properties. Here, $M_{\alpha}$ is given in \cref{definition:calM}. From now on, we make the convention that orange arrows carry the bimodule $M$, and all orange dot vertices represent the structure maps $\mu_M$.
\begin{definition}\label{def:productOnChains}
    Define a product $\pi_{M_\alpha}: C_*(A,M_\alpha) \otimes C_*(A,M_\alpha) \to C_*(A,M_\alpha)[n]$  by 
    \[ \pi_{M_\alpha}(x_1,x_2) = \Pi(\alpha,x_1,x_2), \]
    where $\Pi$ is the map $C^*_{(2)}(A) \otimes C_*(A,M_\alpha) \otimes C_*(A,M_\alpha) \to C_*(A,M_\alpha)$ given by the oriented ribbon quiver
    \[\tikzfig{
        \draw [->-] (0,1.5) arc (90:0:1.5);
        \draw [->-,orange] (1.5,0) arc (0:-90:1.5);
        \draw [-w-] (0,1.5) arc (90:180:1.5);
        \draw [->-,orange] (-1.5,0) arc (180:270:1.5);
        \node [inner sep=0pt,orange] (v1) at (-3,0) {$\times$};
        \node [bullet,orange] (v4) at (1.5,0) {};
        \node [vertex,fill=white,inner sep=4pt] (v3) at (0,1.5) {};
        \node [inner sep=0pt,orange] (v2) at (0,0) {$\times$};
        \node [bullet,orange] (v5) at (0,-1.5) {};
        \node [circ,orange] (v6) at (0,-3) {};
        \node [bullet,orange] (v7) at (-1.5,0) {};
        \draw [->-,shorten <=-3.5pt,orange] (v1) to (v7);
        \draw [->-,shorten <=-3.5pt,orange] (v2) to (v4);
        \draw [->-,orange] (v5) to (v6);
        \node at (0,-0.4) {II};
        \node at (-3,-0.4) {I};
        \node at (-1.3,1.3) {$1$};
        \node at (1.3,1.3) {$2$};;
        \node at (-1.3,-1.3) {$6$};
        \node at (1.3,-1.3) {$7$};
        \node at (-2.25,-0.4) {$3$};
        \node at (0.75,-0.4) {$4$};
        \node at (0.5,-2.25) {$8$};
    }\ ,\ (7\ 6\ 5\ 4\ 3\ 2\ 1).\]
\end{definition}

\begin{proposition}\label{prop:associativityProductCone}
    The chain-level product $\pi_{M_\alpha}$ is a map of complexes inducing an associative product of degree $-n$ on Hochschild homology $HH_*(A,M_\alpha)$. Furthermore, $\pi_{M_\alpha}$ extends the chain-level product $\pi$ on $C_*(A,A)$ described in \cref{align:productonchains}.
\end{proposition}
\begin{proof}
    The fact that $\pi_{M_\alpha}$ is a map of complexes follows from $\alpha$ and $\mu^2_{M_\alpha}$ being closed. To show that the induced map on Hochschild homology is associative we use the definition
    \begin{align*}
		\pi_{M_\alpha}(\pi_{M_\alpha}(x_1,x_2),x_3) &= \Pi(\alpha,\Pi(\alpha,x_1,x_2),x_3)), \\
    	(-1)^{n\deg(x_1)}\pi_{M_\alpha}(x_1,\pi_{M_\alpha}(x_2,x_3)) &= (-1)^{n\deg(x_1)}\Pi(\alpha,x_1,\Pi(\alpha,x_2,x_3))
	\end{align*}
    and draw the diagrams giving the chain-level expressions for these two maps $C^*_{(2)}(A)^{\otimes 2} \otimes (C_*(A,M_{\alpha}))^{\otimes 3} \to C_*(A,M_\alpha)$. We then find a homotopy between them, which we present in \cref{app:associativityCone}.

    The product $\pi_{M_\alpha}$ extends the product on $C_*(A,A)$ because $\mu^2_{AA,Z} = 0$ and the two lines coming out of $\alpha$ are in the image of $A \hookrightarrow M_\alpha$. 
    
\end{proof}

\subsubsection{The cup product}
We note that the vertex $\mu^2_{M_\alpha}$ also gives a product on Hochschild cohomology $C^*(A,M_\alpha)$, analogous to the usual cup product on $C^*(A,A)$.

\begin{definition}
    The chain-level cup product $\underset{M_\alpha}{\smile}: C^*(A,M_\alpha) \otimes C^*(A,M_\alpha) \to C^*(A,M_\alpha)$ is 
     defined by the following diagram
    \[\tikzfig{
	\draw (0,0.7) circle (1.2);
	\node [vertex] (a) at (-0.6,1) {I};
	\node [vertex] (b) at (0.6,1) {II};
	\node [bullet,orange] (c) at (0,0.3) {};
	\draw [->-,orange] (a) to (c);
	\draw [->-,orange] (b) to (c);
	\draw [->-,orange] (c) to (0,-0.5);
        \node at (-0.5,0.4) {$1$};
        \node at (0.5,0.4) {$2$};
        \node at (0.3,-0.1) {$3$};
    }\ ,\ (3\ 2\ 1) \]
    This is a map of complexes and we denote the induced map on cohomology equally by $\underset{M_\alpha}{\smile}$.
\end{definition}

It follows from the definition of $\mu^2_{M_\alpha}$ that the cup product above extends the cup product of ordinary Hochschild cochains $\smile:C_*(A,A)\otimes C_*(A,A) \to C_*(A,A)$.

\begin{proposition}
    The cohomology cup product $\underset{M_\alpha}{\smile}: HH^*(A,M_\alpha)\otimes HH^*(A,M_\alpha) \to HH^*(A,M_\alpha)$ is associative.
\end{proposition}
\begin{proof}
    The homotopy between the corresponding diagrams follows from the existence of $\mu^3_{M_\alpha}$ and the equation it satisfies together with $\mu^2_{M_\alpha}$.
\end{proof}

In general, this cup product is not commutative, but it does satisfy the following compatibility with the product $\pi_{M_\alpha}$ in  \cref{def:productOnChains} expressed in terms of the map of complexes \[g^{M_\alpha}_\alpha: C_*(A,M_\alpha) \to C^*(A,M_\alpha)[n]\] constructed in \cref{ex:galpha}. For simplicity, we write $g^{M_\alpha}_\alpha$ as $g_\alpha$ below. 
\begin{proposition}\label{prop:cupProductCompatibility}
    There is a chain homotopy
    \[ g_\alpha(\pi_{M_\alpha}(-,-)) \simeq g_\alpha(-) \underset{M_\alpha}{\smile} g_\alpha(-) \]
    of maps of complexes $C_*(A,M_\alpha) \otimes C_*(A,M_\alpha) \to C^*(A,M_\alpha)[-2n]$. The induced map
    \[ HH_*(A,M_\alpha) \otimes  HH_*(A,M_\alpha) \to HH^*(A,M_\alpha)[2n] \]
    is $(-1)^n$-commutative.
\end{proposition}
\begin{proof}
    To prove the first assertion, we must write both sides as maps
    \[ C^*_{(2)}(A) \otimes C^*_{(2)}(A) \otimes C_*(A,M_\alpha) \otimes C_*(A,M_\alpha) \to C^*(A,M_\alpha) \]
    given by ribbon quivers (where we input $\alpha \otimes \alpha$ into the first two factors) and find a homotopy between them. For the $(-1)^n$-commutativity, similarly we write the two ribbon quivers and find a homotopy that swaps the two inputs; comparing orientations at the end, we get the correct $(-1)^n$ sign. We present both of these homotopies in \cref{app:cupProductCompatibility}.
\end{proof}

Recall from Remark \ref{remark:nondegeneratequasiisomorphism} that if $\alpha$ is non-degenerate, for any bimodule $M$ the map \[g_\alpha^M: C_*(A,M) \to C^*(A,M)[n]\] is a quasi-isomorphism. Together with the previous proposition, this implies the following.
\begin{corollary}\label{cor:productCommutative}
    If $\alpha$ is non-degenerate, both the homology product $\pi_{M_\alpha}$ and the cohomology cup product $\underset{M_\alpha}{\smile}$ are $(-1)^n$-commutative.
\end{corollary}

\subsubsection{Relation to the categorical formal punctured neighborhood of infinity}
We will now study the relation between the `cone bimodule' $M_{\alpha}$ of \cref{definition:calM}, defined using $\alpha$, and the bimodule $\Ahatinfty$, whose definition is independent of $\alpha$.

Consider the map of $A$-bimodules given by the composition
\[ A \otimes^\LL_A A^\vee \xrightarrow{\sim} C_*(A,A^\vee \otimes A) \xrightarrow{g_\alpha^{A^\vee \otimes A}} C^*(A,A^\vee \otimes A)[n] \]
We also have another map  of $A$-bimodules
\[ A \otimes^\LL_A A^\vee[-n] \xrightarrow{\sim} A^\vee[-n] \xrightarrow{f_\alpha} A \]
Together with the map defining $\Ahatinfty$ and the canonical map $A \to C^*(A,\Hom_\kk(A,A))$ (which is always a quasi-isomorphism), these maps form a square
\[\xymatrix{
    A \otimes^\LL_A A^\vee[-n] \ar[d] \ar[r] & A \ar[d] \\
    C^*(A,A^\vee \otimes A) \ar[r] & C^*(A,\Hom_\kk(A,A))
}\]
\begin{proposition}\label{prop:homotopyAlpha}
    There is a homotopy making the square above commute, and defining a map of bimodules $M_{\alpha} \to \Ahatinfty$. When $\alpha$ is non-degenerate, this map is a quasi-isomorphism of bimodules.
\end{proposition}
\begin{proof}
    Again, for the first claim we write the ribbon quivers giving the two maps and a homotopy between them, in \cref{app:homotopyAlpha}. The second claim follows from the fact (see \cref{remark:nondegeneratequasiisomorphism}) that if $\alpha$ is non-degenerate then $g_\alpha^{A^\vee \otimes A}$ is a quasi-isomorphism (and thus the vertical maps are both quasi-isomorphisms). 
\end{proof}

A quasi-isomorphism $M_{\alpha} \simeq \Ahatinfty$ of bimodules induces a quasi-isomorphism $C_*(A,M_{\alpha}) \simeq C_*(A, \Ahatinfty)$. Combining this with the previous results of this section we have:
\begin{theorem}
    Given a 4-truncated non-degenerate $n$-pre-CY category $(A,m)$, there is a product $\pi_\infty$ on $HH_*(A,\Ahatinfty)[-n]$ and a cup product $\underset{\infty}{\smile}$ on $HH^*(A,\Ahatinfty)$, both associative and $(-1)^n$-commutative, related by $g^{\Ahatinfty}_\alpha(\pi_\infty(-,-)) = g^{\Ahatinfty}_\alpha(-) \underset{\infty}{\smile} g^{\Ahatinfty}_\alpha(-)$.
\end{theorem}

\begin{remark}
   One could obtain a chain-level description of these operations by using the explicit quasi-isomorphism $M_{\alpha} \simeq \Ahatinfty$ and picking an inverse, but the resulting calculations would be rather complicated. We take the point of view that it is simpler to just use $M_{\alpha}$ instead.
\end{remark}

\subsection{The dual product}\label{subsec:dualProduct}
Let us return to the chain-level product $\pi_{M_{\alpha}}$ on the complex $C_*(A,M_{\alpha})$, which we defined in \cref{def:productOnChains} purely from the (truncated) pre-CY structure of $A$. By definition, this complex is a cone of maps of Hochschild chain complexes
\[ C_*(A,M_{\alpha}) = C_*(A,A^\vee)[1-n] \oplus C_*(A,A) \]
with differential combining the Hochschild differentials with the map $F_\alpha$ induced by the map of bimodules $f_\alpha$ in \eqref{align:falpha}. The product $\pi_{M_{\alpha}}$ restricts to a product $\pi$ on the subcomplex $C_*(A,A)$ by construction, but it does not induce a chain operation on the complex $C_*(A,A^\vee)[1-n]$. Note that we always have a pairing
\[ \langle g^{A^\vee}(-), - \rangle \colon C_*(A,A^\vee)[-n] \otimes C_*(A,A) \to \kk \]
whose induced pairing on homology is non-degenerate when $\alpha$ is non-degenerate. Thus a degree $n$ product on $C_*(A,A^\vee)[1-n]$ is exactly the type of operation that could be dual to a coproduct of degree $1-n$ on $C_*(A,A)$. However, recall that the coproduct $\lambda_H$ we defined in \cref{sec:coproduct} only gives a map of complexes after a quotient.

\subsubsection{Homological algebra setup}\label{sec:dualKernel}
For the sake of clarity, let us describe a certain algebraic setup in general. Suppose that we have complexes of $\kk$-modules $X,Y$, together with a pairing, that is, a map of complexes
\[ (-,-) \colon X \otimes Y \to \kk \]
such that the induced map $H_*X \to H_*(Y)^\vee$ is \emph{surjective}. Suppose that we also have a subcomplex $W \subseteq Y$ such that the induced map $H_*(W) \to H_*(Y)$ is \emph{injective}; as a consequence, the long exact sequence in homology splits into short exact sequences
\[ 0 \to H_*(W) \to H_*(Y) \to H_*(Y/W) \to 0 \]
which upon dualization give exact sequences
\[ 0 \to H_*(Y/W)^\vee \to H_*(Y)^\vee \to H_*(W)^\vee \]
Let us denote $Z_*(-)$ and $B_*(-)$ for the cycles and boundaries subcomplexes. We now define a subcomplex of $X$
\[ K \coloneqq \{x \in Z_*(X)\ |\ \forall w \in Z_*(W), (x,w) = 0 \} \subseteq X. \]
and pick any subcomplex $Q$ of $X$. We note that $B_*(X) \subseteq K$.
\begin{lemma}\label{lem:surjectionToDual}
    There is a \emph{surjective} map of graded $\kk$-modules $f \colon K/B_*(Q) \to H_*(Y/W)^\vee$ making the diagram
    \[\xymatrix{
        K/B_*(Q) \ar[r] \ar@{->>}[d]_{f} & H_*(X) \ar@{->>}[d] \\
        H_*(Y/W)^\vee \ar@{^{(}->}[r]^-{p^\vee} & H_*(Y)^\vee 
    }\]
    commute.
\end{lemma}
\begin{proof}
    We first note that $K/B_*(Q) \to K/B_*(X)$ is surjective, so it is enough to prove the statement with $Q=X$. Every class in the image of $K/B_*(X) \hookrightarrow H_*(X)$ pairs with zero with the image of $H_*(W) \to H_*(Y)$, so the map $f$ exists by the universal property of the kernel. It also follows that for any \emph{map of sets} $s \colon H_*(Y/W) \to H_*(Y)$ giving a section (note that without extra splitness assumptions we do not necessarily have a morphism of graded $\kk$-modules), the pairing
    \[ F \colon K/B_*(X) \otimes H_*(Y/W) \to \kk \]
    corresponding to $f$ satisfies $F(a,b) = (i(a),s(b))$. Therefore, given any $\varphi \in H_*(Y/W)^\vee$, we take $p^\vee(\varphi)$, which is a function on $H_*(Y/W)$; any function of that form is given by a linear function on $Z_*(Y)$ satisfying $p^\vee(\varphi)(y) = p^\vee(\varphi)(y + dz + w)$ for all $z\in Y$ and $w\in Z_*(W)$. But then by assumption of surjectivity of $H_*(X) \to H_*(Y)^\vee$, there exists $x \in Z_*(X)$ such that $(x,y) = p^\vee(\varphi)(y)$ for every $y \in Z_*(Y)$ and by linearity we conclude that $(x,w) = 0$ for every $w \in Z_*(W)$. Thus we must have $x \in K$, and its class in $K/B_*(X)$ maps to $\varphi$ as a function on $H_*(Y/W)$, proving surjectivity.
\end{proof}

\subsubsection{Lifting the product}
To lift the product $\pi_{M_{\alpha}}$ to the complex $C_*(A,A^\vee)[1-n]$ or to a subcomplex thereof, we use a pair of nullhomotopies of the map of complexes $F_\alpha$. In the case where the class of $[F_\alpha]$ is non-trivial, we can still get a partial lift from a `partial nullhomotopy'. Again, let $W$ be some subcomplex of $C_*(A,A)$ such that $H_*(W)\to HH_*(A,A)$ is injective.
\begin{definition}
    A \emph{homotopy} of $F_\alpha$ onto $W$ is a pair $(E_0,h)$, where
    \[ E_0 \in W \otimes C_*(A,A) \cap C_*(A,A)\otimes W,\]
    is closed and $h \colon C_*(A,A^\vee)[-n] \to C_*(A,A)[-1]$ is a homotopy of maps $C_*(A,A^\vee)[-n] \to C_*(A,A)$ between $F_\alpha$ and $^\sharp E_0 \circ g^{A^\vee}_\alpha$.
\end{definition}
In other words, since $E_0 \in C_*(A,A) \otimes W$, a partial homotopy as above gives a homotopy between the map $F_\alpha$ and a map landing in the smaller complex $W$. Given any homotopy as above, we then get a map of graded $\kk$-modules $\iota_h \colon C_*(A,A^\vee)[-n+1] \to C_*(A,M_\alpha)$ given by $\iota_h(x) \coloneqq (x, -h(x))$, where we use the isomorphism of graded $\kk$-modules $C_*(A,M_\alpha) \cong C_*(A,A^\vee)[-n+1]\oplus C_*(A,A)$; it follows from the definition above that this map restricts to a \emph{map of complexes}
\[ \iota_h \colon \ker(^\sharp E_0 \circ g^{A^\vee}_\alpha) \to C_*(A,M_\alpha). \]
\begin{proposition}\label{prop:productCone}
    Let $(E_{01},h_1),(E_{02},h_2)$ be two homotopies of $F_\alpha$ onto $W$. Then the map
    \[ \pi_{h_1, h_2}(x_1,x_2) = p \pi_{M_{\alpha}}(\iota_{h_1}(x_1), \iota_{h_2}(x_2)) \]
    where $p$ is the projection $C_*(A,M_\alpha)[n] \to C_*(A,A^\vee)[-n]$ gives a map of complexes 
    \[ \left(\ker(^\sharp E_{01} \circ g^{A^\vee}_\alpha) \otimes C_*(A,A^\vee)[-n+1] \right) \cap \left(C_*(A,A^\vee)[-n+1] \otimes \ker(^\sharp E_{02} \circ g^{A^\vee}_\alpha) \right) \to C_*(A,A^\vee)[1], \]
    whose induced map on homology only depends on the $h_i$ up to $[d,-]$-exact terms.
\end{proposition}

Let us from now on assume that $A$ is connective. Let us denote
\[ Q \coloneqq \ker(^\sharp E_{01} \circ g^{A^\vee}_\alpha) \cap \ker(^\sharp E_{02} \circ g^{A^\vee}_\alpha)\]
and apply the construction of \cref{sec:dualKernel} with
\[ X = C_*(A,A^\vee)[-n], \quad Y = C_*(A,A), \quad (-,-) = \langle g^{A^\vee}_\alpha(-),- \rangle, \]
defining a subcomplex
\[ K \coloneqq \{x \in C_*(A,A^\vee)[-n]\ |\ dx=0\ \text{and}\ \forall w \in W\ \text{such that}\ dw=0, \langle g^{A^\vee}_\alpha(x), w \rangle = 0 \}. \]
Since  $E_{01},E_{02}$ are closed elements of degree zero in $W \otimes C_*(A,A)$ and this complex is supported in nonnegative degrees, we have $E_{01},E_{02} \in Z_0(W) \otimes Z_0(C_*(A,A))$, which implies that $K \subseteq Q$. We can then restrict the operation $\pi_{h_1,h_2}$ to $K \otimes K$.
\begin{definition}\label{def:chainDualProduct}
    With the same assumptions of \cref{prop:productCone}, and assuming $A$ connective, we define the \emph{chain-level dual product} to be the map of complexes given by the restriction
    \[ \pi_{h_1,h_2} \colon K\otimes K \to C_*(A,A^\vee)[1], \]
    and the \emph{homology dual product} to be the corresponding induced map
    \[  \pi_{h_1,h_2} \colon \frac{K}{B_*(Q)} \otimes \frac{K}{B_*(Q)} \to HH_*(A,A^\vee)[1],\]
\end{definition}

\begin{proposition}\label{prop:dualProdCommutativity}
    The homology dual product $\pi_{h_1,h_2}$ only depends on the values of $h_i$ on $K$ up to exact terms, and is $(-1)^n$-commutative when $h_1-h_2$ takes exact values on $K$.
\end{proposition}
\begin{proof}
    Follows from the fact that $\pi_{M_\alpha}$ is a map of complexes which is $(-1)^n$-commutative on homology, and from the fact that the maps
    $\iota_{h_i}$ are maps of complexes $K[1] \to C_*(A,M_\alpha)$, whose induced map on homology only depends on the values $h_i(x)$ for $x \in K$ up to exact terms.
\end{proof}

\begin{definition}\label{def:dualBalanced}
    The quadruple $(W,h_1,h_2,\alpha)$ is \emph{balanced} if the image of the chain-level dual product $\pi_{h_1,h_2}$ lies in $K$.
\end{definition}
When $\kk$ is a field and the quadruple $(W,h_1,h_2,\alpha)$ is balanced, taking homology we get a product
\[ \overline{\pi}_{h_1,h_2} \colon \frac{K}{B_*(Q)}[-n] \otimes \frac{K}{B_*(Q)}[-n] \to \frac{K}{B_*(Q)}[1], \]
which as a consequence of \cref{prop:dualProdCommutativity} is graded commutative when $h_1|_K$ and $h_2|_K$ differ from each other by $[d,-]$-exact terms.

\subsubsection{Associativity of the homology dual product}
Asssume that $\kk$ is a field and that we made a balanced choice of quadruple $(W,h_1,h_2,\alpha)$, so that it makes sense to discuss the associativity properties of $\pi_{h_1,h_2}$.
\begin{proposition}\label{prop:dualProdAssociativity}
    If $A$ is smooth over a field $\kk$, connective with a nondegenerate (4-truncated) pre-CY structure of dimension $n \ge 3$, the quadruple $(W,h_1,h_2,\alpha)$ is balanced and $h_1 - h_2$ takes exact values on $K$, then the homology dual product $\overline\pi_{h_1,h_2}$ on $K/B_*(Q)[-n]$ is associative.
\end{proposition}
\begin{proof}
    Given any homotopy $h \colon C_*(A,A^\vee)[-n] \to C_*(A,A)[-1]$ of $F_\alpha$, we have a map \[ \rho_h \colon C_*(A,M_\alpha) \to C_*(A,A)\]
    given by $\rho_h(x,y) = y + h(x)$. By definition, this map satisfies $\rho_h \circ \iota_h = 0$, and also the identity of endomorphisms of $C_*(A,M_\alpha)$
    \[ \iota_h \circ p + \rho_h = \id. \]
    We now take any three elements $x_1,x_2,x_3$ of $K$ and use this identity to write
    \begin{align*}
        \pi_{h_1,h_2}(\pi_{h_1,h_2}(x_1,x_2),x_3) &= p \pi_{M_\alpha}(\iota_{h_1} p \pi_{M_\alpha}(\iota_{h_1} x_1,\iota_{h_2} x_2),\iota_{h_2} x_3) \\
        &= p \pi_{M_\alpha}(\pi_{M_\alpha}(\iota_{h_1}x_1,\iota_{h_2}x_2),\iota_{h_2}x_3) - p \pi_{M_\alpha}(\rho_{h_1} \pi_{M_\alpha}(\iota_{h_1}x_1,\iota_{h_2}x_2),\iota_{h_2}x_3) \\
        \pi_{h_1,h_2}(x_1,\pi_{h_1,h_2}(x_2,x_3)) &= p \pi_{M_\alpha}(\iota_{h_1} x_1, \iota_{h_2} p \pi_{M_\alpha}(\iota_{h_1} x_2,\iota_{h_2} x_3)) \\
        &= p \pi_{M_\alpha}(\iota_{h_1}x_1, \pi_{M_\alpha}(\iota_{h_2}x_2,\iota_{h_2}x_3)) - p \pi_{M_\alpha}(\iota_{h_1} x_1, \rho_{h_2} \pi_{M_\alpha}(\iota_{h_1} x_2,\iota_{h_2} x_3))
    \end{align*}
    We note that every term in the second and fourth line of the equation above only has maps of complexes, so it is enough to prove the desired associativity for the corresponding maps on homology. On homology, the two former terms of those lines are equal since $\iota_{h_1} = \iota_{h_2}$ and $\pi_{M_\alpha}$ is associative. The two latter terms, involving $\rho_{h_1},\rho_{h_2}$, should be seen as a defect of associativity, and vanish for degree reasons given our assumptions: since $HH_*(A,A)$ is supported in nonnegative degrees, so $\rho_{h_1},\rho_{h_2}$ vanish on negative degrees, and $\alpha$ is nondegenerate, $HH_*(A,A^\vee)[-n] \cong (\HH_*(A,A))^\vee$ is supported in nonpositive degrees so the image of $\pi_{M_\alpha}$ is concentrated in degrees $(-\infty,2-n]$.
\end{proof}

\section{Relations between products on the dual and coproducts}\label{sec:relations}
In this section we establish the relation between the dual products $\pi_{h_1,h_2}$, which we got by choosing homotopies of $f_\alpha$, and the coproduct $\lambda_H$ we get by choosing a trivialization $H$ of $E$. In the process, we prove a symmetry statement for $[E]$ under the non-degeneracy assumption on $\alpha$.

\subsection{Symmetry of Chern character}
Recall that in \cref{def:ChernCharacter} we defined a chain-level Chern character $E \in C_*(A,A)\otimes C_*(A,A)$ given a smooth $A_{\infty}$-category $A$ and chain-level representatives for the vertices $\co$ and $\eta$. Given any $2$-truncated $n$-dimensional pre-CY structure $m = \mu + \alpha$, we can compose both sides of $E$ with $g_\alpha: C_*(A,A) \to C^*(A,A)[n]$ to get an element of degree $-2n$
\[ (\id \otimes g_\alpha)(g_\alpha \otimes \id)E \in C^*(A,A)\otimes C^*(A,A). \]
We specify the order of application of the $g_\alpha$ maps due to the possibly nontrivial sign difference $(g_\alpha \otimes \id)(\id \otimes g_\alpha) =g_\alpha \otimes g_\alpha =  (-1)^n (\id \otimes g_\alpha)(g_\alpha \otimes \id)$.

To make the visualization easier, we cut along an arc going from one boundary component of the elbow to the other, and draw the ribbon graphs in the square as follows, namely we identify the bottom edge with the top edge of the square.  We insert $\alpha \otimes \alpha$ into the circles labeled I and II:
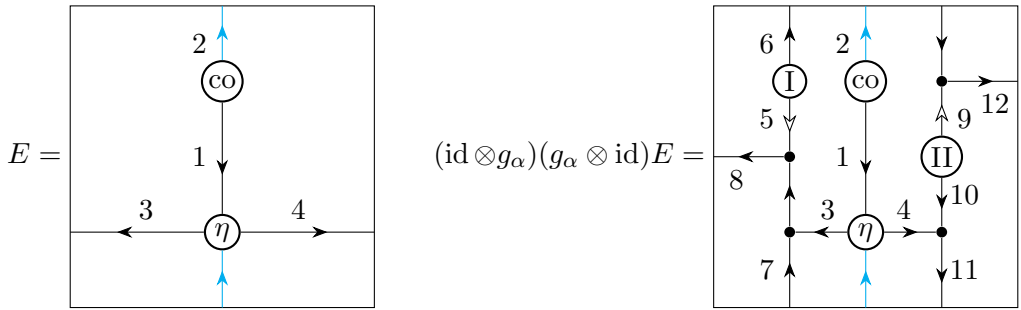
\[ E =  \tikzfig{
    \draw (-2,-2) rectangle (2,2);
    \node [vertex] (co) at (0,1) {$\co$};
    \node [vertex] (eta) at (0,-1) {$\eta$};
    \draw [->-,cyan] (co) to (0,2);
    \draw [->-,cyan] (0,-2) to (eta);
    \draw [->-] (co) to (eta);
    \draw [->-] (eta) to (2,-1);
    \draw [->-] (eta) to (-2,-1);
    \node at (-0.3,0) {$1$};
    \node at (-0.3,1.5) {$2$};
    \node at (-1,-0.7) {$3$};
    \node at (1,-0.7) {$4$};
    } \qquad (\id \otimes g_\alpha)(g_\alpha \otimes \id)E = \tikzfig{
    \draw (-2,-2) rectangle (2,2);
    \node [vertex] (co) at (0,1) {$\co$};
    \node [vertex] (eta) at (0,-1) {$\eta$};
    \node [vertex] (alphaL) at (-1,1) {I};
    \node [vertex] (alphaR) at (1,0) {II};
    \node [bullet] (midL) at (-1,0) {};
    \node [bullet] (botL) at (-1,-1) {};
    \node [bullet] (botR) at (1,-1) {};
    \node [bullet] (topR) at (1,1) {};
    \draw [->-,cyan] (co) to (0,2);
    \draw [->-,cyan] (0,-2) to (eta);
    \draw [->-] (co) to (eta);
    \draw [->-] (eta) to (botL);
    \draw [->-] (eta) to (botR);
    \draw [-w-] (alphaL) to (midL);
    \draw [->-] (botL) to (midL);
    \draw [->-] (alphaL) to (-1,2);
    \draw [->-] (-1,-2) to (botL);
    \draw [->-] (midL) to (-2,0);
    \draw [-w-] (alphaR) to (topR);
    \draw [->-] (alphaR) to (botR);
    \draw [->-] (botR) to (1,-2);
    \draw [->-] (1,2) to (topR);
    \draw [->-] (topR) to (2,1);
    \node at (-0.3,0) {$1$};
    \node at (-0.3,1.5) {$2$};
    \node at (-0.5,-0.7) {$3$};
    \node at (0.5,-0.7) {$4$};
    \node at (-1.3,0.5) {$5$};
    \node at (-1.3,1.5) {$6$};
    \node at (-1.3,-1.5) {$7$};
    \node at (-1.7,-0.3) {$8$};
    \node at (1.3,0.5) {$9$};
    \node at (1.3,-0.5) {$10$};
    \node at (1.3,-1.5) {$11$};
    \node at (1.7,0.7) {$12$};
    }
\]

We can draw another diagram which, when also evaluated on $\alpha \otimes \alpha$, gives another closed element $D_\alpha$ of degree $-2n$ in $C^*(A,A)\otimes C^*(A,A)$:
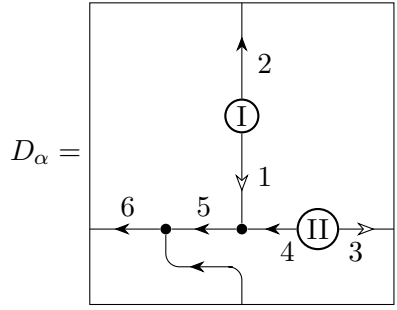
\[ D_\alpha = \tikzfig{
    \draw (-2,-2) rectangle (2,2);
    \node [vertex] (alphaL) at (0,0.5) {I};
    \node [vertex] (alphaR) at (1,-1) {II};
    \node [bullet] (botR) at (0,-1) {};
    \node [bullet] (botL) at (-1,-1) {};
    \draw [-w-] (alphaL) to (botR);
    \draw [->-] (botR) to (botL);
    \draw [->-] (alphaL) to (0,2);
    \draw [->-,rounded corners=5] (-0,-2) |- ++(-0.3,0.5) -| (botL);
    \draw [->-] (botL) to (-2,-1);
    \draw [->-] (alphaR) to (botR);
    \draw [-w-] (alphaR) to (2,-1);
    \node at (0.3,-0.3) {$1$};
    \node at (0.3,1.2) {$2$};
    \node at (1.5,-1.3) {$3$};
    \node at (0.6,-1.3) {$4$};
    \node at (-0.5,-0.7) {$5$};
    \node at (-1.5,-0.7) {$6$};
}\]
Note that pairing with the right-hand output of $D_\alpha$ gives a map
\[ (D_\alpha)^\sharp: C_*(A,A^\vee) \to C^*(A,A) [2n] \]
which is \emph{equal at chain-level} to $g_\alpha \circ F_\alpha$. Besides the chain-level expressions for $\co,\eta$, recall from \cref{sec:anotherDescr} that we fixed a (non-closed) bimodule morphism $\beta:A \to A$ of degree $1$, witnessing the relation between $\co,\eta$
\begin{lemma}\label{lem:Jhomotopy}
    The elements $(\id\otimes g_\alpha)(g_\alpha \otimes \id)E$ and $D_\alpha$ are homologous, i.e. there is an explicit element $J_\alpha \in C^*(A,A)\otimes C^*(A,A)$, depending on the chain-level expressions for $\alpha,\co,\eta$ and $\beta$, satisfying
    \[ dJ_\alpha = (\id\otimes g_\alpha)(g_\alpha \otimes \id)E - D_\alpha \]
\end{lemma}
\begin{proof}
    We write in \cref{app:Jhomotopy} an explicit combination of oriented ribbon quivers giving the element $J_\alpha$.
\end{proof}

Let us denote by $(-)^T$ the action of the generator of $\Z/2\Z$ permuting the two factors.
\begin{lemma}\label{lem:Lhomotopy}
    The elements $D_\alpha$ and $D_\alpha^T$ are homologous, i.e.\ there is an explicit element $L_\alpha \in C^*(A,A)\otimes C^*(A,A)$ satisfying $dL_\alpha = D_\alpha - D_\alpha^T$
\end{lemma}
\begin{proof}
    The element $L_\alpha$ given by evaluating the following combination of diagrams on $\alpha \otimes \alpha$ satisfies the desired property:
    \[\tikzfig{
	\draw (-2,-2) rectangle (2,2);
	\node [vertex] (alphaI) at (0,0.5) {I};
	\node [vertex] (alphaII) at (0,-1) {II};
	\node [bullet] (botL) at (-1,-1) {};
	\draw [-w-] (alphaI) to (alphaII);
	\draw [->-] (alphaI) to (0,2);
	\draw [->-,rounded corners=5] (-0,-2) |- ++(-0.3,0.5) -| (botL);
	\draw [->-] (botL) to (-2,-1);
	\draw [->-] (alphaII) to (botL);
	\draw [-w-] (alphaII) to (2,-1);
	\node at (0.3,-0.3) {$1$};
	\node at (0.3,1.2) {$2$};
	\node at (1,-1.3) {$3$};
	\node at (-0.5,-0.7) {$4$};
	\node at (-1.7,-0.7) {$5$};
} \quad + \quad \tikzfig{
\draw (-2,-2) rectangle (2,2);
\node [vertex] (alphaI) at (0,0.5) {I};
\node [vertex] (alphaII) at (-1,-1) {II};
\node [bullet] (bot) at (0,-1) {};
\draw [-w-] (alphaI) to (bot);
\draw [->-] (alphaI) to (0,2);
\draw [->-,rounded corners=5] (-0,-2) |- ++(-0.3,0.5) -| (alphaII);
\draw [->-] (bot) to (2,-1);
\draw [-w-] (alphaII) to (bot);
\draw [->-] (alphaII) to (-2,-1);
\node at (0.3,-0.3) {$1$};
\node at (0.3,1.2) {$2$};
\node at (1,-1.3) {$3$};
\node at (-0.5,-0.7) {$4$};
\node at (-1.7,-0.7) {$5$};
}\]
\end{proof}

From the above lemmas, we conclude that the elements $(\id \otimes g_\alpha)(g_\alpha \otimes \id) E$ and $$(g_\alpha \otimes \id)(\id \otimes  g_\alpha) E^T = (-1)^n(\id \otimes  g_\alpha)(g_\alpha \otimes \id) E^T$$ are homologous. Therefore we can conclude a symmetry property of the class $[E]$, in the case that $A$ admits a non-degenerate $\alpha$.
\begin{theorem}\label{thm:symmetry}
    If the smooth $A_\infty$-category $A$ admits a nondegenerate (2-truncated) pre-Calabi-Yau structure of dimension $n$, then its class $[E]$ is $(-1)^n$-symmetric. In particular, over a field $\kk$ of characteristic zero, if $A$ admits a weak smooth Calabi-Yau structure of dimension $n$, that is, if there is a quasi-isomorphism of bimodules $A \simeq A^![-n]$, then its class $[E]$ is $(-1)^n$-symmetric.
\end{theorem}

\subsection{Compatibility relation}\label{sec:compatibility}
We will now study the compatibility between the coproducts $\lambda_H$ from \cref{def:chainAlgebraicLoopCoproduct} and the dual products $\pi_{h_1,h_2}$ from \cref{def:chainDualProduct}. This relation will allow us to complete the proof of \cref{thm:thm1intro}. This is the most involved proof in this paper: it involves calculating some rather complicated homotopies, so we will break the proof into parts.

\subsubsection{The square-filling lemma}
Let us start with a lemma which holds for any smooth $A$ with pre-CY structure, not necessarily nondegenerate. We keep this separate from the main proof since we believe it may be of future interest, e.g., when studying operations coming from a possibly degenerate pre-CY structure.

Recall that $G(\varphi)$ gives a homotopy between capping with $\varphi$ onto the left factor of $E$ and onto the right factor of $E$. Passing to Hochschild cochains using $(\id \otimes g_\alpha)(g_\alpha \otimes \id)$, we have a similar description for the cup product. We define
\[ \Gamma(\varphi) =  (\id \otimes g_\alpha)(g_\alpha \otimes \id)G(\varphi) + \tikzfig{
	\draw (-2,-2) rectangle (2,2);
	\node [vertex] (co) at (0,-0.5) {$\co$};
	\node [rectangle,draw,thick] (eta) at (0,1) {$\eta$};
	\node [vertex] (alphaL) at (-1,0) {I};
	\node [vertex] (alphaR) at (1,-1) {II};
	\node [bullet] (topL) at (-1,1) {};
	\node [bullet] (botL) at (-1,-0.7) {};
	\node [vertex] (phi) at (-1,-1.5) {$\varphi$};
	\node [bullet] (midR) at (1,0) {};
	\node [bullet] (topR) at (1,1) {};
	\draw [->-,cyan] (co) to (eta);
	\draw [->-] (co) to (0,-2);
	\draw [->-] (0,2) to (eta);
	\draw [->-] (eta) to (topL);
	\draw [->-] (eta) to (topR);
	\draw [->-] (topR) to (midR);
	\draw [->-] (alphaL) to (topL);
	\draw [-w-] (alphaL) to (botL);
	\draw [->-] (topL) to (-1,2);
	\draw [->-=.9] (-1,-2) to (phi);
	\draw [->-] (phi) to (botL);
	\draw [->-] (botL) to (-2,-0.7);
	\draw [-w-] (alphaR) to (midR);
	\draw [->-] (alphaR) to (1,-2);
	\draw [->-] (1,2) to (topR);
	\draw [->-] (midR) to (2,0);
	\node at (-0.3,-1.2) {$ 1$};
	\node at (-0.3,0) {$ 2$};
	\node at (-0.5,1.3) {$ 3$};
	\node at (0.5,1.3) {$ 4$};
	\node at (-0.7,-0.5) {$ 5$};
	\node at (-1.3,0.5) {$ 6$};
	\node at (-1.3,1.5) {$ 7$};
	\node at (-0.7,-1) {$ 8$};
	\node at (-1.5,-1) {$ 9$};
	\node at (0.7,-0.5) {$10$};
	\node at (0.7,-1.5) {$11$};
	\node at (0.7,0.5) {$12$};
	\node at (1.5,-0.3) {$13$};
} \quad + \quad \tikzfig{
\draw (-2,-2) rectangle (2,2);
\node [vertex] (co) at (0,-0.5) {$\co$};
\node [rectangle,draw,thick] (eta) at (0,1) {$\eta$};
\node [vertex] (alphaL) at (-1,-1) {I};
\node [vertex] (alphaR) at (1,-1.3) {II};
\node [bullet] (topL) at (-1,1) {};
\node [bullet] (midL) at (-1,0) {};
\node [vertex] (phi) at (1,-0.4) {$\varphi$};
\node [bullet] (midR) at (1,0.3) {};
\node [bullet] (topR) at (1,1) {};
\draw [->-,cyan] (co) to (eta);
\draw [->-] (co) to (0,-2);
\draw [->-] (eta) to (topL);
\draw [->-] (eta) to (topR);
\draw [->-] (topL) to (midL);
\draw [->-] (alphaL) to (midL);
\draw [->-] (-1,2) to (topL);
\draw [-w-] (alphaL) to (-1,-2);
\draw [->-] (midL) to (-2,0);
\draw [-w-] (alphaR) to (phi);
\draw [->-] (phi) to (midR);
\draw [->-] (topR) to (midR);
\draw [->-] (alphaR) to (1,-2);
\draw [->-] (1,2) to (topR);
\draw [->-] (midR) to (2,0.3);
\draw [->-] (0,2) to (eta);
\node at (-0.3,-1.4) {$1$};
\node at (-0.3,0) {$2$};
\node at (-0.6,1.3) {$3$};
\node at (0.6,1.3) {$4$};
\node at (-0.7,-1.7) {$5$};
\node at (-0.7,-0.7) {$6$};
\node at (-1.3,0.4) {$7$};
\node at (0.7,0) {$11$};
\node at (-1.5,-0.3) {$8$};
\node at (0.7,-0.8) {$9$};
\node at (0.7,-1.8) {$10$};
\node at (0.7,0.7) {$12$};
\node at (1.5,0.5) {$13$};
} \]
where the two diagrams are taken with orientation $(13\ 12\ \dots\ 1)$ and evaluated with $\alpha \otimes \alpha$ input into I and II. The map $\Gamma: C^*(A,A) \to C^*(A,A)\otimes C^*(A,A)[2n-1]$ then satisfies the equation
\[ d(\Gamma(\varphi)) + \Gamma(d\varphi) =  ((\smile \varphi \otimes \id) - (\id \otimes \varphi \smile))(\id \otimes g_\alpha)(g_\alpha \otimes \id) E \]

The element $J_\alpha$ of \cref{lem:Jhomotopy}, by cupping on the left or on the right, gives a homotopy between each of the two terms on the right-hand side of the equation above and another term involving cupping with the element $D_\alpha$:
\begin{align*} 
	d((\smile \varphi \otimes \id)J_\alpha) + (\smile d\varphi \otimes \id)J_\alpha &= -(\smile \varphi \otimes \id)(\id \otimes g_\alpha)(g_\alpha \otimes \id)(E) + (\smile \varphi \otimes \id)D_\alpha \\
	d((\id \otimes \varphi \smile)J_\alpha) + (\id \otimes d\varphi \smile)J_\alpha &= -(\id \otimes \varphi \smile)(\id \otimes g_\alpha)(g_\alpha \otimes \id)(E) + (\id \otimes \varphi \smile)D_\alpha
\end{align*}
We can easily picture another combination of diagrams which gives a homotopy between the two last terms of the equations above, by `passing' the $\varphi$-vertex through the lines of the diagram. Taking $I_\alpha$ to be the evaluation of the following map on $\alpha \otimes \alpha$:
\begin{align*}
	&+ \quad \tikzfig{
		\draw (-2,-2) rectangle (2,2);
		\node [vertex] (alphaL) at (0.4,0.5) {I};
		\node [vertex] (alphaR) at (1.2,-1) {II};
		\node [bullet] (botR) at (0.4,-1) {};
		\node [bullet] (botL) at (-0.4,-1) {};
		\node [vertex] (phi) at (-1.2,-1) {$\varphi$};
		\draw [-w-] (alphaL) to (botR);
		\draw [->-] (botR) to (botL);
		\draw [->-] (alphaL) to (0.4,2);
		\draw [->-,rounded corners=5] (0.4,-2) |- ++(-0.3,0.5) -| (botL);
		\draw [->-] (botL) to (phi);
		\draw [->-] (phi) to (-2,-1);
		\draw [->-] (alphaR) to (botR);
		\draw [-w-] (alphaR) to (2,-1);
		\node at (0.7,-0.1) {$1$};
		\node at (0.7,1.2) {$2$};
		\node at (0.8,-0.7) {$ 3$};
		\node at (1.6,-0.7) {$ 4$};
		\node at (0,-0.7) {$ 5$};
		\node at (-0.8,-0.7) {$ 6$};
		\node at (-1.6,-0.7) {$ 7$};
	} \quad + \quad \tikzfig{
		\draw (-2,-2) rectangle (2,2);
		\node [vertex] (alphaL) at (0,1) {I};
		\node [vertex] (alphaR) at (1,-1) {II};
		\node [bullet] (botR) at (0,-1) {};
		\node [bullet] (botL) at (-1,-1) {};
		\node [vertex] (phi) at (0,0) {$\varphi$};
		\draw [-w-] (alphaL) to (phi);
		\draw [->-] (phi) to (botR);
		\draw [->-] (botR) to (botL);
		\draw [->-] (alphaL) to (0,2);
		\draw [->-,rounded corners=5] (-0,-2) |- ++(-0.3,0.5) -| (botL);
		\draw [->-] (botL) to (-2,-1);
		\draw [->-] (alphaR) to (botR);
		\draw [-w-] (alphaR) to (2,-1);
		\node at (0.3,0.5) {$1$};
		\node at (0.3,1.6) {$2$};
		\node at (0.5,-1.3) {$ 3$};
		\node at (1.5,-0.7) {$ 4$};
		\node at (0.3,-0.6) {$ 5$};
		\node at (-0.5,-0.7) {$ 6$};
		\node at (-1.5,-0.7) {$ 7$};
	} \\
	&- \quad \tikzfig{
		\draw (-2,-2) rectangle (2,2);
		\node [vertex] (alphaL) at (0,0.5) {I};
		\node [vertex] (alphaR) at (1,-1) {II};
		\node [bullet] (botR) at (0,-1) {};
		\node [bullet] (botL) at (-1,-1) {};
		\node [vertex] (phi) at (1,0) {$\varphi$};
		\draw [-w-] (alphaL) to (botR);
		\draw [->-] (phi) to (alphaR);
		\draw [->-] (botR) to (botL);
		\draw [->-] (alphaL) to (0,2);
		\draw [->-,rounded corners=5] (-0,-2) |- ++(-0.3,0.5) -| (botL);
		\draw [->-] (botL) to (-2,-1);
		\draw [->-] (alphaR) to (botR);
		\draw [-w-] (alphaR) to (2,-1);
		\node at (0.3,-0.3) {$1$};
		\node at (0.3,1.2) {$2$};
		\node at (1.3,-0.5) {$ 3$};
		\node at (1.5,-1.3) {$5$};
		\node at (0.6,-1.3) {$4$};
		\node at (-0.5,-0.7) {$ 6$};
		\node at (-1.7,-0.7) {$ 7$};
	} \quad + \quad \tikzfig{
		\draw (-2,-2) rectangle (2,2);
		\node [vertex] (alphaL) at (-0.4,0.5) {I};
		\node [vertex] (alphaR) at (0.4,-1) {II};
		\node [bullet] (botR) at (-0.4,-1) {};
		\node [bullet] (botL) at (-1.2,-1) {};
		\node [vertex] (phi) at (1.3,-1) {$\varphi$};
		\draw [-w-] (alphaL) to (botR);
		\draw [-w-=.8] (alphaR) to (phi);
		\draw [->-] (phi) to (2,-1);
		\draw [->-] (botR) to (botL);
		\draw [->-] (alphaL) to (-0.4,2);
		\draw [->-,rounded corners=5] (-0.4,-2) |- ++(-0.3,0.5) -| (botL);
		\draw [->-] (botL) to (-2,-1);
		\draw [->-] (alphaR) to (botR);
		\node at (-0.1,-0.2) {$1$};
		\node at (-0.1,1.4) {$2$};
		\node at (0,-1.3) {$ 3$};
		\node at (1.8,-1.3) {$7$};
		\node at (0.9,-1.3) {$4$};
		\node at (-0.8,-0.7) {$5$};
		\node at (-1.7,-1.3) {$6$};
	}
\end{align*}
we have the equation 
\[ d I_\alpha(\varphi) + I_\alpha(d\varphi) = (\smile \varphi \otimes \id)D_\alpha - (\id \otimes \varphi \smile) D_\alpha \]

Heuristically, the homotopies above are the \emph{four sides of a square}, which sits inside of some topological space given by the realization of a cell complex, which should be seen as the space of operations on Hochschild (co)chains that we can realize with our calculus, the cells of which are labelled by our oriented diagrams. One can then naturally ask if there is a combination of cells that `fills in' the squares, and the answer is yes:
\begin{lemma}\label{lem:squareLemma}
    There is a map
    \[ N_\alpha :C^*(A,A) \to C^*(A,A) \otimes C^*(A,A)[2n-2] \]
    depending on $\alpha$ and $\beta$, such that we have an equality in $C^*(A,A) \otimes C^*(A,A)$
    \[ dN_\alpha(\varphi)-N_\alpha(d\varphi) = \Gamma(\varphi) + (\smile \varphi \otimes \id)J_\alpha - (\id \otimes \varphi \smile)J_\alpha - I_\alpha(\varphi) \]
    for every $\varphi \in C^*(A,A)$.
\end{lemma}
\begin{proof}
    We give an expression for the element $N_\alpha$ in \cref{app:squareLemma}.
\end{proof}

\subsubsection{The compatibility relation}
We return to the setting of \cref{sec:dualKernel}, and consider the map $g_\alpha \circ F_\alpha \colon C_*(A,A^\vee)[-n] \to C^*(A,A)[n]$. Suppose that we are given two elements
\[ E_{01}, E_{02} \in W \otimes C_*(A,A) \cap C_*(A,A) \otimes W \]
and two maps $\widetilde h_1, \widetilde h_2 \colon C_*(A,A^\vee)[-n] \to C^*(A,A)[n-1]$ satisfying
\[ [d,\widetilde h_1] = g_\alpha \circ F_\alpha - g_\alpha \circ ^\sharp E_{01} \circ g^{A^\vee}_\alpha, \quad [d,\widetilde h_2] = g_\alpha \circ F_\alpha - g_\alpha \circ ^\sharp E_{02} \circ g^{A^\vee}_\alpha, \]
in other words, $\widetilde h_1,\widetilde h_2$ are homotopies between $g_\alpha \circ F_\alpha$ and maps landing in the subcomplex $g_\alpha(W)$. We can rewrite the pairing between the dual product and a Hochschild cochain in terms of the $\widetilde{h}_i$, by using the following lemma:
\begin{lemma}\label{lem:chainToCochainHomotopy}
    For any pair $(\widetilde{h}_1,\widetilde{h}_2)$ as above, if $\alpha$ is non-degenerate, there exists a pair of homotopies $(E_{01},h_1),(E_{02},h_2)$ of $F_\alpha$ such that for any closed elements $x_1,x_2 \in K$ and $\varphi \in C^*(A,A)$, there is an equality
    \begin{align*}
        \langle \pi_{h_1,h_2}(x_1,x_2),\varphi \rangle &= \Lambda_\alpha(x_1,x_2,\varphi) - (-1)^{\deg(x_1)\deg(x_2)}\langle x_1, \varphi \smile \widetilde{h}_2(x_2)\rangle \\
        &+ (-1)^{\deg(x_2)\deg(\varphi)}\langle x_2, \varphi \smile \widetilde{h}_1(x_1) \rangle 
    \end{align*} 
    where $\Lambda_\alpha$ is a map defined by evaluating a certain combination of oriented ribbon quivers (\cref{app:compatibility}) on $\alpha \otimes \alpha$. Moreover, if $\widetilde{h}_1$ and $\widetilde{h}_2$ restricted to $K$ differ by a $[d,-]$-exact term on $K$, we can also choose $h_1,h_2$ whose restriction to $K$ differ by a $[d,-]$-exact term.
\end{lemma}
\begin{proof}
    Given in \cref{app:compatibility}.
\end{proof}

Suppose now that we are given a trivialization $H$ of $E$ onto $W$. We define:
\begin{align}\label{align:widetildeh}
\begin{aligned}
    E_{01} = (-1)^n E^T_0, \quad &\widetilde{h}_1 = \left( (\id \otimes g_\alpha)(g_\alpha\otimes \id) H  + J_\alpha + L_\alpha \right)^\sharp \\
    E_{02} = E_0, \quad &\widetilde{h}_2 =\ ^\sharp\left((\id \otimes g_\alpha)(g_\alpha\otimes \id) H  + J_\alpha \right)
    \end{aligned}
\end{align}
that is, maps given by pairing either on the right or the left-hand side of the specified elements of $C^*(A,A)\otimes C^*(A,A)$. Recall we have subcomplexes
\[ K \coloneqq \{x \in C_*(A,A^\vee)[-n]\ |\ dx=0\ \text{and}\ \forall w \in W\ \text{such that}\ dw=0, \langle g^{A^\vee}_\alpha(x), w \rangle = 0 \} \]
and $Q \coloneqq \ker(E_0^\sharp \circ g^{A^\vee}_\alpha) \cap \ker(^\sharp E_0 \circ g^{A^\vee}_\alpha)$.
\begin{lemma}\label{lem:compatibilityDeltaTilde}
    For any closed elements $x_1,x_2 \in K,\  \varphi \in C^*(A,A)$, the following equation holds:
    \begin{align*}
        &(-1)^{\deg(x_1)\deg(x_2)}\langle x_2\otimes x_1, (\id \otimes g_\alpha)(g_\alpha \otimes \id) \widetilde{\lambda}_H(\varphi) \rangle \\
        &=\Lambda_\alpha(x_1,x_2,\varphi) -(-1)^{\deg(x_1)\deg(x_2)}\langle x_1, \varphi \smile \widetilde{h}_2(x_2)\rangle
        + (-1)^{\deg(x_2)\deg(\varphi)}\langle x_2, \varphi \smile \widetilde{h}_1(x_1) \rangle 
    \end{align*}
\end{lemma}
\begin{proof}
    Follows from the definition of $\widetilde\lambda_H, \widetilde{h}_1,\widetilde{h}_2$, together with \cref{lem:squareLemma} and a homotopy that we give at the end of \cref{app:compatibility}.
\end{proof}

The following proposition follows from \cref{lem:compatibilityDeltaTilde,lem:chainToCochainHomotopy}:
\begin{proposition}\label{prop:compatibility}
    Given any trivialization $(E_0,H)$ of $H$ onto $W$, if $\alpha$ is nondegenerate there are two homotopies $((-1)^n E^T_0 h_1), (E_0,h_2)$ of $F_\alpha$ such that for any closed elements $x_1,x_2 \in K$ and $y \in C_*(A,A)$ we have the equation
    \[ \langle \pi_{h_1,h_2}(x_1,x_2),g_\alpha(y)\rangle = (-1)^{\deg(x_1)\deg(x_2)}\langle x_2\otimes x_1, (\id \otimes g_\alpha)(g_\alpha \otimes \id)\lambda_H(y) \rangle \]
\end{proposition}
We note that in the proposition above, we do not make any assumption about the `balanced' condition, or impose any symmetry condition on $H$ or $E_0$, and that the result holds with any ring $\kk$ over which the data $\alpha,H$ etc. is defined.
\begin{corollary}
    If $\kk$ is a field, for any choice of homotopies $h_1,h_2$ produced by \cref{prop:compatibility}, the homology loop coproduct
    \[ \lambda_H \colon HH_*(A,A)[n-1] \to \overline{HH}_*(A,A)[n-1] \otimes \overline{HH}_*(A,A)[n-1] \]
    is entirely determined by the homology dual product
    \[ \pi_{h_1,h_2} \colon \frac{K}{B_*(Q)}[n] \otimes \frac{K}{B_*(Q)}[-n] \to HH_*(A,A^\vee)[1]. \] 
\end{corollary}
\begin{proof}
    The pairing we used induces an isomorphism $HH_*(A,A^\vee)[-n] \cong (HH_*(A,A))^\vee$, and by \cref{lem:surjectionToDual} also gives a surjective map $\frac{K}{B_*(Q)} \to H_*(Y/W)^\vee$; the result above follows from these two facts together with the equation of \cref{prop:compatibility}, which holds for all closed elements.
\end{proof}

Finally, we relate the two conditions that we called `balanced'.
\begin{proposition}
    If $\kk$ is a field, when the triple $(W,H,\alpha)$ is balanced, in the sense of \cref{def:balancedText}, then any quadruple $(W,h_1,h_2,\alpha)$ from \cref{prop:compatibility} is balanced, in the sense of \cref{def:dualBalanced}.
\end{proposition}
\begin{proof}
    Recall that $(W,H,\alpha)$ is balanced if the image of $W$ under $\lambda_H$ is contained in $W \otimes C_*(A,A) \cap C_*(A,A) \otimes W$. Over a field we have
    \[ H_*(W \otimes C_*(A,A) \cap C_*(A,A) \otimes W) \cong H_*(W \otimes W) \cong H_*(W) \otimes H_*(W) \]
    so for $[y] \in HH_*(A,A)$, we can represent $\lambda(y)$ by a linear combination of pairs of closed elements of $W$. Thus, applying the equation in \cref{prop:compatibility}, we conclude that for $x_1,x_2 \in K$ we have $\pi_{h_1,h_2}(x_1,x_2) \in K$.
\end{proof}

\begin{remark} \label{rem:moreInfo}
    If $\lambda_H$ happens to be defined over some ring $\kk$, the equation of \cref{prop:compatibility} still holds at chain-level, but it may be that the homology coproduct $\lambda_H$ contain \emph{more information} than the dual coproduct. This is expected, since we lose all the information about torsion classes by taking linear duals. We present an example of this fact in \cref{sec:twoSphereIntegers}.
\end{remark}

\subsubsection{Symmetry of $H$}
Note that the statement of \cref{prop:compatibility} does not depend on any symmetry conditions on $H$. Recall that we have $dH = E-E_0$, and $E_0$ is $(-1)^n$-symmetric by assumption.

We know from \cref{lem:Jhomotopy,lem:Lhomotopy} that upon applying $g_\alpha$ on both sides, we have an explicit relation between $E$ and its transpose:
\[ d(J_\alpha + L_\alpha - J_\alpha^T )= (\id \otimes g_\alpha)(g_\alpha \otimes \id)(E-(-1)^n E^T) \]

Therefore it is natural to consider the following symmetry condition on the complex where $H$ lives.
\begin{definition}\label{def:appropriatelySymmetric}
    A (partial) trivialization $H$ of $E$ is \emph{$\alpha$-symmetric} modulo closed terms in $W$ if the expression
    \[ (\id \otimes g_\alpha)(g_\alpha \otimes \id)(H - (-1)^n H^T) + J_\alpha + L_\alpha -J_\alpha^T \]
    is exact in $C^*(A,A)\otimes C^*(A,A)$ modulo terms in $g_\alpha(Z_*(W)) \otimes g_\alpha(Z_*(W))$.
\end{definition}

\begin{proposition}\label{prop:appropriatelysymmetrich}
    If $H$ is $\alpha$-symmetric modulo closed terms in $W$, then $h_1$ and $h_2$ in \cref{prop:compatibility} can be chosen so that their values on $K$ differ by exact terms. 
\end{proposition}
\begin{proof}
   The element $H$ being $\alpha$-symmetric modulo closed terms in $W$ implies that there exists $\Phi \in C^*(A,A)\otimes C^*(A,A)$ such that
   \[ d\Phi = (\id \otimes g_\alpha)(g_\alpha \otimes \id)(H - (-1)^n H^T) + J_\alpha + L_\alpha -J_\alpha^T + X \]
   where $X \in g_\alpha(Z_*(W)) \otimes g_\alpha(Z_*(W))$. Thus $\widetilde{h}_1$ and $\widetilde{h}_2$ in \cref{align:widetildeh} will differ on $K$ by exact values. Then the result follows from the second assertion of  \cref{lem:chainToCochainHomotopy}.
\end{proof}

\begin{theorem} \label{thm:commCoproduct}
    Let $\kk$ be a field. If $H$ is $\alpha$-symmetric modulo closed terms in $W$, then the coproduct $\overline\lambda_H$ is $(-1)^n$-cocommutative at the homology level. If moreover $A$ is connective and $n \ge 3$, then it is also coassociative at the homology level.
\end{theorem}
\begin{proof}
    Combining \cref{prop:appropriatelysymmetrich,prop:dualProdCommutativity,prop:dualProdAssociativity}, we obtain that the homology dual product $\overline\pi_H \coloneqq \pi_{h_1, h_2}$ is $(-1)^n$-commutative, and associative if $A$ is connective and $n\ge 3$. The result then follows from the fact that the space $K/B_*(Q)$ on which this dual product is defined surjects onto $\overline{HH}_*(A,A)^\vee$, so we can deduce cocommutativity and coassociativity of $\overline\lambda_H$ from the dual properties on $\overline\pi_H$.
\end{proof}
The results above are what we stated in the Introduction as \cref{thm:thm1intro}. We conclude this section by making a remark. The condition of being $\alpha$-symmetric modulo closed terms in $W$ stated in \cref{def:appropriatelySymmetric} requires the term in question to be exact modulo something in $g_\alpha(Z_*(W))\otimes g_\alpha(Z_*(W))$, instead of just modulo something in the larger subcomplex
\[ g_\alpha(W) \otimes C^*(A,A) \cap C^*(A,A) \otimes g_\alpha(W). \]
This may seem excessive, as the value of the coproduct $\lambda_H$ is invariant under changes in $H$ by terms in the complex $W\otimes C_*(A,A) \cap C_*(A,A)\otimes W$, but it is necessary to prove our theorem since elements in $K$ are only required to vanish on closed elements of $W$. However, it turns out that this distinction is irrelevant when $\kk$ is a field of characteristic $\neq 2$, in which case one can state the following variant.
\begin{definition}\label{def:appropriatelySymmetric2}
    Over a field $\kk$ of characteristic $\neq 2$, we say that a (partial) trivialization $H$ of $E$ is \emph{$\alpha$-symmetric} modulo $W$ if
    \[ (\id \otimes g_\alpha)(g_\alpha \otimes \id)(H - (-1)^n H^T) + J_\alpha + L_\alpha -J_\alpha^T \]
    is exact in $C^*(A,A)\otimes C^*(A,A)/g_\alpha(W) \otimes g_\alpha(W)$.
\end{definition}

The following proposition says that the definition above is as good as \cref{def:appropriatelySymmetric} for our purposes.
\begin{proposition}
    Over a field $\kk$ of characteristic $\neq 2$, if $H$ is $\alpha$-symmetric modulo $W$ (in the sense of \cref{def:appropriatelySymmetric2}), then there exists some $\widehat H$ which is $\alpha$-symmetric modulo closed terms in $W$ (in the sense of \cref{def:appropriatelySymmetric}) and such that $\widehat{H} -H \in W \otimes W$. Therefore \cref{thm:commCoproduct} holds also for $H$.
\end{proposition}
\begin{proof}
    Given our assumptions, there exists $X \in W \otimes W$ such that 
    \[ (\id \otimes g_\alpha)(g_\alpha \otimes \id)(H - (-1)^n H^T) + J_\alpha + L_\alpha -J_\alpha^T = \text{exact}\ + (\id \otimes g_\alpha)(g_\alpha \otimes \id) X, \]
    so applying the differential we have
    \[ d(\id \otimes g_\alpha)(g_\alpha \otimes \id)X = -(\id \otimes g_\alpha)(g_\alpha \otimes \id)(E_0 -(-1)^n E_0^T). \]
    Taking $\widehat H = H - X/2$ (which makes sense since we can divide by two) we get
    \[ (\id \otimes g_\alpha)(g_\alpha \otimes \id)(\widehat H - (-1)^n \widehat H^T) + J_\alpha + L_\alpha -J_\alpha^T = \text{exact}\ + (\id \otimes g_\alpha)(g_\alpha \otimes \id) \frac{X + (-1)^n X^T}{2}, \]
    whose last term is closed and therefore by K\"unneth it is homologous to something in $g_\alpha(Z_*(W)) \otimes g_\alpha(Z_*(W))$.
\end{proof}

\section{Examples}\label{sec:examples}
The formalism of pre-CY structures and the associated calculus developed in \cite{kontsevich2021precalabiyau} is suitable for constructing homotopies and higher structures described in terms of diagrams. Evaluating such diagrams involves taking sums over many expressions, which is often impractical. However, in particular examples, one can perform explicit calculations by hand. We would like to present some of these calculations using examples of interest for string topology.

In this section we will study the algebraic loop product and coproduct for algebras given by the homology of the loop spaces of spheres. Spheres are formal and coformal \cite{berglund2014koszul,berglund2017free}, so the graded algebra $A = H_*(\Omega S^N,\kk)$ is quasi-isomorphic to the dg algebra of chains on $\Omega S^N$. We reiterate that, for some general manifold $M$, this homology will not be quasi-isomorphic to the algebra of chains, so looking at the homology algebra $H_*(\Omega M)$ would not give the correct answers. 

One should be careful when comparing the calculations we perform below to the geometric loop coproduct, since the algebraic operation $\lambda_H$ we defined \emph{is not invariant} with respect to quasi-isomorphisms of $A$. In fact, for some choices of $H$, $\lambda_H$ might not have a geometric interpretation and the space of ``geometrically meaningful" trivializations seems to be more restricted. 

Still, we believe that the examples we calculate below could be of interest for string topology. It may seem like this chain-level formalism is overly complicated for calculating such simple examples; nevertheless, we argue that it allows for a systematic treatment of signs which, in some cases, as for the even-dimensional spheres \cref{sec:evenSpheres}, allows us to work over $\Z$ and understand the behavior of torsion classes.

The cases of the circle and the 2-sphere are of particular interest; in both of those cases, there are multiple choices of $\alpha$-symmetric (\cref{def:appropriatelySymmetric}) trivialization $H$, some of which give non-coassociative coproducts $\lambda_H$. These examples show that the condition $n\ge 3$ in \cref{thm:commCoproduct} is necessary.

\subsection{Spheres of odd dimension greater or equal to 3}\label{sec:oddSpheres}
Let us take $A = \kk[t], \deg(t) = 2N \ge 2$, seen as a homologically graded dg algebra over $\kk$ with trivial differential. As mentioned above, this is the homology algebra of $\Omega S^{2N+1}$, and since the sphere is coformal, the Hochschild homology of $A$ calculates the homology of $L S^{2N+1}$. More explicitly, it is given by
\[ HH_*(A,A) = \kk\ \bigoplus_{k \ge 1}\ (\kk[2Nk] \oplus \kk[2Nk+1]) \]
This may be computed using the Koszul resolution of $A$ in \cref{align:resolutionkt} below, and we may pick some simple representatives in the Hochschild chain complex $C_*(A, A)$ (see \cref{subsubsection:chainsandcochains}) for each of those classes: the chain of length one $t^k$ for the generator in degree $2Nk$, for each $k \ge 0$, and the chain of length two $t^{k-1}[t]$ for the generator in degree $2Nk+1$, for each $k \ge 1$. Here, $[t] \in \A[1]$.

Let us also calculate and pick explicit representatives for the non trivial classes in the Hochschild cohomology of $A$. Following the noncommutative geometry analogy between Hochschild cochains and vector fields, we think of $A$ as functions on some space with coordinates $s_k, k \ge 0$, where $s_k$ is dual to $t^k$, and denote by $s_{k_1}s_{k_2}\dots s_{k_p}\del^k$ the Hochschild cochain of length $p$ that sends the basis element $[t^{k_1}]\otimes [t^{k_2}]\otimes \dots \otimes [t^{k_p}]$ to $ t^{k}$ and all other basis elements $[t^{k_1'}] \otimes \dots \otimes [t^{k_q'}]$ to zero. We note that for any $k \ge 0$, the cochain $\del^k$ of length zero and degree $2Nk$ is closed, and so is the cochain $\sum_i i s_i \del^{i+k}$, of length one and degree $2Nk+1$. It turns out that these are all the representatives for the nonzero Hochschild cohomology classes. We summarize our chosen representatives below:
\[\hspace{-0.5cm}\begin{array}{|c||c|c|c|c|c|c|c|c|}
    \hline
    \text{chains}           & \dots & t^2[t] & t^3 & t[t] & t^2 & 1[t] & t & 1\\
    \hline
    \text{deg in\ } HH_*(A,A) & \dots & 6N+1 & 6N & 4N+1 & 4N & 2N+1 & 2N & 0\\
    \hline \hline
    \text{cochains}           & \dots & \del^2 & \sum_i is_i\del^{i+2} & \del^1 & \sum_i is_i\del^{i+1} & \del^0 & \sum_i is_i\del^{i} & \sum_i is_i\del^{i-1} \\
    \hline
    \text{deg in\ } HH^*(A,A) & \dots & 4N & 4N-1 & 2N & 2N-1 & 0 & -1 & -2N-1 \\
    \hline
\end{array}
\]

\subsubsection{Chain-level Chern character}
Recall that even before we pick any notion of orientation on $A$ (smooth CY or pre-CY structure), once we pick chain-level representatives for the canonical vertices $\co$ and $\eta$ we can define the chain-level Chern character of the diagonal bimodule $E$ and the operation $G:C^*(A,A) \to C_*(A,A) \otimes C_*(A,A)[-1]$. 

First, we must choose an explicit representative for the bimodule dual $A^!$. For that, we use the following resolution for the diagonal bimodule $A$:
\begin{align}\label{align:resolutionkt} \widetilde A = (A^e[2N+1] \oplus A^e,d_{\widetilde A}) 
\end{align}
with differential that sends elements of the first summand to the second summand by
\[ d_{\widetilde A}(t^k \otimes t^\ell) = t^{k+1} \otimes t^\ell - t^k \otimes t^{\ell+1} \]
that is, sending the generator $1\otimes 1 \in A^e[2N+1] $ to $ t\otimes 1 -1 \otimes t\in A^e$. We then dualize this bimodule to get a representative of $A^!$ given by
\[ A^! = (A^e \oplus A^e[-2N-1], d_{A^!}). \]
Let us denote the elements of the first and second summand by $R^{k,\ell},S^{k,\ell}$, respectively (denoting $t^k\otimes t^\ell$ in each factor), with differential given by
\[ d_{A^!}(R^{k,\ell}) = S^{k+1,\ell} - S^{k,\ell+1} \]
The bimodule structure is given by the structure maps
\begin{align*}
    \mu(t^i,R^{k,\ell}) &= R^{k+i,\ell} \\
    \mu(R^{k,\ell},t^i) &=  R^{k,\ell+i} \\
    \mu(t^i,S^{k,\ell}) &= S^{k+1,\ell} \\
    \mu(S^{k,\ell},t^i) &= -S^{k,\ell+i}  
\end{align*}

We must now find chain-level representatives for the $\co$ and $\eta$ vertices. These are not unique, but we can easily find a choice that satisfies the necessary conditions. For the coevaluation $\co \in A\otimes^\LL_{A^e} A^!$, we have
\begin{align} \label{align:coexample} \co = 1 \otimes [\varnothing] \otimes R^{0,0} \otimes [\varnothing] - 1 \otimes [t] \otimes S^{0,0} \otimes [\varnothing] - 1 \otimes [\varnothing] \otimes S^{0,0} \otimes [t].
\end{align}
Here, $\co$ lies in   $\bigoplus_{i, j \geq 0} A \otimes \A^{\otimes i} \otimes A^! \otimes \A^{\otimes j}$ and  is illustrated as follows, compare \cref{align:coevfigure}
\begin{align*}
  \co=  \tikzfig{
        \node [vertex] (center) at (0,0) {$\co$};
        \node (left) at (-1.2,0) {$1$};
        \node (right) at (1.4,0) {$R^{0,0}$};
        \draw [->-] (center) to (-1,0);
        \draw [->- ,cyan] (center) to (1,0);
    } -   \tikzfig{
        \node [vertex] (center) at (0,0) {$\co$};
        \node (top) at (0, 1.2) {$[t]$};
        \node (left) at (-1.4,0) {$1$};
        \node (right) at (1.4,0) {$S^{0,0}$};
        \draw [->-] (center) to (top);
        \draw [->-] (center) to (-1,0);
        \draw [->-,cyan] (center) to (1,0);
    } -  \tikzfig{
        \node [vertex] (center) at (0,0) {$\co$};
        \node (bottom) at (0, -1.2) {$[t]$};
        \node (left) at (-1.2,0) {$1$};
        \node (right) at (1.4,0) {$S^{0,0}$};
        \draw [->-] (center) to (bottom);
        \draw [->-] (center) to (-1,0);
        \draw [->-,cyan] (center) to (1,0);
    }
\end{align*}
We check that this element is closed and satisfies the universal property of coevaluation.

For the vertex $\eta$ we can then pick the form given by
\begin{align*}
        \tikzfig{
        \node [rectangle,draw] (center) at (0,0) {$\eta$};
        \node (left) at (-1.3,0) {};
        \node (right) at (1.3,0) {};
        \node (top) at (0,1.3) {$t^n$};
        \node (bottom) at (0,-1.3) {$R^{k,\ell}$};
        \draw [->-] (center) to (-1,0);
        \draw [->-] (center) to (1,0);
        \draw [->-] (0,1) to (center);
        \draw [->-,cyan] (0,-1) to (center);
    } &= \eta(t^n \otimes R^{k,\ell})  = t^\ell \otimes t^{n+k} \\ \\
    \tikzfig{
        \node [rectangle,draw] (center) at (0,0) {$\eta$};
        \node (left) at (-1.3,0) {};
        \node (right) at (1.3,0) {};
        \node (topleft) at (-1,1) {$[t^m]$};
        \node (top) at (0,1.3) {$t^n$};
        \node (bottom) at (0,-1.3) {$S^{k,\ell}$};
        \draw [->-] (center) to (-1,0);
        \draw [->-] (center) to (1,0);
        \draw [->-] (0,1) to (center);
        \draw [->-,cyan] (0,-1) to (center);
        \draw [->-] (topleft) to (center);
    } &= \eta([t^m] \otimes t^n \otimes S^{k,\ell}) = -\sum_{1\leq i \leq m} t^{\ell+i-1} \otimes t^{k+n+m-i} 
\end{align*}
with all other terms of $\eta$ being zero. The first term, involving $R^{k,\ell}$, guarantees that $\eta$ and $\co$ satisfy the required compatibility \cref{align:pairingvertex}, and the second term, involving $S^{k,\ell}$, makes $\eta$ closed under taking the necklace bracket with $\mu$.

Using the vertices above we compute the chain-level Euler character to be $E = 0$, since it is a difference between two equal terms $1\otimes 1$, one obtained from pairing the first term of $\co$ with the first term of $\eta$ and the other from pairing the third term of $\co$ with the second term of $\eta$; all other terms in the evaluation of the diagram are zero. Having $E=0$ is obviously expected since the Euler characteristic of $S^{2N+1}$ is zero, and the differential on Hochschild chains vanishes in degree one since $A$ is commutative.

\subsubsection{The map $G$}
Since $E=0$ at chain-level as explained above, the map $G$ is a map of complexes (see \cref{align:homotopyG}), and induces a map in cohomology
\[ G \colon HH^*(A,A) \to HH_*(A,A)\otimes HH_*(A,A)[-1].\]
Let us calculate the chain-level map on $\varphi$ ranging over the representatives we picked for the nonzero cohomology classes. We note that the only nonzero diagram to evaluate in the expression for $G(\varphi)$ is the term
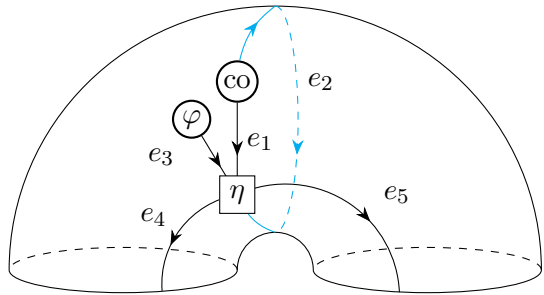
\[\tikzfig{
        \draw (-2,0) [partial ellipse =0:-180:1.5 and 0.3];
        \draw [dashed] (-2,0) [partial ellipse =0:180:1.5 and 0.3];
        \draw (2,0) [partial ellipse =0:-180:1.5 and 0.3];
        \draw [dashed] (2,0) [partial ellipse =0:180:1.5 and 0.3];
        \draw (0.5,0) arc (0:180:0.5);
        \draw (3.5,0) arc (0:180:3.5);
        \node [vertex] (coev) at (-0.5,2.5) {$\co$};
        \node [vertex] (phi) at (-1.1,2) {$\varphi$};
        \node (eta) at (-0.5,1) {};
        \draw [->-] (coev) to (eta);
        \draw [->-] (phi) to (eta);
        \draw (0,2) [->-,partial ellipse=90:-90:0.3 and 1.5,cyan,dashed];
        \draw [->-,bend left=20,cyan] (coev) to (0,3.5);
        \draw [bend left=20,cyan] (0,0.5) to (eta);
        \draw [->-] (eta) arc (115:2.5:1.5); 
        \draw [->-] (eta) arc (100:184.5:1.2); 
        \node [rectangle,draw,fill=white] at (-0.5,1) {$\eta$};
        \node at (-0.2,1.7) {$e_1$};
        \node at (0.6,2.5) {$e_2$};
        \node at (-1.5,1.5) {$e_3$};
        \node at (-1.6,0.7) {$e_4$};
        \node at (1.6,1) {$e_5$};
    }\]
with orientation $(e_5\ e_4\ e_3\ e_2\ e_1)$. We plug in the expressions we picked for $\co$ and $\eta$, and calculate:
\begin{align*}
    G(\del_k) &= \sum_{1 \le i \le k} (t^{i-1} \otimes t^{k-i}[t] - t^{i-1}[t] \otimes t^{k-i} ) \\
    G(\sum_i i s_i \del^{i+k}) &= -\sum_{1 \le j \le k+1} (t^{j-1} \otimes t^{k+1-j} )
\end{align*}
To calculate the second line above, note that $\co$ only contains terms with power of $t$ equal to one (see \cref{align:coexample}) so only the term $G(s_1 \partial^{1+k})$ is nonzero.

The map induced by $G$ on cohomology does not depend on our choice of $\co$ and $\eta$; by the discussion in \cref{sec:covariance}, we could compensate a change in the chain-level expressions for those vertices by a change in $H$, but the space of choices for $H$ (i.e.\ the $HH_0(A,A)\otimes HH_1(A,A)\oplus HH_1(A,A)\otimes HH_0(A,A)$-torsor) is trivial in homology.

\subsubsection{The CY and pre-CY structures}
Let us now pick the orientation on $A$, which for us will be a non-degenerate pre-CY structure. We will prove nondegeneracy by first giving a compatible smooth CY structure:
\begin{proposition}\label{prop:inverseOddSphere}
    The Hochschild chain $\omega = 1[t] \in C_{2N+1}(A,A)$ is non-degenerate.
\end{proposition}
\begin{proof}
    We will prove this by explicitly giving its inverse pre-CY structure. Let $\alpha \in C^{2N+1}_{(2)}(A)$ be given by the formulas
    \begin{align*}
        \tikzfig{
        \node [vertex] (center) at (0,0) {$\alpha$};
        \node (top) at (0,1) {$t^k$};
        \node (left) at (-1.8,0) {$\alpha(\varnothing,t^k)'$};
        \node (right) at (1.8,0) {$\alpha(\varnothing,t^k)''$};
        \draw [->-] (top) to (center);
        \draw [->-] (center) to (-1,0);
        \draw [-w-] (center) to (1,0);
    } &\qquad \alpha(\varnothing, t^k)' \otimes \alpha(\varnothing, t^k)'' = \sum_{0\le i \le k-1} t^i \otimes t^{k-1-i} \\
    \tikzfig{
        \node [vertex] (center) at (0,0) {$\alpha$};
        \node (bot) at (0,-1) {$t^k$};
        \node (left) at (-1.8,0) {$\alpha(t^k,\varnothing)'$};
        \node (right) at (1.8,0) {$\alpha(t^k,\varnothing)''$};
        \draw [->-] (bot) to (center);
        \draw [->-] (center) to (-1,0);
        \draw [-w-] (center) to (1,0);
    } &\qquad \alpha(t^k,\varnothing)' \otimes \alpha(t^k,\varnothing)'' = - \sum_{0\le i \le k-1} t^i \otimes t^{k-1-i}
    \end{align*}
    with all other components of $\alpha$ being zero. Let us explain this better; $\alpha$ only evaluates non-trivially on exactly one input, sending $t^k$ to the tensor $\pm\sum t^i \otimes t^{k-1-i}$, with sign depending on which side this input is on; recall that since the CY dimension is odd, $\alpha$ needs to be antisymmetric under the rotation. We can check that this is a closed element of $C^*_{(2)}(A)$, under the differential given by taking necklace bracket with the Hochschild cochain giving the $A_\infty$ structure $\mu$ (in this case just the multiplication). In this case this means explicitly verifying the following equation
    \[ \tikzfig{
        \node [vertex] (center) at (0,0) {$\alpha$};
        \node [bullet] (mu) at (1,0) {};
        \node (topleft) at (0,1.5) {$t^k$};
        \node (topright) at (1,1.5) {$t^\ell$};
        \draw [->-] (topleft) to (center);
        \draw [->-] (topright) to (mu);
        \draw [->-] (center) to (-1.2,0);
        \draw [-w-] (center) to (mu);
        \draw [->-] (mu) to (2.2,0);
    } + \tikzfig{
        \node [bullet] (mu) at (0,0) {};
        \node [vertex] (center) at (1,0) {$\alpha$};
        \node (topleft) at (0,1.5) {$t^k$};
        \node (topright) at (1,1.5) {$t^\ell$};
        \draw [->-] (topleft) to (mu);
        \draw [->-] (topright) to (center);
        \draw [-w-] (center) to (2.2,0);
        \draw [->-] (center) to (mu);
        \draw [->-] (mu) to (-1.2,0);
    } = \tikzfig{
        \node [vertex] (center) at (0,0) {$\alpha$};
        \node (topleft) at (-0.5,1.5) {$t^k$};
        \node (topright) at (0.5,1.5) {$t^\ell$};
        \node [bullet] (mu) at (0,0.7) {};
        \draw [->-] (topleft) to (mu);
        \draw [->-] (topright) to (mu);
        \draw [->-] (mu) to (center);
        \draw [->-] (center) to (-1,0);
        \draw [-w-] (center) to (1,0);
    }\]
    for all $k, \ell \ge 0$, which holds since both sides are equal to $-\sum_{0\le i \le k+\ell-1} t^i \otimes t^{k+\ell-1-i}$.

    It remains to check that $\alpha$ is indeed an inverse to $\omega$, that is, that the Hochschild cochain given by
    \[\tikzfig{
        \draw (0,0) circle (1.5);
	\node [vertex] (v3) at (0,0.8) {$\alpha$};
	\node [vertex] (v2) at (0,0) {$\omega$};
	\node [bullet] (v4) at (0.8,0) {};
	\node [bullet] (v5) at (0,-0.8) {};
	\draw [->-] (v2) to (v4);
	\draw [-w-,shorten <=6pt] (0,0.8) arc (90:0:0.8);
	\draw [->-,shorten <=6pt] (0,0.8) arc (90:270:0.8);
	\draw [->-] (0.8,0) arc (0:-90:0.8);
	\draw [->-] (v5) to (0,-1.5);
    }\]
    is cohomologous to the unit cochain; in this case one can calculate that it is exactly equal to the unit cochain.
\end{proof}

In order to apply our results, we need the element $\alpha$ we found above to be part of a pre-CY structure. In this case, there is no more trivial terms.
\begin{proposition}
    Taking $m = \mu + \alpha, m_{(\ge 3)}=0$, gives a pre-CY structure of dimension $2N+1$ on $A$.
\end{proposition}
\begin{proof}
    We already verified that $[\mu,\alpha] = 0$, so the only equation left to verify is that $[\alpha,\alpha] = 0$ in $C^*_{(3)}(A)$. Since $\alpha$ is only nonzero evaluated on exactly one input, the only possibly nontrivial term in $[\alpha,\alpha]$ can be on exactly one input. By symmetry it is enough to place it in the first of the three angles; we then have that the only nontrivial terms are
    \begin{align*} [\alpha,\alpha](t^k,\varnothing,\varnothing) &= \tikzfig{
        \node [vertex] (alpha2) at (0,0) {$\alpha$};
        \node [vertex] (alpha1) at (1,0) {$\alpha$};
        \node (in) at (1,-1.2) {$t^k$};
        \draw [->-] (alpha1) to (2.2,0);
        \draw [-w-] (alpha1) to (alpha2);
        \draw [->-] (alpha2) to (-1,1.2);
        \draw [-w-] (alpha2) to (-1,-1.2);
        \draw [->-,bend left=30] (in) to (alpha1);
    } + \tikzfig{
        \node [vertex] (alpha2) at (0,0) {$\alpha$};
        \node [vertex] (alpha1) at (-0.75,-0.9) {$\alpha$};
        \node (in) at (1,-1.2) {$t^k$};
        \draw [->-] (alpha2) to (1.2,0);
        \draw [-w-] (alpha1) to (alpha2);
        \draw [-w-] (alpha2) to (-1,1.2);
        \draw [->-] (alpha1) to (-1.25,-1.5);
        \draw [->-,bend right=30] (in) to (alpha1);
    } \\
    &= -\sum_{i_1+i_2+i_3 = k-1}  t^{i_1} \otimes t^{i_2} \otimes t^{i_3} + \sum_{i_1+i_2+i_3 = k-1}  t^{i_1} \otimes t^{i_2} \otimes t^{i_3} = 0 
    \end{align*}
\end{proof}

\subsubsection{The algebraic loop coproduct}
As we mentioned above, the space of choices for $H$ has trivial homology; combined with the fact that $E=0$ at chain-level, we conclude that any choice of $H$ (that is, with $dH=0$) will only contribute zero to the homology-level coproduct.

We use the element $\alpha$ we found above to calculate the chain-level map $g_\alpha: C_*(A,A) \to C^*(A,A)[2N+1]$ on the chosen representatives of the nonzero classes. We calculate that for every $k \ge 0$ we have
\[ g_\alpha(t^k) = -\sum_i is_i\del^{i+k-1}, \quad g_\alpha(t^k[t]) = \del_k, \]
so composing with the map $G$ we get:
\begin{align*}
    \lambda(t^k) &= \sum_{1 \le i \le k} t^{i-1} \otimes t^{k-i} \\
    \lambda(t^k[t]) &= \sum_{1 \le i \le k} (t^{i-1} \otimes t^{k-i}[t] - t^i[t] \otimes t^{k-1-i})
\end{align*}
As expected from \cref{thm:symmetry}, since the CY dimension $\ge 3$ and $H=0$ is symmetric, the coproduct above is cocommutative and coassociative. It also agrees with the coproduct on the free loop space homology defined topologically. Comparing to the expression in \cite{CieHinOan2}, for example, the identification between our basis and theirs is given by $t^k \leftrightarrow AU^k$ and $t^k[t] \leftrightarrow U^k$.

\subsubsection{The cone bimodule}
For the sake of illustration, let us now describe the dual picture, that is, the product on $HH^*(A,A^\vee)$. For that, we use the bimodule $M = \Cone(f_\alpha)$ which is quasi-isomorphic to the categorical formal punctured neighborhood of infinity since $\alpha$ is non-degenerate. To compute the morphism of bimodules $f_\alpha: B\A[1] \otimes A^\vee \otimes B\A[1] \to A$, let us fix some notation: we denote by $s_k$ the element of degree $-2Nk$ in $A^\vee$ which is dual to $t^k$, that is, which maps $t^\ell$ to $\delta_{\ell-k}$. The set $\{s_k\}_{k\in \N}$ is not a complete basis of the complex $A^\vee$ (which is, after all, just the algebraic dual so `bigger' than $A$), but it will suffice to describe the map $f_\alpha$ on it. We have that the only nontrivial values are
\[ f_\alpha(s_k \otimes t^{\ell}) = -\chi_{k \le \ell-1}\ t^{\ell-1-k}, \quad f_\alpha(t^{\ell} \otimes s_k) = \chi_{k \le \ell-1}\ t^{\ell-1-k} \]
where $\chi_{\dots}$ is the `indicator function', that is, one if $\dots$ is satisfied and zero if it is not. The induced map on Hochschild chains is then $F_\alpha: C_*(A,A^\vee[-2N-1]) \to C_*(A,A)$ is then given by
\[ F_\alpha(s_k[t^{k_1}|\dots|t^{k_p}]) = -\chi_{k \le k_1-1}\ t^{k_1-1-k}[t^{k_2}|\dots|t^{k_p}] + (-1)^{p-1} \chi_{k \le k_p-1}\ t^{k_p-1-k}[t^{k_1}|\dots|t^{k_{p-1}}] \]

We can easily show that this map is zero on cohomology. Since $A$ is smooth $(2N+1)$-CY, we have $HH_*(A,A^\vee[-2N-1]) = HH^*(A,A^\vee) = (HH_*(A,A))^\vee$ which is supported in degrees $0,-2N,-2N+1,-4N,-4N-1,\dots$. Meanwhile, the target of $F_\alpha$ has cohomology supported in degrees $0,2N,2N+1,4N,4N+1,\dots$. Therefore the only possible nontrivial map would be represented in degree zero, but $F_\alpha$ vanishes there.

\subsubsection{The product on the cone}
We have that $C_*(A,M) \simeq \Cone(F_\alpha) = C^*(A,A^\vee[-2N]) \oplus C_*(A,A)$, with the extra differential given by $F_\alpha$ above. We calculate the extension of the pre-CY structure from $A$ to $M$, which depends on $\alpha$ and $\tau$ (zero in this case)\footnote{Again, $M$ as a complex is actually the sum of the bar resolution of $A^\vee[-2N-1]$ and $A$; in this case only the length zero elements have nontrivial products.}
\begin{align*}
    \mu_M(t^k,t^\ell) &= 0+t^{k+\ell},\\
    \mu_M(s_k,t^\ell) &= s_{k-\ell} - \chi_{k\le \ell-1} t^{\ell-1-k},\\
    \mu_M(t^\ell,s_k) &= s_{k-\ell} + \chi_{k\le \ell-1} t^{\ell-1-k},\\
    \mu_M(s_k,s_\ell) &= s_{k+\ell+1} + 0.
\end{align*}

This allows us to calculate the product $\pi_M$ for any two cochains in $C_*(A,M)$. It suffices to compute the product on the chosen representatives; besides the representatives for the classes in $HH_*(A,A)$, we can choose representatives for the nontrivial classes in $HH^*(A,A^\vee)$:
\[ s_k, \ \deg = 2N(k+1) \quad  \text{\ and\ } \quad  s_k[t],\  \deg = 2N(k-1)-1,\  \text{for $k \ge 0$} \]
We compute:
\begin{align*}
    \pi(t^k[t],t^\ell) &= -t^{k+\ell}\\
    \pi(t^k[t],t^\ell[t]) &= t^{k+\ell}[t]\\
    \pi(s_k[t],s_\ell) &= -s_{k+\ell+1},\\
    \pi(s_k[t],s_\ell[t]) &= s_{k+\ell+1}[t],\\
    \pi(s_k[t],t^\ell) &= -s_{k-\ell} + \chi_{k\le \ell-1} t^{\ell-1-k},\\
    \pi(s_k,t^\ell[t]) &= s_{k-\ell} - \chi_{k\le \ell-1} t^{\ell-1-k},\\
    \pi(s_k[t],t^\ell[t]) &= s_{k-\ell}[t] - \chi_{k\le \ell-1} t^{\ell-1-k}[t]
\end{align*}
with the remaining operations either determined from the above by graded skew-commutativity or zero otherwise. To lift the product to $C^*(A,A^\vee)$, one would need to pick nullhomotopies $h_1,h_2$. But even though any nullhomotopy $h$ of $F_\alpha$ must be nonzero, we only need to evaluate it on the cycles $s_k,s_k[t]$ above; the respective images $h(s_k)$ and $h(s_k[t])$ live in degrees $-2N(k+1)$ and $-2Nk+1$, where the complex vanishes. So for any pair of homotopies $h_1,h_2$, the corresponding product on $C_*(A,A^\vee[-2N])$ is just given, up to exact terms, by: 
\begin{align*}
    \pi_{h_1,h_2}(s_k[t],s_\ell) &= -s_{k+\ell+1},\\
    \pi_{h_1,h_2}(s_k[t],s_\ell[t]) &= s_{k+\ell+1}[t].\\
\end{align*}
As expected from \cref{prop:compatibility}, this product is dual to the coproduct $\lambda$, using the pairing
\[ (s_k,t^\ell[t]) \mapsto -\delta_{k-\ell}, \quad (s_k[t],t^\ell) \mapsto \delta_{k-\ell}. \]
Note also that since there are no torsion classes in $HH_*(A,A)$, the product $\pi_{h_1,h_2}$ on $HH_*(A,A^\vee)$ contains the same amount of information as the coproduct $\lambda$ on $HH_*(A,A)$.

\subsection{Spheres of even dimension greater or equal to 4}\label{sec:evenSpheres}
Let us consider the dg algebra $A = \kk[t]$ with $\deg(t)= 2N-1, N \ge 2$, with trivial differential.

In order not to get confused about signs, which will be very important in this example, let us be more precise: as an $A_\infty$-algebra, we have $\mu^{\neq 2}=0$, and $\mu^2$ is given by $\mu^2(t^k,t^\ell) = (-1)^k t^{k+\ell}$, where we view  $\mu^i \in C^2(A, A).$ This affects the signs in the Hochschild chain differential $\partial = L_\mu$, see \cref{ex:capProducts}. 

The Hochschild homology $HH_*(A,A)$ will depend on our choice of ring $\kk$; the only difference between this case and the case of \cref{sec:oddSpheres} is that the differential $C_1(A,A)\to C_0(A,A)$ does not vanish, and instead we have $d(t^{2i-1}[t]) = -2 t^{2i}$ for every $i \ge 1$. Therefore, over $\Z$ we do not have the classes represented by $t^\mathrm{odd}[t]$, and the classes represented by $t^\mathrm{even}$ are $2$-torsion:
\[ HH_*(A,\Z) = \Z \oplus \bigoplus_{1 \le k\ \mathrm{odd}} \left(\Z[(2N-1)k] \oplus \Z[(2N-1)k+1]\right) \oplus \bigoplus_{2 \le \ell\ \mathrm{even}} \frac{\Z}{2\Z}[(2N-1)\ell] \]

Let us also calculate representatives for the nonzero classes in $HH^*(A,A)$. As in the example of \cref{sec:oddSpheres}, they can all be realized by cochains of length zero and one, except that the representatives will change. For length zero we calculate that:
\[ 
    d(\del^k)(t^i) = \begin{cases}
        0 & k\ \text{even} \\
        (-1+(-1)^i)t^{k+i} & k\ \text{odd}
    \end{cases}
\]
so only the cochains $\del^k$ with $k$ even represent nonzero classes, each of degree $k(2N-1)$. As for cochains of length one, denoting we have the following closed elements
\[ \sum_{i \ge 1} i s_i \del^{i+k}, k\ \text{even}, \qquad \sum_{i \ge 1}\frac{1-(-1)^i}{2}s_i \del^{i+k}, k\ \text{odd}. \]
of degree $k(2N-1)-1$. To simplify notation, let us set $\sigma_i = \frac{1-(-1)^i}{2}$, or equivalently $\sigma_i= i\mod 2$; Note that for odd $k\ge 1$, $\sum_{i \ge 1}\sigma_i s_i \del^{i+k}$ has order $2$, since twice this class is equal to $-d(\del^k)$ by the calculation above.  Denoting the $2$-torsion classes in red, we summarize our chosen representative chains in $HH_*(A,\Z)$ and cochains in $HH^*(A,\Z)$:
\[\hspace{-0.3cm}\begin{array}{|c||c|c|c|c|c|c|c|c|}
    \hline
    \text{chains}               & \dots & \textcolor{red}{t^4} & t^2[t] & t^3 & \textcolor{red}{t^2} & 1[t] & t & 1\\
    \hline
    \text{deg }  & \dots & 8N-4 & 6N-2 & 6N-3  & 4N-2 & 2N & 2N-1 & 0\\
    \hline  \hline
    \text{cochains}             & \dots & \textcolor{red}{\sum_{i \ge 1}\sigma_i s_i \del^{i+3}} & \del^2 & \sum_i i s_i \del^{i+2} & \textcolor{red}{\sum_i \sigma_i s_i \del^{i+1}} & \del^0 & \sum_i i \sigma_i \del^i & \sum_i \sigma_i s_i \del^{i-1} \\
    \hline
    \text{deg}  & \dots & 6N-4 & 4N-2 & 4N-3 & 2N-2 & 0 & -1 & -2N \\
    \hline
\end{array}
\]
In other words, over any field $\kk_{\neq 2}$ of characteristic $\neq 2$, the red classes are zero in (co)homology. Over a field $\kk_2$ of characteristic two, the chains $t^\mathrm{odd}[t]$ and cochains $\del^\mathrm{even}$ are closed, and represent nonzero classes in $HH_*(A,\kk_2)$ and $HH^*(A,\kk_2)$, respectively.

\subsubsection{Chain-level Chern character and map $G$}
The description of $A^!$ and the vertices $\co$ and $\ev$ is just like in the case of odd-dimensional spheres, with the only difference being in signs. More precisely, the inverse dualizing bimodule is still given by
\[ A^! = (A^e \oplus A^e[-2N], d_{A^!}) \]
with basis of first and second summand given by elements labeled $R^{k,\ell},S^{k,\ell}$, with differential still given by $d_{A^!}(R^{k,\ell}) = S^{k+1,\ell} - S^{k,\ell+1}$. On the other hand, the bimodule structure maps also have some signs we need to take into account, for example:
\begin{align*}
    \mu(t^i,R^{k,\ell}) &= (-1)^i R^{k+i,\ell} \\
    \mu(R^{k,\ell},t^i) &= (-1)^{k+\ell} R^{k,\ell+i} \\
    \mu(t^i,S^{k,\ell}) &= (-1)^i S^{k+1,\ell} \\
    \mu(S^{k,\ell},t^i) &= (-1)^{k+\ell} S^{k,\ell+i}  
\end{align*}

The elements $\co$ and $\eta$ also differ from the case of \cref{sec:oddSpheres} by signs; we have 
\begin{align*} 
    \co &= 1 \otimes R^{0,0} + 1 \otimes [t] \otimes S^{0,0} + 1 \otimes S^{0,0} \otimes [t]\\
    \eta(t^n; R^{k,\ell}) &= (-1)^{n(k+\ell+1)+(k+1)\ell}t^\ell \otimes t^{k+n} \\
    \eta([t^m],t^n; R^{k,\ell}) &= (-1)^{(m+n)(k+\ell+1) + k\ell + m}\sum_{i=1}^m (-1)^{(i-1)(k+\ell+1)}t^{\ell+i-1}\otimes t^{k+n+m-i}
\end{align*}
The complicated signs above arise from the requirement that $\eta$ be closed under the correct differential $d$ given by taking the necklace bracket with the structure maps. We recall that when exchanging terms past one another we use the reduced degree.

Using the vertices above we compute $E=2\times(1\otimes 1)$, as expected for an even-dimensional sphere. Let us now calculate $G$ on our chosen representatives:
\begin{align*}
    G(\del_k) &= \sum_{i=1}^k (-1)^{i+1} (t^{i-1} \otimes t^{k-i}[t] + t^{i-1}[t] \otimes t^{k-i} ) \\
    G(\sum_i i s_i \del^{i+k}) = G(\sum_i \sigma_i s_i \del^{i+k})&= \sum_{i=1}^{k+1} (-1)^{i+1}(t^{i-1}\otimes t^{k-i})
\end{align*}

\subsubsection{Pre-CY structure and chain-level coproduct}
The CY and pre-CY structure here is also very similar to the one discussed in the case of \cref{sec:oddSpheres}, with difference in signs; here $\alpha$ will symmetric with respect to the $\Z/2\Z$-action changing the marking of the first output, since the dimension is even. That is,
\[ \alpha(\varnothing, t^k) = \alpha(t^k,\varnothing) = \sum_{0 \le i \le k-1} t^i \otimes t^{k-i-1} \]
is the only non-zero term for $m_{(2)} = \alpha$ and there are no higher terms in the pre-CY structure, since we calculate $\alpha \circ \alpha = 0$.

We can then calculate the chain-level map $g_\alpha: C_*(A,A) \to C^*(A,A)[2N]$ on the chosen representatives of the nonzero classes. We calculate that for every $k \ge 0$ we have
\[ g_\alpha(t^k) = \begin{cases}
    \sum_i is_i\del^{i+k-1} & k \text{\ odd},\\
    \sum_i \sigma_i s_i\del^{i+k-1} & k \text{\ even}
\end{cases},  \quad g_\alpha(t^k[t]) = \del^k, \]

To get a chain-level coproduct, we must pick a partial trivialization of $E$. We pick $E_0 = E = 2(1\otimes 1)$ and $W = 2\kk \subseteq C_0(A,A)$, that is, the $\kk$-submodule spanned by a chain of length one given by $2\in A$. Any partial trivialization $H$ satisfies $dH=0$. Since $HH_1(A,A)=C_1(A,A)=0$, the torsor in which $H$ lives has trivial homology, and any choice of $H$ will give the same coproduct; this would hold even if we picked some other quasi-isomorphism algebra $A'\simeq A$. Therefore the chain-level coproduct is given by 
\begin{align*}
    \lambda(t^k) &= \sum_{i=1}^{k} (-1)^{i+1}(t^{i-1} \otimes t^{k-i}) \\
    \lambda(t^k[t]) &= \sum_{i=1}^{k} (-1)^{i+1}(t^{i-1}\otimes t^{k-i}[t] + t^{i-1}[t]\otimes t^{k-i}),
\end{align*}
and we can conclude that $(W,H,\alpha)$ is balanced for any choice of $H$.

\subsubsection{The product on the cone and lift to product on the dual}
The expression for the morphism $f_\alpha$ is also similar to the case of \cref{sec:oddSpheres}, just differing in signs. The induced map on Hochschild chains is
\[ F_\alpha(s_k[t^{k_1}|\dots|t^{k_p}]) = \chi_{k \le k_1-1}\ (-1)^{\sharp_1} t^{k_1-1-k}[t^{k_2}|\dots|t^{k_p}] + \chi_{k \le k_p-1}\ (-1)^{\sharp_2}t^{k_p-1-k}[t^{k_1}|\dots|t^{k_{p-1}}] \]
where $\sharp_1=kk_1, \sharp_2 =(k_p+1)(k+k_1+\dots+k_{p-1}+p+k)+k+k_p$. We recall also that one must change the signs of the pairing between $A^\vee$ and $A$, setting $\ev(s_k,t^\ell) = (-1)^k \delta_{k-\ell}$

We evaluate some of the products on the cone $C_*(A,A^\vee[-2n+1])\otimes C_*(A,A)$:
\begin{align*}
    \pi(t^k[t],t^\ell) &= (-1)^{(k+1)\ell}t^{k+\ell}\\
    \pi(t^k[t],t^\ell[t]) &= (-1)^{(k+1)\ell}t^{k+\ell}[t]\\
    p\pi(t^k[t],s_\ell) &= (-1)^{(k+1)\ell}s_{\ell-k}\\
    p\pi(s_k[t],t^\ell) &= (-1)^{(k+1)\ell}s_{k-\ell}\\
    p\pi(s_k[t],s_\ell) &= (-1)^{k(\ell+1)}s_{k+\ell+1},\\
    p\pi(s_k[t],s_\ell[t]) &= (-1)^{k(\ell+1)}s_{k+\ell+1}[t],
\end{align*}
again with other products determined by commutativity or zero otherwise. We note that the resulting product on $C_*(A,A^\vee)$ does not respect the differential on all of that complex; if we take $k=0, \ell = 2i+1$, the product of closed chains $p\pi(s_0[t],s_{2i+1}) = s_{2i+2}[t]$ is not closed. 

The cohomology of $C_*(A,A^\vee[-2N])$ is concentrated in degrees $\ge 0$ and the cohomology of $C_*(A,A)$ in degrees $\le 0$; therefore the only nontrivial map in homology is in degree zero, where the map is multiplication by $2$. Note that $s_0[t]$ is the only `problematic class'; this is expected, since we know that upon picking a partial homotopy of $F_\alpha$, we will only be able to lift this product to a subcomplex of $C_*(A,A^\vee)$.

Again, we do not have much of a choice for this partial homotopy; we pick $E_0 = E = 2(1\otimes 1)$, which gives a map $q:C_*(A,A^\vee[-2N]) \to 2\kk \subset C_0(A,A)$ given by sending $s_0[t] \to 2$ in degree zero and zero in other degrees. Making any choice of pairs $h_1,h_2$ of partial homotopies, that is, with $dh+hd = F_\alpha -iq$ for $h=h_1,h_2$, we can lift the product to $\ker(q)$, but in fact, though the chain-level expression for any choice of $h$ will be nontrivial, also in this case any two such choices differ by a $[d,-]$-exact term; moreover, when evaluated on chains of lengths one and two we already have $\ ^\sharp E_0=F_\alpha$.

So on our representatives we can ignore the $H$-terms, and for any choice of $H$ we get the following product, up to exact terms:
\begin{align}
    \pi_H(s_k,s_\ell) &= 0,\\
    \pi_H(s_k[t],s_\ell) &= (-1)^{k(\ell+1)}s_{k+\ell+1},\\
    \pi_H(s_k[t],s_\ell[t]) &= (-1)^{k(\ell+1)}s_{k+\ell+1}[t],
\end{align}
As expected, the resulting product on $C_*(A,A^\vee)$ does not respect the differential on all of that complex; if we take $k=0, \ell = 2i+1$, the product of closed chains $p\pi(s_0[t],s_{2i+1}) = -s_{2i+2}[t]$ is not closed. Note that $s_0[t]$ is the only `problematic class'.

\subsubsection{Over a field of characteristic $\neq 2$}\label{sec:evenSphereFieldCharNot2}
Over $\kk=\kk_{\neq 2}$, we have $W \cong \kk_{\neq 2}$, that is, the rank one vector space in degree zero. Since $\lambda_H(1)=0$ we have a map of complexes of vector spaces
\[ \lambda_H: C_*(A,A)/W\to C_*(A,A)/W \otimes C_*(A,A)/W \]

Let us analyze the induced coproduct in homology.
the closed classes in $HH_*(A,\kk_{\neq 2})/W$ represented by $t^k$, which is exact if $k$ is even, and $t^\ell[t]$, only with $\ell$ even. We note that each summand in the expression for $\lambda_H(t^k)$ is closed, and since $[t^\mathrm{even}]=0$, this coproduct in homology is just given by
\[ \lambda(t^k) = -\sum_{i=1}^{(k-1)/2} t^{2i-1} \otimes t^{k-2i} \qquad k \text{\ odd} \]

As for $\lambda(t^\ell[t])$, we note that even though some of the terms in its expression are not individually closed (since $\partial(t^{2i-1}[t]) = -2t^{2i}$), the whole sum is closed under the differential $\partial_{C_*\otimes C_*}$ on the tensor product $C_*(A,A)/\kk \otimes C_*(A,A)/\kk$. Moreover, we have 
\[ b(t^{2i-1}[t] \otimes t^{2j-1}[t]) = -2t^{2i} \otimes t^{2j-1}[t] + 2t^{2i-1}[t] \otimes t^{2j} \]
so every difference $(t^{2i} \otimes t^{2j-1}[t] -t^{2i-1}[t] \otimes t^{2j})$ is zero in homology. Thus the coproduct in homology is given by
\[\lambda(t^\ell[t]) = -\sum_{i=1}^{\ell/2}t^{2i-1} \otimes t^{\ell-2i}[t] \ + \sum_{i=0}^{\ell/2-1} t^{2i}[t] \otimes t^{\ell-1-2i} \qquad \ell \text{\ even}\]

Comparing with the product on the dual, we first calculate
\[ HH_*(A,A^\vee,\kk_{\neq 2}) = \kk_{\neq 2}[2N] \oplus \bigoplus_{i \le 0} (\kk_{\neq 2}[-(2N-1)(2i)+1] \oplus \kk_{\neq 2}[-(2N-1)(2i)+1])\]
with representatives for the nonzero classes given by
\[\begin{array}{|c||c|c|c|c|c|c|c|c|c|}
	\hline
    \text{chains} & s_0[t] & s_1[t] & s_0 & s_3[t] & s_2 & \dots & s_{2i+1}[t] & s_{2i} & \dots\\
    \hline
    \text{deg \ } & 2N & 1 & 0 & -4N+3 & -4N+2 & \dots & -(2N-1)(2i)+1 & -(2N-1)(2i) &  \dots\\
	\hline
\end{array}
\]

The product induced on $H^*(\ker(q)) = \bigoplus_{i>-2n} HH_i(A,A^\vee;\kk_{\neq 2})$ is then given by
\begin{align} 
    \pi_H(s_k,s_\ell) &= 0, & k,\ell \text{\ even} \\
    \pi_H(s_k[t],s_\ell) &= -s_{k+\ell+1}, & k \text{\ odd},\ell \text{\ even}\\
    \pi_H(s_k[t],s_\ell[t]) &= s_{k+\ell+1}[t], & k,\ell \text{\ odd}
\end{align}
which as expected from \cref{thm:thm1intro}, is commutative and associative.

Comparing the coproduct $\lambda_H$ with this product, we check that these operations satisfy the compatibility relation predicted by \cref{prop:compatibility}, and as a consequence $\lambda_H$ is cocommutative and coassociative, which can also be checked directly. As we are over a field, the homology operations $\lambda_H$ and $\pi_H$ contain exactly the same information. Comparing to the product $\pi$ on $HH_*(A,\kk_{\neq 2})$, we also see that Sullivan's relation is satisfied, as expected from \cref{thm:sullivan}.

\subsubsection{Over a field of characteristic $2$}
Over the field $\kk=\kk_2$, basically the same description as above holds, except that for every $i\ge 0$ we have extra nonzero classes $t^{2i+1}[t] \in HH_*(A,\kk_2)$, which get mapped to nonzero classes $\del^{2i}$. On the dual, we have two extra families of nonzero classes, represented by $s_{2i+1}$ and $s_{2i+2}[t]$, for $i\ge 0$. 

Another difference is that $E=0$, so we do not need to worry about quotienting out by some $W$; we get a coproduct
\begin{align*}
    \lambda_H(t^k) &= \sum_{i=1}^{k} t^{i-1} \otimes t^{k-i} \\
    \lambda_H(t^k[t]) &= \sum_{i=1}^{k} t^{i-1}\otimes t^{k-i}[t] + t^{i-1}[t]\otimes t^{k-i}.
\end{align*}
This coproduct is dual to the product on $HH_*(A,A^\vee;\kk_2)$, for any $k,\ell \ge 0$
\begin{align} 
    \pi_H(s_k,s_\ell) &= 0  \\
    \pi_H(s_k[t],s_\ell) &= s_{k+\ell+1}\\
    \pi_H(s_k[t],s_\ell[t]) &= s_{k+\ell+1}[t]
\end{align}
and as a consequence, $\lambda_H$ is again cocommutative and coassociative.

\subsubsection{Interaction with torsion classes for integer coefficients}\label{sec:twoSphereIntegers}
Let us now return to coefficients $\kk=\Z$. The chain-level coproduct $\lambda_H$ is a map of complexes and gives a map on homology:
\[ \lambda_H: HH_*(A,A)/2\Z \to H_*(C_*(A,A)/2\Z \otimes C_*(A,A)/2\Z) [-2N+1] \]
Unlike what we had for field coefficients, the target of this map is not isomorphic to the tensor product of reduced homologies $HH_*(A,\Z)/2\Z\ \otimes\ HH_*(A,\Z)/2\Z$; recall the K\"unneth short exact sequence:
\[ \hspace{-1cm} \bigoplus_{a+b=k} \frac{HH_a(A,A)}{2\Z} \otimes \frac{HH_b(A,A)}{2\Z} \to H_k\left(\frac{C_*(A,A)}{2\Z} \otimes \frac{C_*(A,A)}{2\Z}\right) \to \bigoplus_{a+b=k-1} \Tor^\Z_1\left(\frac{HH_a(A,A)}{2\Z}, \frac{HH_b(A,A)}{2\Z}\right) \]
The third term is nonzero when the degrees are both even multiples of $(2N-1)$, say $a=2i(2N-1),b=2j(2N-1)$, in which case the Tor group contributes a $\Z/2$ summand; let us denote the corresponding 2-torsion class by $\gamma_{i,j}$, with degree $(2i+2j)(2N-1)+1$. We can choose the following representatives for these classes
\begin{align*}
    \gamma_{i,j} &:= t^{2i} \otimes t^{2j-1}[t] -t^{2i-1}[t] \otimes t^{2j}, \quad \text{for\ } i \ge 0, j \ge 1
\end{align*}
With this notation, the coproduct
\[\lambda: HH_*(A,A) \to H_{*-2N+1}\left(\frac{C_*(A,A)}{2\Z} \otimes \frac{C_*(A,A)}{2\Z}\right)\]
is given by
\begin{align*}
    \lambda(t^k) &= \sum_{i=1}^{k} (-1)^{i-1}\ t^{i-1} \otimes t^{k-i} & \text{any\ } k  \\
    \lambda(t^\ell[t]) &= -\sum_{i=1}^{\ell/2} t^{2i-1} \otimes t^{\ell-2i}[t]\ +\ \sum_{i=0}^{\ell/2-1} t^{2i}[t] \otimes t^{\ell-1-2i} + \sum_{i=0}^{\ell/2-1} \gamma_{i,\ell/2-i} & \ell \text{\ even}
\end{align*}
Therefore, from these calculations it becomes clear that the coproduct $\lambda_H$ `reaches' the $2$-torsion classes. The calculation of the product $\pi_H$ is identical to the one in \cref{sec:evenSphereFieldCharNot2}; therefore, as we pointed out in \cref{rem:moreInfo}, over $\Z$ the coproduct $\lambda_H$ has strictly more information than its dual product $\pi_H$.

\subsection{The 2-sphere}
The case of the 2-sphere is of particular interest; it is basically identical to the case of spheres of even dimension $2N \ge 4$, with the sole difference that here we have a nontrivial space of choices for the partial trivialization $H$, so the coproduct is not uniquely defined. The space of choices for $H$, that is, $C_0(A,A)\otimes C_1(A,A) \oplus C_1(A,A) \otimes C_0(A,A)$ is rank two; we write any $H$ as
\[ H = c_+ \cdot 1\otimes t + c_- \cdot t \otimes 1 \]
for scalars $c_+,c_- \in \kk$. The corresponding correction to the chain-level coproduct is only nontrivial on $\lambda(t^k[t])$, for even $k$:
\begin{align*}
    \lambda_H(t^k[t]) &= \sum_{i=1}^{k} (-1)^{i+1}(t^{i-1}\otimes t^{k-i}[t] + t^{i-1}[t]\otimes t^{k-i}) \\
    &+ c_+(t^k\otimes t + 1 \otimes t^{k+1}) + c_-(t^{k+1}\otimes 1 + t^k \otimes t)
\end{align*}
We see that $\lambda_H$ is cocommutative when $H$ is symmetric, as expected from \cref{thm:thm1intro}. 

Let us analyze the corresponding coproduct on $HH_*(A,\kk)/2\kk$ for different coefficients $\kk$. Over a field $\kk_{\neq 2}$ of characteristic $\neq 2$, the correction vanishes, since the classes represented by $t^\mathrm{even}$ are zero, so the coproduct $\lambda_H$ does not depend on $H$ and is always coassociative. Over $\Z$, the correction $c_+(t^k\otimes t + 1 \otimes t^{k+1}) + c_-(t^{k+1}\otimes 1 + t^k \otimes t)$ represents a 2-torsion class in $HH_*(A,\Z)/2\Z$, and is nontrivial for every $\ge 1$.

\subsubsection{Failure of coassocitivity over a field of characteristic 2} \label{subsection:failure}
The most interesting case is over a field $\kk_2$ of characteristic two: assuming $H$ symmetric, so we get a cocommutative coproduct, there are only two choices for $H$, namely $H_0=0$ and $H_1 = 1\otimes t + 1 \otimes t$. The coproducts only differ when evaluated on $t^k[t]$ for any $k \ge 1$:
\begin{align*}
    \lambda_{H_0}(t^k[t]) &= \sum_{i=1}^{k}( t^{i-1}\otimes t^{k-i}[t] + t^{i-1}[t]\otimes t^{k-i}) \\
    \lambda_{H_1}(t^k[t]) &= \sum_{i=1}^{k}( t^{i-1}\otimes t^{k-i}[t] + t^{i-1}[t]\otimes t^{k-i}) + t^k\otimes t + 1 \otimes t^{k+1} + t^{k+1}\otimes 1 + t^k \otimes t
\end{align*}
Both of these coproducts, together with the product $\pi$, satisfy Sullivan's relation, as expected from \cref{thm:sullivan}. We see that the coproduct $\lambda_{H_0}$ is coassociative, but the coproduct $\lambda_{H_1}$ is not. For example, evaluating on $t[t]$ we get
\begin{align*}
    \lambda_{H_1}(\lambda_{H_1}(t[t])') \otimes \lambda_{H_1}(t[t])' &= 1 \otimes t \otimes 1 + t \otimes 1 \otimes 1,\\
    \lambda_{H_1}(t[t])' \otimes \lambda_{H_1}(\lambda_{H_1}(t[t])'') &= 1\otimes 1\otimes t + 1\otimes t \otimes 1
\end{align*}
which represent different classes in $HH_*(A,\kk_2)$. As a corollary, we conclude that there is no automorphism of the coalgebra $HH_*(A,\kk_2)$ intertwining these coproducts, and they are genuinely different. Note that the failure of coassociativity does not contradict \cref{thm:thm1intro}, since we are not in the dimension range covered by that result.

\begin{remark}
    We believe that this discrepancy between the existence of two coproducts is somehow related to the two distinct BV algebra structures on $H_*(LS^2,\F_2)$ found by Menichi \cite{Men09}; see also \cite{poirier2023note} for a very recent perspective and explanation. In our language of CY structures, we think that the difference between the two BV algebras comes from the choice of different lifts to negative cyclic chains. Though we have not explored that circle action in this paper, we believe that there is some relation between the space of choices of trivialization $H$ and the space of negative cyclic lifts; the two different BV algebra structures would then be related to the two nonequivalent coproducts above. We leave these question for future investigations.
\end{remark}

\subsection{The circle}
Let us now describe the example of the circle, which resembles the example of higher-dimensional odd spheres, but is more complicated due to a large space of choices of trivializations. For now, let $\kk$ be any ring, and we take $A = \kk[t^{\pm 1}]$, with $\deg{t} = 0$, which is the homology algebra $H_*(\Omega S^1)$. Its Hochschild homology is the homology of free loop space
\[ HH_*(A,A) = \kk[\Z] \oplus \kk[\Z][1], \]
We pick representatives for the nonzero classes to be $t^k$, in degree zero, and $t^k[t]$ in degree $1$. As for Hochschild cohomology, we have classes $\del^k$ and $\sum_i i s_i \del^{i+k}$, in degrees zero and $-1$, respectively. All these classes are nonzero in (co)homology for any $k\in \Z$.

\subsubsection{The maps $G$ and $\widetilde\lambda_H$}
We can pick representatives for the bimodule $A^!$ exactly as in the case of \cref{sec:oddSpheres} (taking $n=0$) and the vertex $\co$ to be given by the same formula. On the other hand, $\eta$ has to be modified to account for the negative powers of $t$:
\[
    \eta([t^m],t^n;R^{k,\ell}) = -\chi_{1 \le m} \sum_{i=1}^{m} t^{\ell +i-1} \otimes t^{k+n+m-i} + \chi_{m\le -1} \sum_{i=m+1}^{0} t^{\ell+i-1} \otimes t^{k+n+m-i}
\]
We calculate that $E=0$, as expected. Calculating the map $G$ we now have
\begin{align*}
    G(\del_k) = \chi_{1 \le k} \sum_{i=1}^k &(t^{i-1} \otimes t^{k-i}[t] - t^{i-1}[t] \otimes t^{k-i} )- \chi_{k \le -1} \sum_{i=k+1}^0 (t^{i-1} \otimes t^{k-i}[t] - t^{i+1}[t] \otimes t^{k-i} )\\
    G(\sum_i i s_i \del^{i+k}) &= -\chi_{0 \le k} \sum_{i=1}^{k+1} t^{i-1} \otimes t^{k+1-i} + \chi_{k \le -2} \sum_{i=k+2}^0 t^{i-1} \otimes t^{k+1-i}
\end{align*}

Though $E =0$ already at the chain-level, we are free to pick a trivialization $H$ with $dH = 0$, that is, a closed element of degree $-1$ in $C_*(A,A)\otimes C_*(A,A)$. We write $H$ as
\[ H = \sum_{p,q \in \Z \oplus \Z} a_{p,q}t^p \otimes t^q[t] + b_{p,q} t^p[t] \otimes t^q \]
where only finitely many of the $a_{p,q},b_{p,q}$ are nonzero. 

\subsubsection{The CY structure and the chain-level coproduct}
There are multiple nonequivalent CY structures on the algebra $A$. Here we will use one that is of special importance for derived symplectic geometry; see \cite{bozec2023calabi,bozec2023calabimultiplicative} for a discussion of its relation to quasi-hamiltonian quotients. 

We pick the Hochschild chain $\omega=t^{-1}[t] \in C_1(A,A)$ which is non-degenerate; this induces a map $HH^*(A,A)\to HH_*(A,A)[-1]$ given on our representatives by
\[ \sum_i is_i\del^{i+k} \mapsto t^k, \quad \del^{k} \mapsto t^{k-1}[t] \]
for every $k \in \Z$. It turns out that the chain-level inverse $g_\alpha$, with $\alpha \in C^1_{(2)}(A)$ skewsymmetric (as required for a pre-CY structure) only exists if $2$ is a unit in $\kk$. For now, just using the homology-level inverse of the map above (which exists for any $\kk$) we can calculate the coproduct on homology, for each choice of $H$ as above:
\begin{align*}
    \lambda_H(t^k[t]) &= \chi_{0 \le k} \sum_{i=1}^{k+1} (t^{i-1} \otimes t^{k+1-i}[t] - t^{i-1}[t] \otimes t^{k+1-i} ) \\
	&- \chi_{k \le -2} \sum_{i=k+2}^0 (t^{i-1} \otimes t^{k+1-i}[t] - t^{i+1}[t] \otimes t^{k+1-i} )\\
    &+\sum_{p,q \in \Z \oplus \Z} a_{p,q}(t^p[t] \otimes t^{q+k+1} - t^p \otimes t^{q+k+1}[t]) - b_{p,q}(t^{p+k+1}[t]\otimes t^q - t^{p+k+1}\otimes t^q[t])  \\
    \lambda_H(t^k) &= \chi_{0 \le k} \sum_{i=1}^{k+1} t^{i-1} \otimes t^{k+1-i} - \chi_{k \le -2} \sum_{i=k+2}^0 t^{i-1} \otimes t^{k+1-i}\\
    &+\sum_{p,q \in \Z \oplus \Z} a_{p,q}t^p \otimes t^{q+k+1} - b_{p,q} t^{p+k+1}\otimes t^q\\
\end{align*}

Every monomial in the expressions above represents a class in $HH_*(A,A)$, so we can directly read the coproduct induced in homology. We reiterate that for any choice of $H$, the coproduct above, together with the product $\pi$, satisfies Sullivan's relation. As expected, if $H$ is skew-symmetric, meaning $a_{p,q} = -b_{q,p}$, then the coproduct is skew-commutative. Even if we restrict attention to skew-symmetric choices for $H$, there is still a very large family of nonequivalent coproducts, generically not coassociative, as expected since $n=1 <3$. We single out two choices of $a,b$ that lead to particularly nice coproducts.

If $H=0$, then the product obtained is moreover coassociative, and the pair $(\pi,\lambda_0)$ endows $HH_*(A,A)[-1]$ with the structure of an infinitesimal bialgebra. Comparing with the notation used in \cite[Subsection 8.3]{CieHinOan2}, we can identify $t^{k-1} \leftrightarrow AU^k, t^{k-1}[t] \leftrightarrow U^k$ (note the shift with respect to the case of higher-dimensional odd spheres) which gives the coproduct
\begin{align*}
    \lambda_0(U^k) &= \chi_{1 \le k} \sum_{i=1}^{k} (AU^i \otimes U^{k-i+1} \\
	&- U^{i} \otimes AU^{k-i+1})- \chi_{k \le -1} \sum_{i=k+1}^0 (AU^i \otimes U^{k-i+1} - U^i \otimes AU^{k-i+1} )\\
    \lambda_0(AU^k) &= \chi_{1 \le k} \sum_{i=1}^{k} AU^i \otimes AU^{k-i+1}- \chi_{k \le -1} \sum_{i=k+1}^0 AU^i \otimes AU^{k-i+1}\\
\end{align*}
We note the partial similarity of this coproduct with the coproducts on $H_*(LS^1)$ found using Floer-theoretic approaches in \cite{CieHinOan2}. More specifically, the coproduct $\lambda_0$ above is a version of the coproduct denoted $\lambda_{v_-,v_-}$ of Remark 8.8 in \cite{CieHinOan2} with a shift in $k$, that satisfies cocommutativity, coassociativity and Sullivan's relation.

Another particularly nice choice for $H$ is $a_{-1,0} = -b_{0,-1}=1$ with all other $a_{p,q},b_{p,q}$ zero; that is, $H = t^{-1}\otimes 1[t] - 1[t]\otimes t^{-1}$. Again using the notation $U^k,AU^k$, we get the coproduct 
\begin{align*}
    \lambda_H(U^k) &= \chi_{1 \le k} \sum_{i=0}^{k+1} (AU^i \otimes U^{k-i+1} - U^{i} \otimes AU^{k-i+1}) \\
	&- \chi_{k \le -1} \sum_{i=k+2}^{-1} (AU^i \otimes U^{k-i+1} - U^i \otimes AU^{k-i+1} )\\
    \lambda_H(AU^k) &= \chi_{1 \le k} \sum_{i=0}^{k+1} AU^i \otimes AU^{k-i+1}- \chi_{k \le -1} \sum_{i=k+2}^{-1} AU^i \otimes AU^{k-i+1}\\
\end{align*}
which is cocommutative, satisfies Sullivan's relation with respect to $\pi$, but is not coassociative. This coproduct again is a shifted version of the coproduct denoted $\lambda_{v_+,v_+}$ of \emph{op.cit.} 

\begin{remark}
    We do not really understand the geometric significance of the coproducts we calculated above; that would require an investigation about how the choice of $H$ is related to the choice of vector field that leads to all these different coproducts.
\end{remark}

\subsubsection{The pre-CY structure and the cone bimodule}
Let us work over a field $\kk_{\neq 2}$ of characteristic $\neq 2$. We give an explicit inverse $\alpha$ to $\omega$, defined by the formulas
\begin{align*}
    \alpha(t^k, \varnothing) &= \chi_{1\le k}\ \left(\frac{1}{2}(1\otimes t^k + t^k \otimes 1) -\sum_{1 \le i \le k-1}t^i \otimes t^{k-i}\right) \\
    &+ \chi_{k\le -1}\ \left(\frac{1}{2}(1\otimes t^k + t^k \otimes 1) + \sum_{k+1 \le i \le -1}t^i \otimes t^{k-i}\right)
\end{align*}
and $\alpha(\varnothing,t^k) = -\alpha(t^k, \varnothing)$, using the same notation from the proof of \cref{prop:inverseOddSphere}. One can check that this gives a closed element of $C^*_{(2)}(A)$, skewsymmetric and inverse to $\omega$. In order to extend this to a full pre-CY structure, we also need an element with three outputs and zero inputs; by an explicit calculation we verify:
\begin{proposition}
    Taking $m=\mu + \alpha + \tau$, where $\tau = \frac{1}{4}(1\otimes 1 \otimes 1)$, gives a pre-CY structure of dimension 1 on $A$.
\end{proposition}

\subsubsection{Product on the cone}
We again denote $s_k$ for the element of $A^\vee$ dual to $t^k$, for $k \in \Z$; all these elements are of degree zero. We calculate the map $f_\alpha$:
\begin{align*}
    f_\alpha(s_k \otimes t^\ell) &= - \chi_{1\le \ell}\left(\frac{1}{2} \delta_k + \frac{1}{2} \delta_{k-\ell} + \chi_{1 \le k \le \ell-1}\right) t^{\ell-k} + \chi_{\ell \le -1}\left(\frac{1}{2} \delta_k t^\ell + \frac{1}{2} \delta_{k-\ell} + \chi_{\ell+1 \le k \le -1}  \right)t^{\ell-k} \\
    f_\alpha(t^\ell \otimes s_k) &= + \chi_{1\le \ell}\left(\frac{1}{2} \delta_k + \frac{1}{2} \delta_{k-\ell} + \chi_{1 \le k \le \ell-1}  \right) t^{\ell-k}- \chi_{\ell \le -1}\left(\frac{1}{2} \delta_k t^\ell+ \frac{1}{2} \delta_{k-\ell} +\chi_{\ell+1 \le k \le -1}  \right)t^{\ell-k}
\end{align*}
and the map induced on Hochschild chains is
\begin{align*}
    F_\alpha(s_k[t^{k_1}|\dots|t^{k_p}]) = &-\chi_{1 \le k_1} \left(\frac{1}{2} \delta_k + \frac{1}{2}\delta_{k-k_1} + \chi_{1\le k \le k_1-1} \right)t^{k_1-k}[t^{k_2}|\dots|t^{k_p}]\\
    &+\chi_{k_1 \le -1} \left(\frac{1}{2} \delta_k + \frac{1}{2}\delta_{k-k_1} + \chi_{k_1+1\le k \le -1} \right) t^{k_1-k}[t^{k_2}|\dots|t^{k_p}]\\
    &+(-1)^{p-1}\chi_{1 \le k_p} \left(\frac{1}{2} \delta_k + \frac{1}{2}\delta_{k-k_p} + \chi_{1\le k \le k_p-1} \right)t^{k_p-k}[t^{k_1}|\dots|t^{k_{p-1}}]\\
    &-(-1)^{p-1}\chi_{k_p \le -1} \left(\frac{1}{2} \delta_k + \frac{1}{2}\delta_{k-k_p} + \chi_{k_p+1\le k \le -1} \right) t^{k_p-k}[t^{k_1}|\dots|t^{k_{p-1}}]
\end{align*}

We know that the cohomology of the complex $C_*(A,A^\vee[1])$ is supported in degrees 0 and 1, with a partial basis \footnote{note that as a module each one of $HH_{0,1}(A,A^\vee[1])$ is the algebraic dual module $k[\Z]^\vee$} given by the classes $s_k$ and $s_k[t]$, and the cohomology of $C_*(A,A)$ is concentrated in homological degrees $1$ and 0; therefore the only nontrivial component of $F_\alpha$ on cohomology can be in degree zero, which we calculate to be zero. Therefore, again we know that there must be a homotopy between $F_\alpha$ and the zero map.

We follow the same steps as in the case of higher-dimensional odd spheres, and obtain an expression for the product $\pi$ on the cone $C_*(A,A^\vee) \oplus C_*(A,A)$. The expressions get rather complicated, so let us just write the terms we will need. Let us denote $p: C_*(A,A^\vee) \oplus C_*(A,A^\vee) \to C_*(A,A^\vee)$ for the connecting map from the cone.
\begin{align*}
    \pi(t^k[t],t^\ell) &= -t^{k+\ell+1}\\
    \pi(t^k[t],t^\ell[t]) &= t^{k+\ell+1}[t]\\
    p\pi(s_k[t],t^\ell) &= -s_{k-\ell-1}\\
    p\pi(s_k,t^\ell[t]) &= s_{k-\ell-1}\\
    p\pi(s_k[t],t^\ell[t]) &= s_{k-\ell-1}[t]
\end{align*}
The calculation with two inputs in $C_*(A,A^\vee)$ is rather involved, and many diagrams contribute. We present the result in table form:
\begin{align*}
    p\pi(s_k[t],s_\ell) &= \begin{array}{c||c|c|c|}
    1 \le \ell & 0 & -1/2 & -1 \\
    \hline
    \ell =0 & 1/2 & 0 & -1/2 \\
    \hline
    \ell \le -1 & 1 & 1/2 & 0 \\
    \hline \hline
     & k\le-1 & k=0 & 1 \le k 
    \end{array} \times s_{k+\ell-1} \\ \\
    p\pi(s_k[t],s_\ell[t]) &= \begin{array}{c||c|c|c|}
    1 \le \ell & 0 & 1/2 & 1 \\
    \hline
    \ell =0 & -1/2 & 0 & 1/2 \\
    \hline
    \ell \le -1 & -1 & -1/2 & 0 \\
    \hline \hline
     & k\le-1 & k=0 & 1 \le k 
    \end{array} \times s_{k+\ell-1}
\end{align*}

\subsubsection{Difference between homotopies and trivializations}\label{sec:differenceHomotopiesTriv}
We would like to use this example to illustrate the following fact: when $HH_*(A,A)$ is infinite-rank in some degree (as it is the case here), there are more homotopies $h$ of the morphism $F_\alpha$ than trivializations $H$ of the canonical element $E$, even if we know that the maps $\ ^\sharp E \circ g^{A^\vee}_\alpha$ and $F_\alpha$ are homotopic. In other words, not every homotopy $h$ of $F_\alpha$ is equivalent to one that comes as $ ^\sharp H$, that is, from pairing with some trivialization $H$ of $E$.

For example, the simplest homotopy that we can write down for the map $F_\alpha$ above is the map $h: C_*(A,A^\vee) \to C_*(A,A)$ given by
\[ h(s_k[t^{k_1}|\dots|t^{k_p}]) = \left(\frac{1}{2}\delta_k + \chi_{1\le k}\right)t^{-k}[t^{k_1}|\dots|t^{k_p}] \]
If we now compute the resulting product $\pi_h$ on $C_*(A,A^\vee)$, for instance on $(s_k[t],s_\ell)$, the extra terms $p\pi(h(s_k[t]),s_\ell)$ and $p\pi(s_k[t],h(s_\ell))$ add $1/2$ and $1$ to some rows and columns of the table above, and we get that
\[ \pi_h(s_k[t],s_\ell) = s_{k+\ell-1} \]
for all values of $k,\ell$. We see that this product $\pi_h$ on $HH_*(A,A^\vee)$ is not the dual to any coproduct on $HH_*(A,A)$: for any integer $j$, there are infinitely many $k,\ell$ for which $\pi_h(s_k[t],s_\ell) = s_j$, namely, all the pairs with $k+\ell = j+1$. This will be the generic behavior for a homotopy of the map $f_\alpha$.

\vfill
\pagebreak

\appendix
\addtocontents{toc}{\protect\setcounter{tocdepth}{1}}       %
\section{Explicit diagrams}
We left some of the more complicated combinations of diagrams out of the main body of the text for clarity.

\subsection{Conventions}
In order to simplify the notation, we will make some conventions. For some of the proofs, it is unnecessary to specify signs. For example, to prove a quasi-isomorphism between complexes given by the cones of two maps, it is enough to find a homotopy between those maps up to an overall sign. For the proofs where the signs are important, following the discussion in \cref{sec:orientedRibbon}, we can give a sign by
\begin{enumerate}
    \item giving a total ordering of the edges,
    \item giving a total ordering of the vertices where we are going to input the copies of $\alpha \in C^*_{(2)}(A)$.
    \item for each vertex with two or more outgoing arrows, specifying one of them to be the first arrow.
\end{enumerate}
We specify $1.$ by labeling the edges from $1$ to $N$ where $N$ is the number of edges, and always using the ordering $(e_N\ \dots\ e_1)$. We specify $2.$ by numbering the circles where each copy of $\alpha$ will be placed, and $3$ by a white arrowhead.

The statements in this paper are phrased for $A_\infty$-structures, but in all examples of our interest in this paper, we simply have dg structures (that is, an $A_\infty$-structure with $\mu^{\ge 3}=0$). For the sake of simplicity, when giving a combination of diagrams, we will omit terms that are zero for dg-structures. For example, we will omit homotopies such as
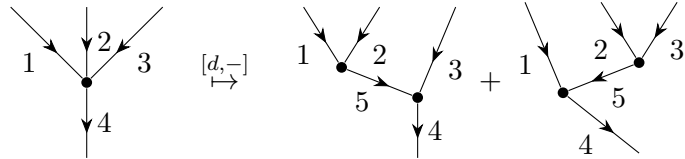
\[ \tikzfig{
\node (mid) [bullet] at (0,0) {};
\draw [->-] (-1,1) to node [auto,swap] {$1$} (mid);
\draw [->-] (0,1) to node [auto] {$2$} (mid);
\draw [->-] (1,1) to node [auto] {$3$} (mid);
\draw [->-] (mid) to node [auto] {$4$} (0,-1);
} \quad \overset{[d,-]}{\mapsto} \quad \tikzfig{
\node (left) [bullet] at (-0.5,0.2) {};
\node (right) [bullet] at (0.5,-0.2) {};
\draw [->-] (-1,1) to node [auto,swap] {$1$} (left);
\draw [->-] (0,1) to node [auto] {$2$} (left);
\draw [->-] (left) to node [auto,swap] {$5$} (right); 
\draw [->-] (1,1) to node [auto] {$3$} (right);
\draw [->-] (right) to node [auto] {$4$} (0.5,-1);
} + \tikzfig{
\node (left) [bullet] at (-0.5,-0.2) {};
\node (right) [bullet] at (0.5,+0.2) {};
\draw [->-] (-1,1) to node [auto,swap] {$1$} (left);
\draw [->-] (0,1) to node [auto,swap] {$2$} (right);
\draw [->-] (right) to node [auto] {$5$} (left); 
\draw [->-] (1,1) to node [auto] {$3$} (right);
\draw [->-] (left) to node [auto,swap] {$4$} (0.5,-1);
}\]
and instead just identify two diagrams that differ by something like the two terms on the right hand side with a minus sign. Similarly, we will omit diagrams where the evaluation vertex $\ev$ receives more than one factor of $A$.

\subsection{Diagrams for infinitesimal bialgebra relation}\label{app:ibl}
We start by giving an explicit expression for the homotopy realizing the compatibility between $G$ and the cup product, which appeared in the proof of \cref{lem:compatibilityWithCup}. Each of the following diagrams lives on the elbow (as in \cref{def:ChernCharacter}) but for ease of visualization we cut along the bottom of the elbow from one boundary component to the other. 
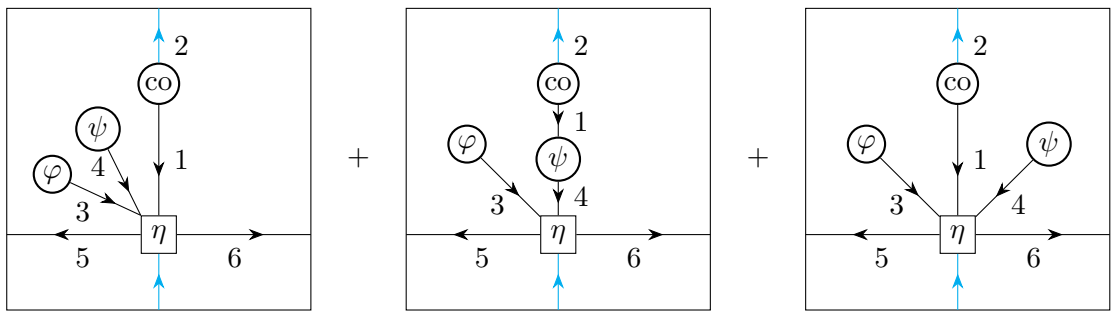
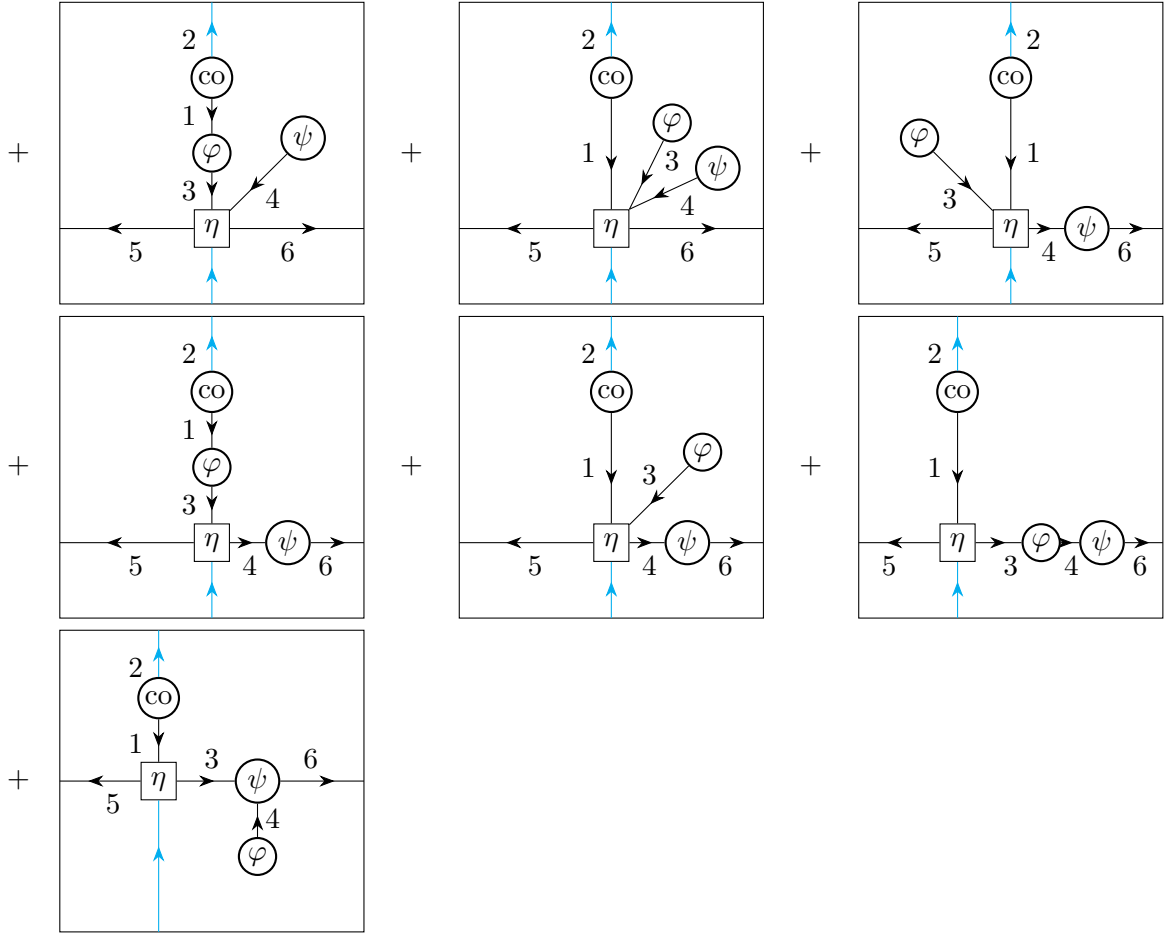
\begin{align*}
    + \quad &\tikzfig{
        \draw (-2,-2) rectangle (2,2);
        \node [vertex] (co) at (0,1) {$\co$};
        \node [rectangle,draw] (eta) at (0,-1) {$\eta$};
        \node [vertex] (phi) at (-1.4,-0.2) {$\varphi$};
        \node [vertex] (psi) at (-0.8,0.4) {$\psi$};
        \draw [->-,cyan] (co) to (0,2);
        \draw [->-,cyan] (0,-2) to (eta);
        \draw [->-] (co) to (eta);
        \draw [->-] (eta) to (2,-1);
        \draw [->-] (eta) to (-2,-1);
        \draw [->-] (phi) to (eta.north west);
        \draw [->-] (psi) to (eta.north west);
        \node at (0.3,0) {$1$};
        \node at (0.3,1.5) {$2$};
        \node at (-1,-0.7) {$3$};
        \node at (-0.8,-0.1) {$4$};
        \node at (-1,-1.3) {$5$};
        \node at (1,-1.3) {$6$};
    } \quad + \quad \tikzfig{
        \draw (-2,-2) rectangle (2,2);
        \node [vertex] (co) at (0,1) {$\co$};
        \node [rectangle,draw] (eta) at (0,-1) {$\eta$};
        \node [vertex] (phi) at (-1.2,0.2) {$\varphi$};
        \node [vertex] (psi) at (0,0) {$\psi$};
        \draw [->-,cyan] (co) to (0,2);
        \draw [->-,cyan] (0,-2) to (eta);
        \draw [->-] (co) to (psi);
        \draw [->-] (eta) to (2,-1);
        \draw [->-] (eta) to (-2,-1);
        \draw [->-] (phi) to (eta);
        \draw [->-] (psi) to (eta);
        \node at (0.3,0.5) {$1$};
        \node at (0.3,1.5) {$2$};
        \node at (-0.8,-0.6) {$3$};
        \node at (0.3,-0.5) {$4$};
        \node at (-1,-1.3) {$5$};
        \node at (1,-1.3) {$6$};
    } \quad + \quad \tikzfig{
        \draw (-2,-2) rectangle (2,2);
        \node [vertex] (co) at (0,1) {$\co$};
        \node [rectangle,draw] (eta) at (0,-1) {$\eta$};
        \node [vertex] (phi) at (-1.2,0.2) {$\varphi$};
        \node [vertex] (psi) at (1.2,0.2) {$\psi$};
        \draw [->-,cyan] (co) to (0,2);
        \draw [->-,cyan] (0,-2) to (eta);
        \draw [->-] (co) to (eta);
        \draw [->-] (eta) to (2,-1);
        \draw [->-] (eta) to (-2,-1);
        \draw [->-] (phi) to (eta);
        \draw [->-] (psi) to (eta);
        \node at (0.3,0) {$1$};
        \node at (0.3,1.5) {$2$};
        \node at (-0.8,-0.6) {$3$};
        \node at (0.8,-0.6) {$4$};
        \node at (-1,-1.3) {$5$};
        \node at (1,-1.3) {$6$};
    }
\end{align*}
\begin{align*}
    + \quad &\tikzfig{
        \draw (-2,-2) rectangle (2,2);
        \node [vertex] (co) at (0,1) {$\co$};
        \node [rectangle,draw] (eta) at (0,-1) {$\eta$};
        \node [vertex] (psi) at (1.2,0.2) {$\psi$};
        \node [vertex] (phi) at (0,0) {$\varphi$};
        \draw [->-,cyan] (co) to (0,2);
        \draw [->-,cyan] (0,-2) to (eta);
        \draw [->-] (co) to (phi);
        \draw [->-] (eta) to (2,-1);
        \draw [->-] (eta) to (-2,-1);
        \draw [->-] (phi) to (eta);
        \draw [->-] (psi) to (eta);
        \node at (-0.3,0.5) {$1$};
        \node at (-0.3,1.5) {$2$};
        \node at (0.8,-0.6) {$4$};
        \node at (-0.3,-0.5) {$3$};
        \node at (-1,-1.3) {$5$};
        \node at (1,-1.3) {$6$};
    } \quad + \quad \tikzfig{
        \draw (-2,-2) rectangle (2,2);
        \node [vertex] (co) at (0,1) {$\co$};
        \node [rectangle,draw] (eta) at (0,-1) {$\eta$};
        \node [vertex] (psi) at (1.4,-0.2) {$\psi$};
        \node [vertex] (phi) at (0.8,0.4) {$\varphi$};
        \draw [->-,cyan] (co) to (0,2);
        \draw [->-,cyan] (0,-2) to (eta);
        \draw [->-] (co) to (eta);
        \draw [->-] (eta) to (2,-1);
        \draw [->-] (eta) to (-2,-1);
        \draw [->-] (phi) to (eta.north east);
        \draw [->-] (psi) to (eta.north east);
        \node at (-0.3,0) {$1$};
        \node at (-0.3,1.5) {$2$};
        \node at (1,-0.7) {$4$};
        \node at (0.8,-0.1) {$3$};
        \node at (-1,-1.3) {$5$};
        \node at (1,-1.3) {$6$};
    } \quad + \quad \tikzfig{
        \draw (-2,-2) rectangle (2,2);
        \node [vertex] (co) at (0,1) {$\co$};
        \node [rectangle,draw] (eta) at (0,-1) {$\eta$};
        \node [vertex] (phi) at (-1.2,0.2) {$\varphi$};
        \node [vertex] (psi) at (1,-1) {$\psi$};
        \draw [->-,cyan] (co) to (0,2);
        \draw [->-,cyan] (0,-2) to (eta);
        \draw [->-] (co) to (eta);
        \draw [->-] (psi) to (2,-1);
        \draw [->-] (eta) to (-2,-1);
        \draw [->-] (phi) to (eta);
        \draw [->-] (eta) to (psi);
        \node at (0.3,0) {$1$};
        \node at (0.3,1.5) {$2$};
        \node at (-0.8,-0.6) {$3$};
        \node at (0.5,-1.3) {$4$};
        \node at (-1,-1.3) {$5$};
        \node at (1.5,-1.3) {$6$};
    } \\
    + \quad &\tikzfig{
        \draw (-2,-2) rectangle (2,2);
        \node [vertex] (co) at (0,1) {$\co$};
        \node [rectangle,draw] (eta) at (0,-1) {$\eta$};
        \node [vertex] (phi) at (0,0) {$\varphi$};
        \node [vertex] (psi) at (1,-1) {$\psi$};
        \draw [->-,cyan] (co) to (0,2);
        \draw [->-,cyan] (0,-2) to (eta);
        \draw [->-] (co) to (phi);
        \draw [->-] (psi) to (2,-1);
        \draw [->-] (eta) to (-2,-1);
        \draw [->-] (phi) to (eta);
        \draw [->-] (eta) to (psi);
        \node at (-0.3,0.5) {$1$};
        \node at (-0.3,1.5) {$2$};
        \node at (-0.3,-0.5) {$3$};
        \node at (0.5,-1.3) {$4$};
        \node at (-1,-1.3) {$5$};
        \node at (1.5,-1.3) {$6$};
    } \quad + \quad \tikzfig{
        \draw (-2,-2) rectangle (2,2);
        \node [vertex] (co) at (0,1) {$\co$};
        \node [rectangle,draw] (eta) at (0,-1) {$\eta$};
        \node [vertex] (phi) at (1.2,0.2) {$\varphi$};
        \node [vertex] (psi) at (1,-1) {$\psi$};
        \draw [->-,cyan] (co) to (0,2);
        \draw [->-,cyan] (0,-2) to (eta);
        \draw [->-] (co) to (eta);
        \draw [->-] (psi) to (2,-1);
        \draw [->-] (eta) to (-2,-1);
        \draw [->-] (phi) to (eta);
        \draw [->-] (eta) to (psi);
        \node at (-0.3,0) {$1$};
        \node at (-0.3,1.5) {$2$};
        \node at (0.5,-0.1) {$3$};
        \node at (0.5,-1.3) {$4$};
        \node at (-1,-1.3) {$5$};
        \node at (1.5,-1.3) {$6$};
    } \quad + \quad \tikzfig{
        \draw (-2,-2) rectangle (2,2);
        \node [vertex] (co) at (-0.7,1) {$\co$};
        \node [rectangle,draw] (eta) at (-0.7,-1) {$\eta$};
        \node [vertex] (phi) at (0.4,-1) {$\varphi$};
        \node [vertex] (psi) at (1.2,-1) {$\psi$};
        \draw [->-,cyan] (co) to (-0.7,2);
        \draw [->-,cyan] (-0.7,-2) to (eta);
        \draw [->-] (co) to (eta);
        \draw [->-] (psi) to (2,-1);
        \draw [->-] (eta) to (-2,-1);
        \draw [->-] (eta) to (phi);
        \draw [->-] (phi) to (psi);
        \node at (-1,0) {$1$};
        \node at (-1,1.5) {$2$};
        \node at (0,-1.3) {$3$};
        \node at (0.8,-1.3) {$4$};
        \node at (-1.6,-1.3) {$5$};
        \node at (1.7,-1.3) {$6$};
    } \\
    + \quad &\tikzfig{
        \draw (-2,-2) rectangle (2,2);
        \node [vertex] (co) at (-0.7,1.1) {$\co$};
        \node [rectangle,draw] (eta) at (-0.7,0) {$\eta$};
        \node [vertex] (phi) at (0.6,-1) {$\varphi$};
        \node [vertex] (psi) at (0.6,0) {$\psi$};
        \draw [->-,cyan] (co) to (-0.7,2);
        \draw [->-,cyan] (-0.7,-2) to (eta);
        \draw [->-] (co) to (eta);
        \draw [->-] (psi) to (2,0);
        \draw [->-] (eta) to (-2,0);
        \draw [->-] (eta) to (psi);
        \draw [->-] (phi) to (psi);
        \node at (-1,0.5) {$1$};
        \node at (-1,1.5) {$2$};
        \node at (0,0.3) {$3$};
        \node at (0.8,-0.5) {$4$};
        \node at (-1.3,-0.3) {$5$};
        \node at (1.3,0.3) {$6$};
    }
\end{align*}

\subsection{Diagrams related to $\Ahatinfty$}
Here we give two homotopies that were omitted from \cref{sec:catFormal}. For both of these cases, we can work without specifying signs, because the relevant statements, quasi-isomorphisms of cones, is independent of an overall sign, and also because the differential of each diagram will have exactly two terms, so one can always arrange the signs so that the intermediate terms cancel.

\subsubsection{Homotopy relating two models for $\Ahatinfty$}\label{app:AhatinftyBimodules}
The following homotopy gives the proof of \cref{prop:isoOfBimodules}:
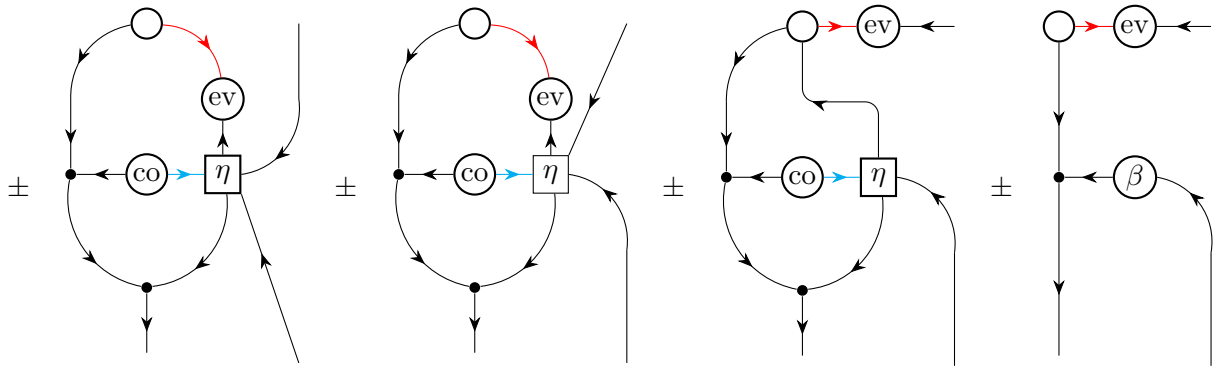
\begin{equation*}
    \pm \quad \tikzfig{
        \node [vertex,inner sep=4pt] (top) at (-1,2) {};
        \node [vertex] (ev) at (0,1) {$\ev$};
        \node [rectangle,draw,thick] (eta) at (0,0) {$\eta$};
        \node [vertex] (coev) at (-1,0) {$\co$};
        \node [bullet] (left) at (-2,0) {};
        \node [bullet] (down) at (-1,-1.5) {};
        \node (down2) at (-1,-2.5) {};
        \draw [->-,red,bend left=40] (top) to (ev);
        \draw [->-,bend right=40] (top) to (-2,1);
        \draw [->-] (-2,1) to (left);
        \draw [->-,bend right=45] (left) to (down);
        \draw [->-,bend left=45] (eta) to (down);
        \draw [->-] (down) to (down2);
        \draw (1,2) to (1,1);
        \draw [->-,bend left=45] (1,1) to (eta);
        \draw [->-,cyan] (coev) to (eta);
        \draw [->-] (coev) to (left);
        \draw [->-] (eta) to (ev);
        \draw [->-] (1,-2.5) to (eta.south east);
} \quad \pm \quad \tikzfig{
        \node [vertex,inner sep=4pt] (top) at (-1,2) {};
        \node [vertex] (ev) at (0,1) {$\ev$};
        \node [rectangle,draw] (eta) at (0,0) {$\eta$};
        \node [vertex] (coev) at (-1,0) {$\co$};
        \node [bullet] (left) at (-2,0) {};
        \node [bullet] (down) at (-1,-1.5) {};
        \node (down2) at (-1,-2.5) {};
        \draw [->-,red,bend left=40] (top) to (ev);
        \draw [->-,bend right=40] (top) to (-2,1);
        \draw [->-] (-2,1) to (left);
        \draw [->-,bend right=45] (left) to (down);
        \draw [->-,bend left=45] (eta) to (down);
        \draw [->-] (down) to (down2);
        \draw [->-] (1,2) to (eta.north east);
        \draw [->-,cyan] (coev) to (eta);
        \draw [->-] (coev) to (left);
        \draw [->-] (eta) to (ev);
        \draw (1,-2.5) to (1,-1);
        \draw [->-,bend right=45] (1,-1) to (eta);
} \quad \pm \quad \tikzfig{
        \node [vertex,inner sep=4pt] (top) at (-1,2) {};
        \node [vertex] (ev) at (0,2) {$\ev$};
        \node [rectangle,draw,thick] (eta) at (0,0) {$\eta$};
        \node [vertex] (coev) at (-1,0) {$\co$};
        \node [bullet] (left) at (-2,0) {};
        \node [bullet] (down) at (-1,-1.5) {};
        \node (down2) at (-1,-2.5) {};
        \draw [->-,red] (top) to (ev);
        \draw [->-,bend right=40] (top) to (-2,1);
        \draw [->-] (-2,1) to (left);
        \draw [->-,bend right=45] (left) to (down);
        \draw [->-,bend left=45] (eta) to (down);
        \draw [->-] (down) to (down2);
        \draw [->-] (1,2) to (ev);
        \draw [->-,cyan] (coev) to (eta);
        \draw [->-] (coev) to (left);
        \draw [->-,rounded corners=5] (eta) |- ++(-0.5,1) -| (top);
        \draw (1,-2.5) to (1,-1);
        \draw [->-,bend right=45] (1,-1) to (eta);
} \quad \pm  \quad \tikzfig{
        \node [vertex,inner sep=4pt] (top) at (-1,2) {};
        \node [vertex] (ev) at (0,2) {$\ev$};
        \node [vertex] (beta) at (0,0) {$\beta$};
        \node [bullet] (mid) at (-1,0) {};
        \node (down2) at (-1,-2.5) {};
        \draw [->-,red] (top) to (ev);
        \draw [->-] (1,2) to (ev);
        \draw [->-] (beta) to (mid);
        \draw [->-] (top) to (mid);
        \draw [->-] (mid) to (down2);
        \draw (1,-2.5) to (1,-1);
        \draw [->-,bend right=45] (1,-1) to (beta);
}
\end{equation*}

\subsubsection{Homotopy relating $E$ to map whose cone is $C_*(A,\Ahatinfty)$}\label{app:relationToE}
The following homotopy gives the proof of \cref{thm:EulerCharacter}, relating our chain-level Chern character and the map $C^*(A,A^\vee) \to C_*(A,A)$ coming from applying Hochschild chains to the map of bimodules defining $\Ahatinfty$:
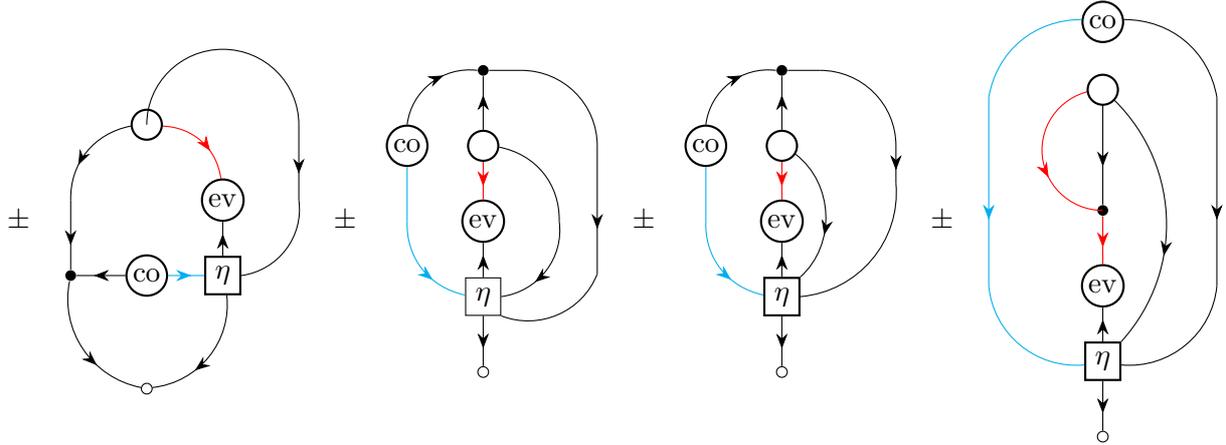
\begin{equation}
\pm \quad \tikzfig{
\node [vertex,inner sep=4pt] (top) at (-1,2) {};
\node [vertex] (ev) at (0,1) {$\ev$};
\node [rectangle,draw,thick] (eta) at (0,0) {$\eta$};
\node [vertex] (coev) at (-1,0) {$\co$};
\node [bullet] (left) at (-2,0) {};
\node [circ] (down) at (-1,-1.5) {};
\draw [->-,red,bend left=40] (top) to (ev);
\draw [->-,bend right=40] (top) to (-2,1);
\draw [->-] (-2,1) to (left);
\draw [->-,bend right=45] (left) to (down);
\draw [->-,bend left=45] (eta) to (down);
\draw [->-] (1,2) to (1,1);
\draw [bend left=45] (1,1) to (eta);
\draw [->-,cyan] (coev) to (eta);
\draw [->-] (coev) to (left);
\draw [->-] (eta) to (ev);
\draw (top) arc (180:0:1);
\node [vertex,inner sep=4pt,fill=white] at (-1,2) {};
} \quad 
\pm \quad \tikzfig{
\node [vertex,inner sep=4pt](in) at (0,2) {};
\node [vertex] (ev) at (0,1) {$\ev$};
\node [rectangle,draw] (eta) at (0,0) {$\eta$};
\node [vertex] (coev) at (-1,2) {$\co$};
\node [bullet] (top) at (0,3) {};
\node [circ] (down) at (0,-1) {};
\draw [->-,red] (in) to (ev);
\draw [->-] (in) to (top);
\draw (top) -- (0.5,3);
\draw [bend left=45] (0.5,3) to (1.5,2);
\draw [->-] (1.5,2) to (1.5,0.3);
\draw [bend left=45] (1.5,0.3) to (eta.south east);
\draw [->-] (eta) to (down);
\draw [->-,bend left=45] (1,1) to (eta);
\draw [cyan] (coev) -- (-1,1);
\draw [->-,cyan,bend right=40] (-1,1) to (eta);
\draw [->-] (eta) to (ev);
\draw [->-,bend left=45] (coev) to (top);
\draw [bend left=40] (in) to (1,1);
} \quad \pm \quad \tikzfig{
\node [vertex,inner sep=4pt](in) at (0,2) {};
\node [vertex] (ev) at (0,1) {$\ev$};
\node [rectangle,draw,thick] (eta) at (0,0) {$\eta$};
\node [vertex] (coev) at (-1,2) {$\co$};
\node [bullet] (top) at (0,3) {};
\node [circ] (down) at (0,-1) {};
\draw [->-,red] (in) to (ev);
\draw [->-] (in) to (top);
\draw (top) -- (0.5,3);
\draw [bend left=45] (0.5,3) to (1.5,2);
\draw [->-] (1.5,2) to (1.5,1.5);
\draw [bend left=45] (1.5,1.5) to (eta.east);
\draw [->-] (eta) to (down);
\draw [cyan] (coev) -- (-1,1);
\draw [->-,cyan,bend right=40] (-1,1) to (eta);
\draw [->-] (eta) to (ev);
\draw [->-,bend left=45] (coev) to (top);
\draw [->-,bend left=50] (in) to (eta.north east);
} \quad \pm \quad \tikzfig{
\node [bullet] (mu1) at (0,0) {};
\node [vertex] (gamma) at (0,1.6) {};
\node [vertex,fill=white] (co) at (0,2.5) {$\co$};
\node [vertex] (ev) at (0,-1) {$\ev$};
\node [rectangle,draw,thick] (eta) at (0,-2) {$\eta$};
\node [circ] (out) at (0,-3) {};
\draw [cyan, bend right=40] (co) to (-1.5,1.5);
\draw [->-,cyan] (-1.5,1.5) to (-1.5,-1);
\draw [cyan,bend right=45] (-1.5,-1) to (eta);
\draw [bend left=40] (co) to (1.5,1.5);
\draw [->-] (1.5,1.5) to (1.5,-1);
\draw [bend left=45] (1.5,-1) to (eta);
\draw [->-,red] (gamma) arc (90:270:0.8);
\draw [->-] (gamma) to (mu1);
\draw [->-,bend left=45] (gamma) to (eta.north east);
\draw [->-] (eta) to (ev);
\draw [->-] (eta) to (out);
\draw [->-,red] (mu1) to (ev);
\node [vertex,fill=white,inner sep=4pt] (gamma) at (0,1.6) {};
}
\end{equation}

\subsection{Diagrams for the cone bimodule $M$}
Here we give some homotopies involving the cone bimodule $M$, which as in the text we denote with orange arrows. For some of these calculations we will need to be careful with signs, so we will include the orientation data.

\subsubsection{Homotopy for associativity of product on cone}\label{app:associativityCone}
Let us make explicit the homotopy proving \cref{prop:associativityProductCone}, that is, realizing the associativity in homology of the product $\pi_M$ on $C_*(A,M)$. For this calculation we will need to be careful with signs. The expressions $\pi_M(x_1,\pi_M(x_2,x_3))$ and $\pi_M(x_1,\pi_M(x_2,x_3))$ are given respectively by evaluating the diagrams
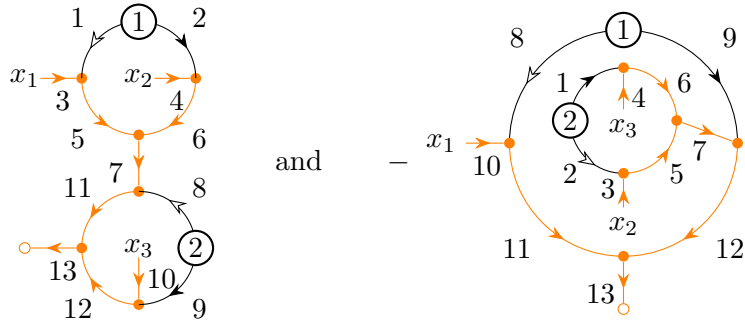
\[\tikzfig{
	\node (v1) at (-1.5,0) {$x_1$};
	\node [bullet,orange] (v4) at (0.75,0) {};
	\node (v2) at (0,0) {$x_2$};
	\node [bullet,orange] (v5) at (0,-0.75) {};
	\node [bullet,orange] (v6) at (0,-1.5) {};
	\node [bullet,orange] (v7) at (-0.75,0) {};
	\node [bullet,orange] (v9) at (0,-3) {};
	\node [bullet,orange] (v10) at (-0.75,-2.25) {};
	\node [circ,orange,fill=white] (v11) at (-1.5,-2.25) {};
	\node (v12) at (0,-2.25) {$x_3$};
	\draw [->-,shorten <=-3.5pt,orange] (v1) to (v7);
	\draw [->-,shorten <=-3.5pt,orange] (v2) to (v4);
	\draw [->-] (0,0.75) arc (90:0:0.75);
	\draw [->-,orange] (0.75,0) arc (0:-90:0.75);
	\draw [-w-] (0,0.75) arc (90:180:0.75);
	\draw [->-,orange] (-0.75,0) arc (180:270:0.75);
	\draw [->-,orange] (v5) to (v6);
	\draw [->-,shorten <=-3.5pt,orange] (v12) to (v9);
	\draw [-w-] (0.75,-2.25) arc (0:90:0.75);
	\draw [->-,orange] (0,-1.5) arc (90:180:0.75);
	\draw [->-] (0.75,-2.25) arc (0:-90:0.75);
	\draw [->-,orange] (0,-3) arc (-90:-180:0.75);
	\draw [->-,orange] (v10) to (v11);
	\node [vertex,fill=white] (v3) at (0,0.75) {$1$};
	\node [vertex,fill=white] (v8) at (0.75,-2.25) {$2$};
	\node at (-0.8,0.8) {$1$};
	\node at (0.8,0.8) {$2$};
	\node at (-1,-0.3) {$3$};
	\node at (0.5,-0.3) {$4$};
	\node at (-0.8,-0.8) {$5$};
	\node at (0.8,-0.8) {$6$};
	\node at (-0.3,-1.25) {$7$};
	\node at (0.8,-1.45) {$8$};
	\node at (0.8,-3.05) {$9$};
	\node at (0.3,-2.65) {$10$};
	\node at (-0.8,-1.45) {$11$};
	\node at (-0.8,-3.05) {$12$};
	\node at (-1,-2.55) {$13$};
}\qquad \text{and} \qquad - \tikzfig{
	\draw [->-] (0,1.5) arc (90:0:1.5);
	\draw [->-,orange] (1.5,0) arc (0:-90:1.5);
	\draw [-w-] (0,1.5) arc (90:180:1.5);
	\draw [->-,orange] (-1.5,0) arc (180:270:1.5);
	\draw [->-] (-0.7,0.3) arc (180:90:0.7);
	\draw [-w-] (-0.7,0.3) arc (180:270:0.7);
	\draw [->-,orange] (0,1) arc (90:0:0.7);
	\draw [->-,orange] (0,-0.4) arc (-90:0:0.7);
	\node (v1) at (-2.4,0) {$x_1$};
	\node [bullet,orange] (v2) at (0,1) {};
	\node [vertex,fill=white] (v3) at (-0.7,0.3) {$2$};
	\node (v4) at (0,0.2) {$x_3$};
	\node [bullet,orange] (v5) at (0.7,0.3) {};
	\node [bullet,orange] (v6) at (0,-0.4) {};
	\node (v7) at (0,-1.1) {$x_2$};
	\node [bullet,orange] (v8) at (1.5,0) {};
	\node [vertex,fill=white] (v9) at (0,1.5) {$1$};
	\node [bullet,orange] (v10) at (0,-1.5) {};
	\node [circ,orange,fill=white] (v11) at (0,-2.2) {};
	\node [bullet,orange] (v12) at (-1.5,0) {};
	\draw [->-,orange] (v1) to (v12);
	\draw [->-,orange] (v5) to (v8);
	\draw [->-,orange] (v4) to (v2);
	\draw [->-,orange] (v7) to (v6);
	\draw [->-,orange] (v10) to (v11);
	\node at (-0.8,0.8) {$1$};
	\node at (0.8,0.8) {$6$};
	\node at (-0.2,-0.6) {$3$};
	\node at (0.2,0.6) {$4$};
	\node at (-0.7,-0.4) {$2$};
	\node at (0.7,-0.4) {$5$};
	\node at (1,-0.1) {$7$};
	\node at (-1.4,1.4) {$8$};
	\node at (1.4,1.4) {$9$};
	\node at (-1.8,-0.3) {$10$};
	\node at (-1.4,-1.4) {$11$};
	\node at (1.4,-1.4) {$12$};
	\node at (-0.3,-2) {$13$};
}\]
on $\alpha \otimes \alpha$; note the minus sign on the latter diagram coming from the composition rules. The homotopy between these maps $C^*_{(2)}(A)^{\otimes 2} \to C_*(A,A)$ is simply given by
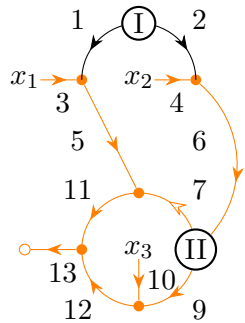
\[\tikzfig{
	\node (v1) at (-1.5,0) {$x_1$};
	\node [bullet,orange] (v4) at (0.75,0) {};
	\node (v2) at (0,0) {$x_2$};
	\node [bullet,orange] (v6) at (0,-1.5) {};
	\node [bullet,orange] (v7) at (-0.75,0) {};
	\node [bullet,orange] (v9) at (0,-3) {};
	\node [bullet,orange] (v10) at (-0.75,-2.25) {};
	\node [circ,orange,fill=white] (v11) at (-1.5,-2.25) {};
	\node (v12) at (0,-2.25) {$x_3$};
	\draw [->-,shorten <=-3.5pt,orange] (v1) to (v7);
	\draw [->-,shorten <=-3.5pt,orange] (v2) to (v4);
	\draw [->-] (0,0.75) arc (90:0:0.75);
	\draw [->-] (0,0.75) arc (90:180:0.75);
	\draw [->-,orange] (v7) to (v6);
	\draw [->-,shorten <=-3.5pt,orange] (v12) to (v9);
	\draw [-w-,orange] (0.75,-2.25) arc (0:90:0.75);
	\draw [->-,orange] (0,-1.5) arc (90:180:0.75);
	\draw [->-,orange] (0.75,-2.25) arc (0:-90:0.75);
	\draw [->-,orange] (0,-3) arc (-90:-180:0.75);
	\draw [->-,orange] (v10) to (v11);
	\node [vertex,fill=white] (v3) at (0,0.75) {I};
	\node [vertex,fill=white] (v8) at (0.75,-2.25) {II};
	\draw [->-, orange,bend left=45] (v4) to (v8);
	\node at (-0.8,0.8) {$1$};
	\node at (0.8,0.8) {$2$};
	\node at (-1,-0.3) {$3$};
	\node at (0.5,-0.3) {$4$};
	\node at (-0.8,-0.8) {$5$};
	\node at (0.8,-0.8) {$6$};
	\node at (0.8,-1.45) {$7$};
	\node at (0.8,-3.05) {$9$};
	\node at (0.3,-2.65) {$10$};
	\node at (-0.8,-1.45) {$11$};
	\node at (-0.8,-3.05) {$12$};
	\node at (-1,-2.55) {$13$};
}\]
together with the diagrams involving only the $\mu_M$ vertices, that following our conventions we leave implicit. Note that we input the $\psi$ map described in \cref{prop:extension} into the vertex with one incoming factor and two outgoing factors of $M$. Comparing signs we get the desired associativity result.

\subsubsection{Compatibility to cup product on $C^*(A,\Ahatinfty)$}\label{app:cupProductCompatibility}
Now we give the two homotopies proving the claims of \cref{prop:cupProductCompatibility}. For the homotopy $ g^M_\alpha(\pi_M(-,-)) \simeq g^M_\alpha(-) \underset{M}{\smile} g^M_\alpha(-)$, we write each side respectively as evaluations of the diagrams
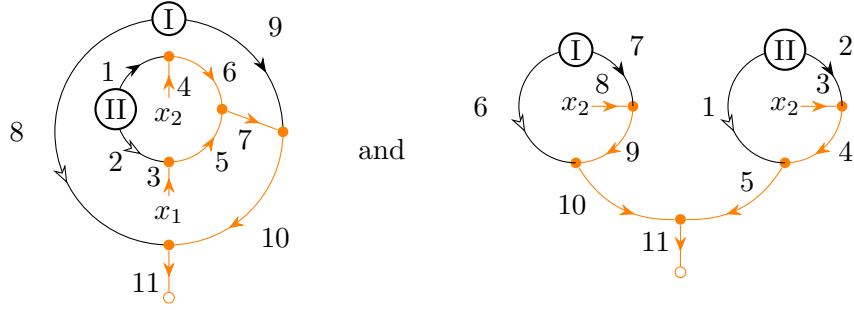
\[  \tikzfig{
	\draw [->-] (0,1.5) arc (90:0:1.5);
	\draw [->-,orange] (1.5,0) arc (0:-90:1.5);
	\draw [-w-] (0,1.5) arc (90:270:1.5);
	\draw [->-] (-0.7,0.3) arc (180:90:0.7);
	\draw [-w-] (-0.7,0.3) arc (180:270:0.7);
	\draw [->-,orange] (0,1) arc (90:0:0.7);
	\draw [->-,orange] (0,-0.4) arc (-90:0:0.7);
	\node [bullet,orange] (v2) at (0,1) {};
	\node [vertex,fill=white] (v3) at (-0.7,0.3) {II};
	\node (v4) at (0,0.2) {$x_2$};
	\node [bullet,orange] (v5) at (0.7,0.3) {};
	\node [bullet,orange] (v6) at (0,-0.4) {};
	\node (v7) at (0,-1.1) {$x_1$};
	\node [bullet,orange] (v8) at (1.5,0) {};
	\node [vertex,fill=white] (v9) at (0,1.5) {I};
	\node [bullet,orange] (v10) at (0,-1.5) {};
	\node [circ,orange,fill=white] (v11) at (0,-2.2) {};
	\draw [->-,orange] (v5) to (v8);
	\draw [->-,orange] (v4) to (v2);
	\draw [->-,orange] (v7) to (v6);
	\draw [->-,orange] (v10) to (v11);
	\node at (-0.8,0.8) {$1$};
	\node at (0.8,0.8) {$6$};
	\node at (-0.2,-0.6) {$3$};
	\node at (0.2,0.6) {$4$};
	\node at (-0.7,-0.4) {$2$};
	\node at (0.7,-0.4) {$5$};
	\node at (1,-0.1) {$7$};
	\node at (-2,0) {$8$};
	\node at (1.4,1.4) {$9$};
	\node at (1.4,-1.4) {$10$};
	\node at (-0.3,-2) {$11$};
} \qquad \text{and} \qquad \tikzfig{
\node [bullet,orange] (v1) at (0.75,0) {};
\node (v2) at (0,0) {$x_2$};
\node [bullet,orange] (v3) at (0,-0.75) {};
\draw [->-,shorten <=-3.5pt,orange] (v2) to (v1);
\draw [->-] (0,0.75) arc (90:0:0.75);
\draw [->-,orange] (0.75,0) arc (0:-90:0.75);
\draw [-w-] (0,0.75) arc (90:270:0.75);
\node [vertex,fill=white] (v4) at (0,0.75) {II};
\node at (-1,0) {$1$};
\node at (0.8,0.8) {$2$};
\node at (0.5,0.3) {$3$};
\node at (0.8,-0.6) {$4$};
\node [bullet,orange] (v5) at (-2,0) {};
\node (v6) at (-2.75,0) {$x_2$};
\node [bullet,orange] (v7) at (-2.75,-0.75) {};
\draw [->-,shorten <=-3.5pt,orange] (v6) to (v5);
\draw [->-] (-2.75,0.75) arc (90:0:0.75);
\draw [->-,orange] (-2,0) arc (0:-90:0.75);
\draw [-w-] (-2.75,0.75) arc (90:270:0.75);
\node [vertex,fill=white] (v8) at (-2.75,0.75) {I};
\node at (-4,0) {$6$};
\node at (-1.95,0.8) {$7$};
\node at (-2.4,0.3) {$8$};
\node [bullet,orange] (mid) at (-1.375,-1.5) {};
\draw [->-,orange,bend left=30] (v3) to (mid);
\draw [->-,orange,bend right=30] (v7) to (mid);
\node [circ,orange,fill=white] (out) at (-1.375,-2.2) {};
\draw [->-,orange] (mid) to (out);
\node at (-0.5,-1) {$5$};
\node at (-2,-0.6) {$9$};
\node at (-2.8,-1.3) {$10$};
\node at (-1.7,-1.8) {$11$};
} \]
Again using our convention and omitting some of the terms, the homotopy between these maps is given by the diagram
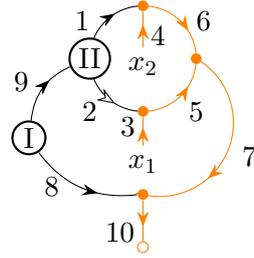
\[\tikzfig{
	\draw [->-] (-0.7,0.3) arc (180:90:0.7);
	\draw [-w-] (-0.7,0.3) arc (180:270:0.7);
	\draw [->-,orange] (0,1) arc (90:0:0.7);
	\draw [->-,orange] (0,-0.4) arc (-90:0:0.7);
	\node [bullet,orange] (v2) at (0,1) {};
	\node [vertex,fill=white] (v3) at (-0.7,0.3) {II};
	\node (v4) at (0,0.2) {$x_2$};
	\node [bullet,orange] (v5) at (0.7,0.3) {};
	\node [bullet,orange] (v6) at (0,-0.4) {};
	\node (v7) at (0,-1.1) {$x_1$};
	\node [vertex,fill=white] (v9) at (-1.5,-0.75) {I};
	\node [bullet,orange] (v10) at (0,-1.5) {};
	\node [circ,orange,fill=white] (v11) at (0,-2.2) {};
	\draw [->-,orange] (v5) arc (60:-108:0.98);
	\draw [->-,bend left=30] (v9) to (v3);
	\draw [->-,bend right=30] (v9) to (v10);
	\draw [->-,orange] (v4) to (v2);
	\draw [->-,orange] (v7) to (v6);
	\draw [->-,orange] (v10) to (v11);
	\node at (-0.8,0.8) {$1$};
	\node at (0.8,0.8) {$6$};
	\node at (-0.2,-0.6) {$3$};
	\node at (0.2,0.6) {$4$};
	\node at (-0.7,-0.4) {$2$};
	\node at (0.7,-0.4) {$5$};
	\node at (1.4,-1) {$7$};
	\node at (-1.2,-1.4) {$8$};
	\node at (-1.6,0) {$9$};
	\node at (-0.3,-2) {$10$};
}\]
Now, for the homotopy proving the second claim, it suffices to take the diagram for $g^M_\alpha(x_1) \smile_M g^M_\alpha(x_2)$ and pass one circle over the other, using $\psi$:
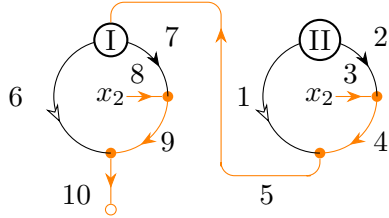
\[\tikzfig{
	\node [bullet,orange] (v1) at (0.75,0) {};
	\node (v2) at (0,0) {$x_2$};
	\node [bullet,orange] (v3) at (0,-0.75) {};
	\draw [->-,shorten <=-3.5pt,orange] (v2) to (v1);
	\draw [->-] (0,0.75) arc (90:0:0.75);
	\draw [->-,orange] (0.75,0) arc (0:-90:0.75);
	\draw [-w-] (0,0.75) arc (90:270:0.75);
	\node [vertex,fill=white] (v4) at (0,0.75) {II};
	\node at (-1,0) {$1$};
	\node at (0.8,0.8) {$2$};
	\node at (0.4,0.3) {$3$};
	\node at (0.8,-0.6) {$4$};
	\node [bullet,orange] (v5) at (-2,0) {};
	\node (v6) at (-2.75,0) {$x_2$};
	\node [bullet,orange] (v7) at (-2.75,-0.75) {};
	\draw [->-,shorten <=-3.5pt,orange] (v6) to (v5);
	\draw [->-] (-2.75,0.75) arc (90:0:0.75);
	\draw [->-,orange] (-2,0) arc (0:-90:0.75);
	\draw [-w-] (-2.75,0.75) arc (90:270:0.75);
	\node [vertex,fill=white] (v8) at (-2.75,0.75) {I};
	\node at (-4,0) {$6$};
	\node at (-1.95,0.8) {$7$};
	\node at (-2.4,0.3) {$8$};
	\draw [->-,orange,rounded corners=5] (v3) |- ++(-0.3,-0.3) -| ++(-1,0.3) |- ++(-0.6,2) -| (v8);
	\node [circ,orange,fill=white] (out) at (-2.75,-1.5) {};
	\draw [->-,orange] (v7) to (out);
	\node at (-0.7,-1.3) {$5$};
	\node at (-2,-0.6) {$9$};
	\node at (-3.2,-1.3) {$10$};
} \]
The $(-1)^n$ factor comes from comparing orientations and exchanging the two $\alpha$ entries, since each one of those inputs has degree $-n$ in $C^*_{(2)}(A)$.

\subsection{Quasi-isomorphism between $M$ and $\Ahatinfty$}\label{app:homotopyAlpha}
For this homotopy proving \cref{prop:homotopyAlpha} we will not need signs, since we are just proving a quasi-isomorphism of cones. The map giving the composition of the left and bottom arrows, that is, $A \otimes^\LL_A A^\vee[-n] \to C^*(A,A^\vee \otimes A) \to C^*(A,\Hom_\kk(A,A))$, and the diagram giving the composition of the top and right arrows, that is, $A \otimes^\LL_A A^\vee \to A \to C^*(A,\Hom_\kk(A,A))$, are given respectively by the diagrams
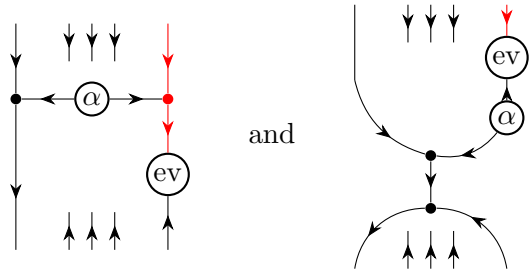
\[ \tikzfig{
	\node [vertex] (mid) at (0,0) {$\alpha$};
	\node [bullet] (left) at (-1,0) {};
	\node [bullet,red] (right) at (1,0) {};
	\node [vertex] (ev) at (1,-1) {$\ev$};
	\draw [->-] (-1,1) to (left);
	\draw [->-] (left) to (-1,-2);
	\draw [->-,red] (1,1) to (right);
	\draw [->-,red] (right) to (ev);
	\draw [->-] (1,-2) to (ev);
	\draw [->-] (mid) to (right);
	\draw [->-] (mid) to (left);
	\draw [->-] (-0.3,1) to (-0.3,0.5);
	\draw [->-] (0,1) to (0,0.5);
	\draw [->-] (0.3,1) to (0.3,0.5);
	\draw [->-] (-0.3,-2) to (-0.3,-1.5);
	\draw [->-] (0,-2) to (0,-1.5);
	\draw [->-] (0.3,-2) to (0.3,-1.5);
} \qquad \text{and} \qquad \tikzfig{
\node [vertex] (ev) at (0,0.8) {$\ev$};
\node [vertex] (alpha) at (0,0) {$\alpha$};
\node [bullet] (mid) at (-1,-0.5) {};
\node [bullet] (down) at (-1,-1.2) {};
\draw [->-,red] (0,1.5) to (ev);
\draw (-2,1.5) to (-2,0.5);
\draw [->-,bend right=30] (-2,0.5) to (mid);
\draw [->-,bend left=30] (alpha) to (mid);
\draw [->-] (alpha) to (ev);
\draw [->-] (mid) to (down);
\draw [->-,bend right=40] (0,-2) to (down);
\draw [->-,bend right=40] (down) to (-2,-2);
\draw [->-] (-1.3,1.5) to (-1.3,1);
\draw [->-] (-1,1.5) to (-1,1);
\draw [->-] (-0.7,1.5) to (-0.7,1);
\draw [->-] (-1.3,-2) to (-1.3,-1.5);
\draw [->-] (-1,-2) to (-1,-1.5);
\draw [->-] (-0.7,-2) to (-0.7,-1.5);
}\]
The homotopy between these diagrams is simply given by the diagram
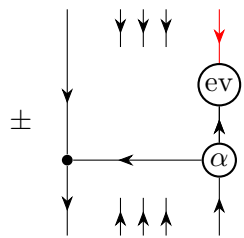
\[ \pm \quad \tikzfig{
	\node [bullet] (left) at (-1,-1) {};
	\node [vertex] (right) at (1,-1) {$\alpha$};
	\node [vertex] (ev) at (1,0) {$\ev$};
	\draw [->-] (-1,1) to (left);
	\draw [->-] (left) to (-1,-2);
	\draw [->-,red] (1,1) to (ev);
	\draw [->-] (right) to (ev);
	\draw [->-] (1,-2) to (right);
	\draw [->-] (right) to (ev);
	\draw [->-] (right) to (left);
	\draw [->-] (-0.3,1) to (-0.3,0.5);
	\draw [->-] (0,1) to (0,0.5);
	\draw [->-] (0.3,1) to (0.3,0.5);
	\draw [->-] (-0.3,-2) to (-0.3,-1.5);
	\draw [->-] (0,-2) to (0,-1.5);
	\draw [->-] (0.3,-2) to (0.3,-1.5);
}\]

\subsection{Diagrams for compatibility between products on the dual and coproducts}
Here we will give the many homotopies that were omitted from \cref{sec:relations}. For all these homotopies we will need orientations.

\subsubsection{Relation between $E$ and $D_\alpha$}\label{app:Jhomotopy}
Below we give the expression for a map $C^*_{(2)}(A)^{\otimes 2} \to C^*(A,A)[-1]$ that, when evaluated on $\alpha \otimes \alpha$, gives the element $J_\alpha$ of \cref{lem:Jhomotopy}.
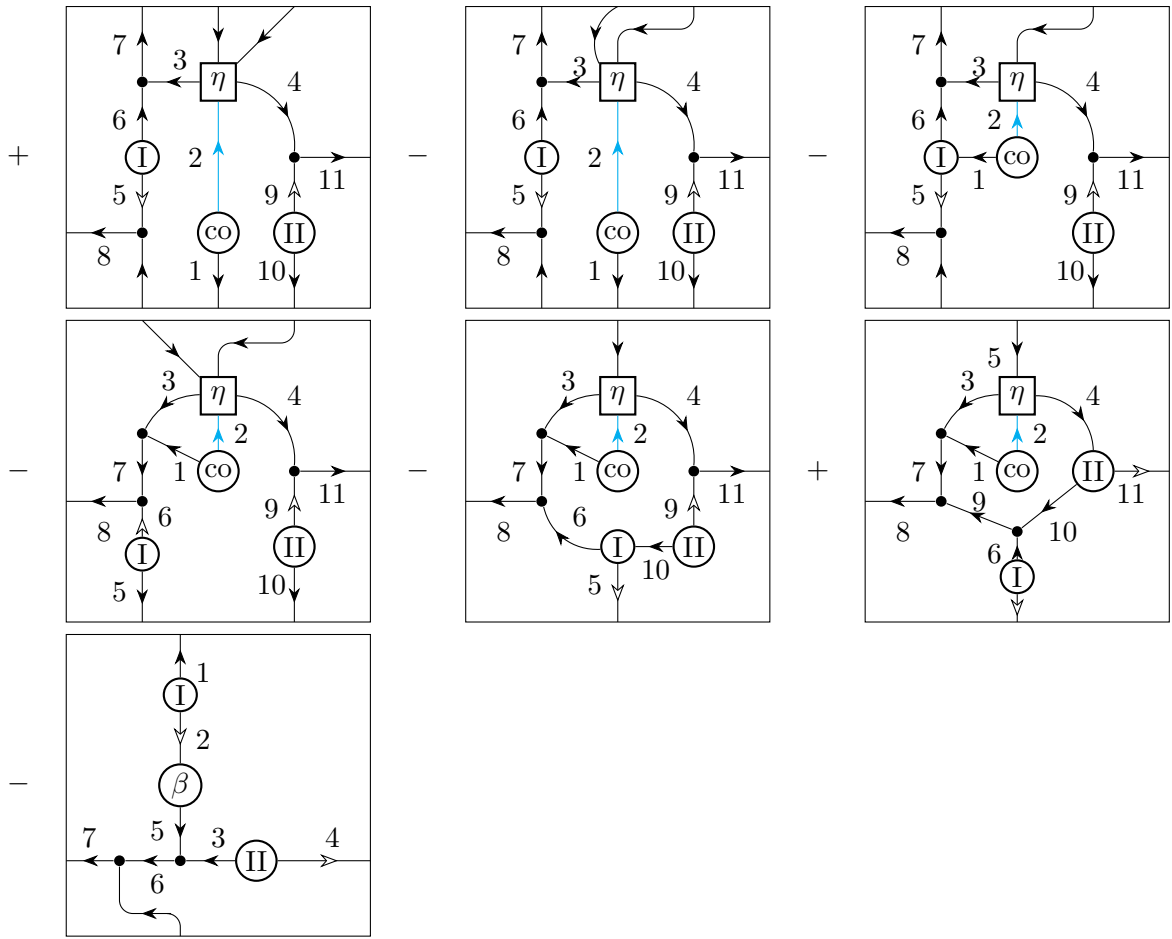
\begin{align*}
	&+ \quad \tikzfig{
		\draw (-2,-2) rectangle (2,2);
		\node [vertex] (co) at (0,-1) {$\co$};
		\node [rectangle,draw,thick] (eta) at (0,1) {$\eta$};
		\node [vertex] (alphaL) at (-1,0) {I};
		\node [vertex] (alphaR) at (1,-1) {II};
		\node [bullet] (topL) at (-1,1) {};
		\node [bullet] (botL) at (-1,-1) {};
		\node [bullet] (midR) at (1,0) {};
		\draw [->-,cyan] (co) to (eta);
		\draw [->-] (co) to (0,-2);
		\draw [->-] (0,2) to (eta);
		\draw [->-] (eta) to (topL);
		\draw [->-,bend left=45] (eta) to (midR);
		\draw [->-] (alphaL) to (topL);
		\draw [-w-] (alphaL) to (botL);
		\draw [->-] (topL) to (-1,2);
		\draw [->-] (-1,-2) to (botL);
		\draw [->-] (botL) to (-2,-1);
		\draw [-w-] (alphaR) to (midR);
		\draw [->-] (alphaR) to (1,-2);
		\draw [->-] (1,2) to (eta);
		\draw [->-] (midR) to (2,0);
		\node at (-0.3,-1.5) {$ 1$};
		\node at (-0.3,0) {$ 2$};
		\node at (-0.5,1.3) {$ 3$};
		\node at (1,1) {$ 4$};
		\node at (-1.3,-0.5) {$ 5$};
		\node at (-1.3,0.5) {$ 6$};
		\node at (-1.3,1.5) {$ 7$};
		\node at (-1.5,-1.3) {$ 8$};
		\node at (0.7,-0.5) {$ 9$};
		\node at (0.7,-1.5) {$ {10}$};
		\node at (1.5,-0.3) {$ {11}$};
	} \quad - \quad \tikzfig{
		\draw (-2,-2) rectangle (2,2);
		\node [vertex] (co) at (0,-1) {$\co$};
		\node [rectangle,draw,thick] (eta) at (0,1) {$\eta$};
		\node [vertex] (alphaL) at (-1,0) {I};
		\node [vertex] (alphaR) at (1,-1) {II};
		\node [bullet] (topL) at (-1,1) {};
		\node [bullet] (botL) at (-1,-1) {};
		\node [bullet] (midR) at (1,0) {};
		\draw [->-,cyan] (co) to (eta);
		\draw [->-] (co) to (0,-2);
		\draw [->-,bend right=45] (0,2) to (eta);
		\draw [->-] (eta) to (topL);
		\draw [->-,bend left=45] (eta) to (midR);
		\draw [->-] (alphaL) to (topL);
		\draw [-w-] (alphaL) to (botL);
		\draw [->-] (topL) to (-1,2);
		\draw [->-] (-1,-2) to (botL);
		\draw [->-] (botL) to (-2,-1);
		\draw [-w-] (alphaR) to (midR);
		\draw [->-] (alphaR) to (1,-2);
		\draw [->-,rounded corners=5] (1,2) -- ++(0,-0.3) -| (eta);
		\draw [->-] (midR) to (2,0);
		\node at (-0.3,-1.5) {$ 1$};
		\node at (-0.3,0) {$ 2$};
		\node at (-0.5,1.2) {$ 3$};
		\node at (1,1) {$ 4$};
		\node at (-1.3,-0.5) {$ 5$};
		\node at (-1.3,0.5) {$ 6$};
		\node at (-1.3,1.5) {$ 7$};
		\node at (-1.5,-1.3) {$ 8$};
		\node at (0.7,-0.5) {$ 9$};
		\node at (0.7,-1.5) {$ {10}$};
		\node at (1.5,-0.3) {$ {11}$};
	} \quad - \quad \tikzfig{
		\draw (-2,-2) rectangle (2,2);
		\node [vertex] (co) at (0,0) {$\co$};
		\node [rectangle,draw,thick] (eta) at (0,1) {$\eta$};
		\node [vertex] (alphaL) at (-1,0) {I};
		\node [vertex] (alphaR) at (1,-1) {II};
		\node [bullet] (topL) at (-1,1) {};
		\node [bullet] (botL) at (-1,-1) {};
		\node [bullet] (midR) at (1,0) {};
		\draw [->-,cyan] (co) to (eta);
		\draw [->-] (co) to (alphaL);
		\draw [->-] (eta) to (topL);
		\draw [->-,bend left=45] (eta) to (midR);
		\draw [->-] (alphaL) to (topL);
		\draw [-w-] (alphaL) to (botL);
		\draw [->-] (topL) to (-1,2);
		\draw [->-] (-1,-2) to (botL);
		\draw [->-] (botL) to (-2,-1);
		\draw [-w-] (alphaR) to (midR);
		\draw [->-] (alphaR) to (1,-2);
		\draw [->-,rounded corners=5] (1,2) -- ++(0,-0.3) -| (eta);
		\draw [->-] (midR) to (2,0);
		\node at (-0.5,-0.3) {$ 1$};
		\node at (-0.3,0.5) {$ 2$};
		\node at (-0.5,1.2) {$ 3$};
		\node at (1,1) {$ 4$};
		\node at (-1.3,-0.5) {$ 5$};
		\node at (-1.3,0.5) {$ 6$};
		\node at (-1.3,1.5) {$ 7$};
		\node at (-1.5,-1.3) {$ 8$};
		\node at (0.7,-0.5) {$ 9$};
		\node at (0.7,-1.5) {$ {10}$};
		\node at (1.5,-0.3) {$ {11}$};
	} \\
	&- \quad \tikzfig{
		\draw (-2,-2) rectangle (2,2);
		\node [vertex] (co) at (0,0) {$\co$};
		\node [rectangle,draw,thick] (eta) at (0,1) {$\eta$};
		\node [vertex] (alphaL) at (-1,-1.1) {I};
		\node [vertex] (alphaR) at (1,-1) {II};
		\node [bullet] (midL) at (-1,0.5) {};
		\node [bullet] (botL) at (-1,-0.4) {};
		\node [bullet] (midR) at (1,0) {};
		\draw [->-,cyan] (co) to (eta);
		\draw [->-] (co) to (midL);
		\draw [->-,bend right=30] (eta) to (midL);
		\draw [->-,bend left=45] (eta) to (midR);
		\draw [->-] (alphaL) to (-1,-2);
		\draw [-w-] (alphaL) to (botL);
		\draw [->-] (midL) to (botL);
		\draw [->-] (-1,2) to (eta);
		\draw [->-] (botL) to (-2,-0.4);
		\draw [-w-] (alphaR) to (midR);
		\draw [->-] (alphaR) to (1,-2);
		\draw [->-,rounded corners=5] (1,2) -- ++(0,-0.3) -| (eta);
		\draw [->-] (midR) to (2,0);
		\node at (-0.5,0) {$ 1$};
		\node at (0.3,0.5) {$ 2$};
		\node at (-0.65,1.2) {$ 3$};
		\node at (1,1) {$ 4$};
		\node at (-1.3,-1.6) {$ 5$};
		\node at (-0.7,-0.6) {$ 6$};
		\node at (-1.3,0) {$ 7$};
		\node at (-1.5,-0.8) {$ 8$};
		\node at (0.7,-0.5) {$ 9$};
		\node at (0.7,-1.5) {$ {10}$};
		\node at (1.5,-0.3) {$ {11}$};
	} \quad - \quad  \tikzfig{
		\draw (-2,-2) rectangle (2,2);
		\node [vertex] (co) at (0,0) {$\co$};
		\node [rectangle,draw,thick] (eta) at (0,1) {$\eta$};
		\node [vertex] (alphaL) at (0,-1) {I};
		\node [vertex] (alphaR) at (1,-1) {II};
		\node [bullet] (midL) at (-1,0.5) {};
		\node [bullet] (botL) at (-1,-0.4) {};
		\node [bullet] (midR) at (1,0) {};
		\draw [->-,cyan] (co) to (eta);
		\draw [->-] (co) to (midL);
		\draw [->-,bend right=30] (eta) to (midL);
		\draw [->-,bend left=45] (eta) to (midR);
		\draw [-w-] (alphaL) to (0,-2);
		\draw [->-,bend left=40] (alphaL) to (botL);
		\draw [->-] (midL) to (botL);
		\draw [->-] (botL) to (-2,-0.4);
		\draw [-w-] (alphaR) to (midR);
		\draw [->-] (alphaR) to (alphaL);
		\draw [->-] (0,2) to (eta);
		\draw [->-] (midR) to (2,0);
		\node at (-0.5,0) {$ 1$};
		\node at (0.3,0.5) {$ 2$};
		\node at (-0.65,1.2) {$ 3$};
		\node at (1,1) {$ 4$};
		\node at (-0.3,-1.5) {$ 5$};
		\node at (-0.5,-0.6) {$ 6$};
		\node at (-1.3,0) {$ 7$};
		\node at (-1.5,-0.8) {$ 8$};
		\node at (0.7,-0.5) {$ 9$};
		\node at (0.5,-1.3) {$ {10}$};
		\node at (1.5,-0.3) {$ {11}$};
	} \quad + \quad \tikzfig{
		\draw (-2,-2) rectangle (2,2);
		\node [vertex] (co) at (0,0) {$\co$};
		\node [rectangle,draw,thick] (eta) at (0,1) {$\eta$};
		\node [vertex] (alphaL) at (0,-1.4) {I};
		\node [vertex] (alphaR) at (1,0) {II};
		\node [bullet] (botL) at (-1,-0.4) {};
		\node [bullet] (midL) at (-1,0.5) {};
		\node [bullet] (botC) at (0,-0.8) {};
		\draw [->-,cyan] (co) to (eta);
		\draw [->-] (co) to (midL);
		\draw [->-,bend right=30] (eta) to (midL);
		\draw [->-,bend left=45] (eta) to (alphaR);
		\draw [-w-=.8] (alphaL) to (0,-2);
		\draw [->-] (botC) to (botL);
		\draw [->-] (midL) to (botL);
		\draw [->-] (botL) to (-2,-0.4);
		\draw [->-] (alphaR) to (botC);
		\draw [->-] (alphaL) to (botC);
		\draw [->-] (0,2) to (eta);
		\draw [-w-] (alphaR) to (2,0);
		\node at (-0.5,0) {$ 1$};
		\node at (0.3,0.5) {$ 2$};
		\node at (-0.65,1.2) {$ 3$};
		\node at (1,1) {$ 4$};
		\node at (-0.3,1.5) {$ 5$};
		\node at (-0.3,-1.1) {$ 6$};
		\node at (-1.3,0) {$ 7$};
		\node at (-1.5,-0.8) {$ 8$};
		\node at (-0.5,-0.4) {$ 9$};
		\node at (0.6,-0.8) {$ {10}$};
		\node at (1.5,-0.3) {$ {11}$};
	} \\
	&- \quad \tikzfig{
		\draw (-2,-2) rectangle (2,2);
		\node [vertex] (alphaL) at (-0.5,1.2) {I};
		\node [vertex] (alphaR) at (0.5,-1) {II};
		\node [vertex] (midL) at (-0.5,0) {$\beta$};
		\node [bullet] (botL) at (-0.5,-1) {};
		\node [bullet] (botLL) at (-1.3,-1) {};
		\draw [-w-] (alphaL) to (midL);
		\draw [->-] (midL) to (botL);
		\draw [->-] (botL) to (botLL);
		\draw [->-] (alphaL) to (-0.5,2);
		\draw [->-=0.5,rounded corners=5] (-0.5,-2) |- ++(-0.3,0.3) -| (botLL);
		\draw [->-] (botLL) to (-2,-1);
		\draw [->-] (alphaR) to (botL);
		\draw [-w-] (alphaR) to (2,-1);
		\node at (-0.2,0.6) {$2$};
		\node at (-0.2,1.5) {$1$};
		\node at (0,-0.7) {$ 3$};
		\node at (1.5,-0.7) {$ 4$};
		\node at (-0.8,-0.6) {$ 5$};
		\node at (-0.8,-1.3) {$ 6$};
		\node at (-1.7,-0.7) {$ 7$};
	}
\end{align*}

\subsubsection{Homotopy for the square-filling Lemma}\label{app:squareLemma}
Here we give the combination of diagrams giving the expression for $N_\alpha$ in \cref{lem:squareLemma}. It is rather complicated: the smallest solution we could find has no less than 53 terms. One can arrive to this solution by `trying to pass the $\varphi$ vertex from one side to the other' on each diagram involved in $J_\alpha$, and then correcting the remaining terms in the differential by hand, adding extra diagrams. Our expression for the element $N_\alpha$ is given by evaluating the following combination of diagrams:%
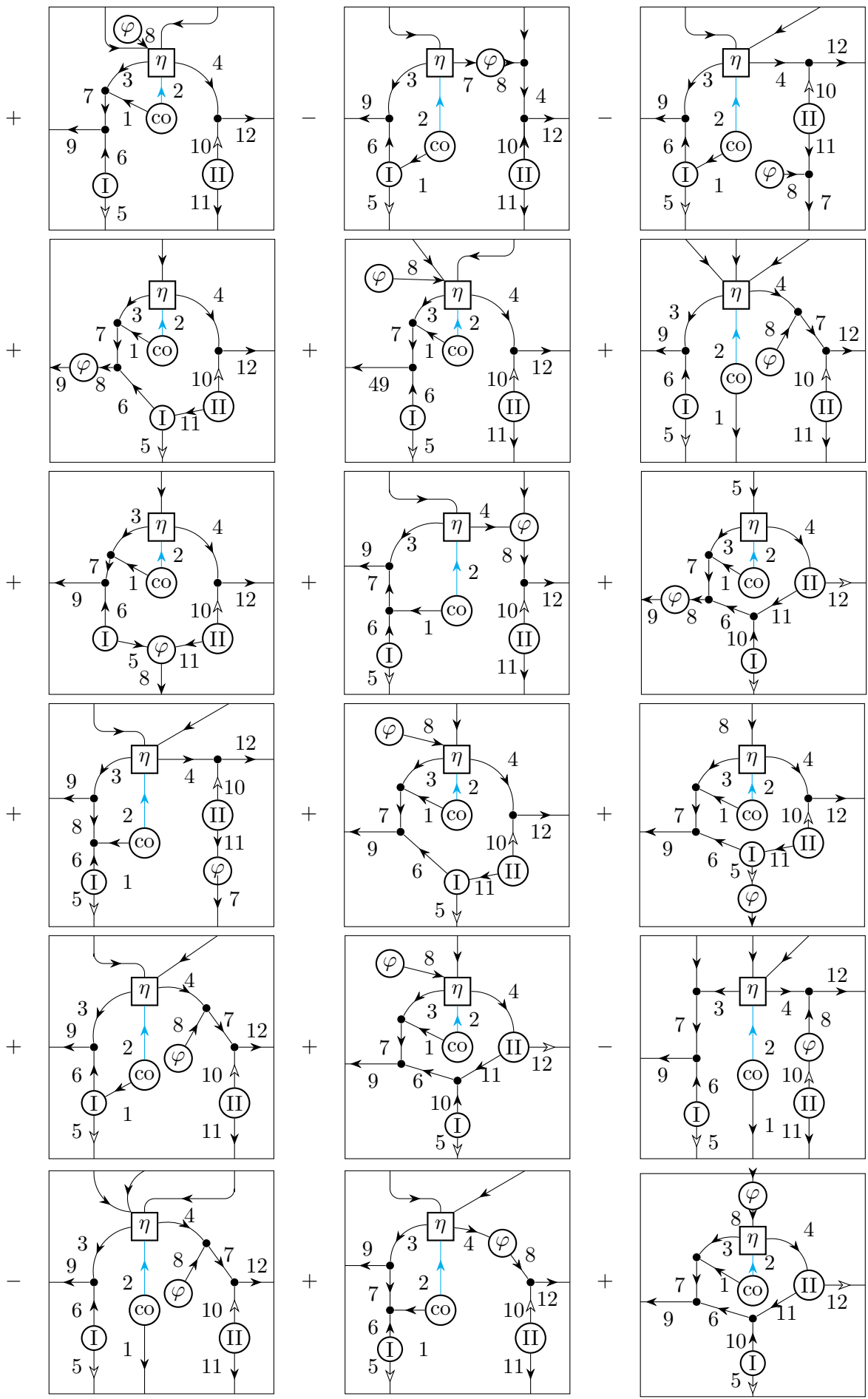
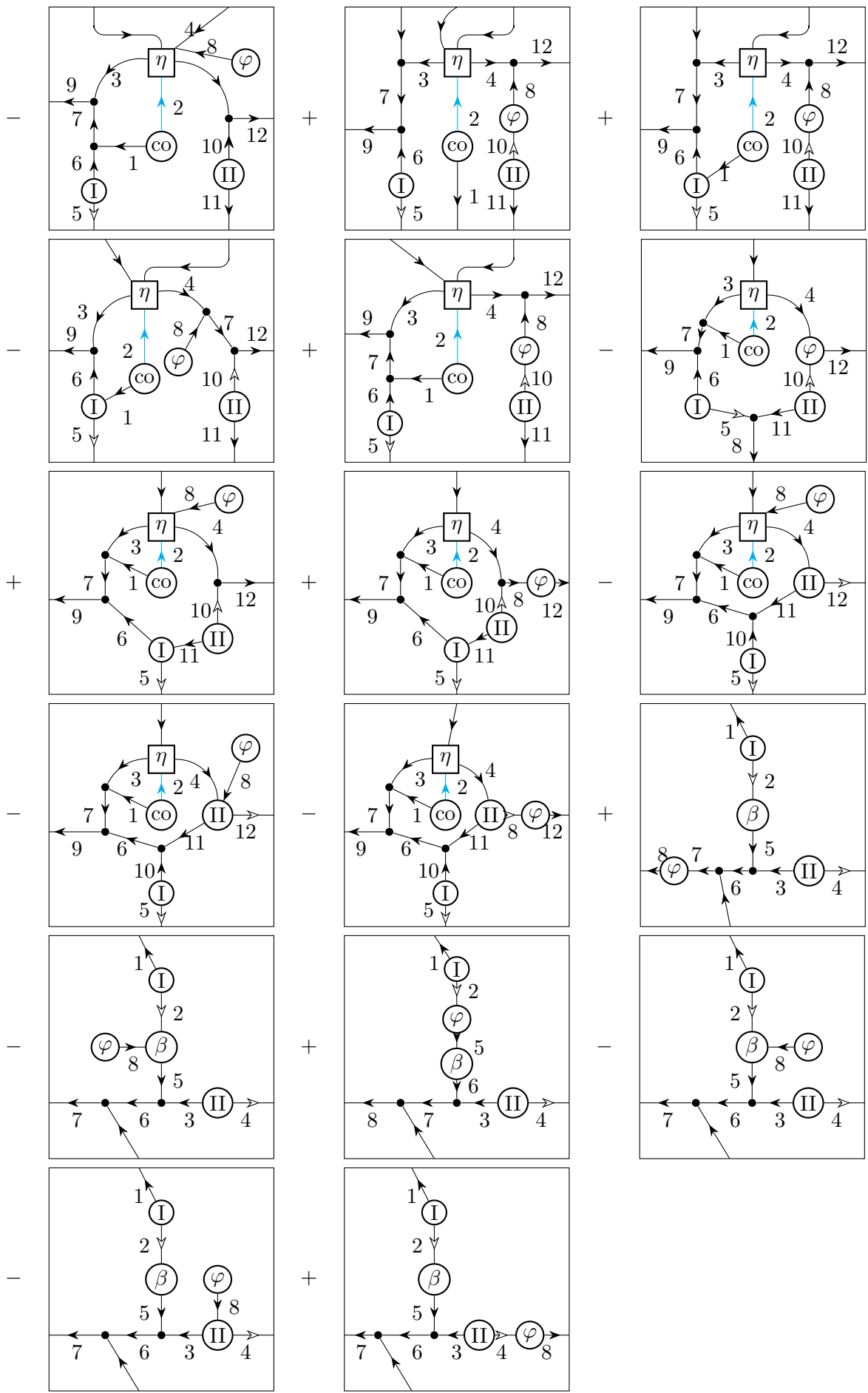
\begin{align*}				
&+ \quad \tikzfig{					%
\draw (-2,-2) rectangle (2,2);
\node [vertex] (co) at (0,-0.5) {$\co$};
\node [rectangle,draw,thick] (eta) at (0,1) {$\eta$};
\node [vertex] (alphaL) at (-1,0) {I};
\node [vertex] (alphaR) at (1,-1) {II};
\node [bullet] (topL) at (-1,1) {};
\node [bullet] (botL) at (-1,-0.7) {};
\node [vertex] (phi) at (-1,-1.5) {$\varphi$};
\node [bullet] (midR) at (1,0) {};
\draw [->-,cyan] (co) to (eta);
\draw [->-] (co) to (0,-2);
\draw [->-] (0,2) to (eta);
\draw [->-] (eta) to (topL);
\draw [->-,bend left=45] (eta) to (midR);
\draw [->-] (alphaL) to (topL);
\draw [-w-] (alphaL) to (botL);
\draw [->-] (topL) to (-1,2);
\draw [->-=.9] (-1,-2) to (phi);
\draw [->-] (phi) to (botL);
\draw [->-] (botL) to (-2,-0.7);
\draw [-w-] (alphaR) to (midR);
\draw [->-] (alphaR) to (1,-2);
\draw [->-] (1,2) to (eta);
\draw [->-] (midR) to (2,0);
\node at (-0.3,-1.2) {$ 1$};
\node at (-0.3,0) {$ 2$};
\node at (-0.5,1.3) {$ 3$};
\node at (1,1) {$ 4$};
\node at (-0.7,-0.5) {$ 5$};
\node at (-1.3,0.5) {$ 6$};
\node at (-1.3,1.5) {$ 7$};
\node at (-0.7,-1) {$ 8$};
\node at (-1.5,-1) {$ 9$};
\node at (0.7,-0.5) {$10$};
\node at (0.7,-1.5) {$11$};
\node at (1.5,-0.3) {$12$};
} \quad 							
- \quad \tikzfig{					%
\draw (-2,-2) rectangle (2,2);
\node [vertex] (co) at (0,-0.5) {$\co$};
\node [rectangle,draw,thick] (eta) at (0,1) {$\eta$};
\node [vertex] (alphaL) at (-1.2,0) {I};
\node [vertex] (alphaR) at (1,-1) {II};
\node [bullet] (topL) at (-1.2,1) {};
\node [bullet] (botL) at (-1.2,-1) {};
\node [bullet] (midR) at (1,0) {};
\node [vertex] (phi) at (-0.7,1.6) {$\varphi$};
\draw [->-,cyan] (co) to (eta);
\draw [->-] (co) to (0,-2);
\draw [->-] (0,2) to (eta);
\draw [->-] (eta) to (topL);
\draw [->-,bend left=45] (eta) to (midR);
\draw [->-] (alphaL) to (topL);
\draw [-w-] (alphaL) to (botL);
\draw [->-] (topL) to (-1.2,2);
\draw [->-] (-1.2,-2) to (botL);
\draw [->-] (botL) to (-2,-1);
\draw [-w-] (alphaR) to (midR);
\draw [->-] (alphaR) to (1,-2);
\draw [->-] (1,2) to (eta);
\draw [->-] (midR) to (2,0);
\draw [->-] (phi) to (eta);
\node at (-0.3,-1.2) {$ 1$};
\node at (-0.3,0) {$ 2$};
\node at (-0.5,0.7) {$ 3$};
\node at (1,1) {$ 4$};
\node at (-1.5,-0.5) {$ 5$};
\node at (-1.5,0.5) {$ 6$};
\node at (-1.5,1.5) {$ 7$};
\node at (-0.3,1.5) {$ 8$};
\node at (-1.6,-1.3) {$ 9$};
\node at (0.7,-0.5) {$10$};
\node at (0.7,-1.5) {$11$};
\node at (1.5,-0.3) {$12$};
} \quad 							
- \quad \tikzfig{					%
\draw (-2,-2) rectangle (2,2);
\node [vertex] (co) at (0,-0.5) {$\co$};
\node [rectangle,draw,thick] (eta) at (0,1) {$\eta$};
\node [vertex] (alphaL) at (-1,0) {I};
\node [vertex] (alphaR) at (1,-1) {II};
\node [bullet] (topL) at (-1,1) {};
\node [bullet] (botL) at (-1,-0.7) {};
\node [vertex] (phi) at (-1,-1.5) {$\varphi$};
\node [bullet] (midR) at (1,0) {};
\draw [->-,cyan] (co) to (eta);
\draw [->-] (co) to (0,-2);
\draw [->-,bend right=45] (0,2) to (eta);
\draw [->-,rounded corners=5] (1,2) -- ++(0,-0.3) -| (eta);
\draw [->-] (eta) to (topL);
\draw [->-,bend left=45] (eta) to (midR);
\draw [->-] (alphaL) to (topL);
\draw [-w-] (alphaL) to (botL);
\draw [->-] (topL) to (-1,2);
\draw [->-=.9] (-1,-2) to (phi);
\draw [->-] (phi) to (botL);
\draw [->-] (botL) to (-2,-0.7);
\draw [-w-] (alphaR) to (midR);
\draw [->-] (alphaR) to (1,-2);
\draw [->-] (midR) to (2,0);
\node at (-0.3,-1.2) {$ 1$};
\node at (-0.3,0) {$ 2$};
\node at (-0.5,1.3) {$ 3$};
\node at (1,1) {$ 4$};
\node at (-0.7,-0.5) {$ 5$};
\node at (-1.3,0.5) {$ 6$};
\node at (-1.3,1.5) {$ 7$};
\node at (-0.7,-1) {$ 8$};
\node at (-1.5,-1) {$ 9$};
\node at (0.7,-0.5) {$10$};
\node at (0.7,-1.5) {$11$};
\node at (1.5,-0.3) {$12$};
}\\									
&+ \quad \tikzfig{					%
	\draw (-2,-2) rectangle (2,2);
	\node [vertex] (co) at (0,-0.5) {$\co$};
	\node [rectangle,draw,thick] (eta) at (0,1) {$\eta$};
	\node [vertex] (alphaL) at (-1.2,0) {I};
	\node [vertex] (alphaR) at (1,-1) {II};
	\node [bullet] (topL) at (-1.2,1) {};
	\node [bullet] (botL) at (-1.2,-1) {};
	\node [bullet] (midR) at (1,0) {};
	\node [vertex] (phi) at (-0.8,1.7) {$\varphi$};
	\draw [->-,cyan] (co) to (eta);
	\draw [->-] (co) to (0,-2);
	\draw [->-,bend right=45] (0,2) to (eta);
	\draw [->-] (eta) to (topL);
	\draw [->-,bend left=45] (eta) to (midR);
	\draw [->-] (alphaL) to (topL);
	\draw [-w-] (alphaL) to (botL);
	\draw [->-] (topL) to (-1.2,2);
	\draw [->-] (-1.2,-2) to (botL);
	\draw [->-] (botL) to (-2,-1);
	\draw [-w-] (alphaR) to (midR);
	\draw [->-] (alphaR) to (1,-2);
	\draw [->-,rounded corners=5] (1,2) -- ++(0,-0.3) -| (eta);
	\draw [->-] (midR) to (2,0);
	\draw [->-] (phi) to (eta);
	\node at (-0.3,-1.2) {$ 1$};
	\node at (-0.3,0) {$ 2$};
	\node at (-0.5,0.7) {$ 3$};
	\node at (1,1) {$ 4$};
	\node at (-1.5,-0.5) {$ 5$};
	\node at (-1.5,0.5) {$ 6$};
	\node at (-1.5,1.5) {$ 7$};
	\node at (-0.5,1.2) {$ 8$};
	\node at (-1.6,-1.3) {$ 9$};
	\node at (0.7,-0.5) {$10$};
	\node at (0.7,-1.5) {$11$};
	\node at (1.5,-0.3) {$12$};
} \quad 
- \quad \tikzfig{					%
		\draw (-2,-2) rectangle (2,2);
		\node [vertex] (co) at (0,-0.5) {$\co$};
		\node [rectangle,draw,thick] (eta) at (0,1) {$\eta$};
		\node [vertex] (alphaL) at (-1,0) {I};
		\node [vertex] (alphaR) at (1,-1) {II};
		\node [bullet] (topL) at (-1,1) {};
		\node [bullet] (botL) at (-1,-0.7) {};
		\node [vertex] (phi) at (-1,-1.5) {$\varphi$};
		\node [bullet] (midR) at (1,0) {};
		\draw [->-,cyan] (co) to (eta);
		\draw [->-,bend left=50] (co) to (phi);
		\draw [->-] (eta) to (topL);
		\draw [->-,bend left=45] (eta) to (midR);
		\draw [->-] (alphaL) to (topL);
		\draw [-w-] (alphaL) to (botL);
		\draw [->-] (topL) to (-1,2);
		\draw [->-=.9] (-1,-2) to (phi);
		\draw [->-] (phi) to (botL);
		\draw [->-] (botL) to (-2,-0.7);
		\draw [-w-] (alphaR) to (midR);
		\draw [->-] (alphaR) to (1,-2);
		\draw [->-,rounded corners=5] (1,2) -- ++(0,-0.3) -| (eta);
		\draw [->-] (midR) to (2,0);
		\node at (-0.3,-1.2) {$ 1$};
		\node at (-0.3,0) {$ 2$};
		\node at (-0.5,1.3) {$ 3$};
		\node at (1,1) {$ 4$};
		\node at (-0.7,-0.5) {$ 5$};
		\node at (-1.3,0.5) {$ 6$};
		\node at (-1.3,1.5) {$ 7$};
		\node at (-0.7,-1) {$ 8$};
		\node at (-1.5,-1) {$ 9$};
		\node at (0.7,-0.5) {$10$};
		\node at (0.7,-1.5) {$11$};
		\node at (1.5,-0.3) {$12$};
	} \quad 
- \quad \tikzfig{							%
	\draw (-2,-2) rectangle (2,2);
	\node [vertex] (co) at (0,0) {$\co$};
	\node [rectangle,draw,thick] (eta) at (0,1) {$\eta$};
	\node [vertex] (alphaL) at (-1,0) {I};
	\node [vertex] (alphaR) at (1,-1) {II};
	\node [bullet] (topL) at (-1,1) {};
	\node [bullet] (botL) at (-1,-0.7) {};
	\node [vertex] (phi) at (-1,-1.5) {$\varphi$};
	\node [bullet] (midR) at (1,0) {};
	\draw [->-,cyan] (co) to (eta);
	\draw [->-] (co) to (alphaL);
	\draw [->-] (eta) to (topL);
	\draw [->-,bend left=45] (eta) to (midR);
	\draw [->-] (alphaL) to (topL);
	\draw [-w-] (alphaL) to (botL);
	\draw [->-] (topL) to (-1,2);
	\draw [->-=.9] (-1,-2) to (phi);
	\draw [->-] (phi) to (botL);
	\draw [->-] (botL) to (-2,-0.7);
	\draw [-w-] (alphaR) to (midR);
	\draw [->-] (alphaR) to (1,-2);
	\draw [->-,rounded corners=5] (1,2) -- ++(0,-0.3) -| (eta);
	\draw [->-] (midR) to (2,0);
	\node at (-0.5,0.3) {$ 1$};
	\node at (0.3,0.5) {$ 2$};
	\node at (-0.5,1.3) {$ 3$};
	\node at (1,1) {$ 4$};
	\node at (-0.7,-0.5) {$ 5$};
	\node at (-1.3,0.5) {$ 6$};
	\node at (-1.3,1.5) {$ 7$};
	\node at (-0.7,-1) {$ 8$};
	\node at (-1.5,-1) {$ 9$};
	\node at (0.7,-0.5) {$10$};
	\node at (0.7,-1.5) {$11$};
	\node at (1.5,-0.3) {$12$};
}\\
&+ \quad \tikzfig{							%
	\draw (-2,-2) rectangle (2,2);
	\node [vertex] (co) at (0,0) {$\co$};
	\node [rectangle,draw,thick] (eta) at (0,1) {$\eta$};
	\node [vertex] (alphaL) at (-1.2,0) {I};
	\node [vertex] (alphaR) at (1,-1) {II};
	\node [bullet] (topL) at (-1.2,1) {};
	\node [bullet] (botL) at (-1.2,-1) {};
	\node [bullet] (midR) at (1,0) {};
	\node [vertex] (phi) at (0,-1) {$\varphi$};
	\draw [->-,cyan] (co) to (eta);
	\draw [->-] (co) to (phi);
	\draw [->-] (phi) to (0,-2);
	\draw [->-,bend right=45] (0,2) to (eta);
	\draw [->-] (eta) to (topL);
	\draw [->-,bend left=45] (eta) to (midR);
	\draw [->-] (alphaL) to (topL);
	\draw [-w-] (alphaL) to (botL);
	\draw [->-] (topL) to (-1.2,2);
	\draw [->-] (-1.2,-2) to (botL);
	\draw [->-] (botL) to (-2,-1);
	\draw [-w-] (alphaR) to (midR);
	\draw [->-] (alphaR) to (1,-2);
	\draw [->-,rounded corners=5] (1,2) -- ++(0,-0.3) -| (eta);
	\draw [->-] (midR) to (2,0);
	\node at (-0.3,-0.5) {$ 1$};
	\node at (0.3,0.5) {$ 2$};
	\node at (-0.5,0.7) {$ 3$};
	\node at (1,1) {$ 4$};
	\node at (-1.5,-0.5) {$ 5$};
	\node at (-1.5,0.5) {$ 6$};
	\node at (-1.5,1.5) {$ 7$};
	\node at (-0.3,-1.5) {$ 8$};
	\node at (-1.6,-1.3) {$ 9$};
	\node at (0.7,-0.5) {$10$};
	\node at (0.7,-1.5) {$11$};
	\node at (1.5,-0.3) {$12$};
}  \quad 
- \quad  \tikzfig{							%
\draw (-2,-2) rectangle (2,2);
\node [vertex] (co) at (0,0) {$\co$};
\node [rectangle,draw,thick] (eta) at (0,1) {$\eta$};
\node [vertex] (alphaL) at (-1.2,0) {I};
\node [vertex] (alphaR) at (1,-1) {II};
\node [bullet] (topL) at (-1.2,1) {};
\node [bullet] (botL) at (-1.2,-1) {};
\node [bullet] (midR) at (1,0) {};
\node [vertex] (phi) at (0,-1) {$\varphi$};
\draw [->-,cyan] (co) to (eta);
\draw [->-] (co) to (phi);
\draw [->-] (phi) to (0,-2);
\draw [->-] (0,2) to (eta);
\draw [->-] (eta) to (topL);
\draw [->-,bend left=45] (eta) to (midR);
\draw [->-] (alphaL) to (topL);
\draw [-w-] (alphaL) to (botL);
\draw [->-] (topL) to (-1.2,2);
\draw [->-] (-1.2,-2) to (botL);
\draw [->-] (botL) to (-2,-1);
\draw [-w-] (alphaR) to (midR);
\draw [->-] (alphaR) to (1,-2);
\draw [->-] (1,2) to (eta);
\draw [->-] (midR) to (2,0);
\node at (-0.3,-0.5) {$ 1$};
\node at (0.3,0.5) {$ 2$};
\node at (-0.5,0.7) {$ 3$};
\node at (1,1) {$ 4$};
\node at (-1.5,-0.5) {$ 5$};
\node at (-1.5,0.5) {$ 6$};
\node at (-1.5,1.5) {$ 7$};
\node at (-0.3,-1.5) {$ 8$};
\node at (-1.6,-1.3) {$ 9$};
\node at (0.7,-0.5) {$10$};
\node at (0.7,-1.5) {$11$};
\node at (1.5,-0.3) {$12$};
} \quad 
+ \quad \tikzfig{							%
\draw (-2,-2) rectangle (2,2);
\node [vertex] (co) at (0,-1) {$\co$};
\node [rectangle,draw,thick] (eta) at (0,0) {$\eta$};
\node [vertex] (alphaL) at (-1.2,0) {I};
\node [vertex] (alphaR) at (1,-1) {II};
\node [bullet] (topL) at (-1.2,1) {};
\node [bullet] (botL) at (-1.2,-1) {};
\node [bullet] (midR) at (1,0) {};
\node [vertex] (phi) at (-0.2,1.4) {$\varphi$};
\draw [->-,cyan] (co) to (eta);
\draw [->-] (co) to (0,-2);
\draw [->-] (phi) to (eta.north west);
\draw [->-,bend right = 80] (0,2) to (eta.north west);
\draw [->-,rounded corners=5] (eta) -| ++(-0.9,0) |- (topL);
\draw [->-] (eta) to (midR);
\draw [->-] (alphaL) to (topL);
\draw [-w-] (alphaL) to (botL);
\draw [->-] (topL) to (-1.2,2);
\draw [->-] (-1.2,-2) to (botL);
\draw [->-] (botL) to (-2,-1);
\draw [-w-] (alphaR) to (midR);
\draw [->-] (alphaR) to (1,-2);
\draw [->-,rounded corners=5] (1,2) -- ++(0,-1) -| (eta);
\draw [->-] (midR) to (2,0);
\node at (-0.3,-1.5) {$ 1$};
\node at (-0.3,-0.5) {$ 2$};
\node at (-0.7,-0.3) {$ 3$};
\node at (0.5,0.3) {$ 4$};
\node at (-1.5,-0.5) {$ 5$};
\node at (-1.5,0.5) {$ 6$};
\node at (-1.5,1.5) {$ 7$};
\node at (-0.4,0.7) {$ 8$};
\node at (-1.6,-1.3) {$ 9$};
\node at (0.7,-0.5) {$10$};
\node at (0.7,-1.5) {$11$};
\node at (1.5,-0.3) {$12$};
}\\
&- \quad \tikzfig{						%
\draw (-2,-2) rectangle (2,2);
\node [vertex] (co) at (0,-0.5) {$\co$};
\node [rectangle,draw,thick] (eta) at (0,0.7) {$\eta$};
\node [vertex] (alphaL) at (-1.2,0) {I};
\node [vertex] (alphaR) at (1,-1) {II};
\node [bullet] (topL) at (-1.2,0.7) {};
\node [bullet] (botL) at (-1.2,-1) {};
\node [bullet] (midR) at (1,0) {};
\node [vertex] (phi) at (0.7,1.6) {$\varphi$};
\draw [->-,cyan] (co) to (eta);
\draw [->-] (co) to (0,-2);
\draw [->-] (0,2) to (eta);
\draw [->-] (eta) to (topL);
\draw [->-,bend left=45] (eta) to (midR);
\draw [->-] (alphaL) to (topL);
\draw [-w-] (alphaL) to (botL);
\draw [->-] (topL) to (-1.2,2);
\draw [->-] (-1.2,-2) to (botL);
\draw [->-] (botL) to (-2,-1);
\draw [-w-] (alphaR) to (midR);
\draw [->-] (alphaR) to (1,-2);
\draw [->-,bend left=80] (1,2) to (eta.north east);
\draw [->-] (midR) to (2,0);
\draw [->-] (phi) to (eta);
\node at (-0.3,-1.2) {$ 1$};
\node at (-0.3,0) {$ 2$};
\node at (-0.6,0.5) {$ 3$};
\node at (1,0.6) {$ 4$};
\node at (-1.5,-0.5) {$ 5$};
\node at (-1.5,0.5) {$ 6$};
\node at (-1.5,1.5) {$ 7$};
\node at (0.3,1.4) {$ 8$};
\node at (-1.6,-1.3) {$ 9$};
\node at (0.7,-0.5) {$10$};
\node at (0.7,-1.5) {$11$};
\node at (1.5,-0.3) {$12$};
} \quad 
- \quad  \tikzfig{							%
\draw (-2,-2) rectangle (2,2);
\node [vertex] (co) at (0,0) {$\co$};
\node [rectangle,draw,thick] (eta) at (0,1) {$\eta$};
\node [vertex] (alphaL) at (-1,-0.6) {I};
\node [vertex] (alphaR) at (1,-1) {II};
\node [bullet] (topL) at (-1,1) {};
\node [bullet] (midL) at (-1,0.2) {};
\node [vertex] (phi) at (-1,-1.5) {$\varphi$};
\node [bullet] (midR) at (1,0) {};
\draw [->-,cyan] (co) to (eta);
\draw [->-] (co) to (midL);
\draw [->-] (eta) to (topL);
\draw [->-,bend left=45] (eta) to (midR);
\draw [->-] (alphaL) to (midL);
\draw [->-] (midL) to (topL);
\draw [-w-] (alphaL) to (phi);
\draw [->-] (topL) to (-1,2);
\draw [->-=.9] (-1,-2) to (phi);
\draw [->-] (phi) to (-2,-1.5);
\draw [-w-] (alphaR) to (midR);
\draw [->-] (alphaR) to (1,-2);
\draw [->-,rounded corners=5] (1,2) -- ++(0,-0.3) -| (eta);
\draw [->-] (midR) to (2,0);
\node at (-0.5,0.3) {$ 1$};
\node at (0.3,0.5) {$ 2$};
\node at (-0.5,1.3) {$ 3$};
\node at (1,1) {$ 4$};
\node at (-0.7,-1) {$5$};
\node at (-1.3,-0.2) {$6$};
\node at (-1.3,0.5) {$7$};
\node at (-1.3,1.5) {$8$};
\node at (-1.5,-1.7) {$ 9$};
\node at (0.7,-0.5) {$10$};
\node at (0.7,-1.5) {$11$};
\node at (1.5,-0.3) {$12$};
} \quad 
+ \quad \tikzfig{						%
\draw (-2,-2) rectangle (2,2);
\node [vertex] (co) at (-0.3,-0.5) {$\co$};
\node [rectangle,draw,thick] (eta) at (-0.3,1) {$\eta$};
\node [vertex] (alphaL) at (-1.2,-1) {I};
\node [vertex] (alphaR) at (1,0) {II};
\node [bullet] (midL) at (-1.2,0) {};
\node [bullet] (midR) at (1,1) {};
\node [vertex] (phi) at (0.3,-1) {$\varphi$};
\node [bullet] (botR) at (1,-1) {};
\draw [->-,cyan] (co) to (eta);
\draw [->-] (co) to (-0.3,-2);
\draw [->-,bend right=45] (-0.3,2) to (eta.north west);
\draw [->-,bend right=45] (eta) to (midL);
\draw [->-] (eta) to (midR);
\draw [->-] (alphaL) to (midL);
\draw [->-,bend right=45] (-1.2,2) to (eta.north west);
\draw [-w-] (alphaL) to (-1.2,-2);
\draw [-w-] (alphaR) to (midR);
\draw [->-] (phi) to (botR);
\draw [->-] (alphaR) to (botR);
\draw [->-] (botR) to (1,-2);
\draw [->-,rounded corners=5] (1,2) -- ++(0,-0.3) -| (eta);
\draw [->-] (midR) to (2,1);
\draw [->-] (midL) to (-2,0);
\node at (-0.6,-1.2) {$1$};
\node at (-0.6,0) {$ 2$};
\node at (-1.4,0.7) {$ 3$};
\node at (0.5,0.7) {$ 4$};
\node at (-1.5,-1.5) {$ 5$};
\node at (-1.5,-0.5) {$ 6$};
\node at (1.3,-1.5) {$ 7$};
\node at (0.7,-1.3) {$ 8$};
\node at (-1.6,0.3) {$ 9$};
\node at (1.3,0.5) {$10$};
\node at (1.3,-0.5) {$11$};
\node at (1.5,1.3) {$12$};
}\\
&+ \quad \tikzfig{							%
	\draw (-2,-2) rectangle (2,2);
	\node [vertex] (co) at (0,0) {$\co$};
	\node [rectangle,draw,thick] (eta) at (0,1) {$\eta$};
	\node [vertex] (alphaL) at (-1,0) {I};
	\node [vertex] (alphaR) at (1,-1) {II};
	\node [bullet] (topL) at (-1,1) {};
	\node [bullet] (botL) at (-1,-1) {};
	\node [vertex] (phi) at (-0.6,1.6) {$\varphi$};
	\node [bullet] (midR) at (1,0) {};
	\draw [->-,cyan] (co) to (eta);
	\draw [->-] (co) to (alphaL);
	\draw [->-] (eta) to (topL);
	\draw [->-,bend left=45] (eta) to (midR);
	\draw [->-] (alphaL) to (topL);
	\draw [-w-] (alphaL) to (botL);
	\draw [->-] (topL) to (-1,2);
	\draw [->-=.9] (-1,-2) to (botL);
	\draw [->-] (botL) to (-2,-1);
	\draw [-w-] (alphaR) to (midR);
	\draw [->-] (alphaR) to (1,-2);
	\draw [->-,rounded corners=5] (1,2) -- ++(0,-0.3) -| (eta);
	\draw [->-] (midR) to (2,0);
	\draw [->-=1] (phi) to (eta.north west);
	\node at (-0.6,-0.2) {$ 1$};
	\node at (0.3,0.5) {$ 2$};
	\node at (-0.6,0.7) {$ 3$};
	\node at (1,1) {$ 4$};
	\node at (-1.3,-0.5) {$5$};
	\node at (-1.3,0.5) {$6$};
	\node at (-1.3,1.5) {$7$};
	\node at (-0.1,1.5) {$8$};
	\node at (-1.5,-1.3) {$ 9$};
	\node at (0.7,-0.5) {$10$};
	\node at (0.7,-1.5) {$11$};
	\node at (1.5,-0.3) {$12$};
} \quad 
- \quad \tikzfig{							%
\draw (-2,-2) rectangle (2,2);
\node [vertex] (co) at (-0.3,-0.5) {$\co$};
\node [rectangle,draw,thick] (eta) at (-0.3,1) {$\eta$};
\node [vertex] (alphaL) at (-1.2,-1) {I};
\node [vertex] (alphaR) at (1,0) {II};
\node [bullet] (midL) at (-1.2,0) {};
\node [bullet] (midR) at (1,1) {};
\node [vertex] (phi) at (0.3,-1) {$\varphi$};
\node [bullet] (botR) at (1,-1) {};
\draw [->-,cyan] (co) to (eta);
\draw [->-] (co) to (-0.3,-2);
\draw [->-] (-0.3,2) to (eta);
\draw [->-,bend right=45] (eta) to (midL);
\draw [->-] (eta) to (midR);
\draw [->-] (alphaL) to (midL);
\draw [->-] (-1.2,2) to (eta.north west);
\draw [-w-] (alphaL) to (-1.2,-2);
\draw [-w-] (alphaR) to (midR);
\draw [->-] (phi) to (botR);
\draw [->-] (alphaR) to (botR);
\draw [->-] (botR) to (1,-2);
\draw [->-] (1,2) to (eta.north east);
\draw [->-] (midR) to (2,1);
\draw [->-] (midL) to (-2,0);
\node at (-0.6,-1.2) {$1$};
\node at (-0.6,0) {$ 2$};
\node at (-1.4,0.7) {$ 3$};
\node at (0.5,0.7) {$ 4$};
\node at (-1.5,-1.5) {$ 5$};
\node at (-1.5,-0.5) {$ 6$};
\node at (1.3,-1.5) {$ 7$};
\node at (0.7,-1.3) {$ 8$};
\node at (-1.6,0.3) {$ 9$};
\node at (1.3,0.5) {$10$};
\node at (1.3,-0.5) {$11$};
\node at (1.5,1.3) {$12$};
} \quad 
- \quad \tikzfig{							%
\draw (-2,-2) rectangle (2,2);
\node [vertex] (co) at (-0.3,-0.5) {$\co$};
\node [rectangle,draw,thick] (eta) at (-0.3,1) {$\eta$};
\node [vertex] (alphaL) at (-1.2,-1) {I};
\node [vertex] (alphaR) at (1.2,-1) {II};
\node [bullet] (midL) at (-1.2,0) {};
\node [bullet] (topR) at (1.2,1) {};
\node [vertex] (phi) at (0.5,1) {$\varphi$};
\node [bullet] (midR) at (1.2,0) {};
\draw [->-,cyan] (co) to (eta);
\draw [->-] (co) to (-0.3,-2);
\draw [->-] (-0.3,2) to (eta);
\draw [->-,bend right=45] (eta) to (midL);
\draw [->-=.9] (eta) to (phi);
\draw [->-=.9] (phi) to (topR);
\draw [->-] (alphaL) to (midL);
\draw [->-] (-1.2,2) to (eta.north west);
\draw [-w-] (alphaL) to (-1.2,-2);
\draw [->-] (midR) to (topR);
\draw [-w-] (alphaR) to (midR);
\draw [->-] (alphaR) to (1.2,-2);
\draw [->-] (1.2,2) to (topR);
\draw [->-] (midR) to (2,0);
\draw [->-] (midL) to (-2,0);
\node at (-0.6,-1.2) {$1$};
\node at (-0.6,0.2) {$ 2$};
\node at (-1.4,0.7) {$ 3$};
\node at (1,0.5) {$4$};
\node at (-1.5,-1.5) {$ 5$};
\node at (-1.5,-0.5) {$ 6$};
\node at (0.2,1.3) {$ 7$};
\node at (0.9,1.3) {$ 8$};
\node at (-1.6,0.3) {$ 9$};
\node at (1.5,-0.5) {$10$};
\node at (1.5,-1.5) {$11$};
\node at (1.6,0.3) {$12$};
}\\
&+ \quad \tikzfig{						%
\draw (-2,-2) rectangle (2,2);
\node [vertex] (co) at (0,0) {$\co$};
\node [rectangle,draw,thick] (eta) at (0,1) {$\eta$};
\node [vertex] (alphaL) at (-0.8,-1.2) {I};
\node [vertex] (alphaR) at (1,-1) {II};
\node [bullet] (topL) at (-0.8,0.5) {};
\node [bullet] (midL) at (-0.8,-0.3) {};
\node [vertex] (phi) at (-1.4,-0.3) {$\varphi$};
\node [bullet] (midR) at (1,0) {};
\draw [->-,cyan] (co) to (eta);
\draw [->-] (co) to (topL);
\draw [->-,bend right=30] (eta) to (topL);
\draw [->-,bend left=45] (eta) to (midR);
\draw [->-] (alphaL) to (midL);
\draw [->-=.9] (midL) to (phi);
\draw [->-] (topL) to (midL);
\draw [-w-=.9] (alphaL) to (-0.8,-2);
\draw [->-=1] (phi) to (-2,-0.3);
\draw [-w-] (alphaR) to (midR);
\draw [->-] (alphaR) to (1,-2);
\draw [->-] (-0.8,2) to (eta.north west);
\draw [->-,rounded corners=5] (1,2) -- ++(0,-0.3) -| (eta);
\draw [->-] (midR) to (2,0);
\node at (-0.5,0) {$1$};
\node at (0.3,0.5) {$ 2$};
\node at (-0.45,0.65) {$ 3$};
\node at (1,1) {$ 4$};
\node at (-0.5,-1.7) {$5$};
\node at (-0.5,-0.8) {$6$};
\node at (-1.1,0.3) {$7$};
\node at (-1.1,-0.6) {$8$};
\node at (-1.8,-0.6) {$ 9$};
\node at (0.7,-0.5) {$10$};
\node at (0.7,-1.5) {$11$};
\node at (1.5,-0.3) {$12$};
} \quad 
+ \quad  \tikzfig{							%
\draw (-2,-2) rectangle (2,2);
\node [vertex] (co) at (0,0) {$\co$};
\node [rectangle,draw,thick] (eta) at (0,1) {$\eta$};
\node [vertex] (alphaL) at (-0.8,-1.2) {I};
\node [vertex] (alphaR) at (1,-1) {II};
\node [bullet] (topL) at (-0.8,0.5) {};
\node [bullet] (midL) at (-0.8,-0.3) {};
\node [vertex] (phi) at (-0.8,1.3) {$\varphi$};
\node [bullet] (midR) at (1,0) {};
\draw [->-,cyan] (co) to (eta);
\draw [->-] (co) to (topL);
\draw [->-,bend right=30] (eta) to (topL);
\draw [->-,bend left=45] (eta) to (midR);
\draw [->-] (alphaL) to (midL);
\draw [->-=.9] (midL) to (-2,-0.3);
\draw [->-] (topL) to (midL);
\draw [-w-=.9] (alphaL) to (-0.8,-2);
\draw [-w-] (alphaR) to (midR);
\draw [->-] (alphaR) to (1,-2);
\draw [->-] (-0.8,2) to (phi);
\draw [->-] (phi) to (eta.north west);
\draw [->-,rounded corners=5] (1,2) -- ++(0,-0.3) -| (eta);
\draw [->-] (midR) to (2,0);
\node at (-0.5,0) {$1$};
\node at (0.3,0.5) {$ 2$};
\node at (-0.45,0.65) {$ 3$};
\node at (1,1) {$ 4$};
\node at (-0.5,-1.7) {$5$};
\node at (-0.5,-0.8) {$6$};
\node at (-1.1,0.3) {$7$};
\node at (-0.3,1.5) {$8$};
\node at (-1.5,-0.6) {$ 9$};
\node at (0.7,-0.5) {$10$};
\node at (0.7,-1.5) {$11$};
\node at (1.5,-0.3) {$12$};
} \quad 
+ \quad \tikzfig{						%
	\draw (-2,-2) rectangle (2,2);
	\node [vertex] (co) at (-0.3,-0.5) {$\co$};
	\node [rectangle,draw,thick] (eta) at (-0.3,1) {$\eta$};
	\node [vertex] (alphaL) at (-1.2,-1) {I};
	\node [vertex] (alphaR) at (1,0) {II};
	\node [bullet] (midL) at (-1.2,0) {};
	\node [bullet] (midR) at (1,1) {};
	\node [vertex] (phi) at (0.3,-1) {$\varphi$};
	\node [bullet] (botR) at (1,-1) {};
	\draw [->-,cyan] (co) to (eta);
	\draw [->-] (co) to (alphaL);
	\draw [->-] (-1.2,2) to (eta.north west);
	\draw [->-,bend right=45] (eta) to (midL);
	\draw [->-] (eta) to (midR);
	\draw [->-] (alphaL) to (midL);
	\draw [-w-] (alphaL) to (-1.2,-2);
	\draw [-w-] (alphaR) to (midR);
	\draw [->-] (phi) to (botR);
	\draw [->-] (alphaR) to (botR);
	\draw [->-] (botR) to (1,-2);
	\draw [->-,rounded corners=5] (1,2) -- ++(0,-0.3) -| (eta);
	\draw [->-] (midR) to (2,1);
	\draw [->-] (midL) to (-2,0);
	\node at (-0.6,-1.2) {$1$};
	\node at (-0.6,0) {$ 2$};
	\node at (-0.8,0.7) {$ 3$};
	\node at (0.5,0.7) {$ 4$};
	\node at (-1.5,-1.5) {$ 5$};
	\node at (-1.5,-0.5) {$ 6$};
	\node at (1.3,-1.5) {$ 7$};
	\node at (0.7,-1.3) {$ 8$};
	\node at (-1.6,0.3) {$ 9$};
	\node at (1.3,0.5) {$10$};
	\node at (1.3,-0.5) {$11$};
	\node at (1.5,1.3) {$12$};
}
\end{align*}

\begin{align*}
	&+ \quad \tikzfig{
	\draw (-2,-2) rectangle (2,2);
	\node [vertex] (co) at (0,0) {$\co$};
	\node [rectangle,draw,thick] (eta) at (0,1) {$\eta$};
	\node [vertex] (alphaL) at (-1,-1.2) {I};
	\node [vertex] (alphaR) at (1,-1) {II};
	\node [bullet] (topL) at (-1,0.5) {};
	\node [bullet] (midL) at (-1,-0.2) {};
	\node [vertex] (phi) at (-0.6,1.6) {$\varphi$};
	\node [bullet] (midR) at (1,0) {};
	\draw [->-,cyan] (co) to (eta);
	\draw [->-] (co) to (topL);
	\draw [->-,bend right=30] (eta) to (topL);
	\draw [->-,bend left=45] (eta) to (midR);
	\draw [->-] (topL) to (midL);
	\draw [->-] (alphaL) to (midL);
	\draw [->-,rounded corners=5] (-1,2) |- (eta.north west);
	\draw [-w-] (alphaL) to (-1,-2);
	\draw [->-] (midL) to (-2,-0.2);
	\draw [-w-] (alphaR) to (midR);
	\draw [->-] (alphaR) to (1,-2);
	\draw [->-,rounded corners=5] (1,2) -- ++(0,-0.3) -| (eta);
	\draw [->-] (midR) to (2,0);
	\draw [->-=1] (phi) to (eta.north west);
	\node at (-0.6,0) {$1$};
	\node at (0.3,0.5) {$2$};
	\node at (-0.6,0.7) {$ 3$};
	\node at (1,1) {$ 4$};
	\node at (-0.7,-1.7) {$5$};
	\node at (-0.7,-0.7) {$6$};
	\node at (-1.3,0.4) {$7$};
	\node at (-0.2,1.5) {$8$};
	\node at (-1.6,-0.5) {$ 9$};
	\node at (0.7,-0.5) {$10$};
	\node at (0.7,-1.5) {$11$};
	\node at (1.5,-0.3) {$12$};} \quad 
	- \quad \tikzfig{
		\draw (-2,-2) rectangle (2,2);
	\node [vertex] (co) at (-0.3,-0.5) {$\co$};
	\node [rectangle,draw,thick] (eta) at (-0.3,1) {$\eta$};
	\node [vertex] (alphaL) at (-1.2,-1) {I};
	\node [bullet] (midL) at (-1.2,0) {};
	\node [vertex] (alphaR) at (1.2,-1) {II};
	\node [bullet] (topR) at (1.2,1) {};
	\node [vertex] (phi) at (0.6,1) {$\varphi$};
	\node [bullet] (midR) at (1.2,0) {};
	\draw [->-,cyan] (co) to (eta);
	\draw [->-] (co) to (alphaL);
	\draw [->-,bend right=45] (eta) to (midL);
	\draw [->-=1] (eta) to (phi);
	\draw [->-] (alphaL) to (midL);
	\draw [-w-] (alphaL) to (-1.2,-2);
	\draw [-w-] (alphaR) to (midR);
	\draw [->-] (phi) to (topR);
	\draw [->-] (alphaR) to (midR);
	\draw [->-] (alphaR) to (1.2,-2);
	\draw [->-] (1.2,2) to (topR);
	\draw [->-,rounded corners=5] (-1.2,2) |- ++(0,-0.5) -| (eta);
	\draw [->-] (midR) to (2,0);
	\draw [->-] (topR) to (midR);
	\draw [->-] (midL) to (-2,0);
	\node at (-0.6,-1.2) {$1$};
	\node at (-0.6,0) {$ 2$};
	\node at (-0.8,0.7) {$ 3$};
	\node at (1.5,0.4) {$ 4$};
	\node at (-1.5,-1.5) {$ 5$};
	\node at (-1.5,-0.5) {$ 6$};
	\node at (0.2,0.6) {$7$};
	\node at (0.8,0.6) {$8$};
	\node at (-1.6,0.3) {$ 9$};
	\node at (0.9,-0.5) {$10$};
	\node at (0.9,-1.5) {$11$};
	\node at (1.7,-0.3) {$12$};} \quad 
	- \quad \tikzfig{
	\draw (-2,-2) rectangle (2,2);
	\node [vertex] (co) at (-0.3,-0.5) {$\co$};
	\node [rectangle,draw,thick] (eta) at (-0.3,1) {$\eta$};
	\node [vertex] (alphaL) at (-1.2,-1) {I};
	\node [vertex] (alphaR) at (1,0) {II};
	\node [bullet] (midL) at (-1.2,0) {};
	\node [bullet] (midR) at (1,1) {};
	\node [vertex] (phi) at (0.3,-1) {$\varphi$};
	\node [bullet] (botR) at (1,-1) {};
	\draw [->-,cyan] (co) to (eta);
	\draw [->-] (co) to (alphaL);
	\draw [->-] (1.2,2) to (eta.north east);
	\draw [->-,bend right=45] (eta) to (midL);
	\draw [->-] (eta) to (midR);
	\draw [->-] (alphaL) to (midL);
	\draw [-w-] (alphaL) to (-1.2,-2);
	\draw [-w-] (alphaR) to (midR);
	\draw [->-] (phi) to (botR);
	\draw [->-] (alphaR) to (botR);
	\draw [->-] (botR) to (1,-2);
	\draw [->-,rounded corners=5] (-1.2,2) -- ++(0,-0.5) -| (eta);
	\draw [->-] (midR) to (2,1);
	\draw [->-] (midL) to (-2,0);
	\node at (-0.6,-1.2) {$1$};
	\node at (-0.6,0) {$ 2$};
	\node at (-0.8,0.7) {$ 3$};
	\node at (0.5,0.7) {$ 4$};
	\node at (-1.5,-1.5) {$ 5$};
	\node at (-1.5,-0.5) {$ 6$};
	\node at (1.3,-1.5) {$ 7$};
	\node at (0.7,-1.3) {$ 8$};
	\node at (-1.6,0.3) {$ 9$};
	\node at (1.3,0.5) {$10$};
	\node at (1.3,-0.5) {$11$};
	\node at (1.5,1.3) {$12$};
	}\\
	&+ \quad \tikzfig{
	\draw (-2,-2) rectangle (2,2);
	\node [vertex] (co) at (0,0) {$\co$};
	\node [rectangle,draw,thick] (eta) at (0,1) {$\eta$};
	\node [vertex] (alphaL) at (0,-1.2) {I};
	\node [vertex] (alphaR) at (1,-1) {II};
	\node [bullet] (topL) at (-0.8,0.5) {};
	\node [bullet] (midL) at (-0.8,-0.3) {};
	\node [vertex] (phi) at (-1.4,-0.3) {$\varphi$};
	\node [bullet] (midR) at (1,0) {};
	\draw [->-,cyan] (co) to (eta);
	\draw [->-] (co) to (topL);
	\draw [->-,bend right=30] (eta) to (topL);
	\draw [->-,bend left=45] (eta) to (midR);
	\draw [->-] (alphaL) to (midL);
	\draw [->-=.9] (midL) to (phi);
	\draw [->-] (topL) to (midL);
	\draw [-w-=.9] (alphaL) to (0,-2);
	\draw [->-=1] (phi) to (-2,-0.3);
	\draw [-w-] (alphaR) to (midR);
	\draw [->-] (alphaR) to (alphaL);
	\draw [->-] (0,2) to (eta);
	\draw [->-] (midR) to (2,0);
	\node at (-0.5,0) {$1$};
	\node at (0.3,0.5) {$ 2$};
	\node at (-0.45,0.65) {$ 3$};
	\node at (1,1) {$ 4$};
	\node at (-0.3,-1.7) {$5$};
	\node at (-0.7,-1) {$6$};
	\node at (-1.1,0.3) {$7$};
	\node at (-1.1,-0.6) {$8$};
	\node at (-1.8,-0.6) {$ 9$};
	\node at (0.7,-0.5) {$10$};
	\node at (0.5,-1.3) {$11$};
	\node at (1.5,-0.3) {$12$};} \quad 
	+ \quad \tikzfig{
	\draw (-2,-2) rectangle (2,2);
	\node [vertex] (co) at (0,0) {$\co$};
	\node [rectangle,draw,thick] (eta) at (0,1) {$\eta$};
	\node [vertex] (alphaL) at (-0.8,-1.2) {I};
	\node [vertex] (alphaR) at (1,-1) {II};
	\node [bullet] (topL) at (-0.8,0.5) {};
	\node [bullet] (midL) at (-0.8,-0.3) {};
	\node [vertex] (phi) at (-1.4,1.3) {$\varphi$};
	\node [bullet] (midR) at (1,0) {};
	\draw [->-,cyan] (co) to (eta);
	\draw [->-] (co) to (topL);
	\draw [->-,bend right=30] (eta) to (topL);
	\draw [->-,bend left=45] (eta) to (midR);
	\draw [->-] (alphaL) to (midL);
	\draw [->-=.9] (midL) to (-2,-0.3);
	\draw [->-] (topL) to (midL);
	\draw [-w-=.9] (alphaL) to (-0.8,-2);
	\draw [->-=1] (phi) to (eta.north west);
	\draw [-w-] (alphaR) to (midR);
	\draw [->-] (alphaR) to (1,-2);
	\draw [->-] (-0.8,2) to (eta.north west);
	\draw [->-,rounded corners=5] (1,2) -- ++(0,-0.3) -| (eta);
	\draw [->-] (midR) to (2,0);
	\node at (-0.5,0) {$1$};
	\node at (0.3,0.5) {$ 2$};
	\node at (-0.45,0.65) {$ 3$};
	\node at (1,1) {$ 4$};
	\node at (-0.5,-1.7) {$5$};
	\node at (-0.5,-0.8) {$6$};
	\node at (-1.1,0.3) {$7$};
	\node at (-0.8,1.5) {$8$};
	\node at (-1.4,-0.6) {$49$};
	\node at (0.7,-0.5) {$10$};
	\node at (0.7,-1.5) {$11$};
	\node at (1.5,-0.3) {$12$};	
	} \quad 
	+ \quad \tikzfig{
	\draw (-2,-2) rectangle (2,2);
	\node [vertex] (co) at (-0.3,-0.5) {$\co$};
	\node [rectangle,draw,thick] (eta) at (-0.3,1) {$\eta$};
	\node [vertex] (alphaL) at (-1.2,-1) {I};
	\node [bullet] (midL) at (-1.2,0) {};
	\node [bullet] (midR) at (1.3,0) {};
	\node [bullet] (topR) at (0.8,0.7) {};
	\node [vertex] (alphaR) at (1.3,-1) {II};
	\node [vertex] (phi) at (0.3,-0.2) {$\varphi$};
	\draw [->-,cyan] (co) to (eta);
	\draw [->-] (co) to (-0.3,-2);
	\draw [->-] (-0.3,2) to (eta);
	\draw [->-,bend right=45] (eta) to (midL);
	\draw [->-,bend left=30] (eta) to (topR);
	\draw [->-] (alphaL) to (midL);
	\draw [->-] (-1.2,2) to (eta.north west);
	\draw [-w-] (alphaL) to (-1.2,-2);
	\draw [-w-] (alphaR) to (midR);
	\draw [->-] (phi) to (topR);
	\draw [->-] (topR) to (midR);
	\draw [->-] (alphaR) to (1.3,-2);
	\draw [->-] (1,2) to (eta.north east);
	\draw [->-] (midR) to (2,0);
	\draw [->-] (midL) to (-2,0);
	\node at (-0.6,-1.2) {$1$};
	\node at (-0.6,0) {$ 2$};
	\node at (-1.4,0.7) {$ 3$};
	\node at (0.5,1.2) {$ 4$};
	\node at (-1.5,-1.5) {$ 5$};
	\node at (-1.5,-0.5) {$ 6$};
	\node at (1.2,0.5) {$ 7$};
	\node at (0.3,0.4) {$ 8$};
	\node at (-1.6,0.3) {$ 9$};
	\node at (0.9,-0.5) {$10$};
	\node at (0.9,-1.5) {$11$};
	\node at (1.7,0.3) {$12$};
	}\\
	&+ \quad \tikzfig{
	\draw (-2,-2) rectangle (2,2);
	\node [vertex] (co) at (0,0) {$\co$};
	\node [rectangle,draw,thick] (eta) at (0,1) {$\eta$};
	\node [vertex] (alphaL) at (-1,-1) {I};
	\node [vertex] (alphaR) at (1,-1) {II};
	\node [bullet] (topL) at (-0.9,0.5) {};
	\node [bullet] (midL) at (-1,0) {};
	\node [vertex] (phi) at (0,-1.2) {$\varphi$};
	\node [bullet] (midR) at (1,0) {};
	\draw [->-,cyan] (co) to (eta);
	\draw [->-] (co) to (topL);
	\draw [->-,bend right=30] (eta) to (topL);
	\draw [->-,bend left=45] (eta) to (midR);
	\draw [->-] (alphaL) to (midL);
	\draw [->-=.9] (midL) to (-2,0);
	\draw [->-] (topL) to (midL);
	\draw [->-=.9] (alphaL) to (phi);
	\draw [->-=1] (phi) to (0,-2);
	\draw [-w-] (alphaR) to (midR);
	\draw [->-] (alphaR) to (phi);
	\draw [->-] (0,2) to (eta);
	\draw [->-] (midR) to (2,0);
	\node at (-0.5,0) {$1$};
	\node at (0.3,0.5) {$ 2$};
	\node at (-0.45,1.2) {$ 3$};
	\node at (1,1) {$ 4$};
	\node at (-0.5,-1.4) {$5$};
	\node at (-0.7,-0.5) {$6$};
	\node at (-1.2,0.3) {$7$};
	\node at (-0.3,-1.7) {$8$};
	\node at (-1.5,-0.3) {$ 9$};
	\node at (0.7,-0.5) {$10$};
	\node at (0.5,-1.4) {$11$};
	\node at (1.5,-0.3) {$12$};
} \quad 
	+ \quad \tikzfig{
	\draw (-2,-2) rectangle (2,2);
	\node [vertex] (co) at (-0,-0.5) {$\co$};
	\node [rectangle,draw,thick] (eta) at (-0,1) {$\eta$};
	\node [vertex] (alphaL) at (-1.2,-1.3) {I};
	\node [bullet] (midL) at (-1.2,0.3) {};
	\node [bullet] (botL) at (-1.2,-0.5) {};
	\node [vertex] (alphaR) at (1.2,-1) {II};
	\node [vertex] (phi) at (1.2,1) {$\varphi$};
	\node [bullet] (midR) at (1.2,0) {};
	\draw [->-,cyan] (co) to (eta);
	\draw [->-] (co) to (botL);
	\draw [->-,bend right=45] (eta) to (midL);
	\draw [->-] (eta) to (phi);
	\draw [->-] (alphaL) to (botL);
	\draw [->-] (botL) to (midL);
	\draw [-w-] (alphaL) to (-1.2,-2);
	\draw [-w-] (alphaR) to (midR);
	\draw [->-] (alphaR) to (1.2,-2);
	\draw [->-] (1.2,2) to (phi);
	\draw [->-,rounded corners=5] (-1.2,2) |- ++(0,-0.5) -| (eta);
	\draw [->-] (midR) to (2,0);
	\draw [->-] (phi) to (midR);
	\draw [->-] (midL) to (-2,0.3);
	\node at (-0.5,-0.8) {$1$};
	\node at (0.3,0.2) {$ 2$};
	\node at (-0.8,0.7) {$ 3$};
	\node at (0.5,1.3) {$ 4$};
	\node at (-1.5,-1.7) {$ 5$};
	\node at (-1.5,-0.8) {$ 6$};
	\node at (-1.5,0) {$7$};
	\node at (0.9,0.5) {$8$};
	\node at (-1.6,0.6) {$ 9$};
	\node at (0.9,-0.5) {$10$};
	\node at (0.9,-1.5) {$11$};
	\node at (1.7,-0.3) {$12$};
} \quad 
	+ \quad \tikzfig{
	\draw (-2,-2) rectangle (2,2);
	\node [vertex] (co) at (0,0) {$\co$};
	\node [rectangle,draw,thick] (eta) at (0,1) {$\eta$};
	\node [vertex] (alphaL) at (0,-1.4) {I};
	\node [vertex] (alphaR) at (1,0) {II};
	\node [bullet] (topL) at (-0.8,0.5) {};
	\node [bullet] (midL) at (-0.8,-0.3) {};
	\node [vertex] (phi) at (-1.4,-0.3) {$\varphi$};
	\node [bullet] (midC) at (0,-0.6) {};
	\draw [->-,cyan] (co) to (eta);
	\draw [->-] (co) to (topL);
	\draw [->-,bend right=30] (eta) to (topL);
	\draw [->-,bend left=45] (eta) to (alphaR);
	\draw [->-] (midC) to (midL);
	\draw [->-=.9] (midL) to (phi);
	\draw [->-] (topL) to (midL);
	\draw [-w-=.9] (alphaL) to (0,-2);
	\draw [->-=1] (phi) to (-2,-0.3);
	\draw [->-] (alphaR) to (midC);
	\draw [->-] (alphaL) to (midC);
	\draw [->-] (0,2) to (eta);
	\draw [-w-] (alphaR) to (2,0);
	\node at (-0.5,0) {$1$};
	\node at (0.3,0.5) {$ 2$};
	\node at (-0.45,0.65) {$ 3$};
	\node at (1,1) {$ 4$};
	\node at (-0.3,1.7) {$5$};
	\node at (-0.5,-0.7) {$6$};
	\node at (-1.1,0.3) {$7$};
	\node at (-1.1,-0.6) {$8$};
	\node at (-1.8,-0.6) {$ 9$};
	\node at (-0.3,-1) {$10$};
	\node at (0.5,-0.6) {$11$};
	\node at (1.5,-0.3) {$12$};}\\
	&+ \quad \tikzfig{
	\draw (-2,-2) rectangle (2,2);
	\node [vertex] (co) at (-0.3,-0.5) {$\co$};
	\node [rectangle,draw,thick] (eta) at (-0.3,1) {$\eta$};
	\node [vertex] (alphaL) at (-1.2,-1.2) {I};
	\node [vertex] (alphaR) at (1,0) {II};
	\node [bullet] (topL) at (-1.2,0.3) {};
	\node [bullet] (midL) at (-1.2,-0.5) {};
	\node [bullet] (midR) at (1,1) {};
	\node [vertex] (phi) at (1,-1) {$\varphi$};
	\draw [->-,cyan] (co) to (eta);
	\draw [->-] (co) to (midL);
	\draw [->-] (1.2,2) to (eta.north east);
	\draw [->-,bend right=45] (eta) to (topL);
	\draw [->-] (topL) to (midL);
	\draw [->-] (eta) to (midR);
	\draw [->-] (alphaL) to (midL);
	\draw [-w-] (alphaL) to (-1.2,-2);
	\draw [-w-] (alphaR) to (midR);
	\draw [->-] (alphaR) to (phi);
	\draw [->-] (botR) to (1,-2);
	\draw [->-,rounded corners=5] (-1.2,2) -- ++(0,-0.5) -| (eta);
	\draw [->-] (midR) to (2,1);
	\draw [->-] (topL) to (-2,0.3);
	\node at (-0.6,-1.2) {$1$};
	\node at (-0.6,0) {$ 2$};
	\node at (-0.8,0.7) {$ 3$};
	\node at (0.5,0.7) {$ 4$};
	\node at (-1.5,-1.5) {$ 5$};
	\node at (-1.5,-0.8) {$ 6$};
	\node at (1.3,-1.5) {$ 7$};
	\node at (-1.5,-0.2) {$ 8$};
	\node at (-1.6,0.6) {$ 9$};
	\node at (1.3,0.5) {$10$};
	\node at (1.3,-0.5) {$11$};
	\node at (1.5,1.3) {$12$};
	} \quad 
	+ \quad \tikzfig{
	\draw (-2,-2) rectangle (2,2);
	\node [vertex] (co) at (0,0) {$\co$};
	\node [rectangle,draw,thick] (eta) at (0,1) {$\eta$};
	\node [vertex] (alphaL) at (0,-1.2) {I};
	\node [vertex] (alphaR) at (1,-1) {II};
	\node [bullet] (topL) at (-1,0.5) {};
	\node [bullet] (midL) at (-1,-0.3) {};
	\node [vertex] (phi) at (-1.2,1.5) {$\varphi$};
	\node [bullet] (midR) at (1,0) {};
	\draw [->-,cyan] (co) to (eta);
	\draw [->-] (co) to (topL);
	\draw [->-,bend right=30] (eta) to (topL);
	\draw [->-,bend left=45] (eta) to (midR);
	\draw [->-] (alphaL) to (midL);
	\draw [->-=.9] (midL) to (-2,-0.3);
	\draw [->-] (topL) to (midL);
	\draw [-w-=.9] (alphaL) to (0,-2);
	\draw [->-=1] (phi) to (eta.north west);
	\draw [-w-] (alphaR) to (midR);
	\draw [->-] (alphaR) to (alphaL);
	\draw [->-] (0,2) to (eta);
	\draw [->-] (midR) to (2,0);
	\node at (-0.5,0) {$1$};
	\node at (0.3,0.5) {$ 2$};
	\node at (-0.45,0.65) {$ 3$};
	\node at (1,1) {$ 4$};
	\node at (-0.3,-1.7) {$5$};
	\node at (-0.7,-1) {$6$};
	\node at (-1.3,0) {$7$};
	\node at (-0.5,1.6) {$8$};
	\node at (-1.5,-0.6) {$ 9$};
	\node at (0.7,-0.5) {$10$};
	\node at (0.5,-1.3) {$11$};
	\node at (1.5,-0.3) {$12$};
	} \quad 
	+ \quad \tikzfig{
	\draw (-2,-2) rectangle (2,2);
	\node [vertex] (co) at (0,0) {$\co$};
	\node [rectangle,draw,thick] (eta) at (0,1) {$\eta$};
	\node [vertex] (alphaL) at (0,-0.7) {I};
	\node [vertex] (alphaR) at (1,-0.5) {II};
	\node [bullet] (topL) at (-1,0.5) {};
	\node [bullet] (midL) at (-1,-0.3) {};
	\node [vertex] (phi) at (0,-1.5) {$\varphi$};
	\node [bullet] (midR) at (1,0.3) {};
	\draw [->-,cyan] (co) to (eta);
	\draw [->-] (co) to (topL);
	\draw [->-,bend right=30] (eta) to (topL);
	\draw [->-,bend left=45] (eta) to (midR);
	\draw [->-] (alphaL) to (midL);
	\draw [->-=.9] (midL) to (-2,-0.3);
	\draw [->-] (topL) to (midL);
	\draw [-w-=.9] (alphaL) to (phi);
	\draw [->-=1] (phi) to (0,-2);
	\draw [-w-] (alphaR) to (midR);
	\draw [->-] (alphaR) to (alphaL);
	\draw [->-] (0,2) to (eta);
	\draw [->-] (midR) to (2,0.3);
	\node at (-0.5,0) {$1$};
	\node at (0.3,0.5) {$ 2$};
	\node at (-0.45,0.65) {$ 3$};
	\node at (1,1) {$ 4$};
	\node at (-0.3,-1) {$5$};
	\node at (-0.7,-0.8) {$6$};
	\node at (-1.3,0) {$7$};
	\node at (-0.5,1.6) {$8$};
	\node at (-1.5,-0.6) {$ 9$};
	\node at (0.7,0) {$10$};
	\node at (0.5,-0.8) {$11$};
	\node at (1.5,0) {$12$};
	}\\
	&+ \quad \tikzfig{
	\draw (-2,-2) rectangle (2,2);
	\node [vertex] (co) at (-0.3,-0.5) {$\co$};
	\node [rectangle,draw,thick] (eta) at (-0.3,1) {$\eta$};
	\node [vertex] (alphaL) at (-1.2,-1) {I};
	\node [bullet] (midL) at (-1.2,0) {};
	\node [bullet] (midR) at (1.3,0) {};
	\node [bullet] (topR) at (0.8,0.7) {};
	\node [vertex] (alphaR) at (1.3,-1) {II};
	\node [vertex] (phi) at (0.3,-0.2) {$\varphi$};
	\draw [->-,cyan] (co) to (eta);
	\draw [->-] (co) to (alphaL);
	\draw [->-,rounded corners=5] (-1.2,2) |- ++(0,-0.5) -| (eta);
	\draw [->-,bend right=45] (eta) to (midL);
	\draw [->-,bend left=30] (eta) to (topR);
	\draw [->-] (alphaL) to (midL);
	\draw [-w-] (alphaL) to (-1.2,-2);
	\draw [-w-] (alphaR) to (midR);
	\draw [->-] (phi) to (topR);
	\draw [->-] (topR) to (midR);
	\draw [->-] (alphaR) to (1.3,-2);
	\draw [->-] (1,2) to (eta.north east);
	\draw [->-] (midR) to (2,0);
	\draw [->-] (midL) to (-2,0);
	\node at (-0.6,-1.2) {$1$};
	\node at (-0.6,0) {$ 2$};
	\node at (-1.4,0.7) {$ 3$};
	\node at (0.5,1.2) {$ 4$};
	\node at (-1.5,-1.5) {$ 5$};
	\node at (-1.5,-0.5) {$ 6$};
	\node at (1.2,0.5) {$ 7$};
	\node at (0.3,0.4) {$ 8$};
	\node at (-1.6,0.3) {$ 9$};
	\node at (0.9,-0.5) {$10$};
	\node at (0.9,-1.5) {$11$};
	\node at (1.7,0.3) {$12$};
	} \quad 
	+ \quad \tikzfig{
	\draw (-2,-2) rectangle (2,2);
	\node [vertex] (co) at (0,0) {$\co$};
	\node [rectangle,draw,thick] (eta) at (0,1) {$\eta$};
	\node [vertex] (alphaL) at (0,-1.4) {I};
	\node [vertex] (alphaR) at (1,0) {II};
	\node [bullet] (topL) at (-1,0.5) {};
	\node [bullet] (midL) at (-1,-0.3) {};
	\node [vertex] (phi) at (-1.2,1.5) {$\varphi$};
	\node [bullet] (midC) at (0,-0.6) {};
	\draw [->-,cyan] (co) to (eta);
	\draw [->-] (co) to (topL);
	\draw [->-,bend right=30] (eta) to (topL);
	\draw [->-,bend left=45] (eta) to (alphaR);
	\draw [->-] (alphaL) to (midC);
	\draw [->-] (midC) to (midL);
	\draw [->-=.9] (midL) to (-2,-0.3);
	\draw [->-] (topL) to (midL);
	\draw [-w-=.9] (alphaL) to (0,-2);
	\draw [->-=1] (phi) to (eta.north west);
	\draw [->-] (alphaR) to (midC);
	\draw [->-] (0,2) to (eta);
	\draw [-w-] (alphaR) to (2,0);
	\node at (-0.5,0) {$1$};
	\node at (0.3,0.5) {$ 2$};
	\node at (-0.45,0.65) {$ 3$};
	\node at (1,1) {$ 4$};
	\node at (-0.3,-1.7) {$5$};
	\node at (-0.7,-0.6) {$6$};
	\node at (-1.3,0) {$7$};
	\node at (-0.5,1.6) {$8$};
	\node at (-1.5,-0.6) {$ 9$};
	\node at (-0.3,-1) {$10$};
	\node at (0.6,-0.5) {$11$};
	\node at (1.5,-0.3) {$12$};
	} \quad 
	- \quad \tikzfig{
	\draw (-2,-2) rectangle (2,2);
	\node [vertex] (co) at (0,-0.5) {$\co$};
	\node [rectangle,draw,thick] (eta) at (0,1) {$\eta$};
	\node [vertex] (alphaL) at (-1,-1.2) {I};
	\node [vertex] (alphaR) at (1,-1) {II};
	\node [bullet] (topL) at (-1,1) {};
	\node [bullet] (midL) at (-1,-0.2) {};
	\node [vertex] (phi) at (1,0) {$\varphi$};
	\node [bullet] (topR) at (1,1) {};
	\draw [->-,cyan] (co) to (eta);
	\draw [->-] (co) to (0,-2);
	\draw [->-] (eta) to (topL);
	\draw [->-] (eta) to (topR);
	\draw [->-] (topL) to (midL);
	\draw [->-] (alphaL) to (midL);
	\draw [->-] (-1,2) to (topL);
	\draw [-w-] (alphaL) to (-1,-2);
	\draw [->-] (midL) to (-2,-0.2);
	\draw [-w-] (alphaR) to (phi);
	\draw [->-] (phi) to (topR);
	\draw [->-] (alphaR) to (1,-2);
	\draw [->-] (1,2) to (eta.north east);
	\draw [->-] (topR) to (2,1);
	\draw [->-] (0,2) to (eta);
	\node at (0.3,-1.4) {$1$};
	\node at (0.3,0) {$2$};
	\node at (-0.6,0.7) {$3$};
	\node at (0.6,0.7) {$4$};
	\node at (-0.7,-1.7) {$5$};
	\node at (-0.7,-0.7) {$6$};
	\node at (-1.3,0.4) {$7$};
	\node at (1.3,0.5) {$8$};
	\node at (-1.6,-0.5) {$ 9$};
	\node at (0.7,-0.5) {$10$};
	\node at (0.7,-1.5) {$11$};
	\node at (1.5,1.3) {$12$};
	}\\
	&- \quad \tikzfig{
	\draw (-2,-2) rectangle (2,2);
	\node [vertex] (co) at (-0.3,-0.5) {$\co$};
	\node [rectangle,draw,thick] (eta) at (-0.3,1) {$\eta$};
	\node [vertex] (alphaL) at (-1.2,-1) {I};
	\node [bullet] (midL) at (-1.2,0) {};
	\node [bullet] (midR) at (1.3,0) {};
	\node [bullet] (topR) at (0.8,0.7) {};
	\node [vertex] (alphaR) at (1.3,-1) {II};
	\node [vertex] (phi) at (0.3,-0.2) {$\varphi$};
	\draw [->-,cyan] (co) to (eta);
	\draw [->-] (co) to (-0.3,-2);
	\draw [->-,bend right=45] (-0.3,2) to (eta.north west);
	\draw [->-,bend right=45] (eta) to (midL);
	\draw [->-,bend left=30] (eta) to (topR);
	\draw [->-] (alphaL) to (midL);
	\draw [->-,bend right=45] (-1.2,2) to (eta.north west);
	\draw [-w-] (alphaL) to (-1.2,-2);
	\draw [-w-] (alphaR) to (midR);
	\draw [->-] (phi) to (topR);
	\draw [->-] (topR) to (midR);
	\draw [->-] (alphaR) to (1.3,-2);
	\draw [->-,rounded corners=5] (1.3,2) |- ++(0,-0.5) -| (eta);
	\draw [->-] (midR) to (2,0);
	\draw [->-] (midL) to (-2,0);
	\node at (-0.6,-1.2) {$1$};
	\node at (-0.6,0) {$ 2$};
	\node at (-1.4,0.7) {$ 3$};
	\node at (0.5,1.2) {$ 4$};
	\node at (-1.5,-1.5) {$ 5$};
	\node at (-1.5,-0.5) {$ 6$};
	\node at (1.2,0.5) {$ 7$};
	\node at (0.3,0.4) {$ 8$};
	\node at (-1.6,0.3) {$ 9$};
	\node at (0.9,-0.5) {$10$};
	\node at (0.9,-1.5) {$11$};
	\node at (1.7,0.3) {$12$};} \quad 
	+ \quad \tikzfig{
	\draw (-2,-2) rectangle (2,2);
	\node [vertex] (co) at (-0.3,-0.5) {$\co$};
	\node [rectangle,draw,thick] (eta) at (-0.3,1) {$\eta$};
	\node [vertex] (alphaL) at (-1.2,-1.2) {I};
	\node [bullet] (topL) at (-1.2,0.3) {};
	\node [bullet] (midL) at (-1.2,-0.5) {};
	\node [bullet] (midR) at (1.3,0) {};
	\node [vertex] (alphaR) at (1.3,-1) {II};
	\node [vertex] (phi) at (0.8,0.7) {$\varphi$};
	\draw [->-,cyan] (co) to (eta);
	\draw [->-] (co) to (midL);
	\draw [->-] (1.2,2) to (eta.north east);
	\draw [->-,bend right=45] (eta) to (topL);
	\draw [->-] (topL) to (midL);
	\draw [->-,bend left=10] (eta) to (phi);
	\draw [->-] (alphaL) to (midL);
	\draw [->-] (phi) to (midR);
	\draw [-w-] (alphaL) to (-1.2,-2);
	\draw [-w-] (alphaR) to (midR);
	\draw [->-] (alphaR) to (1.3,-2);
	\draw [->-,rounded corners=5] (-1.2,2) -- ++(0,-0.5) -| (eta);
	\draw [->-] (midR) to (2,0);
	\draw [->-] (topL) to (-2,0.3);
	\node at (-0.6,-1.2) {$1$};
	\node at (-0.6,0) {$ 2$};
	\node at (-0.8,0.7) {$ 3$};
	\node at (0.2,0.7) {$ 4$};
	\node at (-1.5,-1.5) {$ 5$};
	\node at (-1.5,-0.8) {$ 6$};
	\node at (1,-1.5) {$11$};
	\node at (-1.5,-0.2) {$7$};
	\node at (-1.6,0.6) {$ 9$};
	\node at (1.3,0.5) {$8$};
	\node at (1,-0.5) {$10$};
	\node at (1.6,-0.3) {$12$};
	} \quad 
	+ \quad \tikzfig{
	\draw (-2,-2) rectangle (2,2);
	\node [vertex] (co) at (0,-0.1) {$\co$};
	\node [rectangle,draw,thick] (eta) at (0,0.8) {$\eta$};
	\node [vertex] (alphaL) at (0,-1.4) {I};
	\node [vertex] (alphaR) at (1,0) {II};
	\node [bullet] (topL) at (-1,0.5) {};
	\node [bullet] (midL) at (-1,-0.3) {};
	\node [vertex] (phi) at (0,1.6) {$\varphi$};
	\node [bullet] (midC) at (0,-0.6) {};
	\draw [->-,cyan] (co) to (eta);
	\draw [->-] (co) to (topL);
	\draw [->-,bend right=30] (eta) to (topL);
	\draw [->-,bend left=45] (eta) to (alphaR);
	\draw [->-] (alphaL) to (midC);
	\draw [->-] (midC) to (midL);
	\draw [->-=.9] (midL) to (-2,-0.3);
	\draw [->-] (topL) to (midL);
	\draw [-w-=.9] (alphaL) to (0,-2);
	\draw [->-] (alphaR) to (midC);
	\draw [->-] (0,2) to (phi);
	\draw [->-] (phi) to (eta);
	\draw [-w-] (alphaR) to (2,0);
	\node at (-0.5,0) {$1$};
	\node at (0.3,0.4) {$ 2$};
	\node at (-0.45,0.65) {$ 3$};
	\node at (1,1) {$ 4$};
	\node at (-0.3,-1.7) {$5$};
	\node at (-0.7,-0.6) {$6$};
	\node at (-1.3,0) {$7$};
	\node at (-0.3,1.2) {$8$};
	\node at (-1.5,-0.6) {$ 9$};
	\node at (-0.3,-1) {$10$};
	\node at (0.6,-0.5) {$11$};
	\node at (1.5,-0.3) {$12$};}\\
\end{align*}

\begin{align*}
	&- \quad \tikzfig{
	\draw (-2,-2) rectangle (2,2);
	\node [vertex] (co) at (-0,-0.5) {$\co$};
	\node [rectangle,draw,thick] (eta) at (-0,1) {$\eta$};
	\node [vertex] (alphaL) at (-1.2,-1.3) {I};
	\node [bullet] (midL) at (-1.2,0.3) {};
	\node [bullet] (botL) at (-1.2,-0.5) {};
	\node [vertex] (alphaR) at (1.2,-1) {II};
	\node [vertex] (phi) at (1.5,1) {$\varphi$};
	\node [bullet] (midR) at (1.2,0) {};
	\draw [->-,cyan] (co) to (eta);
	\draw [->-] (co) to (botL);
	\draw [->-,bend right=45] (eta) to (midL);
	\draw [->-,bend left=45] (eta) to (midR);
	\draw [->-] (alphaL) to (botL);
	\draw [->-] (botL) to (midL);
	\draw [-w-] (alphaL) to (-1.2,-2);
	\draw [-w-] (alphaR) to (midR);
	\draw [->-] (alphaR) to (midR);
	\draw [->-] (alphaR) to (1.2,-2);
	\draw [->-] (1.2,2) to (eta.north east);
	\draw [->-,rounded corners=5] (-1.2,2) |- ++(0,-0.5) -| (eta);
	\draw [->-] (midR) to (2,0);
	\draw [->-] (phi) to (eta.north east);
	\draw [->-] (midL) to (-2,0.3);
	\node at (-0.5,-0.8) {$1$};
	\node at (0.3,0.2) {$ 2$};
	\node at (-0.8,0.7) {$ 3$};
	\node at (0.5,1.6) {$4$};
	\node at (-1.5,-1.7) {$ 5$};
	\node at (-1.5,-0.8) {$ 6$};
	\node at (-1.5,0) {$7$};
	\node at (0.9,1.3) {$8$};
	\node at (-1.6,0.6) {$ 9$};
	\node at (0.9,-0.5) {$10$};
	\node at (0.9,-1.5) {$11$};
	\node at (1.7,-0.3) {$12$};
	} \quad
	+ \quad \tikzfig{
	\draw (-2,-2) rectangle (2,2);
	\node [vertex] (co) at (0,-0.5) {$\co$};
	\node [rectangle,draw,thick] (eta) at (0,1) {$\eta$};
	\node [vertex] (alphaL) at (-1,-1.2) {I};
	\node [vertex] (alphaR) at (1,-1) {II};
	\node [bullet] (topL) at (-1,1) {};
	\node [bullet] (midL) at (-1,-0.2) {};
	\node [vertex] (phi) at (1,0) {$\varphi$};
	\node [bullet] (topR) at (1,1) {};
	\draw [->-,cyan] (co) to (eta);
	\draw [->-] (co) to (0,-2);
	\draw [->-] (eta) to (topL);
	\draw [->-] (eta) to (topR);
	\draw [->-] (topL) to (midL);
	\draw [->-] (alphaL) to (midL);
	\draw [->-] (-1,2) to (topL);
	\draw [-w-] (alphaL) to (-1,-2);
	\draw [->-] (midL) to (-2,-0.2);
	\draw [-w-] (alphaR) to (phi);
	\draw [->-] (phi) to (topR);
	\draw [->-] (alphaR) to (1,-2);
	\draw [->-,rounded corners=5] (1,2) |- ++(0,-0.5) -| (eta);
	\draw [->-] (topR) to (2,1);
	\draw [->-,bend right=45] (0,2) to (eta.north west);
	\node at (0.3,-1.4) {$1$};
	\node at (0.3,0) {$2$};
	\node at (-0.6,0.7) {$3$};
	\node at (0.6,0.7) {$4$};
	\node at (-0.7,-1.7) {$5$};
	\node at (-0.7,-0.7) {$6$};
	\node at (-1.3,0.4) {$7$};
	\node at (1.3,0.5) {$8$};
	\node at (-1.6,-0.5) {$ 9$};
	\node at (0.7,-0.5) {$10$};
	\node at (0.7,-1.5) {$11$};
	\node at (1.5,1.3) {$12$};
	} \quad 
	+ \quad \tikzfig{
	\draw (-2,-2) rectangle (2,2);
	\node [vertex] (co) at (0,-0.5) {$\co$};
	\node [rectangle,draw,thick] (eta) at (0,1) {$\eta$};
	\node [vertex] (alphaL) at (-1,-1.2) {I};
	\node [vertex] (alphaR) at (1,-1) {II};
	\node [bullet] (topL) at (-1,1) {};
	\node [bullet] (midL) at (-1,-0.2) {};
	\node [vertex] (phi) at (1,0) {$\varphi$};
	\node [bullet] (topR) at (1,1) {};
	\draw [->-,cyan] (co) to (eta);
	\draw [->-] (co) to (alphaL);
	\draw [->-] (eta) to (topL);
	\draw [->-] (eta) to (topR);
	\draw [->-] (topL) to (midL);
	\draw [->-] (alphaL) to (midL);
	\draw [->-] (-1,2) to (topL);
	\draw [-w-] (alphaL) to (-1,-2);
	\draw [->-] (midL) to (-2,-0.2);
	\draw [-w-] (alphaR) to (phi);
	\draw [->-] (phi) to (topR);
	\draw [->-] (alphaR) to (1,-2);
	\draw [->-,rounded corners=5] (1,2) |- ++(0,-0.5) -| (eta);
	\draw [->-] (topR) to (2,1);
	\node at (-0.5,-1) {$1$};
	\node at (0.3,0) {$2$};
	\node at (-0.6,0.7) {$3$};
	\node at (0.6,0.7) {$4$};
	\node at (-0.7,-1.7) {$5$};
	\node at (-1.3,-0.7) {$6$};
	\node at (-1.3,0.4) {$7$};
	\node at (1.3,0.5) {$8$};
	\node at (-1.6,-0.5) {$ 9$};
	\node at (0.7,-0.5) {$10$};
	\node at (0.7,-1.5) {$11$};
	\node at (1.5,1.3) {$12$};
	}\\
	&- \quad \tikzfig{
	\draw (-2,-2) rectangle (2,2);
	\node [vertex] (co) at (-0.3,-0.5) {$\co$};
	\node [rectangle,draw,thick] (eta) at (-0.3,1) {$\eta$};
	\node [vertex] (alphaL) at (-1.2,-1) {I};
	\node [bullet] (midL) at (-1.2,0) {};
	\node [bullet] (midR) at (1.3,0) {};
	\node [bullet] (topR) at (0.8,0.7) {};
	\node [vertex] (alphaR) at (1.3,-1) {II};
	\node [vertex] (phi) at (0.3,-0.2) {$\varphi$};
	\draw [->-,cyan] (co) to (eta);
	\draw [->-] (co) to (alphaL);
	\draw [->-,bend right=45] (eta) to (midL);
	\draw [->-,bend left=30] (eta) to (topR);
	\draw [->-] (alphaL) to (midL);
	\draw [->-,rounded corners=5] (1.2,2) |- ++(0,-0.5) -| (eta);
	\draw [-w-] (alphaL) to (-1.2,-2);
	\draw [-w-] (alphaR) to (midR);
	\draw [->-] (phi) to (topR);
	\draw [->-] (topR) to (midR);
	\draw [->-] (alphaR) to (1.3,-2);
	\draw [->-] (-1,2) to (eta.north west);
	\draw [->-] (midR) to (2,0);
	\draw [->-] (midL) to (-2,0);
	\node at (-0.6,-1.2) {$1$};
	\node at (-0.6,0) {$ 2$};
	\node at (-1.4,0.7) {$ 3$};
	\node at (0.5,1.2) {$ 4$};
	\node at (-1.5,-1.5) {$ 5$};
	\node at (-1.5,-0.5) {$ 6$};
	\node at (1.2,0.5) {$ 7$};
	\node at (0.3,0.4) {$ 8$};
	\node at (-1.6,0.3) {$ 9$};
	\node at (0.9,-0.5) {$10$};
	\node at (0.9,-1.5) {$11$};
	\node at (1.7,0.3) {$12$};} \quad 
	+ \quad \tikzfig{
	\draw (-2,-2) rectangle (2,2);
	\node [vertex] (co) at (-0,-0.5) {$\co$};
	\node [rectangle,draw,thick] (eta) at (-0,1) {$\eta$};
	\node [vertex] (alphaL) at (-1.2,-1.3) {I};
	\node [bullet] (midL) at (-1.2,0.3) {};
	\node [bullet] (botL) at (-1.2,-0.5) {};
	\node [vertex] (alphaR) at (1.2,-1) {II};
	\node [vertex] (phi) at (1.2,0) {$\varphi$};
	\node [bullet] (topR) at (1.2,1) {};
	\draw [->-,cyan] (co) to (eta);
	\draw [->-] (co) to (botL);
	\draw [->-,bend right=45] (eta) to (midL);
	\draw [->-] (eta) to (topR);
	\draw [->-] (alphaL) to (botL);
	\draw [->-] (botL) to (midL);
	\draw [-w-] (alphaL) to (-1.2,-2);
	\draw [-w-] (alphaR) to (phi);
	\draw [->-] (phi) to (topR);
	\draw [->-] (alphaR) to (1.2,-2);
	\draw [->-] (-1.2,2) to (eta.north west);
	\draw [->-,rounded corners=5] (1,2) |- ++(0,-0.5) -| (eta);
	\draw [->-] (topR) to (2,1);
	\draw [->-] (midL) to (-2,0.3);
	\node at (-0.5,-0.8) {$1$};
	\node at (-0.3,0.2) {$ 2$};
	\node at (-0.8,0.6) {$ 3$};
	\node at (0.6,0.7) {$4$};
	\node at (-1.5,-1.7) {$ 5$};
	\node at (-1.5,-0.8) {$ 6$};
	\node at (-1.5,-0.2) {$7$};
	\node at (1.5,0.5) {$8$};
	\node at (-1.6,0.6) {$ 9$};
	\node at (1.5,-0.5) {$10$};
	\node at (1.5,-1.5) {$11$};
	\node at (1.7,1.3) {$12$};
	} \quad 
	- \quad \tikzfig{
	\draw (-2,-2) rectangle (2,2);
	\node [vertex] (co) at (0,0) {$\co$};
	\node [rectangle,draw,thick] (eta) at (0,1) {$\eta$};
	\node [vertex] (alphaL) at (-1,-1) {I};
	\node [vertex] (alphaR) at (1,-1) {II};
	\node [bullet] (topL) at (-0.9,0.5) {};
	\node [bullet] (midL) at (-1,0) {};
	\node [bullet] (botC) at (0,-1.2) {};
	\node [vertex] (phi) at (1,0) {$\varphi$};
	\draw [->-,cyan] (co) to (eta);
	\draw [->-] (co) to (topL);
	\draw [->-,bend right=30] (eta) to (topL);
	\draw [->-,bend left=45] (eta) to (phi);
	\draw [->-] (alphaL) to (midL);
	\draw [->-=.9] (midL) to (-2,0);
	\draw [->-] (topL) to (midL);
	\draw [-w-=.9] (alphaL) to (botC);
	\draw [->-=1] (botC) to (0,-2);
	\draw [-w-] (alphaR) to (phi);
	\draw [->-] (alphaR) to (botC);
	\draw [->-] (0,2) to (eta);
	\draw [->-] (phi) to (2,0);
	\node at (-0.5,0) {$1$};
	\node at (0.3,0.5) {$ 2$};
	\node at (-0.45,1.2) {$ 3$};
	\node at (1,1) {$ 4$};
	\node at (-0.5,-1.4) {$5$};
	\node at (-0.7,-0.5) {$6$};
	\node at (-1.2,0.3) {$7$};
	\node at (-0.3,-1.7) {$8$};
	\node at (-1.5,-0.3) {$ 9$};
	\node at (0.7,-0.5) {$10$};
	\node at (0.5,-1.4) {$11$};
	\node at (1.5,-0.3) {$12$};
	d}\\
	&+ \quad \tikzfig{
	\draw (-2,-2) rectangle (2,2);
	\node [vertex] (co) at (0,0) {$\co$};
	\node [rectangle,draw,thick] (eta) at (0,1) {$\eta$};
	\node [vertex] (alphaL) at (0,-1.2) {I};
	\node [vertex] (alphaR) at (1,-1) {II};
	\node [bullet] (topL) at (-1,0.5) {};
	\node [bullet] (midL) at (-1,-0.3) {};
	\node [vertex] (phi) at (1.2,1.5) {$\varphi$};
	\node [bullet] (midR) at (1,0) {};
	\draw [->-,cyan] (co) to (eta);
	\draw [->-] (co) to (topL);
	\draw [->-,bend right=30] (eta) to (topL);
	\draw [->-,bend left=45] (eta) to (midR);
	\draw [->-] (alphaL) to (midL);
	\draw [->-=.9] (midL) to (-2,-0.3);
	\draw [->-] (topL) to (midL);
	\draw [-w-=.9] (alphaL) to (0,-2);
	\draw [->-] (phi) to (eta.north east);
	\draw [-w-] (alphaR) to (midR);
	\draw [->-] (alphaR) to (alphaL);
	\draw [->-] (0,2) to (eta);
	\draw [->-] (midR) to (2,0);
	\node at (-0.5,0) {$1$};
	\node at (0.3,0.5) {$ 2$};
	\node at (-0.45,0.65) {$ 3$};
	\node at (1,1) {$ 4$};
	\node at (-0.3,-1.7) {$5$};
	\node at (-0.7,-1) {$6$};
	\node at (-1.3,0) {$7$};
	\node at (0.5,1.6) {$8$};
	\node at (-1.5,-0.6) {$ 9$};
	\node at (0.7,-0.5) {$10$};
	\node at (0.5,-1.3) {$11$};
	\node at (1.5,-0.3) {$12$};
	} \quad 
	+ \quad \tikzfig{
	\draw (-2,-2) rectangle (2,2);
	\node [vertex] (co) at (0,0) {$\co$};
	\node [rectangle,draw,thick] (eta) at (0,1) {$\eta$};
	\node [vertex] (alphaL) at (0,-1.2) {I};
	\node [vertex] (alphaR) at (0.8,-0.8) {II};
	\node [bullet] (topL) at (-1,0.5) {};
	\node [bullet] (midL) at (-1,-0.3) {};
	\node [vertex] (phi) at (1.5,0) {$\varphi$};
	\node [bullet] (midR) at (0.8,0) {};
	\draw [->-,cyan] (co) to (eta);
	\draw [->-] (co) to (topL);
	\draw [->-,bend right=30] (eta) to (topL);
	\draw [->-,bend left=30] (eta) to (midR);
	\draw [->-] (alphaL) to (midL);
	\draw [->-=.9] (midL) to (-2,-0.3);
	\draw [->-] (topL) to (midL);
	\draw [-w-=.9] (alphaL) to (0,-2);
	\draw [-w-] (alphaR) to (midR);
	\draw [->-] (alphaR) to (alphaL);
	\draw [->-] (0,2) to (eta);
	\draw [->-] (midR) to (phi);
	\draw [->-=1] (phi) to (2,0);
	\node at (-0.5,0) {$1$};
	\node at (0.3,0.5) {$ 2$};
	\node at (-0.45,0.65) {$ 3$};
	\node at (0.7,1) {$ 4$};
	\node at (-0.3,-1.7) {$5$};
	\node at (-0.7,-1) {$6$};
	\node at (-1.3,0) {$7$};
	\node at (1.1,-0.3) {$8$};
	\node at (-1.5,-0.6) {$ 9$};
	\node at (0.5,-0.4) {$10$};
	\node at (0.5,-1.3) {$11$};
	\node at (1.7,-0.5) {$12$};
	} \quad 
	- \quad \tikzfig{
	\draw (-2,-2) rectangle (2,2);
	\node [vertex] (co) at (0,0) {$\co$};
	\node [rectangle,draw,thick] (eta) at (0,1) {$\eta$};
	\node [vertex] (alphaL) at (0,-1.4) {I};
	\node [vertex] (alphaR) at (1,0) {II};
	\node [bullet] (topL) at (-1,0.5) {};
	\node [bullet] (midL) at (-1,-0.3) {};
	\node [vertex] (phi) at (1.2,1.5) {$\varphi$};
	\node [bullet] (midC) at (0,-0.6) {};
	\draw [->-,cyan] (co) to (eta);
	\draw [->-] (co) to (topL);
	\draw [->-,bend right=30] (eta) to (topL);
	\draw [->-,bend left=45] (eta) to (alphaR);
	\draw [->-] (alphaL) to (midC);
	\draw [->-] (midC) to (midL);
	\draw [->-=.9] (midL) to (-2,-0.3);
	\draw [->-] (topL) to (midL);
	\draw [-w-=.9] (alphaL) to (0,-2);
	\draw [->-=1] (phi) to (eta.north east);
	\draw [->-] (alphaR) to (midC);
	\draw [->-] (0,2) to (eta);
	\draw [-w-] (alphaR) to (2,0);
	\node at (-0.5,0) {$1$};
	\node at (0.3,0.5) {$ 2$};
	\node at (-0.45,0.65) {$ 3$};
	\node at (1,1) {$ 4$};
	\node at (-0.3,-1.7) {$5$};
	\node at (-0.7,-0.6) {$6$};
	\node at (-1.3,0) {$7$};
	\node at (0.5,1.6) {$8$};
	\node at (-1.5,-0.6) {$ 9$};
	\node at (-0.3,-1) {$10$};
	\node at (0.6,-0.5) {$11$};
	\node at (1.5,-0.3) {$12$};}\\
	&- \quad \tikzfig{
	\draw (-2,-2) rectangle (2,2);
	\node [vertex] (co) at (0,0) {$\co$};
	\node [rectangle,draw,thick] (eta) at (0,1) {$\eta$};
	\node [vertex] (alphaL) at (0,-1.4) {I};
	\node [vertex] (alphaR) at (1,0) {II};
	\node [bullet] (topL) at (-1,0.5) {};
	\node [bullet] (midL) at (-1,-0.3) {};
	\node [vertex] (phi) at (1.5,1.2) {$\varphi$};
	\node [bullet] (midC) at (0,-0.6) {};
	\draw [->-,cyan] (co) to (eta);
	\draw [->-] (co) to (topL);
	\draw [->-,bend right=30] (eta) to (topL);
	\draw [->-,bend left=45] (eta) to (alphaR);
	\draw [->-] (alphaL) to (midC);
	\draw [->-] (midC) to (midL);
	\draw [->-=.9] (midL) to (-2,-0.3);
	\draw [->-] (topL) to (midL);
	\draw [-w-=.9] (alphaL) to (0,-2);
	\draw [->-=1] (phi) to (alphaR);
	\draw [->-] (alphaR) to (midC);
	\draw [->-] (0,2) to (eta);
	\draw [-w-] (alphaR) to (2,0);
	\node at (-0.5,0) {$1$};
	\node at (0.3,0.5) {$ 2$};
	\node at (-0.45,0.65) {$ 3$};
	\node at (0.6,0.6) {$ 4$};
	\node at (-0.3,-1.7) {$5$};
	\node at (-0.7,-0.6) {$6$};
	\node at (-1.3,0) {$7$};
	\node at (1.5,0.6) {$8$};
	\node at (-1.5,-0.6) {$ 9$};
	\node at (-0.3,-1) {$10$};
	\node at (0.6,-0.5) {$11$};
	\node at (1.5,-0.3) {$12$};
	} \quad 
	- \quad \tikzfig{
	\draw (-2,-2) rectangle (2,2);
	\node [vertex] (co) at (-0.2,0) {$\co$};
	\node [rectangle,draw,thick] (eta) at (-0.2,1) {$\eta$};
	\node [vertex] (alphaL) at (-0.2,-1.4) {I};
	\node [vertex] (alphaR) at (0.6,0) {II};
	\node [bullet] (topL) at (-1.2,0.5) {};
	\node [bullet] (midL) at (-1.2,-0.3) {};
	\node [vertex] (phi) at (1.4,0) {$\varphi$};
	\node [bullet] (midC) at (-0.2,-0.6) {};
	\draw [->-,cyan] (co) to (eta);
	\draw [->-] (co) to (topL);
	\draw [->-,bend right=30] (eta) to (topL);
	\draw [->-,bend left=30] (eta) to (alphaR);
	\draw [->-] (alphaL) to (midC);
	\draw [->-] (midC) to (midL);
	\draw [->-=.9] (midL) to (-2,-0.3);
	\draw [->-] (topL) to (midL);
	\draw [-w-=.9] (alphaL) to (-0.2,-2);
	\draw [->-] (alphaR) to (midC);
	\draw [->-] (0,2) to (eta);
	\draw [-w-] (alphaR) to (phi);
	\draw [->-] (phi) to (2,0);
	\node at (-0.7,0) {$1$};
	\node at (0.1,0.5) {$ 2$};
	\node at (-0.65,0.65) {$ 3$};
	\node at (0.6,0.8) {$ 4$};
	\node at (-0.5,-1.7) {$5$};
	\node at (-0.9,-0.6) {$6$};
	\node at (-1.5,0) {$7$};
	\node at (1,-0.3) {$8$};
	\node at (-1.7,-0.6) {$ 9$};
	\node at (-0.5,-1) {$10$};
	\node at (0.4,-0.5) {$11$};
	\node at (1.7,-0.3) {$12$};
	} \quad 
	+ \quad \tikzfig{
	\draw (-2,-2) rectangle (2,2);
	\node [vertex] (alphaL) at (-0,1.2) {I};
	\node [vertex] (alphaR) at (1,-1) {II};
	\node [vertex] (midL) at (0,0) {$\beta$};
	\node [bullet] (botC) at (0,-1) {};
	\node [bullet] (botL) at (-0.6,-1) {};
	\node [vertex] (phi) at (-1.4,-1) {$\varphi$};
	\draw [-w-] (alphaL) to (midL);
	\draw [->-] (midL) to (botC);
	\draw [->-] (alphaL) to (-0.4,2);
	\draw [->-] (-0.4,-2) to (botL);
	\draw [->-] (botL) to (phi);
	\draw [->-] (phi) to (-2,-1);
	\draw [->-] (alphaR) to (botC);
	\draw [->-] (botC) to (botL);
	\draw [-w-] (alphaR) to (2,-1);
	\node at (0.3,0.6) {$2$};
	\node at (-0.4,1.5) {$1$};
	\node at (0.5,-1.3) {$ 3$};
	\node at (1.5,-1.3) {$ 4$};
	\node at (0.3,-0.6) {$ 5$};
	\node at (-0.3,-1.3) {$ 6$};
	\node at (-1,-0.7) {$7$};
	\node at (-1.6,-0.7) {$8$};
	}\\
	&- \quad \tikzfig{
	\draw (-2,-2) rectangle (2,2);
	\node [vertex] (alphaL) at (-0,1.2) {I};
	\node [vertex] (alphaR) at (1,-1) {II};
	\node [vertex] (midL) at (0,0) {$\beta$};
	\node [bullet] (botC) at (0,-1) {};
	\node [bullet] (botL) at (-1,-1) {};
	\node [vertex] (phi) at (-1,0) {$\varphi$};
	\draw [-w-] (alphaL) to (midL);
	\draw [->-] (midL) to (botC);
	\draw [->-] (alphaL) to (-0.4,2);
	\draw [->-] (-0.4,-2) to (botL);
	\draw [->-] (botL) to (-2,-1);
	\draw [->-] (alphaR) to (botC);
	\draw [->-] (botC) to (botL);
	\draw [-w-] (alphaR) to (2,-1);
	\draw [->-] (phi) to (midL);
	\node at (0.3,0.6) {$2$};
	\node at (-0.4,1.5) {$1$};
	\node at (0.5,-1.3) {$ 3$};
	\node at (1.5,-1.3) {$ 4$};
	\node at (0.3,-0.6) {$ 5$};
	\node at (-0.3,-1.3) {$ 6$};
	\node at (-1.5,-1.3) {$7$};
	\node at (-0.5,-0.3) {$8$};
	} \quad 
	+ \quad \tikzfig{
	\draw (-2,-2) rectangle (2,2);
	\node [vertex] (alphaL) at (-0,1.4) {I};
	\node [vertex] (alphaR) at (1,-1) {II};
	\node [vertex] (midL) at (0,-0.3) {$\beta$};
	\node [bullet] (botC) at (0,-1) {};
	\node [bullet] (botL) at (-1,-1) {};
	\node [vertex] (phi) at (0,0.5) {$\varphi$};
	\draw [-w-] (alphaL) to (phi);
	\draw [->-] (midL) to (botC);
	\draw [->-] (alphaL) to (-0.4,2);
	\draw [->-] (-0.4,-2) to (botL);
	\draw [->-] (botL) to (-2,-1);
	\draw [->-] (alphaR) to (botC);
	\draw [->-] (botC) to (botL);
	\draw [-w-] (alphaR) to (2,-1);
	\draw [->-] (phi) to (midL);
	\node at (0.3,1) {$2$};
	\node at (-0.4,1.5) {$1$};
	\node at (0.5,-1.3) {$ 3$};
	\node at (1.5,-1.3) {$ 4$};
	\node at (0.4,0) {$ 5$};
	\node at (0.3,-0.7) {$ 6$};
	\node at (-0.5,-1.3) {$7$};
	\node at (-1.5,-1.3) {$8$};
	} \quad 
	- \quad \tikzfig{
	\draw (-2,-2) rectangle (2,2);
	\node [vertex] (alphaL) at (-0,1.2) {I};
	\node [vertex] (alphaR) at (1,-1) {II};
	\node [vertex] (midL) at (0,0) {$\beta$};
	\node [bullet] (botC) at (0,-1) {};
	\node [bullet] (botL) at (-1,-1) {};
	\node [vertex] (phi) at (1,0) {$\varphi$};
	\draw [-w-] (alphaL) to (midL);
	\draw [->-] (midL) to (botC);
	\draw [->-] (alphaL) to (-0.4,2);
	\draw [->-] (-0.4,-2) to (botL);
	\draw [->-] (botL) to (-2,-1);
	\draw [->-] (alphaR) to (botC);
	\draw [->-] (botC) to (botL);
	\draw [-w-] (alphaR) to (2,-1);
	\draw [->-] (phi) to (midL);
	\node at (-0.3,0.6) {$2$};
	\node at (-0.4,1.5) {$1$};
	\node at (0.5,-1.3) {$ 3$};
	\node at (1.5,-1.3) {$ 4$};
	\node at (-0.3,-0.6) {$ 5$};
	\node at (-0.3,-1.3) {$ 6$};
	\node at (-1.5,-1.3) {$7$};
	\node at (0.5,-0.3) {$8$};
	}\\
	&- \quad \tikzfig{
	\draw (-2,-2) rectangle (2,2);
	\node [vertex] (alphaL) at (-0,1.2) {I};
	\node [vertex] (alphaR) at (1,-1) {II};
	\node [vertex] (midL) at (0,0) {$\beta$};
	\node [bullet] (botC) at (0,-1) {};
	\node [bullet] (botL) at (-1,-1) {};
	\node [vertex] (phi) at (1,0) {$\varphi$};
	\draw [-w-] (alphaL) to (midL);
	\draw [->-] (midL) to (botC);
	\draw [->-] (alphaL) to (-0.4,2);
	\draw [->-] (-0.4,-2) to (botL);
	\draw [->-] (botL) to (-2,-1);
	\draw [->-] (alphaR) to (botC);
	\draw [->-] (botC) to (botL);
	\draw [-w-] (alphaR) to (2,-1);
	\draw [->-] (phi) to (alphaR);
	\node at (-0.3,0.6) {$2$};
	\node at (-0.4,1.5) {$1$};
	\node at (0.5,-1.3) {$ 3$};
	\node at (1.5,-1.3) {$ 4$};
	\node at (-0.3,-0.6) {$ 5$};
	\node at (-0.3,-1.3) {$ 6$};
	\node at (-1.5,-1.3) {$7$};
	\node at (1.3,-0.5) {$8$};
	} \quad 
	+ \quad \tikzfig{
		\draw (-2,-2) rectangle (2,2);
		\node [vertex] (alphaL) at (-0.4,1.2) {I};
		\node [vertex] (alphaR) at (0.4,-1) {II};
		\node [vertex] (midL) at (-0.4,0) {$\beta$};
		\node [bullet] (botC) at (-0.4,-1) {};
		\node [bullet] (botL) at (-1.4,-1) {};
		\node [vertex] (phi) at (1.3,-1) {$\varphi$};
		\draw [-w-] (alphaL) to (midL);
		\draw [->-] (midL) to (botC);
		\draw [->-] (alphaL) to (-0.8,2);
		\draw [->-] (-0.8,-2) to (botL);
		\draw [->-] (botL) to (-2,-1);
		\draw [->-] (alphaR) to (botC);
		\draw [->-] (botC) to (botL);
		\draw [-w-] (alphaR) to (phi);
		\draw [->-] (phi) to (2,-1);
		\node at (-0.7,0.6) {$2$};
		\node at (-0.8,1.5) {$1$};
		\node at (0,-1.3) {$ 3$};
		\node at (0.8,-1.3) {$ 4$};
		\node at (-0.7,-0.6) {$ 5$};
		\node at (-0.7,-1.3) {$ 6$};
		\node at (-1.7,-1.3) {$7$};
		\node at (1.6,-1.3) {$8$};
	} \quad 
 \end{align*}

\subsubsection{Homotopies for compatibility}\label{app:compatibility}
Here we give the expressions that we omitted from \cref{sec:compatibility}. To simplify these pictures, which involve the pairing between elements in $C_*(A,A^\vee)$ and $C^*(A,A)$, we will make the abbreviation
\[ \tikzfig{
\node [vertex] (ev) at (0,0) {$\ev_i$};
\draw [->-] (1,0) to (ev);
} = \tikzfig{\node (x) at (0,0) {$x_i$};
\node [vertex] (ev) at (1,0) {$\ev$};
\draw [->-,red] (x) to (ev);
\draw [->-] (2,0) to (ev);} \]
for some element $x_i \in C^*(A,A^\vee)$. \footnote{As for the orientation, we will not be switching the order of evaluation of $x_1,x_2$, so we can just fix some ordering of the edges connecting $x_i$ to $\ev$.}

The following combination of diagrams gives a map
\[ \Lambda: C^*_{(2)}(A)^{\otimes 2} \otimes C_*(A,A^\vee)^{\otimes 2} \otimes C^*(A,A) \to \kk \]
which when evaluated on $\alpha\otimes \alpha\otimes x_1\otimes x_2\otimes \varphi$ gives the element $\Lambda_\alpha(x_1,x_2,\varphi)$ of \cref{lem:chainToCochainHomotopy}:
\begin{align*}
	&+ \quad \tikzfig{
		\draw [->-] (0,1.2) arc (90:0:1.2);
		\draw [-w-] (0,-1.2) arc (-90:0:1.2);
		\draw [-w-] (0,1.2) arc (90:180:1.2);
		\draw [->-] (0,-1.2) arc (270:180:1.2);
		\node (v1) at (-2.4,0) {$\ev_1$};
		\node [bullet] (v4) at (1.2,0) {};
		\node [vertex,fill=white] (v3) at (0,1.2) {I};
		\node (v2) at (0,0) {$\ev_2$};
		\node [vertex,fill=white] (v5) at (0,-1.2) {II};
		\node [vertex] (v6) at (0,-2.4) {$\varphi$};
		\node [bullet] (v7) at (-1.2,0) {};
		\draw [->-] (v7) to (v1);
		\draw [->-] (v4) to (v2);
		\draw [->-] (v6) to (v5);
		\node at (-1.8,-0.3) {$1$};
		\node at (0.8, -0.3) {$2$};
		\node at (-1.1,1.1) {$3$};
		\node at (1.1,1.1) {$4$};
		\node at (-1.1,-1.1) {$5$};
		\node at (1.1,-1.1) {$6$};
		\node at (0.3,-1.8) {$7$};
	} \quad + \tikzfig{
		\draw [->-] (0,1.2) arc (90:0:1.2);
		\draw [->-] (0,-1.2) arc (-90:0:1.2);
		\draw [-w-] (0,1.2) arc (90:180:1.2);
		\draw [-w-] (-1.2,0) arc (180:270:1.2);
		\node (v1) at (-2.4,0) {$\ev_1$};
		\node [bullet] (v4) at (1.2,0) {};
		\node [vertex,fill=white] (v3) at (0,1.2) {I};
		\node (v2) at (0,0) {$\ev_2$};
		\node [bullet] (v5) at (0,-1.2) {};
		\node [vertex] (v6) at (0,-2.4) {$\varphi$};
		\node [vertex,fill=white] (v7) at (-1.2,0) {II};
		\draw [->-] (v7) to (v1);
		\draw [->-] (v4) to (v2);
		\draw [->-] (v6) to (v5);
		\node at (-1.8,-0.3) {$1$};
		\node at (0.8, -0.3) {$2$};
		\node at (-1.1,1.1) {$3$};
		\node at (1.1,1.1) {$4$};
		\node at (-1.1,-1.1) {$5$};
		\node at (1.1,-1.1) {$6$};
		\node at (0.3,-1.8) {$7$};
	} \quad + \tikzfig{
		\draw [->-] (0,1.2) arc (90:0:1.2);
		\draw [-w-] (0,1.2) arc (90:180:1.2);
		\draw [->-] (1.2,0) arc (0:-90:1.2);
		\draw [->-] (0,-1.2) arc (-90:-180:1.2);
		\node (v1) at (-2.4,0) {$\ev_1$};
		\node [vertex,fill=white] (v4) at (1.2,0) {II};
		\node [vertex,fill=white] (v3) at (0,1.2) {I};
		\node (v2) at (0,0) {$\ev_2$};
		\node [bullet] (v5) at (0,-1.2) {};
		\node [vertex] (v6) at (0,-2.4) {$\varphi$};
		\node [bullet] (v7) at (-1.2,0) {};
		\draw [->-] (v7) to (v1);
		\draw [-w-] (v4) to (v2);
		\draw [->-] (v6) to (v5);
		\node at (-1.8,-0.3) {$1$};
		\node at (0.8, -0.3) {$2$};
		\node at (-1.1,1.1) {$3$};
		\node at (1.1,1.1) {$4$};
		\node at (-1.1,-1.1) {$5$};
		\node at (1.1,-1.1) {$6$};
		\node at (0.3,-1.8) {$7$};
	} \\
	&+ \tikzfig{
		\draw [->-] (-1.2,0) arc (180:0:1.2);
		\draw [-w-] (0,-1.2) arc (-90:0:1.2);
		\draw [->-] (0,-1.2) arc (270:180:1.2);
		\node (ev1) at (-3.2,0) {$\ev_1$};
		\node [bullet] (right) at (1.2,0) {};
		\node [bullet] (left) at (-1.2,0) {};
		\node [vertex,fill=white] (alphaI) at (-2,0) {I};
		\node (ev2) at (0,0) {$\ev_2$};
		\node [vertex,fill=white] (alphaII) at (0,-1.2) {II};
		\node [vertex] (phi) at (0,-2.4) {$\varphi$};
		\draw [-w-] (alphaI) to (ev1);
		\draw [->-] (phi) to (alphaII);
		\draw [->-] (right) to (ev2);
		\draw [->-] (alphaI) to (left);
		\node at (-2.7,-0.3) {$1$};
		\node at (0.8, -0.3) {$2$};
		\node at (0,1.4) {$3$};
		\node at (-1.1,-1.1) {$4$};
		\node at (1.1,-1.1) {$5$};
		\node at (0.3,-1.8) {$6$};
		\node at (-1.5,-0.3) {$7$};
	} \quad + \quad \tikzfig{
		\draw [->-] (-1.2,0) arc (180:0:1.2);
		\draw [-w-] (0,-1.2) arc (-90:0:1.2);
		\draw [->-] (0,-1.2) arc (270:180:1.2);
		\node (ev1) at (-4.2,0) {$\ev_1$};
		\node [bullet] (right) at (1.2,0) {};
		\node [bullet] (left) at (-1.2,0) {};
		\node [vertex,fill=white] (alphaI) at (-3,0) {I};
		\node (ev2) at (0,0) {$\ev_2$};
		\node [vertex,fill=white] (alphaII) at (0,-1.2) {II};
		\node [vertex] (phi) at (-2,0) {$\varphi$};
		\draw [-w-] (alphaI) to (ev1);
		\draw [->-] (phi) to (left);
		\draw [->-] (right) to (ev2);
		\draw [->-] (alphaI) to (phi);
		\node at (-3.5,-0.3) {$1$};
		\node at (0.8, -0.3) {$2$};
		\node at (-2.4,-0.3) {$3$};
		\node at (0,1.4) {$4$};
		\node at (-1.1,-1.1) {$5$};
		\node at (1.1,-1.1) {$6$};
		\node at (-1.5,-0.3) {$7$};
	}\\
	&+ \quad  \tikzfig{
		\draw [->-] (0,1.2) arc (90:0:1.2);
		\draw [-w-] (0,1.2) arc (90:180:1.2);
		\draw [->-] (0,-1.2) arc (-90:-180:1.2);
		\draw [->-] (1.2,0) arc (0:-90:1.2);
		\node [vertex,fill=white] (alphaI) at (-2,0) {I};
		\node (ev1) at (-3.2,0) {$\ev_1$};
		\node [bullet] (left) at (-1.2,0) {};
		\node [bullet] (bot) at (0,-1.2) {};
		\node [vertex,fill=white] (phi) at (1.2,0) {$\varphi$};
		\node [vertex,fill=white] (alphaII) at (0,1.2) {II};
		\node (ev2) at (0,0) {$\ev_2$};
		\draw [-w-] (alphaI) to (ev1);
		\draw [->-] (alphaI) to (left);
		\draw [->-] (bot) to (ev2);
		\node at (-2.7,-0.3) {$1$};
		\node at (0.3, -0.6) {$2$};
		\node at (-1.5,-0.3) {$3$};
		\node at (1.1,1.1) {$4$};
		\node at (-1.1,1.1) {$7$};
		\node at (-1.1,-1.1) {$5$};
		\node at (1.1,-1.1) {$6$};
	} \quad - \quad  \tikzfig{
		\draw [->-] (0,1.2) arc (90:0:1.2);
		\draw [-w-] (0,1.2) arc (90:180:1.2);
		\draw [->-] (0,-1.2) arc (-90:-180:1.2);
		\draw [->-] (1.2,0) arc (0:-90:1.2);
		\node [vertex,fill=white] (alphaI) at (-2,0) {I};
		\node (ev1) at (-3.2,0) {$\ev_1$};
		\node [bullet] (left) at (-1.2,0) {};
		\node [bullet] (right) at (1.2,0) {};
		\node [vertex,fill=white] (phi) at (0.5,0) {$\varphi$};
		\node [vertex,fill=white] (alphaII) at (0,1.2) {II};
		\node (ev2) at (-0.6,0) {$\ev_2$};
		\draw [-w-] (alphaI) to (ev1);
		\draw [->-] (alphaI) to (left);
		\draw [->-] (right) to (phi);
		\draw [->-] (phi) to (ev2);
		\node at (-2.7,-0.3) {$1$};
		\node at (0, -0.3) {$2$};
		\node at (-1.5,-0.3) {$3$};
		\node at (1.1,1.1) {$4$};
		\node at (-1.1,1.1) {$7$};
		\node at (0.9,-0.3) {$5$};
		\node at (0,-1.4) {$6$};
	} \\
	&- \quad \tikzfig{
		\draw [->-] (-1.2,0) arc (180:0:1.2);
		\draw [-w-] (-1.2,0) arc (-180:0:1.2);
		\node [vertex,fill=white] (alphaI) at (-2.2,0) {I};
		\node (ev1) at (-3.2,0) {$\ev_1$};
		\node [bullet] (right) at (1.2,0) {};
		\node [vertex,fill=white] (phi) at (0.2,-0.8) {$\varphi$};
		\node [vertex,fill=white] (alphaII) at (-1.2,0) {II};
		\node [bullet] (mid) at (0.4,0) {};
		\node (ev2) at (-0.6,0) {$\ev_2$};
		\draw [-w-=0.8] (alphaI) to (ev1);
		\draw [->-] (alphaI) to (alphaII);
		\draw [->-] (right) to (mid);
		\draw [->-] (phi) to (mid);
		\draw [->-] (mid) to (ev2);
		\node at (-2.7,-0.3) {$1$};
		\node at (0, -0.3) {$2$};
		\node at (-1.8,-0.3) {$3$};
		\node at (0.8,-0.3) {$5$};
		\node at (0,1.4) {$4$};
		\node at (0.6,-0.5) {$7$};
		\node at (0,-1.4) {$6$};
	}  +  \tikzfig{
		\draw [->-] (1.2,0) arc (0:180:1.2);
		\draw [->-] (0,-1.2) arc (-90:-180:1.2);
		\draw [-w-] (0,-1.2) arc (-90:0:1.2);
		\node [vertex,fill=white] (phi) at (-2,0) {$\varphi$};
		\node (ev1) at (-3.2,0) {$\ev_1$};
		\node [bullet] (left) at (-1.2,0) {};
		\node [bullet] (right) at (1.2,0) {};
		\node [vertex,fill=white] (alphaI) at (0.4,0) {I};
		\node [vertex,fill=white] (alphaII) at (0,-1.2) {II};
		\node (ev2) at (-0.6,0) {$\ev_2$};
		\draw [->-] (phi) to (ev1);
		\draw [->-] (left) to (phi);
		\draw [-w-] (alphaI) to (right);
		\draw [->-] (alphaI) to (ev2);
		\node at (-2.7,-0.3) {$1$};
		\node at (0, 0.3) {$2$};
		\node at (-1.5,-0.3) {$3$};
		\node at (0.8,0.3) {$7$};
		\node at (0,1.4) {$4$};
		\node at (-1.1,-1.1) {$5$};
		\node at (1.1,-1.1) {$6$};
	} - \tikzfig{
		\draw [-w-] (1.2,0) arc (0:180:1.2);
		\draw [->-] (1.2,0) arc (0:-90:1.2);
		\draw [->-] (0,-1.2) arc (-90:-180:1.2);
		\node [vertex,fill=white] (phi) at (0,-2) {$\varphi$};
		\node (ev1) at (-2.2,0) {$\ev_1$};
		\node [bullet] (left) at (-1.2,0) {};
		\node [bullet] (bottom) at (0,-1.2) {};
		\node [vertex,fill=white] (alphaI) at (0.4,0) {I};
		\node [vertex,fill=white] (alphaII) at (1.2,0) {II};
		\node (ev2) at (-0.6,0) {$\ev_2$};
		\draw [->-] (phi) to (bottom);
		\draw [-w-=.8] (alphaI) to (alphaII);
		\draw [->-] (alphaI) to (ev2);
		\draw [->-] (left) to (ev1);
		\node at (0, 0.3) {$2$};
		\node at (-1.5,-0.3) {$1$};
		\node at (0.8,0.3) {$7$};
		\node at (0,1.4) {$3$};
		\node at (-1.1,-1.1) {$4$};
		\node at (1.1,-1.1) {$5$};
		\node at (0.3,-1.5) {$6$};
	}
\end{align*}

Let us now give the proof of \cref{lem:chainToCochainHomotopy} Since $g_\alpha$ is a quasi-isomorphism, it is enough to prove the converse statement, namely, that given $(h_1,h_2)$ we can find appropriate $(\widetilde{h}_1,\widetilde{h}_2)$. The expression $\langle \pi_{h_1,h_2}(x_1,x_2), \varphi \rangle$ is given by evaluating the following combination of diagrams:
\begin{align*}
	&+ \tikzfig{
		\draw [->-] (0,1.2) arc (90:0:1.2);
		\draw [-w-] (0,-1.2) arc (-90:0:1.2);
		\draw [-w-] (0,1.2) arc (90:180:1.2);
		\draw [->-] (0,-1.2) arc (270:180:1.2);
		\node (v1) at (-2.4,0) {$\ev_1$};
		\node [bullet] (v4) at (1.2,0) {};
		\node [vertex,fill=white] (v3) at (0,1.2) {I};
		\node (v2) at (0,0) {$\ev_2$};
		\node [vertex,fill=white] (v5) at (0,-1.2) {II};
		\node [vertex] (v6) at (0,-2.4) {$\varphi$};
		\node [bullet] (v7) at (-1.2,0) {};
		\draw [->-] (v7) to (v1);
		\draw [->-] (v4) to (v2);
		\draw [->-] (v6) to (v5);
		\node at (-1.8,-0.3) {$1$};
		\node at (0.8, -0.3) {$2$};
		\node at (-1.1,1.1) {$3$};
		\node at (1.1,1.1) {$4$};
		\node at (-1.1,-1.1) {$5$};
		\node at (1.1,-1.1) {$6$};
		\node at (0.3,-1.8) {$7$};
	} \quad + \quad \tikzfig{
		\draw [->-] (0,1.2) arc (90:0:1.2);
		\draw [->-] (0,-1.2) arc (-90:0:1.2);
		\draw [-w-] (0,1.2) arc (90:180:1.2);
		\draw [-w-] (-1.2,0) arc (180:270:1.2);
		\node (v1) at (-2.4,0) {$\ev_1$};
		\node [bullet] (v4) at (1.2,0) {};
		\node [vertex,fill=white] (v3) at (0,1.2) {I};
		\node (v2) at (0,0) {$\ev_2$};
		\node [bullet] (v5) at (0,-1.2) {};
		\node [vertex] (v6) at (0,-2.4) {$\varphi$};
		\node [vertex,fill=white] (v7) at (-1.2,0) {II};
		\draw [->-] (v7) to (v1);
		\draw [->-] (v4) to (v2);
		\draw [->-] (v6) to (v5);
		\node at (-1.8,-0.3) {$1$};
		\node at (0.8, -0.3) {$2$};
		\node at (-1.1,1.1) {$3$};
		\node at (1.1,1.1) {$4$};
		\node at (-1.1,-1.1) {$5$};
		\node at (1.1,-1.1) {$6$};
		\node at (0.3,-1.8) {$7$};
	} \quad + \quad \tikzfig{
		\draw [->-] (0,1.2) arc (90:0:1.2);
		\draw [-w-] (0,1.2) arc (90:180:1.2);
		\draw [->-] (1.2,0) arc (0:-90:1.2);
		\draw [->-] (0,-1.2) arc (-90:-180:1.2);
		\node (v1) at (-2.4,0) {$\ev_1$};
		\node [vertex,fill=white] (v4) at (1.2,0) {II};
		\node [vertex,fill=white] (v3) at (0,1.2) {I};
		\node (v2) at (0,0) {$\ev_2$};
		\node [bullet] (v5) at (0,-1.2) {};
		\node [vertex] (v6) at (0,-2.4) {$\varphi$};
		\node [bullet] (v7) at (-1.2,0) {};
		\draw [->-] (v7) to (v1);
		\draw [-w-] (v4) to (v2);
		\draw [->-] (v6) to (v5);
		\node at (-1.8,-0.3) {$1$};
		\node at (0.8, -0.3) {$2$};
		\node at (-1.1,1.1) {$3$};
		\node at (1.1,1.1) {$4$};
		\node at (-1.1,-1.1) {$5$};
		\node at (1.1,-1.1) {$6$};
		\node at (0.3,-1.8) {$7$};
	} \\
	&- \tikzfig{
		\draw [->-] (0,1.2) arc (90:0:1.2);
		\draw [->-] (-1.2,0) arc (-180:-90:1.2);
		\draw [-w-] (0,1.2) arc (90:180:1.2);
		\draw [->-] (0,-1.2) arc (270:360:1.2);
		\node (h1) at (-2.8,0) {$h_1(x_1)$};
		\node [bullet] (right) at (1.2,0) {};
		\node [vertex,fill=white] (top) at (0,1.2) {$\alpha$};
		\node (ev2) at (0,0) {$\ev_2$};
		\node [bullet] (bottom) at (0,-1.2) {};
		\node [vertex] (phi) at (0,-2.4) {$\varphi$};
		\node [bullet] (left) at (-1.2,0) {};
		\draw [->-] (h1) to (left);
		\draw [->-] (right) to (ev2);
		\draw [->-] (phi) to (bottom);
		\node at (-1.8,-0.3) {$1$};
		\node at (0.8, -0.3) {$2$};
		\node at (-1.1,1.1) {$3$};
		\node at (1.1,1.1) {$4$};
		\node at (-1.1,-1.1) {$5$};
		\node at (1.1,-1.1) {$6$};
		\node at (0.3,-1.8) {$7$};
	} \quad + \quad \tikzfig{
		\draw [->-] (0,1.2) arc (90:0:1.2);
		\draw [->-] (0,-1.2) arc (-90:-180:1.2);
		\draw [-w-] (0,1.2) arc (90:180:1.2);
		\draw [->-] (0,-1.2) arc (-90:0:1.2);
		\node (ev1) at (-2.4,0) {$\ev_1$};
		\node [bullet] (right) at (1.2,0) {};
		\node [vertex,fill=white] (top) at (0,1.2) {$\alpha$};
		\node (h2) at (-0.2,0) {$h_2(x_2)$};
		\node [bullet] (bottom) at (0,-1.2) {};
		\node [vertex] (phi) at (0,-2.4) {$\varphi$};
		\node [bullet] (left) at (-1.2,0) {};
		\draw [->-] (left) to (ev1);
		\draw [->-] (h2) to (right);
		\draw [->-] (phi) to (bottom);
		\node at (-1.8,-0.3) {$1$};
		\node at (0.8, -0.3) {$2$};
		\node at (-1.1,1.1) {$3$};
		\node at (1.1,1.1) {$4$};
		\node at (-1.1,-1.1) {$5$};
		\node at (1.1,-1.1) {$6$};
		\node at (0.3,-1.8) {$7$};
	}
\end{align*}

We replace the term containing $h_1(x_1)$ using the following homotopy:
\begin{align*} 
	&+\tikzfig{
		\draw [-w-] (0,1.2) arc (90:0:1.2);
		\draw [->-] (-1.2,0) arc (-180:0:1.2);
		\draw [->-] (0,1.2) arc (90:180:1.2);
		\node (h1) at (-2.6,0) {$h_1(x_1)$};
		\node [bullet] (right) at (1.2,0) {};
		\node [vertex,fill=white] (top) at (0,1.2) {$\alpha$};
		\node (ev2) at (0,0) {$\ev_2$};
		\node [vertex] (phi) at (0,2.4) {$\varphi$};
		\node [bullet] (left) at (-1.2,0) {};
		\draw [->-] (h1) to (left);
		\draw [->-] (phi) to (top);
		\draw [->-] (right) to (ev2);
		\node at (-1.6,-0.3) {$1$};
		\node at (0.8, -0.3) {$2$};
		\node at (-1.1,1.1) {$3$};
		\node at (1.1,1.1) {$4$};
		\node at (0,-1.4) {$5$};
		\node at (0.3,1.8) {$6$};
	} \quad - \tikzfig{
		\draw [-w-] (-1.2,0) arc (180:0:1.2);
		\draw [->-] (0,-1.2) arc (-90:0:1.2);
		\draw [->-] (0,-1.2) arc (270:180:1.2);
		\node (h1) at (-3.5,0) {$h_1(x_1)$};
		\node [bullet] (right) at (1.2,0) {};
		\node [bullet] (left) at (-1.2,0) {};
		\node (ev2) at (0,0) {$\ev_2$};
		\node [vertex,fill=white] (alphaII) at (0,-1.2) {II};
		\node [vertex] (phi) at (-2,0) {$\varphi$};
		\draw [->-] (phi) to (left);
		\draw [->-] (right) to (ev2);
		\draw [->-] (h1) to (phi);
		\node at (0.8, -0.3) {$2$};
		\node at (-2.4,-0.3) {$1$};
		\node at (0,1.4) {$5$};
		\node at (-1.1,-1.1) {$4$};
		\node at (1.1,-1.1) {$3$};
		\node at (-1.5,-0.3) {$6$};
	} \\
	&- \quad \tikzfig{
		\draw [-w-] (0,1.2) arc (90:0:1.2);
		\draw [->-] (0,1.2) arc (90:180:1.2);
		\draw [->-] (1.2,0) arc (0:-90:1.2);
		\draw [->-] (-1.2,0) arc (180:270:1.2);
		\node (h1) at (-2.6,0) {$h_1(x_1)$};
		\node [bullet] (left) at (-1.2,0) {};
		\node [bullet] (right) at (1.2,0) {};
		\node [vertex,fill=white] (alpha) at (0,1.2) {$\alpha$};
		\node (ev2) at (0,0) {$\ev_2$};
		\node [vertex,fill=white] (phi) at (1.2,0) {$\varphi$};
		\node [bullet] (bottom) at (0,-1.2) {};
		\draw [->-] (h1) to (left);
		\draw [->-] (bottom) to (ev2);
		\node at (-1.6,-0.3) {$1$};
		\node at (-0.3, -0.6) {$2$};
		\node at (-1.1,1.1) {$3$};
		\node at (1.1,1.1) {$4$};
		\node at (1.1,-1.1) {$6$};
		\node at (-1.1,-1.1) {$5$};
	} \quad + \quad \tikzfig{
	\draw [->-] (-1.2,0) arc (180:0:1.2);
	\draw [-w-] (0,-1.2) arc (-90:-180:1.2);
	\draw [->-] (0,-1.2) arc (-90:0:1.2);
	\node (h1) at (-2.5,0) {$h_1(x_1)$};
	\node [bullet] (left) at (-1.2,0) {};
	\node [bullet] (right) at (1.2,0) {};
	\node [vertex,fill=white] (phi) at (0.4,0) {$\varphi$};
	\node [vertex,fill=white] (alpha) at (0,-1.2) {$\alpha$};
	\node (ev2) at (-0.6,0) {$\ev_2$};
	\draw [->-] (h1) to (left);
	\draw [->-] (right) to (phi);
	\draw [->-] (phi) to (ev2);
	\node at (0, 0.3) {$2$};
	\node at (-1.5,-0.3) {$1$};
	\node at (0.8,0.3) {$5$};
	\node at (0,1.4) {$3$};
	\node at (-1.1,-1.1) {$4$};
	\node at (1.1,-1.1) {$6$};
	} \\
	&- \quad \tikzfig{
	\draw [-w-] (-1.2,0) arc (180:0:1.2);
	\draw [->-] (-1.2,0) arc (-180:0:1.2);
	\node (h1) at (-2.5,0) {$h_1(x_1)$};
	\node [bullet] (right) at (1.2,0) {};
	\node [vertex,fill=white] (phi) at (0.2,-0.8) {$\varphi$};
	\node [vertex,fill=white] (alpha) at (-1.2,0) {$\alpha$};
	\node [bullet] (mid) at (0.4,0) {};
	\node (ev2) at (-0.6,0) {$\ev_2$};
	\draw [->-] (h1) to (alpha);
	\draw [->-] (right) to (mid);
	\draw [->-] (phi) to (mid);
	\draw [->-] (mid) to (ev2);
	\node at (-1.5,-0.3) {$1$};
	\node at (0, -0.3) {$2$};
	\node at (0.8,-0.3) {$5$};
	\node at (0,1.4) {$3$};
	\node at (0.6,-0.5) {$6$};
	\node at (0,-1.4) {$4$};
	}
\end{align*}
And the term containing $h_2(x_2)$ using the following homotopy:
\[ - \quad \tikzfig{
\draw [-w-] (1.2,0) arc (0:180:1.2);
\draw [->-] (1.2,0) arc (0:-90:1.2);
\draw [->-] (0,-1.2) arc (-90:-180:1.2);
\node [vertex,fill=white] (phi) at (0,-2) {$\varphi$};
\node (ev1) at (-2.2,0) {$\ev_1$};
\node [bullet] (left) at (-1.2,0) {};
\node [bullet] (bottom) at (0,-1.2) {};
\node (h2) at (-0.1,0) {$h_2(x_2)$};
\node [vertex,fill=white] (alpha) at (1.2,0) {$\alpha$};
\draw [->-] (phi) to (bottom);
\draw [->-=0.8] (h2) to (alpha);
\draw [->-] (left) to (ev1);
\node at (-1.5,-0.3) {$1$};
\node at (0.8,-0.3) {$2$};
\node at (0,1.4) {$3$};
\node at (-1.1,-1.1) {$4$};
\node at (1.1,-1.1) {$5$};
\node at (0.3,-1.5) {$6$};
} \quad + \quad \tikzfig{
\draw [->-] (1.2,0) arc (0:180:1.2);
\draw [->-] (0,-1.2) arc (-90:-180:1.2);
\draw [-w-] (0,-1.2) arc (-90:0:1.2);
\node [vertex,fill=white] (phi) at (-2,0) {$\varphi$};
\node (ev1) at (-3.2,0) {$\ev_1$};
\node [bullet] (left) at (-1.2,0) {};
\node [bullet] (right) at (1.2,0) {};
\node (h2) at (-0.1,0) {$h_2(x_2)$};
\node [vertex,fill=white] (alpha) at (0,-1.2) {$\alpha$};
\draw [->-] (phi) to (ev1);
\draw [->-] (left) to (phi);
\draw [->-] (h2) to (right);
\node at (-2.7,-0.3) {$1$};
\node at (-1.5,-0.3) {$6$};
\node at (0.8,-0.3) {$2$};
\node at (0,1.4) {$3$};
\node at (-1.1,-1.1) {$4$};
\node at (1.1,-1.1) {$5$};
}\]
Counting the remaining terms gives a combination of diagrams, which when written in terms of cup products, gives the desired result.

We now give the last homotopy that was missing, for the proof of \cref{lem:compatibilityDeltaTilde}:
\begin{align*}
	&+ \quad \tikzfig{
		\draw [->-] (-1.2,0) arc (180:0:1.2);
		\draw [-w-] (0,-1.2) arc (-90:0:1.2);
		\draw [->-] (0,-1.2) arc (270:180:1.2);
		\node (ev1) at (-2.4,0) {$\ev_1$};
		\node [bullet] (right) at (1.2,0) {};
		\node [vertex,fill=white] (alphaI) at (-1.2,0) {I};
		\node (ev2) at (0,0) {$\ev_2$};
		\node [vertex,fill=white] (alphaII) at (0,-1.2) {II};
		\node [vertex] (phi) at (0,-2.4) {$\varphi$};
		\draw [-w-] (alphaI) to (ev1);
		\draw [->-] (phi) to (alphaII);
		\draw [->-] (right) to (ev2);
		\node at (-1.8,-0.3) {$1$};
		\node at (0.8, -0.3) {$2$};
		\node at (0,1.4) {$3$};
		\node at (-1.1,-1.1) {$4$};
		\node at (1.1,-1.1) {$5$};
		\node at (0.3,-1.8) {$6$};
	} + \tikzfig{
		\draw [->-] (-1.2,0) arc (180:0:1.2);
		\draw [-w-] (0,-1.2) arc (-90:0:1.2);
		\draw [->-] (0,-1.2) arc (270:180:1.2);
		\node (ev1) at (-2.4,0) {$\ev_1$};
		\node [bullet] (right) at (1.2,0) {};
		\node [vertex,fill=white] (alphaI) at (-1.2,0) {I};
		\node (ev2) at (0,0) {$\ev_2$};
		\node [vertex,fill=white] (alphaII) at (0,-1.2) {II};
		\node [vertex] (phi) at (-1.8,-1.5) {$\varphi$};
		\draw [-w-] (alphaI) to (ev1);
		\draw [->-] (phi) to (alphaI);
		\draw [->-] (right) to (ev2);
		\node at (-1.8,-0.3) {$1$};
		\node at (0.8, -0.3) {$2$};
		\node at (0,1.4) {$3$};
		\node at (-0.7,-0.7) {$4$};
		\node at (1.1,-1.1) {$5$};
		\node at (-1.4,-1) {$6$};
	} + \tikzfig{
		\draw [->-] (-1.2,0) arc (180:0:1.2);
		\draw [-w-] (0,-1.2) arc (-90:0:1.2);
		\draw [->-] (0,-1.2) arc (270:180:1.2);
		\node (ev1) at (-3.4,0) {$\ev_1$};
		\node [bullet] (right) at (1.2,0) {};
		\node [vertex,fill=white] (alphaI) at (-1.2,0) {I};
		\node (ev2) at (0,0) {$\ev_2$};
		\node [vertex,fill=white] (alphaII) at (0,-1.2) {II};
		\node [vertex] (phi) at (-2.2,0) {$\varphi$};
		\draw [-w-] (alphaI) to (phi);
		\draw [->-] (phi) to (ev1);
		\draw [->-] (right) to (ev2);
		\node at (-2.8,-0.3) {$1$};
		\node at (-1.8,-0.3) {$6$};
		\node at (0.8, -0.3) {$2$};
		\node at (0,1.4) {$3$};
		\node at (-1.1,-1.1) {$4$};
		\node at (1.1,-1.1) {$5$};
	} \\
	&- \quad \tikzfig{
		\draw [-w-] (1.2,0) arc (0:90:1.2);
		\draw [-w-] (0,1.2) arc (90:180:1.2);
		\draw [->-] (1.2,0) arc (0:-90:1.2);
		\draw [->-] (0,-1.2) arc (-90:-180:1.2);
		\node (ev1) at (-3.4,0) {$\ev_1$};
		\node [vertex,fill=white] (right) at (1.2,0) {II};
		\node [vertex,fill=white] (alphaI) at (0,1.2) {I};
		\node (ev2) at (0,0) {$\ev_2$};
		\node [vertex] (phi) at (-2.2,0) {$\varphi$};
		\node [bullet] (left) at (-1.2,0) {};
		\draw [->-] (phi) to (ev1);
		\draw [->-] (left) to (phi);
		\draw [->-] (alphaI) to (ev2);
		\node at (-2.8,-0.3) {$1$};
		\node at (0.3, 0.8) {$2$};
		\node at (-1.1,1.1) {$3$};
		\node at (1.1,1.1) {$4$};
		\node at (0,-1.4) {$5$};
		\node at (-1.8,-0.3) {$6$};
	} \quad - \quad \tikzfig{
	\draw [->-] (-1.2,0) arc (180:0:1.2);
	\draw [->-] (0,-1.2) arc (-90:0:1.2);
	\draw [-w-] (-1.2,0) arc (-180:-90:1.2);
	\node (ev1) at (-3.4,0) {$\ev_1$};
	\node [bullet] (right) at (1.2,0) {};
	\node [vertex,fill=white] (alphaII) at (-1.2,0) {II};
	\node (ev2) at (0,0) {$\ev_2$};
	\node [vertex,fill=white] (phi) at (0,1.2) {$\varphi$};
	\node [vertex,fill=white] (alphaI) at (-2.2,0) {I};
	\draw [-w-] (alphaI) to (ev1);
	\draw [->-] (alphaI) to (alphaII);
	\draw [->-] (right) to (ev2);
	\node at (-2.8,-0.3) {$1$};
	\node at (-1.8,-0.3) {$3$};
	\node at (0.8, -0.3) {$2$};
	\node at (-1.1,1.1) {$4$};
	\node at (0,-1.4) {$5$};
	\node at (1.1,1.1) {$6$};
	} \\
	&- \quad \tikzfig{
	\draw [->-] (-1.2,0) arc (180:0:1.2);
	\draw [-w-] (-1.2,0) arc (-180:0:1.2);
	\node [vertex,fill=white] (alphaI) at (-2.2,0) {I};
	\node (ev1) at (-3.2,0) {$\ev_1$};
	\node [bullet] (right) at (1.2,0) {};
	\node [vertex,fill=white] (phi) at (0.4,0) {$\varphi$};
	\node [vertex,fill=white] (alphaII) at (-1.2,0) {II};
	\node (ev2) at (-0.5,0) {$\ev_2$};
	\draw [-w-=0.8] (alphaI) to (ev1);
	\draw [->-] (alphaI) to (alphaII);
	\draw [->-] (right) to (phi);
	\draw [->-] (phi) to (ev2);
	\node at (-2.7,-0.3) {$1$};
	\node at (0, -0.3) {$2$};
	\node at (-1.8,-0.3) {$3$};
	\node at (0.8,-0.3) {$5$};
	\node at (0,1.4) {$4$};
	\node at (0,-1.4) {$6$};
	} \quad + \quad \tikzfig{
	\draw [->-] (-1.2,0) arc (180:0:1.2);
	\draw [-w-] (0,-1.2) arc (-90:0:1.2);
	\draw [->-] (0,-1.2) arc (270:180:1.2);
	\node (ev1) at (-3.2,0) {$\ev_1$};
	\node [bullet] (right) at (1.2,0) {};
	\node [vertex,fill=white] (alphaI) at (-2,0) {I};
	\node (ev2) at (0,0) {$\ev_2$};
	\node [vertex,fill=white] (alphaII) at (0,-1.2) {II};
	\node [vertex,fill=white] (phi) at (-1.2,0) {$\varphi$};
	\draw [-w-] (alphaI) to (ev1);
	\draw [->-] (right) to (ev2);
	\draw [->-] (alphaI) to (phi);
	\node at (-3.5,-0.3) {$1$};
	\node at (0.8, -0.3) {$2$};
	\node at (-2.4,-0.3) {$3$};
	\node at (0,1.4) {$4$};
	\node at (-1.1,-1.1) {$5$};
	\node at (1.1,-1.1) {$6$};
	\node at (-1.5,-0.3) {$7$};
	} \\
	&+ \quad \tikzfig{
	\draw [->-] (-1.2,0) arc (180:0:1.2);
	\draw [->-] (0,-1.2) arc (-90:0:1.2);
	\draw [->-] (-1.2,0) arc (180:270:1.2);
	\node (ev1) at (-2.4,0) {$\ev_1$};
	\node [bullet] (right) at (1.2,0) {};
	\node [vertex,fill=white] (tau) at (-1.2,0) {$\tau$};
	\node (ev2) at (0,0) {$\ev_2$};
	\node [bullet] (bot) at (0,-1.2) {};
	\node [vertex] (phi) at (0,-2.4) {$\varphi$};
	\draw [-w-] (tau) to (ev1);
	\draw [->-] (phi) to (bot);
	\draw [->-] (right) to (ev2);
	\node at (-1.8,-0.3) {$1$};
	\node at (0.8, -0.3) {$2$};
	\node at (0,1.4) {$3$};
	\node at (-1.1,-1.1) {$4$};
	\node at (1.1,-1.1) {$5$};
	\node at (0.3,-1.8) {$6$};
	} \quad - \quad \tikzfig{
	\draw [->-] (1.2,0) arc (0:180:1.2);
	\draw [->-] (0,-1.2) arc (-90:-180:1.2);
	\draw [->-] (1.2,0) arc (0:-90:1.2);
	\node (ev1) at (-2.4,0) {$\ev_1$};
	\node [bullet] (left) at (-1.2,0) {};
	\node [vertex,fill=white] (tau) at (1.2,0) {$\tau$};
	\node (ev2) at (0,0) {$\ev_2$};
	\node [bullet] (bot) at (0,-1.2) {};
	\node [vertex] (phi) at (0,-2.4) {$\varphi$};
	\draw [-w-] (tau) to (ev2);
	\draw [->-] (phi) to (bot);
	\draw [->-] (left) to (ev1);
	\node at (-1.8,-0.3) {$1$};
	\node at (0.8, -0.3) {$2$};
	\node at (0,1.4) {$3$};
	\node at (-1.1,-1.1) {$4$};
	\node at (1.1,-1.1) {$5$};
	\node at (0.3,-1.8) {$6$};
	}
\end{align*}

\printbibliography
    
\Addresses
\end{document}